\documentclass[a4paper, 10pt]{amsart}
\usepackage[top=80pt,bottom=65pt,left=70pt,right=70pt]{geometry}
\usepackage{graphicx, eepic}
\usepackage{times}

\normalfont
\usepackage{bbm}
\usepackage{xcolor}
\usepackage{verbatim}
\usepackage{mathrsfs}
\usepackage{kotex}
\usepackage{amscd}
\usepackage{float}
\usepackage[all]{xy}
\usepackage[subnum]{cases}
\ifpdf
\usepackage{tikz}
\usepackage{textcomp}
\fi
\usetikzlibrary{arrows,matrix,positioning}

\usepackage{amsthm,amsmath,amsfonts,amssymb}
\usepackage[colorlinks=true, linkcolor=blue, citecolor=red, filecolor=violet, urlcolor=violet, backref=page]{hyperref}
\usepackage{multirow, url}
\usepackage{color}
\usepackage[all]{xy}

\theoremstyle{plain}
\newtheorem{theorem}{Theorem}[section]
\newtheorem{thmx}{Theorem}

\newtheorem{proposition}[theorem]{Proposition}
\newtheorem{lemma}[theorem]{Lemma}
\newtheorem{corollary}[theorem]{Corollary}

\theoremstyle{definition}

\newtheorem{definition}[theorem]{Definition}
\newtheorem{example}[theorem]{Example}

\theoremstyle{remark}
\newtheorem{remark}[theorem]{Remark}

\renewcommand{\bar}{\overline}

\renewcommand{\hat}{\widehat}

\newcommand{\dge}{\rotatebox[origin=c]{45}{$\ge$}}
\newcommand{\uge}{\rotatebox[origin=c]{315}{$\ge$}}

\newcommand{\updots}{\hbox to1.65em{\rotatebox[origin=c]{45}{$\cdots$}}}
\newcommand{\dndots}{\hbox to1.65em{\rotatebox[origin=c]{315}{$\cdots$}}}

\newcommand{\lmd}[1]{\hbox to1.65em{$\hfill \lambda_{#1} \hfill$}}

\newcommand{\C}{\mathbb{C}}

\newcommand{\N}{\mathbb{N}}
\newcommand{\Q}{\mathbb{Q}}
\newcommand{\R}{\mathbb{R}}

\newcommand{\Z}{\mathbb{Z}}
\newcommand{\p}{\mathbb{P}}
\newcommand{\F}{\mathcal{F}}
\newcommand{\CP}{\mathbb{C}P}

\newcommand{\uu}{\mathfrak{u}}
\newcommand{\fa}{\mathfrak{a}}
\newcommand{\fb}{\mathfrak{b}}
\newcommand{\fc}{\mathfrak{c}}
\newcommand{\fd}{\mathfrak{d}}

\newcommand{\hookuparrow}{\mathrel{\rotatebox[origin=c]{90}{$\hookrightarrow$}}}

\def\pa{\partial}
\def\mcal{\mathcal}
\def\frak{\mathfrak}
\def\scr{\mathscr}
\def\ev{\textup{ev}}
\def\bev{\textbf{\textup{ev}}}

\numberwithin{equation}{section} 

\begin{document}

\title{Lagrangian fibers of Gelfand-Cetlin systems}

\author{Yunhyung Cho}
\address{Department of Mathematics Education, Sungkyunkwan University, Seoul, Republic of Korea}
\email{yunhyung@skku.edu}

\author{Yoosik Kim}
\address{Department of Mathematics and Statistics, Boston University, Boston, MA, USA}
\email{kimyoosik27@gmail.com, yoosik@bu.edu}

\author{Yong-Geun Oh}
\address{Center for Geometry and Physics, Institute for Basic Science (IBS),  Pohang, Republic of Korea, and Department of Mathematics, POSTECH, Pohang, Republic of Korea}
\email{yongoh1@postech.ac.kr}

\begin{abstract}
Motivated by the study of Nishinou-Nohara-Ueda on the Floer
thoery of Gelfand-Cetlin systems over complex partial flag manifolds,
we provide a complete description of the topology of Gelfand-Cetlin fibers.
We prove that all fibers are \emph{smooth} isotropic submanifolds and
give a complete description of the fiber to be Lagrangian in terms of combinatorics
of Gelfand-Cetlin polytope. Then we study (non-)displaceability of Lagrangian fibers.
After a few combinatorial and numercal tests for the displaceability, using the bulk-deformation
of Floer cohomology by Schubert cycles, we prove that
every full flag manifold $\mcal{F}(n)$ ($n \geq 3$) with a monotone Kirillov-Kostant-Souriau
symplectic form carries a continuum of non-displaceable Lagrangian tori which degenerates
to a non-torus fiber in the Hausdorff limit. In particular, the Lagrangian $S^3$-fiber in
$\mcal{F}(3)$ is non-displaceable the question of which was raised by Nohara-Ueda
who computed its Floer cohomology to be vanishing.
\end{abstract}
\maketitle
\setcounter{tocdepth}{1} 
\tableofcontents

%------------------------------------------------------------------------------------------------------------------------
\section{Introduction}
\label{secIntroduction}

A partial flag manifold $\mathrm{GL}(n, \C)/P$ can be defined as the orbit $\mcal{O}_\lambda$ 
of the diagonal matrix whose diagonal entries are given by a decreasing sequence $\lambda$ of integers 
under the conjugate $U(n)$-action. 
It comes with a $U(n)$-invariant K\"{a}hler form (unique up to scaling), so-called a {\em Kirillov-Kostant-Souriau symplectic form}. 
Guillemin-Sternberg \cite{GS} constructed a completely integrable system on $\mcal{O}_\lambda$ respecting the circle actions arising from the Cartan subgroups of connected closed subgroups that form a nested sequence of $U(n)$. 
They named it a \emph{Gelfand-Cetlin system} (\emph{GC system} for short) as the torus actions was used by Gelfand-Cetlin \cite{GC} to decompose the space of holomorphic sections of the $U(n)$-equivariant line bundle corresponding to $\lambda$, 
which is the irreducible representation of the highest weight $\lambda$ by the Borel-Weil theorem. 
GC systems are of main interest in this article. 

A projective variety admitting a toric degeneration can be studied through the toric variety located at the center.
Partial flag manifolds are known to have toric degenerations by works of Gonciulea-Lakshmibai \cite{GL}, Caldero \cite{Ca} and Kogan-Miller \cite{KM}. 
They have been versatile tools to study algebraic and symplectic geometry of partial flag manifolds for understanding Schubert varieties, studying mirror symmetry, constructing integrable systems, estimating Gromov width, etc by taking advantage of combinatorial obejcts associated with the toric varieties.
At the level of completely integrable systems, Nishinou-Nohara-Ueda \cite{NNU} constructed a toric degeneration of a GC system to a toric moment map. 
They computed (deformed) Floer cohomology of torus fibers. En route, they constructed a Landau-Ginzburg mirror by classifying holomorphic discs of Maslov index two bounded by a torus fiber on a partial flag manifold. It can be thought as Floer theoretical realization of Landau-Ginzburg models appearing in the context of closed mirror symmetry \cite{EHX, Gi, BCKV}.

Just like the toric cases, the image of any GC system is a polytope and the fiber over any point in its interior is a Lagrangian torus. However, contrary to the toric cases, the action of the big torus does not extend to a partial flag variety, which is reflected to an appearance of non-torus Lagrangian fibers. 
For the purpose of homological mirror symmetry, those Lagrangians cannot be ignored since a certain non-toric Lagrangian is indeed a non-zero object over the Novikov ring over the field of complex numbers according to the work of Nohara-Ueda \cite{NU2}. 
This feature contrasts with the toric cases where a certain collection of Lagrangian toric fibers together with deformation data can split-generate the Fukaya category of a toric manifold investigated by the third named author with Abouzaid, Fukaya, Ohta, and Ono \cite{AFOOO}. 
It is also responsible for the fact that the number of critical points of the superpotential is sometimes strictly less than the sum of betti numbers of the ambient variety. 
By the work of Nohara-Ueda \cite{NU1} and Harada-Kaveh \cite{HK}, one can produce completely integrable systems on a partial flag manifold coming from various toric degenerations, but one still misses critical points having some critical values if toric fibers are only considered. 
Therefore, non-toric fibers or their (quasi-)equivalent objects should be taken into consideration to do mirror symmetry, which motivates us to classify non-toric Lagrangian fibers in this paper. 

Besides the aspect of the Fukaya category of partial flag manifolds, the process revealing the topology of fibers is itself interesting as it involves combinatorics of ladder diagrams and tableaux. In representation theory and Schubert calculus, Young diagrams and tableaux are very effective combinatorial tools to study crystal graphs for representations of quantum groups and the product structure on Schubert cocycles for instance, see \cite{HoK, Ful}. 
In this paper, we device a combinatorial procedure on ladder diagrams and tableaux, which are analogues to Young diagrams and tableaux, in order to understand the topology of GC fibers and symplectic geometry of partial flag manifolds. 
The first named author with An and Kim \cite{ACK} found an order-preserving one-to-one correspondence between the faces of a GC polytope and certain subgraphs of the ladder diagram. 
By playing ``\textup{3D-TETRIS}\textregistered \," with certain shapes of blocks on the subgraph corresponding to a face of a GC polytope, one can tell the topology of fibers over the relative interior of the face and whether the fibers are Lagrangian or not. 
Those types of combinatorics on a ladder diagram will be further developed in \cite{CKO} to obtain topological and geometric data on partial flag manifolds. 

In the second part of the paper, based on the above detailed description of Lagrangian fibers, we test displaceability and non-displaceability of Lagrangian fibers. In Nishinou-Nohara-Ueda \cite{NNU}, they showed the Lagrangian torus fiber at the center of a GC polytope is non-displaceable. Other than toric fibers, there are not many things to be known about non-displaceability of non-toric fibers. As far as the authors know, the only known non-displaceable non-toric GC fibers are in some limited cases of Grassmannians such as the monotone $U(2)$-fiber in $\mathrm{Gr}(2,4)$ in Nohara-Ueda \cite{NU2} and the monotone $U(n)$-fiber in $\mathrm{Gr}(n, 2n)$ in Evans-Lekili \cite{EL2}. 
In this article, we deal with the case of complete flag manifolds $\mcal{F}(n)$ for $n \geq 3$ equipped with a monotone Kirillov-Kostant-Souriau symplectic form. We detect several non-displaceable non-toric Lagrangian GC fibers diffeomorphic to $U(m) \times T^s$. For this purpose, we consider line segments connecting the center of the polytope and the centers of certain faces having Lagrangian fibers. Using Lagrangian Floer theory deformed by ambient cycles developed by the third named author with Fukaya, Ohta and Ono \cite{FOOO}, every toric fiber over the line segments is shown to be non-displaceable. More specifically, for deformation we employ Schubert cycles, which degenerate into unions of toric subvarieties corresponding to facets of a moment polytope according to Kogan, Kogan-Miller \cite{K,KM}. Then, non-displaceability follows because non-toric fibers are realized as the Hausdorff limit of non-displaceable toric fibers. In the particular case for $\mcal{F}(3)$, we completely classify the non-displaceable GC fibers, answering questions of Nohara-Ueda \cite{NU2} and Pabiniak \cite{Pa}. 

We hope that our detailed description of the geometry of Lagrangian fibers
will have similar effect on the Floer theory and homological mirror symmetry
to that of Cho and the third named author's article \cite{CO}, especially as the testing
ground of homological mirror symmetry for the non-toric Fano manifolds.
For instance, it is expected to be used to recover Rietsch's Landau-Ginzburg mirror \cite{Ri} that is preferred in the aspect of closed mirror symmetry. It consists of a partial compactification of an algebraic torus together with a holomorphic function on it. Smoothing faces containing a face admitting Lagrangians leads to a Lagrangian torus fibration with singular fibers. By gluing \emph{formal} deformation spaces of immersed Lagrangians and Lagrangian tori in the scheme of Cho-Hong-Lau \cite{CHL}, one can recover the mirror and then the category of matrix factorizations should serve as the mirror of the Fukaya category of partial flag manifolds. It will be invesigated by the second named author with Hong and Lau. 

\subsection*{Acknowledgements}
The first named author is supported by the National Research Foundation of Korea grant funded by the Korea government (MSIP; Ministry of Science, ICT \& Future Planning) (NRF-2017R1C1B5018168). The third named author is supported by IBS-R003-D1. This work was initiated when the first named author and the second named author were affiliated to IBS-CGP and supported by IBS-R003-D1. 
The second named author would like to thank Cheol-Hyun Cho, Hansol Hong, Siu-Cheong Lau for explaining their upcoming work, Yuichi Nohara for explaining work with Ueda, and Byung Hee An, Morimichi Kawasaki and Fumihiko Sanda for helpful comments.

%------------------------------------------------------------------------------------
%------------------------------------------------------------------------------------
\part{The topology of Gelfand-Cetlin fibers}\label{part_thetopologyofgelfand}
\vspace{0.2cm}

%--------------------------------------------------------
\section{Introduction to Part~\ref{part_thetopologyofgelfand}.}

A partial flag manifold arises as the orbit $\mcal{O}_\lambda$ of an element $\lambda$ in the dual Lie algebra of $\frak{u}(n)$ under the co-adjoint action of the unitary group $U(n)$ 
equipped with a $U(n)$-invariant K\"{a}hler form $\omega_\lambda$ (unique up to scaling), 
so-called a {\em Kirillov-Kostant-Souriau symplectic form} (\emph{KKS form} for short). 
On a co-adjoint orbit, Guillemin and Sternberg \cite{GS} built a completely integrable system 
\[
	\Phi_\lambda \colon (\mathcal{O}_\lambda, \omega_\lambda) \rightarrow \R^{\dim_\C \mathcal{O}_\lambda},
\]
which is called a \emph{Gelfand-Cetlin system}, or a \emph{GC system}.  
The image of $\mcal{O}_\lambda$ under $\Phi_\lambda$ is a polytope, denoted by $\Delta_\lambda$. We call it a {\em Gelfand-Cetlin polytope}, abbreviated as a \emph{GC polytope}.

The first part of this article studies the topology of the fibers of a GC system. 
By the celebrated Arnold-Liouville theorem, the fiber over any interior point is a Lagrangian torus. 
Moreover, the fiber over each point not in the interior of any GC polytope turns out to be a smooth isotropic manifold,
see Theorem \ref{theoremA}.
It is notable because there is almost no control on the fiber over a non-regular value of a general completely integrable system. 
Furthermore, a GC fiber not in the interior can be Lagrangian, which shows 
one marked difference between GC systems and toric moment maps (on toric manifolds). 
Table~\ref{featuresofgefiber} summarizes similarities and differences between them. 

\vspace{0.1cm}
\begin{center}
\begin{table}[H]
  \begin{tabular}{| l || c | c | }  \hline
      & \,\, Gelfand-Cetlin fiber \,\, &   Toric moment fiber \\ \hline \hline
    over any point & \multicolumn{2}{|c|}{isotropic submanifold}  \\ \hline 
    over any interior point & \multicolumn{2}{|c|}{Lagrangian torus} \\ \hline
    \multirow{3}{*} {\shortstack[l]{over any point in the relative\\ interior of a $k$-dim face}}
     & \multicolumn{2}{|c|}{$\pi_1 \textup{(fiber)} = \Z^k$, \, $\pi_2 \textup{(fiber)} = 0$}  \\ \cline{2-3}
    			& not necessarily torus & $k$-dimensional torus   \\ \cline{2-2}\cline{2-3}
			& can be Lagrangian & can \emph{not} be Lagrangian \\ \hline
  \end{tabular}
\bigskip
\caption{Features of Gefand-Cetlin fibers and toric fibers}\label{featuresofgefiber}
\end{table}
\end{center}
\vspace{-0.5cm}

To show the above, we will describe each GC fiber 
as the total space of the iterated bundle constructed by playing a game with various ``\textup{LEGO}\textregistered \, blocks".
A game manual will be provided in Section~\ref{secLagrangianFibersOfGelfandCetlinSystems} and~\ref{secIteratedBundleStructuresOnGelfandCetlinFibers}. 
The first main result of the article, obtained from the game, is stated as follows. 

\begin{thmx}[Theorem~\ref{theorem_main}]\label{theoremA}
	Let $\Phi_\lambda$ be the Gelfand-Cetlin system on
	the co-adjoint orbit $(\mathcal{O}_\lambda, \omega_\lambda)$ for $\lambda \in \mathfrak{u}(n)^*$ and let $\Delta_\lambda$ be the corresponding Gelfand-Cetlin
	polytope. For any point $\textbf{\textup{u}} \in \Delta_\lambda$, the fiber $\Phi_\lambda^{-1}(\textbf{\textup{u}})$ is an isotropic submanifold of $(\mathcal{O}_\lambda, \omega_\lambda)$ and is the
	total space of an iterated bundle
	\[
		\Phi^{-1}_\lambda(\textbf{\textup{u}}) = E_{n-1} \stackrel{p_{n-1}} \longrightarrow E_{n-2} \stackrel{p_{n-2}} \longrightarrow \cdots \stackrel{p_2} \longrightarrow E_1
		\stackrel{p_1} \longrightarrow E_0= \mathrm{point}
	\]
	such that the fiber at each stage is either a point or a product of odd dimensional spheres.
	Two fibers $\Phi^{-1}_\lambda(\textbf{\textup{u}}_1)$ and $\Phi^{-1}_\lambda(\textbf{\textup{u}}_2)$ are diffeomorphic if two points $\textbf{\textup{u}}_1$ and $\textbf{\textup{u}}_2$ are contained in the relative interior of the same face.
\end{thmx}

A GC system $\Phi_\lambda$ is known to admit a toric degeneration from $\Phi_\lambda$ to a toric integrable system $\Phi$ on a toric variety. 
On the toric degeneration of algebraic varieties in stages constructed by Kogan and Miller \cite{KM}, Nishinou, Nohara and Ueda constructed the degeneration (see Section \ref{secDegenerationsOfFibersToTori})
of the GC system $\Phi_\lambda$ into the moment map $\Phi$ of the (singular) projective toric variety 
$X_\lambda$ associated with $\Delta_\lambda$ using Ruan's technique \cite{R} of gradient-Hamiltonian flows, see Theorem 1.2 in \cite{NNU}. Then, we obtain the following commutative diagram
\begin{equation}\label{equation_toric_degeneration_diagram}
	\xymatrix{
		  (\mathcal{O}_\lambda, \omega_\lambda)  \ar[dr]_{\Phi_\lambda} \ar[rr]^{ \phi}
                              & & (X_\lambda, \omega)
      \ar[dl]^{\Phi} \\
  & \Delta_\lambda &}
\end{equation}
where $\phi$ is a continuous map from $\mathcal{O}_\lambda$ onto $X_\lambda$. Moreover, the map 
$\phi$ induces a symplectomorphism from $\Phi^{-1}_\lambda(\mathring{\Delta}_\lambda)$ to $\Phi^{-1}(\mathring{\Delta}_\lambda)$ where $\mathring{\Delta}_\lambda$ is the interior of ${\Delta}_\lambda$. As a consequence, one can see that each fiber located at $\mathring{\Delta}_\lambda$ is a Lagrangian torus.

According to the work of Batyrev, Ciocan-Fontanine, Kim, and Van Straten \cite{BCKV}, this degeneration process given by the map $\phi$ in~\eqref{equation_toric_degeneration_diagram} can be interpreted as a smoothing of conifold staratum, which is (a half of) a \emph{conifold} transition. Namely, through the map $\phi$, a partial flag manifold is deformed into a singular toric variety having conifold strata. The following theorem indeed visualizes how each GC fiber degenerates into a toric fiber. Specifically, through $\phi$, every odd-dimensional sphere of dimension $>1$ appeared in each stage of the iterated bundle $\{E_\bullet\}$ simultaneously contracts to a point and each $S^1$-factor persists. 

\begin{thmx}[Theorem ~\ref{theorem_contraction}]\label{theoremB}
	Let $\textbf{\textup{u}}$ be a point lying on the relative interior of an $r$-dimensional face. Then every $S^1$-factor appeared 
	in any stage of the iterated bundle given in Theorem \ref{theoremA} comes out as a trivial factor so that 
	\[
		\Phi_\lambda^{-1}(\textbf{\textup{u}}) \cong T^r \times B(\textbf{\textup{u}})
	\]
	where $B(\textbf{\textup{u}})$ is the iterated bundle obtained from the original bundle by removing all $S^1$-factors. 
	Moreover, the map $\phi \colon \Phi_\lambda^{-1}(\textbf{\textup{u}}) \rightarrow \Phi_{\vphantom{\lambda}}^{-1}(\textbf{\textup{u}})$ is nothing but the projection $T^r \times B(\textbf{\textup{u}}) \rightarrow T^r$ on the first factor.
\end{thmx}

Because of Theorem \ref{theoremA}, other fibers located in the interior of a face $f$ of $\Delta_\lambda$
are all Lagrangian as soon as the fiber of one single point at the interior of $f$ is Lagrangian. In this sense, it is reasonable to call such a face $f$ 
a \emph{Lagrangian face}. Since every fiber is isotropic again by Theorem \ref{theoremA}, in order to show that a face $f$ is Lagrangian, it suffices to check that the dimension of the fiber over 
a point in the interior of $f$ is exactly the half of the real 
dimension of $\mathcal{O}_\lambda$. We will classify all Lagrangian faces of a GC polytope by providing a refined combinatorial procedure testing whether a face is Lagrangian or not, see Corollary~\ref{corollary_L_fillable}. 

The remaing parts of Part I is organized as follows. 
In Section~\ref{secTheGelfandCetlinSystems}, we review a construction and properties of GC systems. 
Section~\ref{secLadderDiagramAndItsFaceStructure} discusses the face structure of a GC polytope in terms of certain graphs of the ladder diagram associated with the polytope. Section~\ref{secLagrangianFibersOfGelfandCetlinSystems} is devoted to introduce combinatorics that will be used to describe the GC fibers and classify all Lagrangian faces.
In Section~\ref{secIteratedBundleStructuresOnGelfandCetlinFibers}, we provide the proof of Theorem~\ref{theoremA} by relating the combinatorics on ladder diagrams to the topology of fibers. 
Finally, the proof of Theorem~\ref{theoremB} will be given in Section~\ref{secDegenerationsOfFibersToTori}.

\vspace{0.2cm}
%------------------------------------------------------------------------------------
\section{Gelfand-Cetlin systems}
\label{secTheGelfandCetlinSystems}

In this section, we briefly overview Gelfand-Cetlin systems on partial flag manifolds.

For a positive integer $r \in \mathbb{N}$, let $n_0, n_1, \cdots, n_r, n_{r+1}$
be a finite sequence of non-negative integers such that
	\begin{equation}\label{nidef}
		0 = n_0 < n_1 < n_2 < \cdots <n_r < n_{r+1} = n.
	\end{equation}
The {\em partial flag manifold} $\F(n_1, \cdots, n_r; n)$ is the space of nested sequences of complex vector subspaces whose dimensions are $n_1, \cdots, n_r$, respectively.
That is,
	\[
		\F(n_1, \cdots, n_r; n) = \{V_{\bullet} := 0 \subset V_1 \subset \cdots \subset V_r \subset \C^n ~|~ \dim_{\C} V_i = n_i \mbox{ for $i= 1, \cdots, r$}\}.
	\]
An element $V_\bullet$ of $\F(n_1, \cdots, n_r; n)$ is called a {\em flag}.

Note that the linear $U(n)$-action on $\C^n$ induces a transitive $U(n)$-action on $\F(n_1, \cdots, n_r; n)$ and each flag 
has an isotropy subgroup isomorphic to $U(k_1) \times \cdots \times U(k_{r+1})$ where
	\begin{equation}\label{kidef}
		k_i = n_i - n_{i-1}
	\end{equation}
for $i=1,\cdots,r+1$. Thus, $\F(n_1, \cdots, n_r; n)$ is a homogeneous space diffeomorphic to $U(n) / (U(k_1) \times \cdots \times U(k_{r+1}))$.
In particular, we have
	\begin{equation}\label{dimofflagma}
		\dim_{\R} \F(n_1, \cdots, n_r; n) = n^2 - \sum_{i=1}^{r+1} k_i^2.
	\end{equation}
For notational simplicity, we denote by $\mcal{F}(n)$ the \emph{complete flag manifold} $\mcal{F}(1, 2, \cdots, n-1; n)$.

\vspace{0.3cm}

%------------------------------------------------------------------------------------
\subsection{Description of $\F(n_1, \cdots, n_r; n)$ as a co-adjoint orbit of $U(n)$}~\label{ssecDescriptionOfFMathfrakNCoadjointOrbitOfUN}
\vspace{0.2cm}

Consider the conjugate action of $U(n)$ on itself. The action fixes the identity matrix $I_n \in U(n)$, which induces the $U(n)$-action,
called the {\em adjoint action} and denoted by $Ad$, on the Lie algebra $\mathfrak{u}(\mathfrak{n}) :=T_{I_n} U(n)$.
Note that $\uu(n)$ is the set of $(n \times n)$ skew-hermitian matrices
	\[
		\uu(n) = \{ A \in M_n(\C) ~|~ A^* = -A \}, \quad A^* = \overline{A}^t
	\]
and the adjoint action can be written as
	\[
		\begin{array}{ccccc}
			Ad \colon & U(n) \times \uu(n)& \rightarrow & \uu(n) \\[0.1em]
	      	   &  (M, A) & \mapsto  &  MAM^*.\\
		\end{array}
	\]
The {\em co-adjoint action} $Ad^*$ of $U(n)$ is the action on the dual Lie algebra $\uu(n)^*$ induced by $Ad$, explicitly given by
	\[
		\begin{array}{ccccc}
			Ad^* \colon  & U(n) \times \uu(n)^*& \rightarrow & \uu(n)^* \\[0.1em]
	         	   &  (M, X) & \mapsto  &  X_M\\
		\end{array}
	\]
where $X_M \in \uu(n)^*$ is defined by $X_M(A) = X(M^*AM)$ for every $A \in \uu(n)$.

\begin{proposition}[p.51 in \cite{Au}]\label{proposition_herm_coadjoint}
	Let $\mathcal{H}_n = i\uu(n) \subset M_n(\C)$ be the set of
	$(n \times n)$ hermitian matrices with the conjugate $U(n)$-action. Then there is a $U(n)$-equivariant $\R$-vector space isomorphism
	$\phi \colon \mathcal{H}_n \rightarrow \uu(n)^*$.
\end{proposition}

Henceforth, we always think of the co-adjoint action of $U(n)$ on $\uu(n)^*$ as the conjugate $U(n)$-action on $\mathcal{H}_n$.
Let $\lambda = \{\lambda_1, \cdots, \lambda_n\}$ be a non-increasing sequence of real numbers such that
	\begin{equation}\label{lambdaidef}
		\lambda_1 = \cdots = \lambda_{n_1} > \lambda_{n_1 + 1} = \cdots = \lambda_{n_2} > \cdots > \lambda_{n_r +1} =
		\cdots = \lambda_{n_{r+1}} (= \lambda_n)
	\end{equation}
and let $I_\lambda = \text{diag}(\lambda_1, \cdots, \lambda_n) \in \mathcal{H}_n$ be the diagonal matrix whose $i$-th diagonal entry is $\lambda_i$ for $i=1,\cdots,n$.
Then the isotropy subgroup of $I_\lambda$ is given by 
	\[
		\begin{pmatrix}
			U(k_1) & 0 & \cdots & 0 \\
			0	& U(k_2) & \cdots & 0 \\
			\vdots & \vdots & \ddots & \vdots \\
			0 & 0 & \cdots &  U(k_{r+1}) 
		\end{pmatrix} \subset U(n).
	\]
If we denote by $\mathcal{O}_\lambda$ the $U(n)$-orbit of $I_\lambda$, then we have 
\[
	\mcal{O}_\lambda \cong U(n) / \left(U(k_1) \times \cdots \times U(k_{r+1}) \right),
\] which is diffeomorphic to $\F(n_1, \cdots, n_r; n)$. We call $\mcal{O}_\lambda$ the \emph{co-adjoint orbit associated with eigenvalue pattern} $\lambda$.

\begin{remark}\label{remark_property_hermitian_matrix}
	Any two similar matrices have the same eigenvalues with the same multiplicities and any hermitian matrix is unitarily diagonalizable.
	Thus the co-adjoint orbit $\mathcal{O}_{\lambda}$ is the set of all hermitian matrices having the eigenvalue pattern $\lambda$ respecting the multiplicities.
\end{remark}

\vspace{0.2cm}
%------------------------------------------------------------------------------------
\subsection{Symplectic structure on a co-adjoint orbit $\mathcal{O}_{\lambda}$}~\label{ssecSymplecticStructureOnMathcalOLambda}
\vspace{0.2cm}

For any compact Lie group $G$ with the Lie algebra $\mathfrak{g}$ and for any dual element $\lambda \in \mathfrak{g}^*$,
there is a {\em canonical} $G$-invariant symplectic form $\omega_{\lambda}$, called  the \emph{Kirillov-Kostant-Souriau symplectic form} (\emph{KKS form} shortly),
on the orbit $\mathcal{O}_{\lambda}$. 
Furthermore, $\mathcal{O}_\lambda$ admits a unique $G$-invariant K\"{a}hler metric compatible with $\omega_\lambda$, and therefore
$(\mathcal{O}_\lambda, \omega_\lambda)$ forms a K\"{a}hler manifold. We refer the reader to \cite[p.150]{Br} for more details.

The KKS form $\omega_\lambda$ can be described more explicitly 
in the case where $G = U(n)$ as below.
For each $h \in \mathcal{H}_n$, we define a real-valued skew-symmetric bilinear form $\widetilde{\omega}_h$ 
on $\uu(n) = i\mathcal{H}_n$ by
	\[
		\widetilde{\omega}_h(X,Y) := \text{tr}(ih[X,Y]) = \text{tr}(iY[X,h]), \quad X, Y \in \uu(n).
	\]
The kernel of $\widetilde{\omega}_h$ is then
	\[
		\text{ker} ~\widetilde{\omega}_h = \{ X \in \uu(n) ~|~ [X,h] = 0 \}.
	\]
Since $\mathcal{O}_\lambda$ is a homogeneous $U(n)$-space, we may express each tangent space $T_h \mathcal{O}_\lambda$ at a point $h \in \mcal{O}_\lambda$ as 
	\[
		T_h \mathcal{O}_\lambda = \{ [X,h] \in T_h \mathcal{H}_n = \mathcal{H}_n~|~ X \in \uu(n) \}.
	\]
Then we get a non-degenerate two form $\omega_\lambda$ on $\mathcal{O}_\lambda$ defined by
	\[
		\left(\omega_{\lambda}\right)_h([X,h], [Y,h]) := \widetilde{\omega}_h(X,Y), \quad h \in \mathcal{O}_\lambda, \quad X,Y \in \uu(n).
	\]
The closedness of $\omega_\lambda$ then follows from the Jacobi identity on $\uu(n)$, see \cite[p.52]{Au} for instance.

\begin{remark}
The diffeomorphism type of $\mathcal{O}_\lambda$ does not depend on
the choice of $\lambda$ but on $k_\bullet$'s. However, the symplectic form $\omega_\lambda$
depends on the choice of $\lambda$. For instance, two co-adjoint orbits $\mathcal{O}_{\lambda}$ and $\mathcal{O}_{\lambda'}$
have $k_1 = k_2 = 1$ when $\lambda = \{1,-1\}$ and $\lambda' = \{1,0\}$ and both orbits
are diffeomorphic to $U(2) / \left( U(1) \times U(1)\right) \cong \p^1$.
However, the symplectic area of $(\mathcal{O}_\lambda, \omega_\lambda)$ and $(\mathcal{O}_{\lambda'}, \omega_{\lambda'})$ are one
and two, respectively.
\end{remark}

Any partial flag manifold is known to be a Fano manifold and hence it admits a monotone K\"{a}hler form.
The following proposition gives a complete description of the monotonicity of $\omega_\lambda$.
\begin{proposition}[p.653-654 in \cite{NNU}]\label{proposition_monotone_lambda}
	The symplectic form $\omega_\lambda$ on $\mathcal{O}_\lambda$ satisfies
	\[
		c_1(T\mathcal{O}_\lambda) = [\omega_\lambda]
	\]
	if and only if
\[
  \lambda = (\underbrace{n-n_1, \cdots}_{k_1} ~,
  \underbrace{n-n_1-n_2, \cdots}_{k_2} ~, \cdots ~, \underbrace{n-n_{r-1}-n_r, \cdots}_{k_r} ~, \underbrace{-n_r, \cdots, -n_r}_{k_{r+1}} )
  +  (\underbrace{m, \cdots, m}_{n}),
  \label{canonical}
\]
for some $m \in \R$.
\end{proposition}

\vspace{0.2cm}
%------------------------------------------------------------------------------------
\subsection{Completely integrable system on $\mathcal{O}_\lambda$}~
\label{ssecCompletelyIntegrableSystemOnMathcalOLambda}
\vspace{0.2cm}

We adorn a co-adjoint orbit $(\mcal{O}_\lambda, \omega_\lambda)$ with a completely integrable system, called the Gelfand-Cetlin system. We recall a standard definition of a completely integrable system.

\begin{definition}\label{definition_CIS_smooth}
A {\em completely integrable system} on a $2n$-dimensional
symplectic manifold $(M,\omega)$ is an $n$-tuple of smooth functions
	\[
		\Phi := (\Phi_1, \cdots, \Phi_n) \colon M \rightarrow \R^n
	\]
such that
\begin{enumerate}
	\item $\{\Phi_i, \Phi_j\} = 0$ for every $1 \leq i, j \leq n$ and
	\item $d\Phi_1, \cdots, d\Phi_n$ are linearly independent on an open dense subset of $M$.
\end{enumerate}
\end{definition}

If $\Phi$ is a proper map, the Arnold-Liouville theorem states that for any regular value $\textbf{\textup{u}} \in \R^n$ of $\Phi$ the preimage $\Phi^{-1}(\textbf{\textup{u}})$ is a Lagrangian torus. However, if $\textbf{\textup{u}}$ is a critical value, the fiber might not be a manifold in general.

Harada and Kaveh \cite{HK} proved that any smooth projective variety equips a completely integrable system whenever it admits a flat toric degeneration. 
(See Section \ref{secDegenerationsOfFibersToTori} for more details.)
However, the terminology ``completely integrable system'' used in \cite{HK} is a weakened version of Definition \ref{definition_CIS_smooth} in the following sense. 

\begin{definition}\label{definition_CIS_continuous}
	A {\em (continuous) completely integrable system} on a $2n$-dimensional symplectic manifold $(M,\omega)$ is a collection of $n$ \emph{continuous} functions
		\[
			\Phi := \{\Phi_1, \cdots, \Phi_n\} \colon M \rightarrow \R^n
		\]
	such that there exists an open dense subset $\mathcal{U}$ of $M$ on which
	\begin{enumerate}
		\item each $\Phi_i$ is smooth,
		\item $\{\Phi_i, \Phi_j\} = 0$ for every $1 \leq i, j \leq n$, and
		\item $ d\Phi_1, \cdots, d\Phi_n $ are linearly independent.
	\end{enumerate}
\end{definition}

For any co-adjoint orbit $(\mathcal{O}_\lambda, \omega_\lambda)$, Guillemin and Sternberg \cite{GS} constructed a completely integrable system (in the sense of Definition \ref{definition_CIS_continuous})
	\[
		\Phi_\lambda \colon \mathcal{O}_\lambda \rightarrow \R^{\dim_{\C} \mathcal{O}_\lambda},
	\]
called {\em the Gelfand-Cetlin system on $\mathcal{O}_\lambda$} (\emph{GC system} for short) with respect to the KKS symplectic form $\omega_\lambda$. 
The GC system on $\mathcal{O}_\lambda$ is in general continuous but not smooth. From now on, a completely integrable system will be meant to be a conitnuous completely integrable system in Definition~\ref{definition_CIS_continuous} unless mentioned.

We briefly recall a construction of the GC system on $\mathcal{O}_\lambda$ as follows. (See also \cite[p.7-9]{NNU}.)
Let $n_\bullet$'s and $k_\bullet$'s be given in~\eqref{nidef} and~\eqref{kidef}, respectively, and 
let $\lambda$ be a non-increasing sequence satisfying \eqref{lambdaidef}.
From~\eqref{dimofflagma}, it follows that
 	\[
 		\dim_{\R} \mathcal{O}_\lambda = 2\dim_{\C} \mathcal{O}_\lambda = n^2 - \sum_{i=1}^{r+1} k_i^2.
 	\]
For any $x \in \mathcal{O}_\lambda \subset \mathcal{H}_n$, let $x^{(k)}$ be the $(k \times k)$ leading principal minor of $x$ for each $k=1,\cdots,n-1$.
Since $x^{(k)}$ is also a hermitian matrix, the eigenvalues are all real. 
Let
\begin{equation}\label{equation_index_global}
	\mcal{I} = \{(i,j) ~|~ i,j \in \N, ~i+j \leq n + 1\}
\end{equation}
be an index set. We then define the real-valued function
	\[
		\Phi_\lambda^{i,j}\colon \mcal{O}_\lambda \to \R, \quad (i,j) \in \mcal{I}
	\]
where $\Phi_\lambda^{i,j}(x)$ is assigned to be the $i$-th largest eigenvalue of $x^{(i+j-1)}$. Note that the eigenvalues of $x^{(k)}$ are arranged by
\[
	\Phi_\lambda^{1,k}(x) \geq \Phi_\lambda^{2,k-1}(x) \geq \cdots \geq \Phi_\lambda^{k,1}(x)
\]
in the descending order. Collecting all $\Phi_\lambda^{i,j}$'s for $(i,j) \in \mcal{I}$, 
we obtain $\Phi_\lambda$.
	
\begin{definition}\label{definition_GC_system}
Let $\lambda$ be given in \eqref{lambdaidef}. 
The {\em Gelfand-Cetlin system $\Phi_\lambda$ associated with $\lambda$} is defined by the collection of real-valued functions
\begin{equation}\label{philambdaij}
	\Phi_\lambda := \left( \Phi_\lambda^{i,j} \right)_{(i,j) \in \mcal{I}} \colon \mathcal{O}_\lambda \rightarrow \R^{\frac{n(n+1)}{2}}.
\end{equation}
\end{definition}

\begin{remark}
We label each component of $\Phi_\lambda$ by multi-index $(i,j) \in \mcal{I}$ such that $\Phi_\lambda^{i,j}$ corresponds to the lattice point $(i,j) \in \Z^2$ 
in a ladder diagram in Definition \ref{definition_ladder_diagram}. (See also Figure \ref{figure_GC_to_ladder}.)
Notice that the labeling of components of $\Phi_\lambda$ used in \cite{NNU} is different from ours. 
\end{remark}

Now, we consider the coordinate system of $\R^{n(n+1)/2}$ 
\begin{equation}\label{equation_coordinate}
\left\{ \left( u_{i,j} \right) \in \R^{n(n+1)/2} ~|~ (i,j) \in \mcal{I} \right\}
\end{equation}
where each component $u_{i,j}$ records the values of $\Phi^{i,j}_\lambda$.
Fix $x \in \mcal{O}_\lambda$ and let $(u_{i,j})_{(i,j) \in \mcal{I}} = \Phi_\lambda(x)$, that is, $u_{i,j} = \Phi_\lambda^{i,j}(x)$ for each $(i,j) \in \mcal{I}$. 
By the min-max principle, the components $u_{i,j}$'s satisfy the following pattern: 

\vspace{0.2cm}
\begin{equation} 
\begin{alignedat}{17}
  \lambda_1 &&&& \lambda_2 &&&& \lambda_3 && \cdots && \lambda_{n-1} &&&& \lambda_n  \\
  & \uge && \dge && \uge && \dge &&&&&& \uge && \dge & \\
  && u_{1,n-1} &&&& u_{2, n-2} &&&&&&&& u_{n-1, 1} && \\
  &&& \uge && \dge &&&&&&&& \dge &&& \\
  &&&& u_{1, n-2} &&&&&&&& u_{n-2, 1} &&&& \\
  &&&&& \uge &&&&&& \dge &&&&& \\
  &&&&&& \dndots &&&& \updots &&&&&& \\
  &&&&&&& \uge && \dge &&&&&&& \\
  &&&&&&&& u_{1,1} &&&&&&&&&
\end{alignedat}
\label{equation_GC-pattern}
\vspace{0.2cm}
\end{equation}

Note that each $\Phi_\lambda^{j, n+1-j}$ assigns the $j$-th largest eigenvalue of $x^{(n)} = x$, and therefore $\Phi_\lambda^{j, n+1-j}(x) = \lambda_j$ 
for every $x \in \mcal{O}_\lambda$ and $j=1,\cdots, n$.
Furthermore, by our assumption~\eqref{lambdaidef}, we have $\lambda_{n_{i-1} + 1} = \cdots =  \lmd {n_i}$ for every $i=1,\cdots, r+1$.
Then it is an immediate consequence of (\ref{equation_GC-pattern}) that for all $j = n_{i-1}+1, \cdots, n_i -1$ and each $k = j, \cdots, n_i - 1$, we have 
$$
\Phi_\lambda^{j, n + 1 -k}(x) = \lambda_{n_i}
$$
for all $x \in \mcal{O}_\lambda$.
Thus, there are exactly $\frac{1}{2}(n^2 - \sum_{i=1}^{r+1} k_i^2)$ non-constant
functions among $\{\Phi^{j,k}_\lambda\}_{(j,k) \in \mcal{I}}$ on $\mathcal{O}_\lambda$.
Let
\begin{equation}\label{collectionofindicesset}
	\mcal{I}_\lambda := \{ (i,j) \in \mcal{I} ~|~ \Phi^{i,j}_\lambda \textup{ is \emph{not} a constant function on $\mcal{O}_\lambda$}\}. 
\end{equation}
Collecting all non-constant components of $\Phi_\lambda$, we rename 
\begin{equation}\label{Gelfand-Cetlinsystemdef}
\Phi_\lambda = \left( \Phi^{i,j}_\lambda \right)_{(i,j) \in \mcal{I}_\lambda} \colon \mathcal{O}_\lambda \to \R^{| \mcal{I}_\lambda |}, \quad 
| \mcal{I}_\lambda | = \dim_{\C} \mathcal{O}_\lambda = \frac{1}{2}(n^2 - \sum_{i=1}^{r+1} k_i^2)
\end{equation}
as the {\em Gelfand-Cetlin system}.
By abuse of notation, the collection is still denoted by $\Phi_\lambda$.
Guillemin and Sternberg \cite{GS} prove that $\Phi_\lambda$ satisfies all properties given in Definition \ref{definition_CIS_continuous},
and hence it is a completely integrable system on $\mathcal{O}_\lambda$ \footnote{In general, the Gelfand-Cetlin system is never smooth on the whole space $\mathcal{O}_\lambda$ unless $\mathcal{O}_\lambda$ is a projective space.}.
We will not distinguish two notations ~\eqref{philambdaij} and~\eqref{Gelfand-Cetlinsystemdef} unless any confusion arises.

\begin{definition}
	The {\em Gelfand-Cetlin polytope}\footnote{It is straightforward to see that $\Delta_\lambda$ is a convex polytope, since $\Delta_\lambda$ is the intersection
	of half-spaces
	defined by inequalities in (\ref{equation_GC-pattern}).} $\Delta_\lambda$ is the collection of points $(u_{i,j})$ satisfying (\ref{equation_GC-pattern}). A Gefland-Cetlin polytope will be abbreviated to a \emph{GC polytope} for simplicity.
\end{definition}

Indeed, we have another description of a GC polytope.

\begin{proposition}[\cite{GS}]
	The Gelfand-Cetlin polytope $\Delta_\lambda$ coincides with the image of $\mcal{O}_\lambda$ under $\Phi_\lambda$. 
\end{proposition}

\vspace{0.2cm}
%------------------------------------------------------------------------------------
\subsection{Smoothness of $\Phi_\lambda$}~
\label{ssecSmoothnessOfPhiLambda}
\vspace{0.2cm}

Let $\lambda$ be given in \eqref{lambdaidef} and let $\Phi_\lambda$ be the GC system 
on $(\mcal{O}_\lambda, \omega_\lambda)$. In general, $\Phi_\lambda$ is \emph{not} smooth on the whole 
$\mcal{O}_\lambda$. However, the following proposition due to Guillemin-Sternberg 
states that $\Phi_\lambda$ is smooth on $\mcal{O}_\lambda$ {\em almost everywhere}.
\begin{proposition}[Proposition 5.3, p.113, and p.122 in \cite{GS}]\label{proposition_GS_smooth}
	For each $(i,j) \in \mcal{I}_\lambda$, the component $\Phi_\lambda^{i,j}$ is smooth at 
	$z \in \mcal{O}_\lambda$ if 
	\begin{equation}\label{equation_smooth_region}
		\Phi_\lambda^{i+1, j}(z) < \Phi_\lambda^{i, j}(z) < \Phi_\lambda^{i, j+1}(z). 
	\end{equation}
	In particular, $\Phi_\lambda$ is smooth on the open dense subset 
	$\Phi_\lambda^{-1}(\mathring{\Delta}_\lambda)$ of $\mathcal{O}_\lambda$ where $\mathring{\Delta}_\lambda$ is the interior of $\Delta_\lambda$.
	Furthermore, $d\Phi_\lambda^{i,j}(z) \neq 0$ for every point $z$ satisfying \eqref{equation_smooth_region}.  
\end{proposition}

One important remark is that a Hamiltonian trajectory of each $\Phi_\lambda^{i,j}$ passing through a point 
$z \in \mcal{O}_\lambda$ satisfying \eqref{equation_smooth_region} is periodic with integer period. 
Therefore, each $\Phi_\lambda^{i,j}$ generates a Hamiltonian circle action on the subset of $\mathcal{O}_\lambda$ 
on which $\Phi_\lambda^{i,j}$ is smooth.
See \cite[Theorem 3.4 and Section 5]{GS} for more details.

\vspace{0.2cm}
%------------------------------------------------------------------------------------
\section{Ladder diagram and its face structure}
\label{secLadderDiagramAndItsFaceStructure}

In order to visualize a GC polytope, it is convenient to employ an alternative description of its face structure 
in terms of certain graphs in the ladder diagram provided by the first named author with An and Kim in \cite{ACK}. 
The goal of this section is to review the description of the face structure.

We begin by the definition of a ladder diagram.
Let $\lambda = \{\lambda_1, \cdots, \lambda_n\}$ be given in \eqref{lambdaidef}. 
Then $\lambda$ uniquely determines $n_\bullet$'s and $k_\bullet$'s in~\eqref{nidef} and~\eqref{kidef}, respectively. 

\begin{definition}[\cite{BCKV}, \cite{NNU}]\label{definition_ladder_diagram}
Let $\Gamma_{\Z^2} \subset \R^2$ be the square grid graph satisfying
\begin{enumerate}
\item its vertex set is $\Z^2 \subset \R^2$ and
\item each vertex $(a,b) \in \Z^2$ connects to exactly four vertices $(a, b \pm 1)$ and $(a \pm 1, b)$.
\end{enumerate}
The {\em ladder diagram} $\Gamma (n_1, \cdots, n_r; n)$ is defined as the induced subgraph of $\Gamma_{\Z^2}$ 
that is formed from the set $V_{\Gamma (n_1, \cdots, n_r; n)}$ of vertices given by
\[
	V_{\Gamma (n_1, \cdots, n_r; n)} := \bigcup_{j=0}^r \,\, \left\{ (a,b) \in \Z^2 ~\big{|}~ \, (a,b) \in [n_j, n_{j+1}] \times [0,n-n_{j+1}] \right\}.
\]
As $\lambda$ determines $n_\bullet$'s, we simply denote $\Gamma(n_1, \cdots, n_r; n)$ by $\Gamma_\lambda$. We call $\Gamma_\lambda$ the \emph{ladder diagram associated with} $\lambda$.
\end{definition}

\begin{figure}[H]
	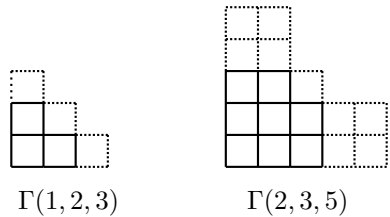
	\bigskip
	\caption{\label{figure_ld123} Ladder diagrams $\Gamma(1,2;3)$ and $\Gamma(2,3;5)$.}	
\end{figure}
\vspace{-0.2cm}

We call the vertex of $\Gamma_\lambda$ located at $(0,0)$ {\em the origin}.
Also, we call $v \in V_\Gamma$ a {\em top vertex} if $v$ is a farthest vertex from the origin 
with respect to the taxicab metric.
Equivalently, a vertex $v = (a,b) \in V_\Gamma$ is a top vertex if $a+b = n$. 

\vspace{0.2cm}
\begin{figure}[H]
	\scalebox{0.9}{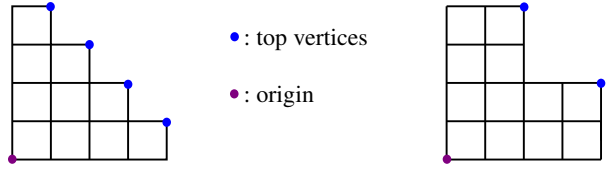}
	\vspace{0.2cm}
	\caption{\label{figure_top_bottom} Top vertices for $\Gamma(1,2,3,4;5)$ and $\Gamma(2,4;6)$.}	
\end{figure}
\vspace{-0.3cm}		

\begin{definition}[Definition 2.2 in \cite{BCKV}]
	A {\em positive path} on a ladder diagram $\Gamma_\lambda$ is a shortest path from the origin to some top vertex in $\Gamma_\lambda$.
\end{definition}

Now, we define the face structure of $\Gamma_\lambda$ as follows.

\begin{definition}[Definition 1.5 in \cite{ACK}]\label{definition_face}
	Let $\Gamma_\lambda$ be a ladder diagram.
	\begin{itemize}
		\item A subgraph $\gamma$ of $\Gamma_\lambda$ is called a {\em face} of $\Gamma_\lambda$ if
			\begin{enumerate}
				\item $\gamma$ contains all top vertices of $\Gamma_\lambda$, 
				\item $\gamma$ can be represented by a union of positive paths.
			\end{enumerate}
		\item For two faces $\gamma$ and $\gamma'$ of $\Gamma_\lambda$, $\gamma$ is said to be a {\em face} of $\gamma'$ if $\gamma \subset \gamma'$.
		\item The \emph{dimension} of a face $\gamma$ is defined by
		\[
			\dim \gamma := \text{rank}_\Z~H_1(\gamma; \Z),
		\]
		regarding $\gamma$ as a $1$-dimensional CW-complex. In other words, $\dim \gamma$ is the number of minimal cycles in $\gamma$.
	\end{itemize}
\end{definition}

It is straightforward from Definition~\ref{definition_face} that for any two faces $\gamma$ and
$\gamma'$ of $\Gamma_\lambda$, $\gamma \cup \gamma'$ is also a face of $\Gamma_\lambda$, which
is the smallest face containing $\gamma$ and $\gamma'$.

We now characterize the face structure of the GC polytope $\Delta_\lambda$
in terms of the face structure of the ladder diagram $\Gamma_\lambda$.
\begin{theorem}[Theorem 1.9 in \cite{ACK}]\label{theorem_equiv_CG_LD}
For a sequence $\lambda$ of real numbers given in~\eqref{lambdaidef},
let $\Gamma_\lambda$ be the ladder diagram and $\Delta_\lambda$ be the Gelfand-Cetlin polytope. 
Then there exists a bijective map
	\[
			\{~\text{faces of}~\Gamma_\lambda \} \stackrel{\Psi} \longrightarrow \{~\text{faces of} ~\Delta_\lambda\}
	\]
such that for faces $\gamma$ and $\gamma'$ of $\Gamma_\lambda$
\begin{itemize}
	\item (Order-preserving) $\gamma \subset \gamma'$ if and only if $\Psi(\gamma) \subset \Psi(\gamma')$,
	\item (Dimension) $\dim \Psi(\gamma) = \dim \gamma$.
\end{itemize}
\end{theorem}

\begin{example}\label{example_123}
	Let $\lambda = \{\lambda_1, \lambda_2, \lambda_3 \}$ with $\lambda_1 > \lambda_2 > \lambda_3$.
	The co-adjoint orbit $(\mcal{O}_\lambda, \omega_\lambda)$ is diffeomorphic to the complete flag manifold $\F(3)$.
	Let $\Gamma_\lambda$ be the ladder diagram associated with $\lambda$ as in Figure
	\ref{figure_ldF123}.
	Here, the blue dots are top vertices and the purple dot is the origin of $\Gamma_\lambda$.
	\begin{figure}[ht]
		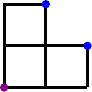
		\caption{\label{figure_ldF123} Ladder diagram $\Gamma_\lambda$.}	
	\end{figure}	
		
	The zero, one, two, and three-dimensional faces of $\Gamma_\lambda$ are respectively listed in Figure \ref{figure_zero_dim_face_F123}, \ref{figure_one_dim_face_F123}, \ref{figure_two_dim_face_F123}, and \ref{figure_three_dim_face_F123}.
	Here,
	$v_i$ denotes a vertex for $i \in \{1, \cdots, 7\}$,
	$e_{ij}$ is the edge containing $v_i$ and $v_j$, $f_I$ is  the facet
	containing all $v_i$'s for $i \in I$, and $I_{1234567}$ is the three dimensional face, i.e., the whole $\Gamma_\lambda$.

	\vspace{0.5cm}
	\begin{figure}[ht]
		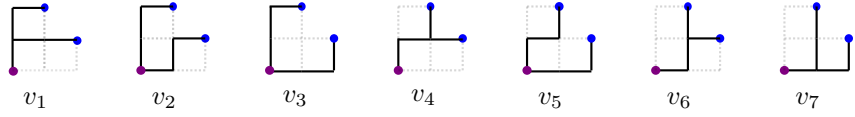
		\caption{\label{figure_zero_dim_face_F123} The zero-dimensional faces of $\Gamma_\lambda$.}	
	\end{figure}		
	
	\begin{figure}[ht]
		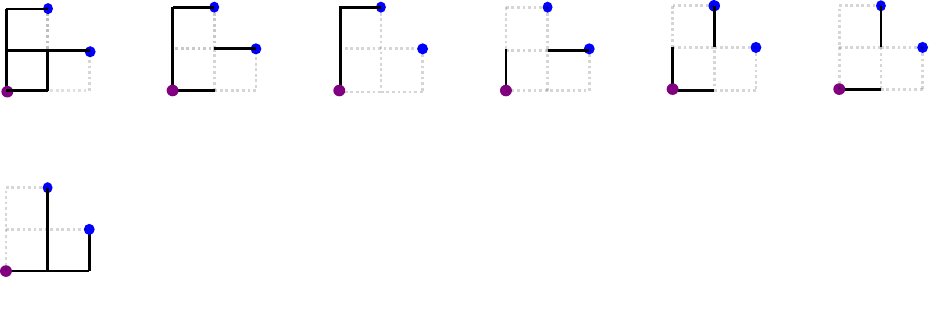
		\caption{\label{figure_one_dim_face_F123} The one-dimensional faces of $\Gamma_\lambda$.}	
	\end{figure}		
	
	\begin{figure}[H]
		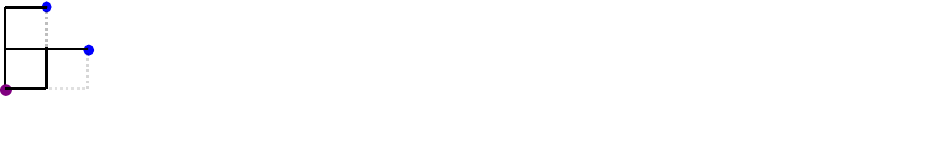
		\bigskip
		\caption{\label{figure_two_dim_face_F123} The two-dimensional faces of $\Gamma_\lambda$.}	
	\end{figure}		

	\vspace{-0.5cm}	
	\begin{figure}[H]
		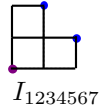
		\bigskip
		\caption{\label{figure_three_dim_face_F123} The three-dimensional face of $\Gamma_\lambda$.}	
	\end{figure}

	The GC polytope $\Delta_\lambda$ is given in Figure \ref{figure_GC_moment_123}.
	(See Figure 5 in \cite{K} or Figure 4 in \cite{NNU}.)
	We can easily see that the correspondence $\Psi(v_i) = w_i$ of vertices naturally
	extends to the set of faces of $\Gamma_\lambda$, satisfying
	(1) and (2) in Theorem \ref{theorem_equiv_CG_LD}.
		
	\begin{figure}[ht]
	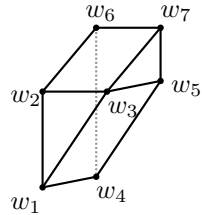
	\caption{\label{figure_GC_moment_123} The GC polytope $\Delta_\lambda$ for $\lambda = \{ \lambda_1 > \lambda_2 > \lambda_3 \}$.}	
	\end{figure}		
\end{example}

For our convenience, we describe each point in $\Delta_\lambda$ by using $\Gamma_\lambda$ \emph{with a filling}, putting each component ${u_{i,j}}$ of the coordinate system 
\eqref{equation_coordinate} of $\R^{|\mcal{I}|} = \R^{\frac{n(n+1)}{2}}$ into the unit box whose top-right vertex is $(i,j)$ of $\Gamma_\lambda$.
The GC pattern~\eqref{equation_GC-pattern} implies that $\{u_{i,j} \}_{(i,j) \in \mcal{I}}$ is
\begin{equation}\label{rem:diagram-pattern}
\begin{cases}
(1) \,\, \text{\emph{increasing} along the columns of $\Gamma_\lambda$ }, and \\
(2) \,\, \text{\emph{decreasing} along the rows of $\Gamma_\lambda$}
\end{cases}
\end{equation}
allowing repetitions (cf. a Young tableau in \cite{Ful}). 

\begin{example}
	Let $\mcal{O}_\lambda \simeq \mathcal{F}(3)$ be the co-adjoint orbit from Example~\ref{example_123}. 
	Recall that the pattern~\eqref{equation_GC-pattern} consists of the following inequalities:
	$$
		u_{1,2} \geq u_{1,1},\,\,\, u_{1,1} \geq u_{2,1},\,\,\, \lambda_1 \geq u_{1,2},\,\,\, u_{1,2} \geq \lambda_2,\,\,\, \lambda_2 \geq u_{2,1},\,\,\, u_{2,1} \geq \lambda_3.
	$$
	The ladder diagram $\Gamma_\lambda$ with a filling by variables $u_{i,j}$'s is as in Figure~\ref{figure_GC_to_ladder}.
	\begin{figure}[H]
	\scalebox{0.9}{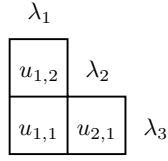}
	\caption{\label{figure_GC_to_ladder} The ladder diagram $\Gamma_\lambda$ with $u_{i,j}$'s variables.}	
\end{figure}	
\end{example}

Now, we explain how the map $\Psi$ in Theorem \ref{theorem_equiv_CG_LD}
is in general defined. For a given  face $\gamma$ of $\Gamma_\lambda$, consider $\gamma$ with a filling by the coordinate system $\{u_{i,j}\}$. 
The image of $\gamma$ under $\Psi$ is the intersection of facets supported by the hyperplanes that are given by equating two adjacent variables $u_{i,j}$'s \emph{not} divided by any positive paths. 

\begin{example}
Suppose that $\lambda = \{4,4,3,2,1\}$ and let $\gamma$ be a face given as in Figure \ref{figure_one_to_one_face}.
Then, the corresponding face $\Psi(\gamma)$ in $\Delta_\lambda$ is defined by
\[
	\Psi(\gamma) =  \Delta_\lambda \cap  \{ u_{2,1} = u_{1,1} \} \cap  \{ u_{1,1} = u_{1,2} \} \cap  \{ u_{2,2} = u_{1,2} \} \cap  \{ u_{2,1} = u_{2,2} \} \cap  \{ u_{2,2} = u_{2,3} \} \cap  \{ u_{3,1} = 4 \} .
\]
\begin{figure}[H]
	\scalebox{0.9}{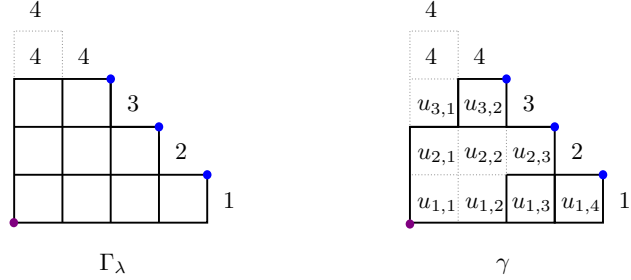}
	\caption{\label{figure_one_to_one_face} The bijection $\Psi$}
\end{figure}
\end{example}

%----------------------------------------------------------------------------------------------------
\section{Classification of Lagrangian fibers}
\label{secLagrangianFibersOfGelfandCetlinSystems}

Our first main theorem \ref{theoremA}, which will be proven in Section \ref{secIteratedBundleStructuresOnGelfandCetlinFibers},
states that every fiber of the GC system $\Phi_\lambda$ on a co-adjoint orbit $(\mathcal{O}_\lambda, \omega_\lambda)$
is an isotropic submanifold and is the total space of certain iterated bundle
\begin{equation}\label{cha4iteratedbund}
		E_{n-1} \stackrel {p_{n-1}} \longrightarrow E_{n-2} \stackrel{p_{n-2}} \longrightarrow \cdots \stackrel{p_2} \longrightarrow E_1
		\stackrel{p_1} \longrightarrow E_0= \mathrm{point}
\end{equation}
such that the fiber at each stage is either a point or a product of odd dimensional spheres.
In this section, we provide a combinatorial way of ``reading off'' the topology of the fiber of each projection map $p_i$ from the ladder diagram (Theorem \ref{theorem_main}).
Furthermore, we classify all positions of Lagrangian fibers in the Gelfand-Cetlin polytope
(Corollary~\ref{corollary_L_fillable}).

We first consider a $2n$-dimensional compact symplectic toric manifold $(M,\omega)$ with a moment map $\Phi \colon M \to \R^n.$
It is a smooth completely integrable system on $M$ in the sense of Definition \ref{definition_CIS_smooth} and the 
Atiyah-Guillemin-Sternberg convexity theorem \cite{At, GS2} yields that the image $\Delta := \Phi(M)$ is an $n$-dimensional convex polytope.
It is well-known that for any $k$-dimensional face $f$ of $\Delta$ and a point $\textbf{\textup{u}} \in \mathring{f}$ in its relative interior $\mathring{f}$ of $f$.
the fiber $\Phi^{-1}(\textbf{\textup{u}})$ is a $k$-dimensional isotropic torus. 
In particular, a fiber $\Phi^{-1}(\textbf{\textup{u}})$ is Lagrangian if and only if $\textbf{\textup{u}} \in \mathring{\Delta}$. .

In contrast, in the GC system case the preimage of a point in the inteior of a $k$-th dimensional face of the GC polytope $\Delta_\lambda$ might 
have the dimension greater than $k$ as the torus action does not extend to the whole space.
In particular, $\Phi_\lambda$ might admit a Lagrangian fiber over a point not contained in the interior of $\Delta_\lambda$.

\begin{definition}\label{definition_lagrangian_face}
	We call a face $f$ of $\Delta_\lambda$ {\em Lagrangian} if it contains a point $\textbf{\textup{u}}$ in its relative interior $\mathring{f}$
	such that the fiber $\Phi_\lambda^{-1}(\textbf{\textup{u}})$ is Lagrangian. Also, we call a face $\gamma$ of $\Gamma_\lambda$ \emph{Lagrangian} if the corresponding face
	$f_\gamma := \Psi(\gamma)$ of $\Delta_\lambda$ is Lagrangian where $\Psi$ is given in Theorem~\ref{theorem_equiv_CG_LD}.
\end{definition}

\begin{remark}
We will see later that if $\textbf{\textup{u}}$ and $\textbf{\textup{u}}'$ are contained in the relative interior of a same face of $\Delta_\lambda$, then 
$\Phi_\lambda^{-1}(\textbf{\textup{u}})$ and $\Phi_\lambda^{-1}(\textbf{\textup{u}}')$ are diffeomorphic.
In particular if $f$ is a Lagrangian face of $\Delta_\lambda$, then every fiber over any point in $\mathring{f}$ is Lagrangian, see Corollary~\ref{corollary_L_fillable}.
\end{remark}

\begin{example}[\cite{K,NNU}]\label{example_NNU_S3}
	For a complete flag manifold $\F(3) \simeq \mcal{O}_\lambda$ of complex dimension three from Example~\ref{example_123}, the fiber $\Phi_\lambda^{-1}(w_3)$ is 
	a Lagrangian submanifold diffeomorphic to $S^3$ where the vertex $w_3$ is in Figure~\ref{figure_GC_moment_123}.
	See Example 3.8 in \cite{NNU}.
\end{example}

From now on,  we tacitly identify faces of the ladder diagram $\Gamma_\lambda$ with faces of the Gelfand-Cetlin polytope $\Delta_\lambda$ via the map $\Psi$ in Theorem \ref{theorem_equiv_CG_LD}.
For example, ``a point $r$ in a face $\gamma$'' of $\Gamma_\lambda$ means a point $r$ in the face $\Psi(\gamma) = f_\gamma$ of $\Delta_\lambda$.

\vspace{0.2cm}

%------------------------------------------------------------------------------------
\subsection{$W$-shaped blocks and $M$-shaped blocks}~
\label{ssecWShapedBlocks}
	
\vspace{0.2cm}

For each $(a,b) \in \Z^2$, let $\square^{(a,b)}$ be the simple closed region bounded by the unit square in $\R^2$
whose vertices are lattice points in $\Z^2$ and the top-right vertex is $(a,b)$. The region $\square^{(a,b)}$ is simply said to be the \emph{box} at $(a, b)$.
\begin{definition}\label{definition_W_shaped_block}
	For each integer $k \geq 1$, a {\em $k$-th $W$-shaped block} denoted by $W_k$, or simply {\em a $W_k$-block}, is defined by
	\[
		W_k := \bigcup_{(a,b)} \square^{(a,b)}
	\]
	where the union is taken over all $(a,b)$'s in $(\Z_{\geq 1})^2$ such that $k+1 \leq a+b \leq k+2$.
	A lattice point closest from the origin (with respect to the taxicab metric) in the $W_k$-block is called a \emph{bottom vertex}.
\end{definition}

	The following figures illustrate $W_1$, $W_2$, and $W_3$ where the red dots in each figure
	indicate the vertices over which the union is taken in Definition \ref{definition_W_shaped_block}. The purple dots are bottom vertices of each $W$-blocks.
	\begin{figure}[H]
	\scalebox{0.9}{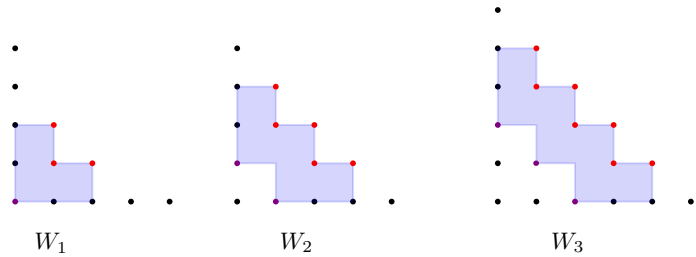}
	\caption{\label{figure_W_k_shaped_block} $W$-shaped blocks}
	\end{figure}	
\vspace{-0.5cm}

For a given diagram $\Gamma_\lambda$, the set of edges of each face $\gamma$ divides $W_k$
into several pieces of simple closed regions.
\begin{definition}\label{defn:W-gamma} For the $W_k$-block and a face $\gamma$ of $\Gamma_\lambda$, we denote by $W_k(\gamma)$ the $W_k$-block with `walls', 
where a wall is an edge of $\gamma$ lying on the interior of the $W_k$-block.
\end{definition}

\begin{example}\label{example_dividing_block_123}
	In Example \ref{example_123}, we consider the vertex $v_3$ of $\Gamma_\lambda$ 
	in Figure~\ref{figure_zero_dim_face_F123}. There are no edges of $v_3$ inside $W_1$ and hence
	$W_1(v_3) = W_1$ with no walls. The $W_2$-block is divided by $v_3$ into three pieces of simple
	closed regions so that $W_2(v_3)$ is $W_2$ with two walls indicated with red line segments in Figure~\ref{figure_W_k_shaped_block_v3}.
\vspace{0.2cm}
	\begin{figure}[H]
	\scalebox{0.9}{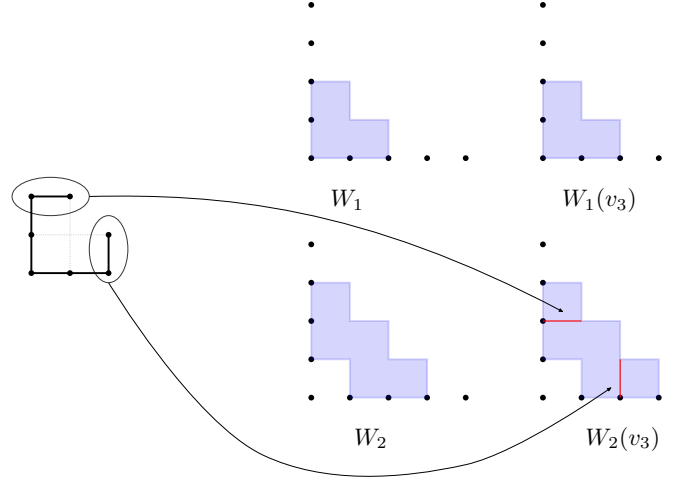}
	\caption{\label{figure_W_k_shaped_block_v3}  $W_i(v_3)$-blocks}
	\end{figure}
\vspace{-0.4cm}			
\end{example}

Next, we introduce the notion of \emph{$M$-shaped blocks}.

\begin{definition}\label{definition_M_shaped_block}
	For each positive integer $k \geq 1$, a {\em $k$-th $M$-shaped block} denoted by $M_k$, or simply an {\em $M_k$-block}, is defined, \emph{up to translation in} $\R^2$,
	by
	\[
		M_k := \bigcup_{(a,b)} \square^{(a,b)}
	\]
	where the union is taken over all $(a,b)$'s in $(\Z_{\geq 1})^2$ such that
	\begin{itemize}
		\item $k+1 \leq a+b \leq k+2$,
		\item $(a,b) \neq (k+1,1)$, and
		\item $(a,b) \neq (1,k+1)$.
	\end{itemize}
\begin{figure}[H]
	\scalebox{0.95}{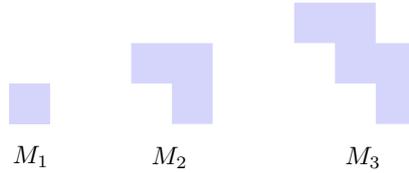}
		\vspace{-0.3cm}			
	\caption{\label{figure_block_sphere} $M$-shaped blocks.}
\end{figure}	
\vspace{-0.5cm}			
\end{definition}

\begin{remark}
	Note that $M_k$ can be obtained from $W_k$ by deleting two boxes $\square^{(k+1,1)}$ and $\square^{(1,k+1)}$.
	The reader should keep in mind that each $W$-shaped block $W_k$ is located at the specific position, but $M_k$ is not
	since it is defined up to translation in $\R^2$.
\end{remark}

For each divided simple closed region $\mcal{D}$ in $W_k(\gamma)$, we assign a topological space to $S_k (\mcal{D})$   by the following rule :
\begin{equation*}
S_k(\mathcal{D}) =
\begin{cases}
S^{2\ell-1} \quad \mbox{if $\mathcal{D}$ is the $M_\ell$-block and $\mcal{D}$ contains at least one bottom vertex of the $W_k$-block,} \\
 \mathrm{point} \quad \mbox{otherwise.}
\end{cases}
\end{equation*}
We then put
\begin{equation}\label{equation_block_sphere}
	S_k(\gamma) := \prod_{\mathcal{D} \subset W_k(\gamma)} S_k(\mathcal{D})
\end{equation}
where the product is taken over all simple closed regions in $W_k(\gamma)$ distinguished by walls coming from edges of $\gamma$.

\begin{example}\label{example_computation_S_2_v3}
	Again, consider the vertex $v_3$ of $\Gamma_\lambda$ in Example \ref{example_dividing_block_123}.
	Note that $S_1(v_3) = \mathrm{pt}$ since $W_1(v_3)$ consists of one simple closed region $W_1$ which is not an $M$-shaped block.
	In contrast, $W_2(v_3)$ consists of three simple closed regions $\mathcal{D}_1$, $\mathcal{D}_2$, and $\mathcal{D}_3$
	as in the figure below. (See also Figure \ref{figure_W_k_shaped_block_v3}.)
	Even if $\mcal{D}_1$ and $\mcal{D}_3$ are $M_1$-blocks, they do not contain any bottom vertices so that 
	$S_2(\mcal{D}_1) = \mathrm{point}$ and $S_2(\mcal{D}_3) = \mathrm{point}$.
	Observe that $\mathcal{D}_2$ is an $M_2$-block containing a bottom vertex of $W_2$. 	
	Therefore, we have
	\[
		S_2(v_3) = S_2(\mathcal{D}_1) \times S_2(\mathcal{D}_2) \times S_2(\mathcal{D}_3) \cong S^3.
	\]
	For $k>2$, $W_k(v_3)$ has no walls and hence $W_k(v_3)$ consists of one simple closed region $W_k$ which is never an
	$M$-shaped block. Thus $S_k(v_3) = \mathrm{point}$ for every positive integer $k>2$.

	\begin{figure}[H]
		\scalebox{0.7}{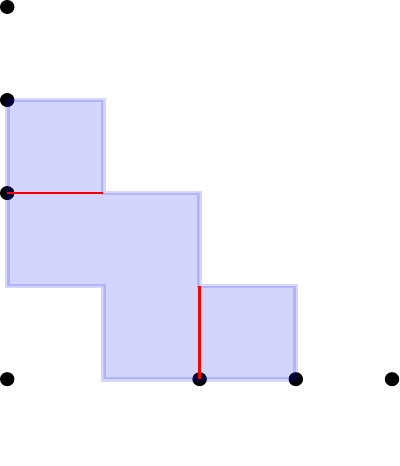}	
	\end{figure}
\end{example}

\begin{example}\label{example_gr24_W_block}
	Let $\lambda = \{t,t,0,0\}$ with $t > 0$. Then, the co-adjoint orbit $\mathcal{O}_\lambda$ is a complex Grassmannian $\mathrm{Gr}(2,4)$.
	Let $\gamma$ be the one-dimensional face of $\Gamma_\lambda$ given by
	\begin{figure}[H]
	\scalebox{0.9}{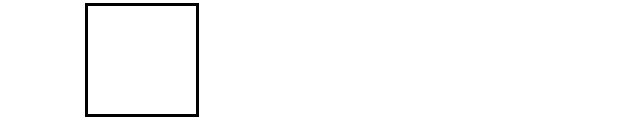}
	\end{figure}			
	As we see in the following figure, $W_1(\gamma)$ consists of one simple closed region that is not an
	$M$-shaped block. Thus we have $S_1(\gamma) = \mathrm{point}$.
	
	\begin{figure}[H]
		\scalebox{0.9}{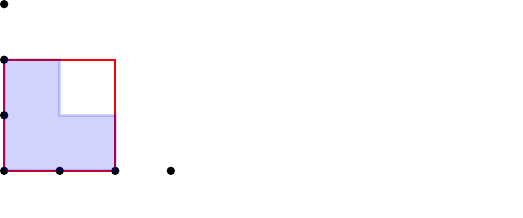}
	\end{figure}
	\vspace{-0.1cm}
	
	On the other hand, $W_2(\gamma)$ has two walls (red line segments in the figure below) and is exactly the same as $W_2(v_3)$ in Example 		
	\ref{example_computation_S_2_v3}. Thus we have $S_2(\gamma) = S^3$.
	
	\begin{figure}[H]
		\scalebox{0.9}{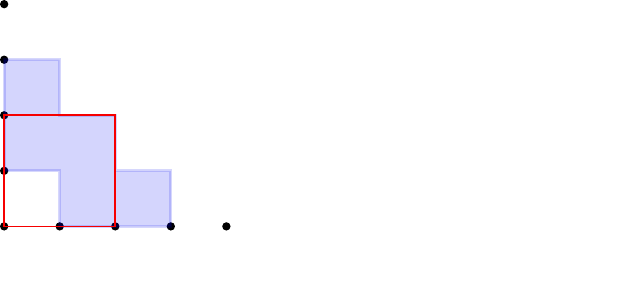}
	\end{figure}
	\vspace{-0.1cm}
	
	Finally, $W_3(\gamma)$ has two walls as we see below, and therefore there are three simple closed regions, 
	$\mathcal{D}_1$, $\mathcal{D}_2$, and $\mathcal{D}_3$, in $W_3(\gamma)$.
	
	\begin{figure}[H]
		\scalebox{0.9}{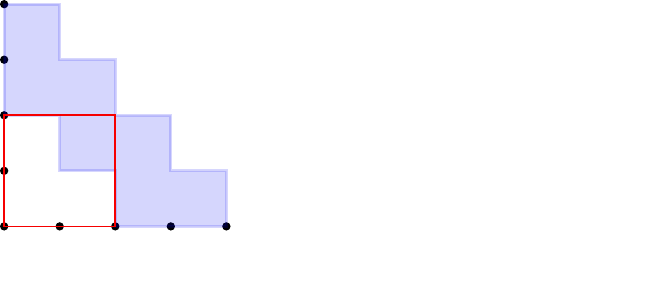}
	\end{figure}
	\vspace{-0.1cm}
	
	Since $\mathcal{D}_1$ and $\mathcal{D}_3$ are not $M$-shaped blocks, we have $S_3(\mathcal{D}_1) = S_3(\mathcal{D}_3) = \mathrm{point}$.
	Since $\mathcal{D}_2 = M_1$ and contains a botton vertex of $W_3$, we have $S_3(\mathcal{D}_2) = S^1$.
	Therefore, we have 
	\[
		S_3(\gamma) = S_3(\mathcal{D}_1) \times S_3(\mathcal{D}_2) \times S_3(\mathcal{D}_3) \cong S^1.
	\]
	For $k>3$, $W_k(\gamma)$ consists only one simple closed region, the $W_k$-block itself, and it is never an $M$-shaped block.
	Thus $S_k(\gamma) = \mathrm{point}$ for $k>3$.

\end{example}

\begin{proposition}\label{proposition_M_1_block_simple_closed_region}
	Let $\lambda$ be a sequence of real numbers satisfying~\eqref{lambdaidef} and $\gamma$ be a face of the ladder diagram $\Gamma_\lambda$.
	For each $i \geq 1$, let $\{\mathcal{D}_1, \cdots, \mathcal{D}_{m_i} \}$ be simple closed regions in $W_i(\gamma)$ such that
	$
		S_i(\mathcal{D}_j) = S^1
	$ for every $j=1,\cdots, m_i$.
	Then we have
	\[
		\dim \gamma = \sum_{i=1}^{n-1} m_i.
	\]
\end{proposition}

\begin{proof}
	Note that $\dim \gamma$ is the number of minimal cycles in $\gamma$ by Definition \ref{definition_face}. Also, each minimal cycle $\sigma$ in $\gamma$
	can be represented by the union of two shortest paths connecting the bottom-left vertex and the top-right vertex of $\sigma$. We denote by
	$v_\sigma$ the top-right vertex of $\sigma$. 
	Note that $\square^{v_\sigma}$
	(blue-colored region in Figure \ref{figure_M_1_block_simple_closed_region}) is contained in the simple closed region bounded by $\sigma$.
	Therefore, if we denote by $\Sigma := \{\sigma_1, \cdots, \sigma_m\}$ the set of minimal cycles in $\gamma$, then
	there is a one-to-one correspondence
	between $\Sigma$ and $\{v_{\sigma_i} \}_{1 \leq i \leq m}$.
	
	On the other hand, observe that $\square^{v_\sigma}$ is appeared as an $M_1$-block in $W_i(\gamma)$ where
	\[
		i+1 = a+b, \quad v_\sigma = (a,b)
	\]	
	for each $\sigma \in \Sigma$.
	Also, every $M_1$-block appeared in $W_i(\gamma)$ for some $i$ should be one of such $\square^{v_\sigma}$'s.
	Consequently, there is a one-to-one correspondence between $\Sigma$ and $M_1$-blocks appeared in $W_i(\gamma)$ for $i \geq 1$.
	Since $|\Sigma| = \dim \gamma$ by definition, this completes the proof.
\end{proof}
	\vspace{-0.5cm}
	\begin{figure}[H]
		\scalebox{0.9}{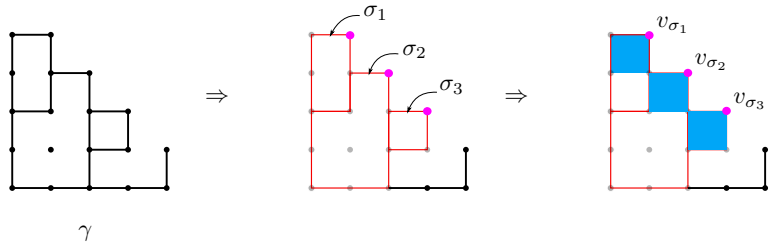}
		\caption{\label{figure_M_1_block_simple_closed_region} Correspondence between minimal cycles and $M_1$-blocks}
	\end{figure}
	\vspace{-0.3cm}

Now, we state one of our main theorem which characterizes the topology of each fiber of $\Phi_\lambda$, where the proof will be given in Section \ref{section4}.

\begin{theorem}\label{theorem_main}
	Let $\lambda = \{ \lambda_1, \cdots, \lambda_n \}$ be a non-increasing sequence of real numbers satisfying~\eqref{lambdaidef}.
	Let $\gamma$ be a face of $\Gamma_\lambda$ and $f_\gamma = \Psi(\gamma)$ be the corresponding face of $\Delta_\lambda$
	described in Theorem \ref{theorem_equiv_CG_LD}.
	For any point $\textbf{\textup{u}}$ in the relative interior of $f_\gamma$, the fiber $\Phi_\lambda^{-1}(\textbf{\textup{u}})$ is an isotropic submanifold
	of $(\mathcal{O}_\lambda, \omega_\lambda)$ and is the total space of an iterated bundle
	\begin{equation}\label{theoremmaindia}
		\Phi_\lambda^{-1}(\textbf{\textup{u}}) \cong \bar{S_{n-1}}(\gamma) \xrightarrow {p_{n-1}} \bar{S_{n-2}}(\gamma)
		 \rightarrow \cdots
		 \xrightarrow{p_2} \bar{S_1}(\gamma) \xrightarrow{p_1} \bar{S_0}(\gamma) := \mathrm{point}
	\end{equation}
	where
	$p_k \colon \bar{S_k}(\gamma) \rightarrow \bar{S_{k-1}}(\gamma)$ is an $S_k(\gamma)$-bundle over $\bar{S_{k-1}}(\gamma)$ for $k=1,\cdots, n-1$.
	In particular, the dimension of $\Phi_\lambda^{-1}(\textbf{\textup{u}})$ is given by 
	\[
		\dim \Phi_\lambda^{-1}(\textbf{\textup{u}}) = \sum_{k=1}^{n-1} \dim S_k(\gamma).
	\]
\end{theorem}

We illustrate Theorem \ref{theorem_main} by the following examples.

\begin{example}\label{example_lag_fiber_123_vertex}
	Let $\lambda$ be given in Example \ref{example_123}. 
	Using Theorem \ref{theorem_main}, we will compute the fiber $\Phi_\lambda^{-1}(v_i)$ of each vertex $v_i$ given in Figure~\ref{figure_zero_dim_face_F123}
	for $i=1,\cdots,7$. 
	\begin{figure}[H]
		\scalebox{0.8}{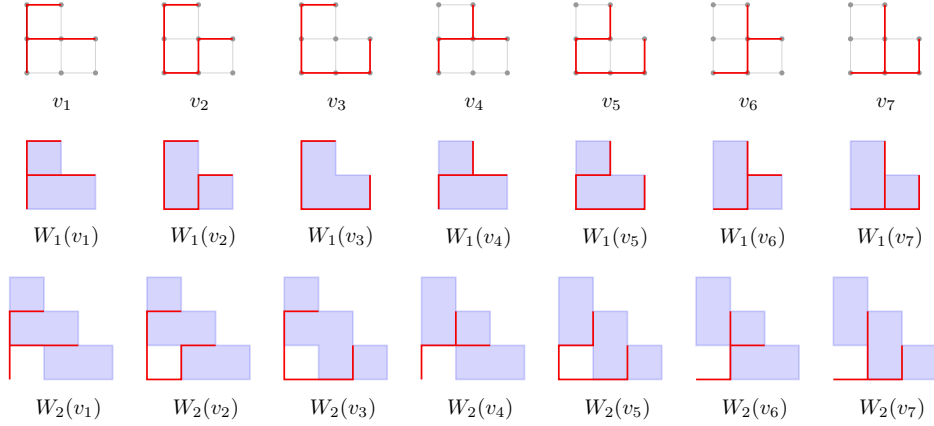}
		\caption{\label{figure_inverse_vertex_123} $W_1(v_i)$'s and $W_2(v_i)$'s for $\F (3)$}
	\end{figure}
		\vspace{-0.5cm}

	Figure \ref{figure_inverse_vertex_123} shows that $S_1(v_i) = \mathrm{pt}$ for every $i=1,\cdots,7$ since each $W_1(v_i)$ does not contain
	any $M$-shaped block containing a bottom vertex of $W_1$, that is, the origin $(0,0)$. Also, we can easily check that $S_2(v_i) = \mathrm{point}$ unless
	$i = 3$. When $i=3$, there is one $M$-shaped block $M_2$ inside $W_2(v_3)$ containing a bottom vertex of $W_2$. Thus we have
	\[
		S_2(v_3) = \mathrm{point} \times S^3 \times \mathrm{point} \cong S^3. 
	\] 
	Since $\Phi_\lambda^{-1}(v_i)$ is an $S_2(v_i)$-bundle over $S_1(v_i)$
	and $S_1(v_i)$ is a point for every $i=1,\cdots, 7$, by Theorem \ref{theorem_main}, we obtain
	\[
		\Phi_\lambda^{-1}(v_i) =
		\begin{cases}
			S^3 \quad &\mbox{ if } i = 3 \\
			\mathrm{point} \quad &\mbox{ otherwise.}
		\end{cases}
	\]
\end{example}

\begin{example}\label{example_lag_fiber_123_edge}	
	Again, we consider Example \ref{example_123} and compute the fibers over the points in the relative interior of some higher dimensional face of $\Delta_\lambda$ as follows.
	
	Consider the edge $e = e_{12}$ in Figure~\ref{figure_one_dim_face_F123} and let $\textbf{\textup{u}} \in \mathring{e}_{12}$. 
	By Theorem \ref{theorem_main}, $\Phi_\lambda^{-1}(\textbf{\textup{u}})$ is an $S_2(e_{12})$-bundle 
	over $S_1(e_{12})$ so that
	it is diffeomorphic to $S^1$. See the figure below. 
	For any other edge $e$ of $\Gamma_\lambda$,
	we can show that $\Phi_\lambda^{-1}(\textbf{\textup{u}}) \cong S^1$ for every point $\textbf{\textup{u}} \in \mathring{e}$ in a similar way. 
	\vspace{-0.2cm}
	\begin{figure}[H]
	\scalebox{0.9}{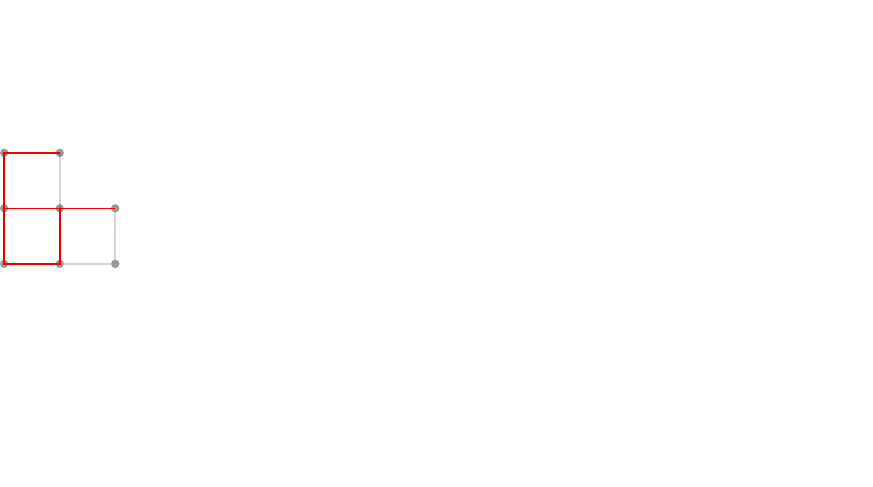}
	\end{figure}
	\vspace{-0.1cm}	
	
	For a two-dimensional face of $\Gamma_\lambda$, we first consider the face $f_{1345}$ of $\Delta_\lambda$ described in Figure
	\ref{figure_two_dim_face_F123}. 	
	Again by Theorem \ref{theorem_main}, 
	we have $\Phi_\lambda^{-1}(\textbf{\textup{u}}) \cong T^2$, an $S_2(f_{1345})$-bundle over $S_1(f_{1345})$, for every point $\textbf{\textup{u}} \in \mathring{f}_{1345}$. 
	Similarly, we obtain $\Phi_\lambda^{-1}(\textbf{\textup{u}}) \cong T^2$ for every interior point $\textbf{\textup{u}}$ of any two-dimensional face of $\Delta_\lambda$.

	\vspace{-0.2cm}
	\begin{figure}[H]
	\scalebox{0.9}{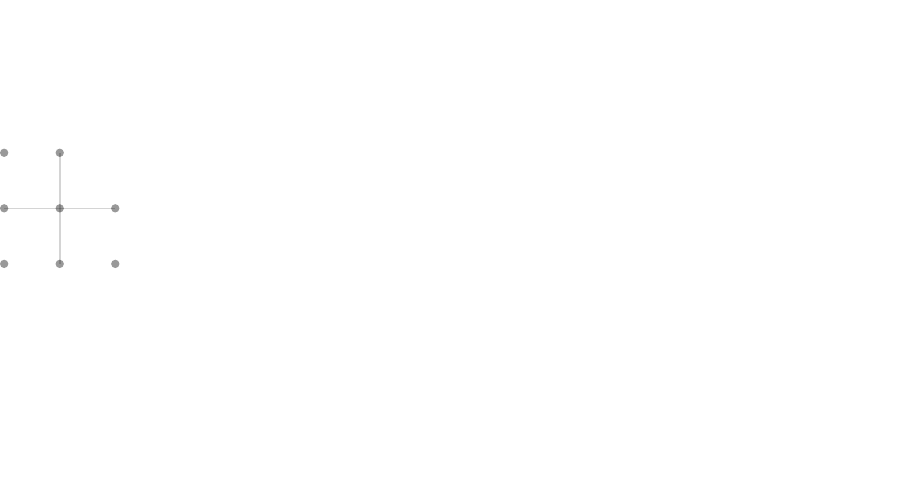}
	\end{figure}	
	\vspace{-0.1cm}	
	Finally, consider $I_{1234567}$, the improper face of $\Delta_\lambda$.
	Then Theorem \ref{theorem_main} tells us that 
	$\Phi_\lambda^{-1}(\textbf{\textup{u}})$ is an $S^1 \times S^1$-bundle over $S^1$ for every interior point $\textbf{\textup{u}}$ of $\Gamma_\lambda$.
	In fact, the Arnold-Liouville theorem implies that the bundle is trivial, i.e., $\Phi_\lambda^{-1}(\textbf{\textup{u}})$ is a torus $T^3$ for every 
	$\textbf{\textup{u}} \in \mathring{\Delta_\lambda}$, see also Theorem~\ref{theorem_contraction}.
	\vspace{-0.2cm}
	\begin{figure}[H]
	\scalebox{0.9}{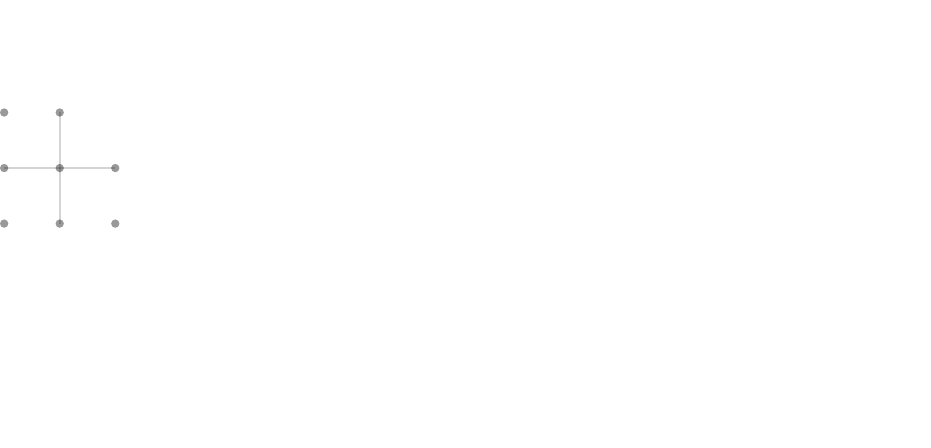}
	\end{figure}	
	\vspace{-0.2cm}
		
	Consequently, a Lagrangian fiber of the GC system on $(\mathcal{O}_\lambda, \omega_\lambda)$ is diffeomorphic to either $T^3$ (a fiber over an
	interior point of $\Delta_\lambda$) or $S^3$ (a fiber over $v_3$). Other fibers are isotropic but not Lagrangian submanifolds of
	$(\mathcal{O}_\lambda, \omega_\lambda)$ for dimensional reasons.
\end{example}

\begin{remark}\label{remark_su(3)}
In general, one should \emph{not} expect that every iterated bundle in~\eqref{theoremmaindia} is trivial. Namely, $\Phi^{-1}_\lambda(\textbf{\textup{u}})$ might not be the product 
space $\prod_{k=1}^{n-1} S_k (\gamma)$. For instance, consider the co-adjoint orbit $\mcal{O}_\lambda \simeq \mcal{F}(2,3;5)$ associated with 
$\lambda = (3,3,0,-3,-3)$ as in Figure~\ref{Fig_SU33}. By Theorem~\ref{theorem_main}, the GC fiber $\Phi_\lambda^{-1}(\textbf{0})$ over the origin is an $S^3$-bundle over $S^5$. 
Meanwhile, Proposition 2.7 in \cite{NU2} implies that $\Phi_\lambda^{-1}(\textbf{0})$ is $SU(3)$. It is however well-known that $SU(3)$ is not homotopy equivalent to $S^5 \times S^3$.
	\begin{figure}[H]
		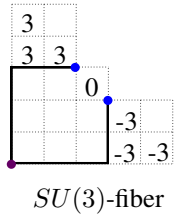
		\caption{\label{Fig_SU33} $SU(3)$-fiber}
	\end{figure}
\end{remark}

%------------------------------------------------------------------------------------
\subsection{Classification of Lagrangian faces}~
\label{ssecClassificationOfLagrangianFaces}

 \vspace{0.2cm}
Recall that Theorem \ref{theorem_main} implies that a face $\gamma$ is Lagrangian if and only if the GC fiber over an interior point of $\gamma$ is of dimension  $\dim_{\C} \mathcal{O}_\lambda$. 
Therefore to determine whether $\gamma$ is Lagrangian or not, it is sufficient to check that 
 \[
	\sum_{k=1}^{n-1} \dim S_k(\gamma) = \frac{1}{2} \dim_{\R} \mathcal{O}_\lambda.
 \]
In this section, we present an efficient way of checking whether a given face of $\Gamma_\lambda$ is Lagrangian or not by using so called ``L-shaped blocks''. 

\begin{definition}\label{definition_L_shaped_block}
	For each positive integer $k \in \Z_{\geq 1}$ and every lattice point $(p,q) \in \Z^2 \subset \R^2$, {\em a $k$-th L-shaped block at $(p,q)$}, 
	or simply {\em an $L_k$-block at $(p,q)$}, is the closed region defined by 
	\[
		L_k(p,q) := \bigcup_{(a,b)} \square^{(a,b)}
	\]
	where the union is taken over all $(a,b)$'s in $\Z^2$ such that
	\begin{itemize}
		\item $(a,b) = (p,q+i)$ \quad for $0 \leq i \leq k-1$, and
		\item $(a,b) = (p+i,q)$ \quad for $0 \leq i \leq k-1$.
	\end{itemize}
\end{definition}
	\vspace{-0.5cm}
	\begin{figure}[H]
	\scalebox{0.9}{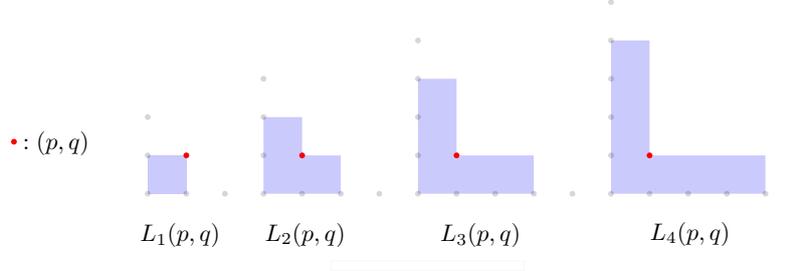}
	\caption{\label{figure_L_shape} $L$-shaped blocks}
	\end{figure}
	\vspace{-0.5cm}
	
\begin{remark}\label{rem:linearly-orderd}
    Note that GC patterns described in \eqref{equation_GC-pattern} are
    linearly ordered on any of $W$-shaped, $M$-shaped, and $L$-shaped
    blocks in the direction from the right or bottom most block to the left or top most
    block.
\end{remark}

\begin{definition}\label{definition_rigid}
	Let $\gamma$ be a face of $\Gamma_\lambda$.
	For a given positive integer $k \in \Z_{\geq 1}$ and a lattice point $(p,q) \in \Z^2$, we say that $L_k(p,q)$ is {\em rigid in} $\gamma$ if
	\begin{enumerate}
		\item the interior of $L_k(p,q)$ does not contain an edge of $\gamma$ and
		\item the rightmost edge and the top edge of $L_k(p,q)$ should be edges of $\gamma$.
	\end{enumerate}
\end{definition}

\begin{example}
	Let us consider $\Gamma = \Gamma(2,5;7)$ and let $\gamma$ be given as follows.
	\begin{figure}[H]
	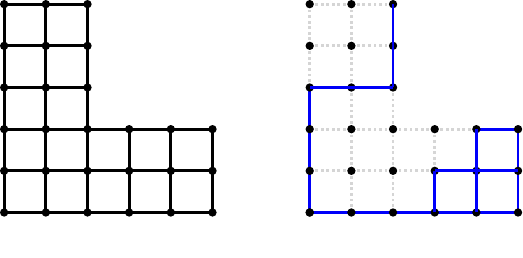
	\end{figure}
	In this example, there are exactly four rigid $L_k$-blocks : $L_3(1,1)$, $L_1(4,1)$, $L_1(5,1)$, and $L_1(5,2)$.
	\begin{figure}[H]
	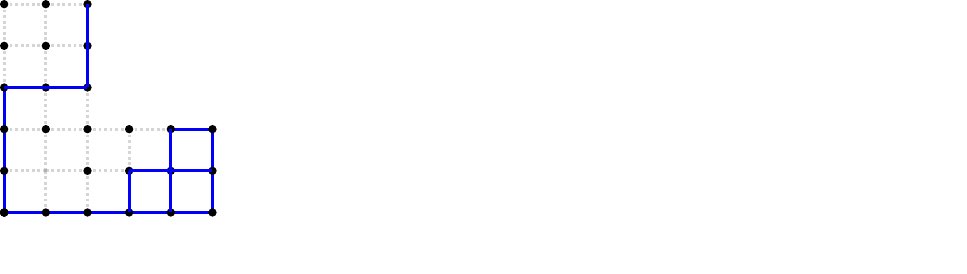
	\end{figure}
	We can check that any other $L$-blocks are \emph{not} rigid.
	For instance, $L_3(2,1)$ is not rigid since its interior contains an edge of $\gamma$. 
	\begin{figure}[H]
	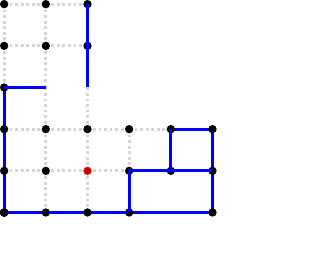
	\end{figure}
	Also, $L_2(2,2)$ is not rigid because its rightmost edge is \emph{not} an edge of $\gamma$ violating the condition (2) in Definition
	\ref{definition_rigid}.
	\begin{figure}[H]
		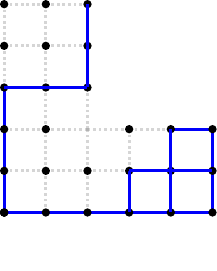
	\end{figure}
\end{example}

The following lemma follows from the min-max principle of GC pattern \eqref{equation_GC-pattern}, 
or more specifically from~\eqref{rem:diagram-pattern}.

\begin{lemma}\label{lemma_L_Q}
	Suppose that $L_k(a,b)$ is rigid in a face $\gamma$ of $\Gamma_\lambda$.
	Let $Q_k(a,b)$ be the closed region defined by
	\[
		Q_k(a,b) := \bigcup_{0 \leq i,j \leq k-1} \square^{(a+i,b+j)},
	\]
	i.e., $Q_k(a,b)$ is the square of size $(k \times k)$ that contains $L_k(a,b)$.
	Then there are no edges of $\gamma$ in the interior of $Q_k(a,b)$.
\end{lemma}

\begin{proof}
	If $k=1$, then $L_1(a,b) = Q_1(a,b)$ and it has no edge of $\gamma$ in its interior so that there is nothing to prove. Now, assume that $k \geq 2$
	and suppose that there is an edge $e = [v_0v_1]$ of $\gamma$ contained in the interior of $Q_k(a,b)$.
	Then, without loss of generality, we may assume that
	$v_0 = (a_0, b_0)$ is in the interior of $Q_k(a,b)$ so that $a \leq a_0 < a+k-1$ and $b \leq b_0 < b+k-1$. 
	
	By Definition \ref{definition_face}, there exists a positive path $\delta$ contained in $\gamma$ passing through $v_0$. 
	Then $\delta$ should pass through the interior of $L_k(a,b)$ since $\delta$ contains a shortest path from the origin of $\Gamma_\lambda$ and $v_0$, 
	which intersects the interior of $L_k(a,b)$. This contradicts the rigidity (1) in Definition \ref{definition_rigid}.
\end{proof}

\begin{lemma}\label{lemma_L_disjoint}
	If two different $L$-blocks $L_i(a,b)$ and $L_j(c,d)$ are rigid in the same face $\gamma$, 
	then they cannot be overlapped, i.e.,
	\[
		\mathring{L_i}(a,b) \cap \mathring{L_j}(c,d) = \emptyset
	\]
	where $\mathring{L_i}(a,b)$ and $\mathring{L_j}(c,d)$ denote the interior of $\mathring{L_i}(a,b)$ and $\mathring{L_j}(c,d)$, respectively.
\end{lemma}

\begin{proof}
	If $(a,b) = (c,d)$ and $i \neq j$, then both $L_i(a,b)$ and $L_i(c,d)$ cannot be rigid since at least one of these two blocks violates (1) in Definition~\ref{definition_rigid}. 
	Suppose that $(a,b) \neq (c,d)$ and $\mathring{L_i}(a,b) \cap \mathring{L_j}(c,d) \neq \emptyset$. Without loss of generality, we may assume that $j \geq i$.
	Then it is straightforward that either the top edge or the rightmost edge of $L_i(a,b)$ lies on the interior of
	$Q_j(c,d)$. It leads to a contradiction to Lemma \ref{lemma_L_Q} since
	the top edge and the rightmost edge of $L_i(a,b)$ are edges of $\gamma$ by Definition \ref{definition_L_shaped_block}.
	\begin{figure}[H]
		\scalebox{0.9}{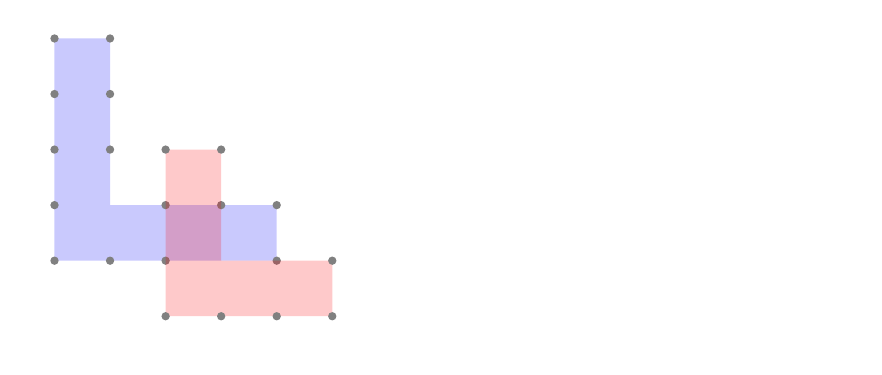}
	\end{figure}
	\vspace{-0.2cm}
\end{proof}

\begin{proposition}\label{proposition_classification_lagrangian_face}
	Let $\lambda = (\lambda_1, \cdots, \lambda_n)$ be given in \eqref{lambdaidef}. 
	For a given face $\gamma$ of $\Gamma_\lambda$, let $\mathcal{L}(\gamma)$ be the set of all rigid L-shaped blocks in $\gamma$.
	Then, for any point $\textbf{\textup{u}}$ in the relative interior of the face $f_\gamma = \Psi(\gamma)$ of $\Delta_\lambda$, we have 
	\[
		\dim \Phi_\lambda^{-1}(\textbf{\textup{u}}) = \sum_{L_k(a,b) \in \mathcal{L}(\gamma)} |L_k(a,b)| = \sum_{L_k(a,b) \in \mathcal{L}(\gamma)} (2k-1)
	\]	
	where $|L_k(a,b)| = 2k - 1$ is the are of $L_k(a,b)$.
\end{proposition}

\begin{proof}
	By Theorem \ref{theorem_main}, it is enough to show that
	\[
		\sum_{k=1}^{n-1} \dim S_k(\gamma) =  \sum_{L_k(a,b) \in \mathcal{L}(\gamma)} |L_k(a,b)|
	\] where 
	\[
		S_k(\gamma) = \prod_{\mathcal{D} \subset W_k(\gamma)} S_k(\mathcal{D})
	\]
	as defined in (\ref{equation_block_sphere}).
	
	Let $\mathcal{D}$ be a simple closed region in $W_k(\gamma)$. Suppose that $\mathcal{D}$ contains a bottom vertex of $W_k$ and
	$\mathcal{D} = M_j(a,b)$ for some $j \geq 1$ where $M_j(a,b)$ denotes the $j$-th $M$-shaped block whose top-left vertex is
	$(a,b)$.
	Then there are two edges $e_1$ and $e_2$ of $\gamma$ on the boundary of $M_j$ as we see below.
	\begin{figure}[h]
	\scalebox{0.9}{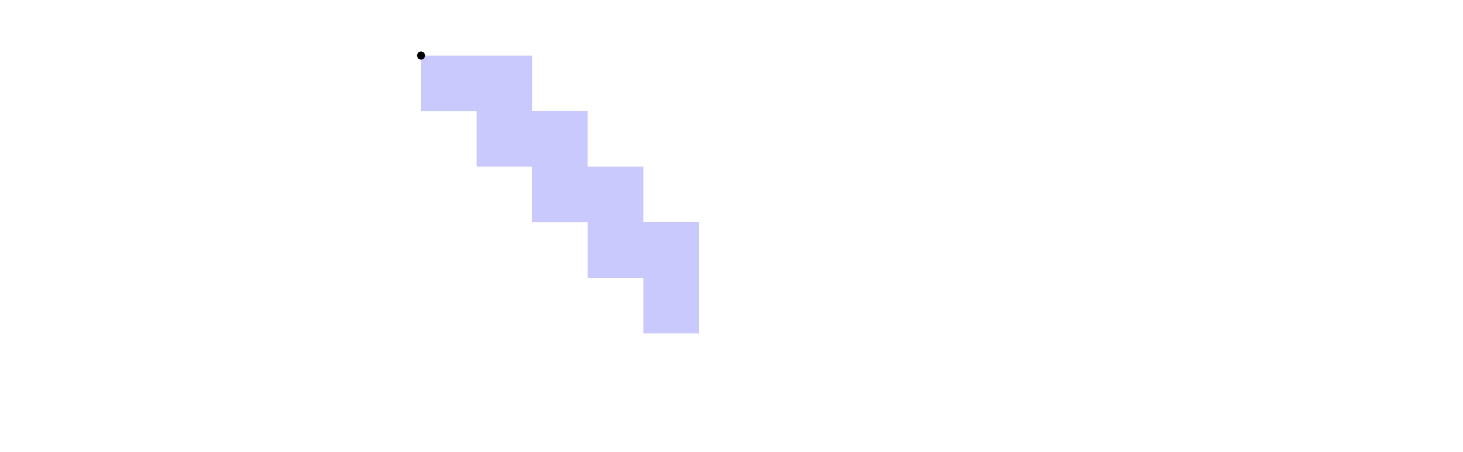}
	\end{figure}
	By~\eqref{rem:diagram-pattern},
	there is no edge of $\gamma$ in the interior of $Q_j(a+1,b-j+1)$. Thus $L_j(a+1,b-j+1)$
	is a rigid $L_j$-block in $\gamma$.
	
	Similarly, for any rigid $L$-shaped block $L_j(a+1,b-j+1)$ in $\gamma$ (for some $j \in \Z_{\geq 0}$ and $(a,b) \in (\Z_{\geq 1})^2$), 
	we can find a closed region $\mcal{D}$ in some $W_k(\gamma)$ such that 
	$\mcal{D}$ is $M_j(a,b)$ that contains a bottom vertex of $W_k$. 
	Therefore, 
	there is a one-to-one correspondence between the set of rigid $L$-shaped blocks in $\gamma$ and the set of simple closed regions $\mathcal{D}$ appeared
	in some $W_k(\gamma)$ such that
	such that $S_k(\mathcal{D}) = S^{2j-1}$. In other words,
	\[
		\begin{array}{rcl}
			\displaystyle \bigcup_{j=1} \,\, \{\textrm{rigid~$L_j$}$-$\mathrm{blocks~in~}\gamma \}  & \Leftrightarrow & \displaystyle \bigcup_{k=1}^{n-1} \bigcup_{j=1}
			\,\, \left\{ \mathcal{D} \subset W_k(\gamma) ~|~ S_k(\mathcal{D}) = S^{2j-1} \right\} \\ [0.5em]
			L_j(a+1,b-j+1) & \leftrightarrow & M_j(a,b).
		\end{array}
	\]
	Moreover, we have $|M_j| = |L_j| = \dim S_k(\mathcal{D}) = 2j-1$, and hence it completes the proof.
\end{proof}

The following corollary is derived from Theorem \ref{theorem_main},
Lemma \ref{lemma_L_disjoint} and Proposition \ref{proposition_classification_lagrangian_face}.

\begin{corollary}\label{corollary_L_fillable}
Let $\gamma$ and $f_\gamma$ be as in Proposition \ref{proposition_classification_lagrangian_face}.
Then the followings are equivalent.
	\begin{enumerate}
		\item For an interior point $\textbf{\textup{u}}$ of $f_\gamma$, the fiber $\Phi_\lambda^{-1}(\textbf{\textup{u}})$ is a Lagrangian submanifold of $(\mathcal{O}_\lambda, \omega)$.
		\item For any interior point $\textbf{\textup{u}}$ of $f_\gamma$, the fiber $\Phi_\lambda^{-1}(\textbf{\textup{u}})$ is a Lagrangian submanifold of $(\mathcal{O}_\lambda, \omega)$.
		\item The set of rigid L-shaped blocks in $\gamma$ covers the whole $\Gamma_\lambda$.
	\end{enumerate}
\end{corollary}

Also, we have the following corollary which follows from Corollary \ref{corollary_L_fillable}.
\begin{corollary}\label{corollary_unless_projective}
	A Gelfand-Cetlin system $\Phi_\lambda \colon \mathcal{O}_\lambda \rightarrow \Delta_\lambda$ on a co-adjoint orbit $(\mathcal{O}_\lambda, \omega_\lambda)$
	always possesses a non-torus Lagrangian fiber unless $\mathcal{O}_\lambda$ is a projective space.
\end{corollary}

\begin{example}\label{example_gr26}
	Let $\lambda = \{t,t,0,0,0,0\}$ with $t >0$. The co-adjoint orbit $\mathcal{O}_\lambda$ is a complex Grassmannian $\mathrm{Gr}(2,6)$
	of two planes in $\C^6$ and the corresponding ladder diagram
	$\Gamma_\lambda$ is given as below.
	\begin{figure}[H]
	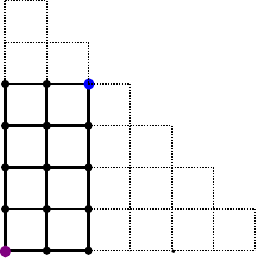
	\end{figure}
	Observe that any faces of $\Gamma_\lambda$ do not admit rigid $L_k$-blocks of $k > 2$. 
	Also, note that there are three Lagrangian faces $\gamma_1$, $\gamma_2$ and $\gamma_3$ of $\Gamma_\lambda$
	which have only one rigid $L_2$-block as follows.
	\begin{figure}[H]
	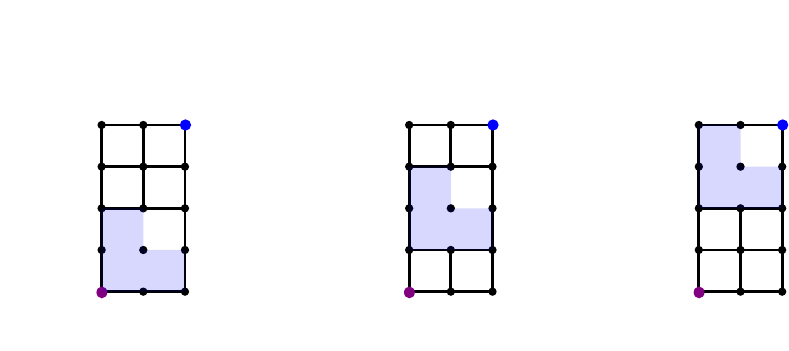
	\end{figure}
	Finally, there is exactly one Lagrangian face $\gamma_4$ which has two rigid $L_2$-blocks as below.
	\begin{figure}[H]
	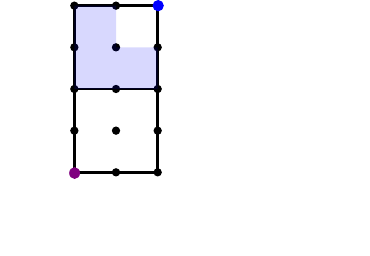
	\end{figure}
	Thus there are exactly four proper Lagrangian faces $\gamma_1, \gamma_2, \gamma_3$ and $\gamma_4$ of $\Gamma_\lambda$.
\end{example}

\begin{example}
	Let $\lambda = \{3,2,1,0\}$. Then the co-adjoint orbit $\mathcal{O}_\lambda$ is a complete flag manifold $\mcal{F}(4)$ and the corresponding ladder diagram $\Gamma_\lambda$ is as follows.
	\begin{figure}[H]
	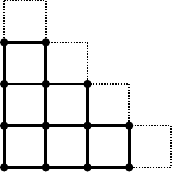
	\end{figure}
	We can easily see that any face of $\Gamma_\lambda$ does not have a rigid $L_k$-block for $k \geq 3$.
	There are exactly three Lagrangian faces of $\Gamma_\lambda$ containing one rigid $L_2$-block as below.
	 \begin{figure}[H]
	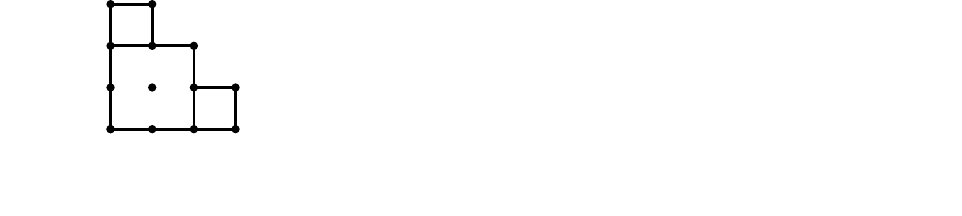
	\end{figure}
	Also, it is not hard to see that there is no Lagrangian face that contains more than one $L_2$-block.
	Thus $\gamma_1$, $\gamma_2$, and $\gamma_3$ are the only proper Lagrangian faces of $\Gamma_\lambda$.
\end{example}

\vspace{0.1cm}
%------------------------------------------------------------------------------------
\section{Iterated bundle structures on Gelfand-Cetlin fibers}
\label{secIteratedBundleStructuresOnGelfandCetlinFibers}
\label{section4}

In this section, for each point $\textbf{\textup{u}}$ in a GC polytope $\Delta_\lambda$, we will construct the iterated bundle $E_\bullet$ described in Section \ref{secLagrangianFibersOfGelfandCetlinSystems} whose total space is the fiber $\Phi^{-1}_\lambda(\textbf{\textup{u}})$. Using this construction, we complete the proof of Theorem \ref{theorem_main} by showing that
each $\Phi^{-1}_\lambda(\textbf{\textup{u}})$ is an isotropic submanifold of $\mathcal{O}_\lambda$.

	For a fixed integer $k \geq 1$, consider sequences  $\fa = (a_1, \cdots, a_{k+1})$ and $\fb = (b_1, \cdots, b_k)$ of real numbers satisfying
	\begin{equation}\label{patternaandb}
		a_1 \geq b_1 \geq a_2 \geq b_2 \geq \cdots \geq a_k \geq b_k \geq a_{k+1}.
	\end{equation}
	Denoting by $\mathcal{H}_\ell$ the set of $(\ell \times \ell)$ hermitian matrices for $\ell \geq 1$ and by $\mathrm{sp}(x)$ the spectrum of $x$, we let
	\[
			\mathcal{O}_\fa =  \left\{ x \in \mathcal{H}_{k+1} ~\big{|}~  \mathrm{sp}(x) = \{a_1, \cdots, a_{k+1}\} \right\}
	\]
	be the co-adjoint $U(k+1)$-orbit of the diagonal matrix $I_\fa := \mathrm{diag}(a_1,\cdots, a_{k+1})$ in $\mathcal{H}_{k+1}  \cong \uu(k+1)^*$. 
	
	Consider a subspace
	\begin{align*}
		\mathcal{A}_{k+1}(\fa,\fb) & = \left\{ x \in \mathcal{H}_{k+1} ~\big{|}~ \mathrm{sp}(x) = \{a_1, \cdots, a_{k+1}\}, ~\mathrm{sp}(x^{(k)}) = \{b_1, \cdots, b_k\} \right\}
	\end{align*}
	where $x^{(k)}$ denotes the $(k \times k)$ leading principal minor submatrix of $x$.
	It naturally comes with the projection map
	\[
		\begin{array}{ccccc}
		\rho_{k+1} \colon & \mathcal{A}_{k+1}(\fa,\fb) & \rightarrow & \mathcal{O}_\fb \\
		                     &                   x                     & \mapsto & x^{(k)}\\
		\end{array}
	\]
	from $\mathcal{A}_{k+1}(\fa,\fb)$ to the co-adjoint $U(k)$-orbit $\mathcal{O}_\fb$ of the diagonal matrix $I_\fb$ in $\mathcal{H}_k \cong \uu(k)^*$.

	Let $W_k(\fa,\fb)$ be the $k$-th $W$-shaped block $W_k$ together with walls defined by the equalities of $a_i$'s and $b_j$'s as in Figure~\ref{figure_block_a_b}.
	By comparing the divided regions by the walls on $W_k(\fa,\fb)$ with $M$-shaped blocks as in~\eqref{equation_block_sphere},
	we define a topological space $S_k(\fa,\fb)$, which is either a single point or a product space of odd dimensional spheres.
	
	\begin{figure}[H]
	\scalebox{0.9}{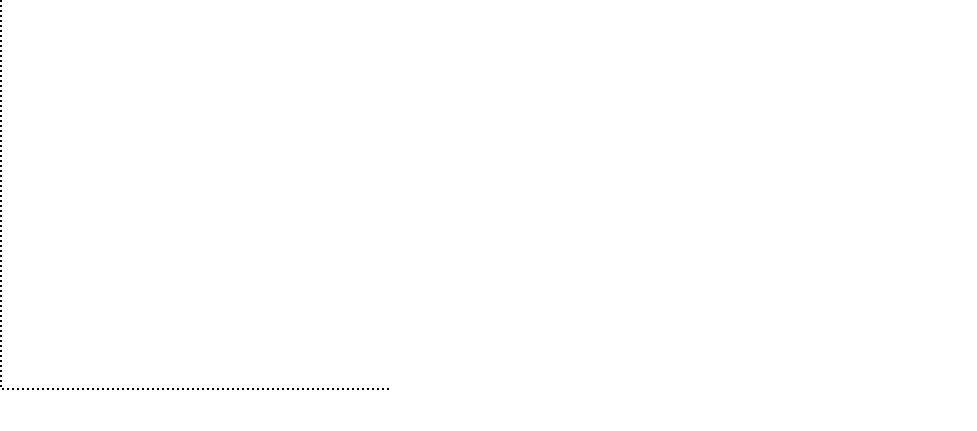}
	\caption{\label{figure_block_a_b} $W_k(\fa, \fb)$}
	\end{figure}	

	\begin{example}
	\begin{enumerate}
	\item
		For $\fa = (a_1,a_2,a_3) = (5,4,2)$ and $\fb = (b_1, b_2) = (4,2)$, $W_2(\fa,\fb)$ is divided by three simply closed regions $\mcal{D}_1, \mcal{D}_2$ and $\mcal{D}_3$. Since $\mcal{D}_1$ does not contain any bottom vertices and neither $\mcal{D}_2$ nor $\mcal{D}_3$ match with $M$-shaped blocks, $S_2(\fa,\fb) = S_2(\mathcal{D}_1) \times S_2(\mathcal{D}_2)  \times S_2(\mathcal{D}_3)  \cong \mathrm{point}$.
		 \begin{figure}[H]
				\scalebox{0.9}{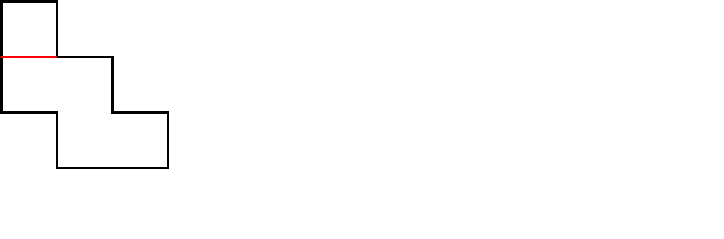}
		\end{figure}
		\vspace{-0.2cm}

		\item		For $\fa = (5,4,2)$ and $\fb = (4,4)$, $W_2(\fa,\fb)$ is divided by three simply closed regions $\mcal{D}_1, \mcal{D}_2$ and $\mcal{D}_3$. Observe that $\mcal{D}_1$ and $\mcal{D}_3$ do not contain any bottom vertices and $\mathcal{D}_2$ is an $M_2$-block containing bottom vertices of $W_2$. Therefore, we have $S_2(\fa,\fb) = \mathrm{point} \times S^3 \times \mathrm{point} \cong S^3.$

		\begin{figure}[H]
				\scalebox{0.9}{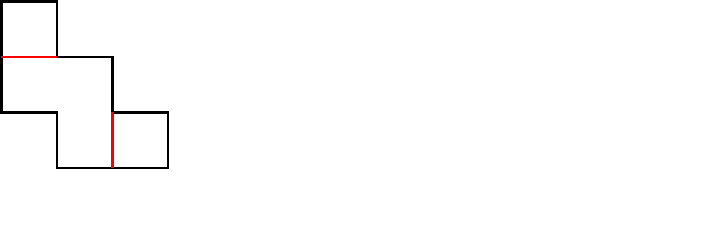}
		\end{figure}
		\vspace{-0.3cm}	
	\end{enumerate}
	\end{example}
	
	\begin{proposition}\label{proposition_A_k_S_k_bundle}\label{proposition_bundle}
		With the notations above, $\rho_{k+1} \colon \mathcal{A}_{k+1}(\fa,\fb) \rightarrow \mathcal{O}_\fb$ is an $S_k(\fa,\fb)$-bundle over $\mathcal{O}_\fb$.
	\end{proposition}
	
	Before starting the proof of Proposition \ref{proposition_A_k_S_k_bundle}, as preliminaries, we introduce some notations and prove lemmas. We denote by $\widetilde{\mathcal{A}}_{k+1}(\fa,\fb)$ the set of matrices in $\mathcal{A}_{k+1}(\fa,\fb)$ whose $(k \times k)$ leading principal minor is the diagonal matrix $I_\fb$. So, a matrix in $\widetilde{\mathcal{A}}_{k+1}(\fa,\fb)$ is of the form
	\[
		Z_{(\fa,\fb)}(z) =   \begin{pmatrix} b_1 & & 0 & \bar{z}_1\\
                	  & \ddots & & \vdots\\
                  	0 & & b_k & \bar{z}_k\\
                  	z_1 & \hdots & z_k & z_{k+1}
		  \end{pmatrix}
		  \vspace{0.2cm}
	\]
	for $z = (z_1, \cdots, z_k) \in \C^k$. Since $Z_{(\fa,\fb)}(z)$ has the eigenvalues $\fa = \{a_1, \cdots, a_{k+1} \}$, the $(k+1,k+1)$-entry of $Z_{(\fa,\fb)}(z)$ is to be constant
	\[
		z_{k+1} = \sum_{i=1}^{k+1} a_i - \sum_{i=1}^k b_i \in \R
	\]
	by computing the trace of $Z_{(\fa,\fb)}(z)$.
	The characteristic polynomial of $Z_{(\fa,\fb)}(z)$ is expressed as
	\begin{equation}\label{equation_det}
		\det(xI - Z_{(\fa,\fb)}(z)) = (x - z_{k+1}) \cdot \prod_{i=1}^k (x-b_i) - \sum_{j=1}^k \left( \frac{|z_j|^2}{x-b_j} \cdot \prod_{i=1}^k (x-b_i) \right) = 0,
	\end{equation}
	whose zeros are $x = a_1, \cdots, a_{k+1}$ by our assumption. By inserting $x = a_1, \cdots, a_{k+1}$ into~\eqref{equation_det}, we obtain the system of $(k+1)$ equations of real coefficients, which are linear with respect to $(|z_1|^2, \cdots, |z_k|^2) \in (\R_{\geq 0})^k$.
	We sometimes regard an element $\widetilde{\mathcal{A}}_{k+1}(\fa,\fb)$ as an element $\C^k$ under the identification $Z_{(\fa, \fb)}(z) \mapsto z$.
	The following lemma implies that the solution space is \emph{never} empty as long as $(\fa, \fb)$ obeys~\eqref{patternaandb}.
	
	\begin{lemma}[Lemma 3.5 in \cite{NNU}]\label{lemma_NNU}
		Let $a_1, b_1, \cdots, a_k, b_k$, and $a_{k+1}$ be real numbers satisfying~\eqref{patternaandb}.
	Then there exists $z_1, \dots, z_k \in \mathbb{C}$ and $z_{k+1} \in \mathbb{R}$ such that
	\[
		  \begin{pmatrix} b_1 & & 0 & \bar{z}_1\\
                	  & \ddots & & \vdots\\
                  	0 & & b_k & \bar{z}_k\\
                  	z_1 & \hdots & z_k & z_{k+1}
		  \end{pmatrix}
	\]
	has eigenvalues $a_1, \dots , a_{k+1}$.
	\end{lemma}

	\begin{lemma}\label{lemma_fiber_torus}
		Suppose in addition that $a_1, b_1, \cdots, a_k, b_k$, and $a_{k+1}$ in Lemma~\ref{lemma_NNU} are all distinct.
		Then there exist positive numbers $\delta_1, \cdots, \delta_k$ such that
		\[
			\begin{array}{ccl}
			\widetilde{\mathcal{A}}_{k+1}(\fa,\fb) & = & \left \{
				  \begin{pmatrix} b_1 & & 0 & \bar{z}_1\\
                	  & \ddots & & \vdots\\
                  	0 & & b_k & \bar{z}_k\\
                  	z_1 & \hdots & z_k & z_{k+1}
		  \end{pmatrix}
		\in \mathcal{A}_{k+1}(\fa,\fb)
		  ~:~ (z_1,\cdots, z_k) \in \C^k, ~|z_i|^2 = \delta_i \textup{ for } i=1,\cdots, k \right \} \\[0.8em]
		  \end{array}
		\]
		where $z_{k+1} = \sum_{i=1}^{k+1} a_i - \sum_{i=1}^k b_i$.
		In particular, we have $\widetilde{\mathcal{A}}_{k+1}(\fa,\fb) \cong T^k$.
	\end{lemma}
	\begin{proof}
		We first note that if $|z_j|^2 = 0$ for some $j$, then the equation (\ref{equation_det}) (with respect to $x$) has a solution $x=b_j$. It implies that
		$b_j \in \{a_1, \cdots, a_{k+1} \}$, which contradicts to our assumption that $a_i$'s and $b_j$'s are all distinct.
		Thus, it is enough to show existence and uniqueness of
		a solution $(|z_1|^2, \cdots, |z_k|^2)$ of the system of linear equations in (\ref{equation_det}).
		
		The existence immediately follows from Lemma \ref{lemma_NNU}.
		Let
		\[
			\{|z_1|^2 = \delta_1 > 0, \cdots, |z_k|^2 = \delta_k > 0\}
		\]
		be one of solutions of (\ref{equation_det}) so that
		$\widetilde{\mathcal{A}}_{k+1}(\fa,\fb)$ contains a real $k$-torus
		\[
			T^k = \{(z_1, \cdots, z_k) \in \C^k ~|~ |z_i|^2 = \delta_i,~ i=1,\cdots,k \},
		\]
		which yields that
		$\dim_\R \widetilde{\mathcal{A}}_{k+1}(\fa,\fb) \geq k$.
		Since (\ref{equation_det}) is a system of non-homogeneous \emph{linear} equations with respect to the variables $|z_1|^2, \cdots, |z_k|^2$,
		the set
		$$
			\left\{ (|z_1|^2, \cdots, |z_k|^2) ~|~ (z_1, \cdots, z_k) \in \widetilde{\mathcal{A}}_{k+1}(\fa,\fb) \right\}
		$$
		is an affine subspace of $\R^k$. Therefore,
		$\dim_\R \widetilde{\mathcal{A}}_{k+1}(\fa,\fb) = k$ if and only if the equations (\ref{equation_det})
		has a unique solution $(|z_1|^2, \cdots, |z_k|^2) = (\delta_1, \cdots, \delta_k)$. It suffices to show that
		$\dim_\R \widetilde{A}_{k+1}(\fa,\fb) = k$.
		
		Note that $\mathcal{A}_{k+1}(\fa,\fb)$ is an $ \widetilde{\mathcal{A}}_{k+1}(\fa,\fb)$-bundle over $\mathcal{O}_{\fb}$ whose projection map is
		\[
			\begin{array}{ccccc}
					\rho_{k+1}  \colon & \mathcal{A}_{k+1}(\fa,\fb) & \rightarrow & \mathcal{O}_\fb \\
					                    &                x             & \mapsto    &  x^{(k)}.\\
			\end{array}
		\]
		More precisely, for each element $y \in \mathcal{O}_\fb$, there exists a unitary matrix $g_y \in U(k)$ (depending on $y$) such that
		\[
			g_y y g_y^{-1} = I_\fb
		\]
		where $I_\fb$ is the diagonal matrix $\textup{diag}(b_1, \cdots, b_k)$.
		Then the preimage $\rho_{k+1}^{-1}(y)$ of $y$ can be identified with $\widetilde{\mathcal{A}}_{k+1}(\fa,\fb)$ via
		\[
		\begin{array}{ccl}
			\rho_{k+1}^{-1}(y) & \longrightarrow & \widetilde{\mathcal{A}}_{k+1}(\fa,\fb) \\[0.5em]
			          Y  =  \begin{pmatrix} y & * \\
					                  	*^t & z_{k+1} \end{pmatrix}
					                  	& \mapsto & \begin{pmatrix} g_y & 0 \\
					                  	0 & 1 \end{pmatrix}
					                  	\cdot  Y \cdot
					                  	\begin{pmatrix} g_y^{-1} & 0 \\
					                  	0 & 1 \end{pmatrix}\\[2em]
					                  	
					                  	& & = \begin{pmatrix} g_y \cdot y \cdot g_y^{-1} & g_y \cdot * \\
					                  	*^t \cdot g_y^{-1} & z_{k+1}
								 \end{pmatrix} \\[2em]
								 & & =
								 \begin{pmatrix} I_\fb & g_y \cdot * \\
					                  	*^t \cdot g_y^{-1} & z_{k+1}
								 \end{pmatrix}
		\end{array}
		\]
		so that $\mathcal{A}_{k+1}(\fa,\fb)$ is an $ \widetilde{\mathcal{A}}_{k+1}(\fa,\fb)$-bundle over $\mathcal{O}_{\fb}$ via $\rho_{k+1}$.
		
		Now, we consider a sequence of real numbers $\fc = \{c_1, \cdots, c_{k-1}\}$ such that
		\[
		b_1 > c_1 > \cdots > b_{k-1} > c_{k-1} > b_k.
		\]
		Restricting the fibration $\rho_{k+1}$ to $\mathcal{A}_k(\fb,\fc)$, we similary have an $\widetilde{\mathcal{A}}_{k+1}(\fa,\fb)$-bundle over
		$\mathcal{A}_k(\fb,\fc)$.
		Note that its total space is the collection of $(k+1)\times (k+1)$ hermitian matrices such that the spectra,
		the spectra of the $(k \times k)$ leading principal minor and the spectra of the $(k-1) \times (k-1)$ leading principal minor are resepectively $\fa, \fb$ and $\fc$.
		
		By a similar way described as above, we see that
		$\mathcal{A}_k(\fb,\fc)$ is an $\widetilde{\mathcal{A}}_k(\fb,\fc)$-bundle over $\mathcal{O}_\fc$
		with the projection map $\rho_k \colon \mathcal{A}_k(\fb,\fc) \rightarrow \mathcal{O}_\fc$ such that
		$\dim_\R \widetilde{\mathcal{A}}_k(\fb,\fc) \geq k-1$. Taking a sequence of real numbers $\fd = \{d_1, \cdots, d_{k-2}\}$
		so that
		\[
		c_1 > d_1 > \cdots > c_{k-2} > d_{k-2} > c_{k-1},
		\]
		the restriction of $\rho_k$ to $\mathcal{A}_{k-1}(\fc,\fd)$ induces an $\widetilde{\mathcal{A}}_k(\fb,\fc)$-bundle over $\mathcal{A}_{k-1}(\fc,\fd)$.
		
		Proceeding this procedure inductively, we end up obtaining a tower of bundles such that the total space $E$ is a generic fiber of the GC system of $\mcal{O}_\fa$. 
		Namely, $E$ is the preimage of a point in the interior of the GC polytope
		$\Delta_\fa$. By Proposition \ref{proposition_GS_smooth}, $E$ is a smooth manifold of dimension
		\[
			\dim_\R E = \frac{1}{2} \dim_\R \mathcal{O}_\fa = \frac{k(k+1)}{2}.
		\]
		On the other hand, by our construction, the dimension of $E$ is the sum of dimensions of all fibers of $\rho_i$'s for $i=2,\cdots,k+1$ so that
		\[
			\dim E = \dim \widetilde{\mathcal{A}}_{k+1} + \dim \widetilde{\mathcal{A}}_{k} + \cdots + \dim \widetilde{\mathcal{A}}_{2}.
		\]
		Since $\dim \widetilde{\mathcal{A}}_{i+1} \geq i$ for each $i$, we get $\dim \widetilde{\mathcal{A}}_{i+1} = i$ for every $i=1,\cdots,k$.
		Lemma~\ref{lemma_fiber_torus} is established.
	\end{proof}
	
	Note that Lemma \ref{lemma_fiber_torus} deals with the case where $a_j \not \in \{b_1, \cdots, b_k\}$ for $j=1,\cdots,k+1$.
	Now, let us consider the case where $a_{j+1} \in \{b_1, \cdots, b_k\}$ for some $j \in \{0,1,\cdots, k\}$. 			
	Denoting the multiplicity of $a_{j+1}$ in $\fa$ by $\ell$, without loss of generality, we assume that $a_j > a_{j+1} = \cdots = a_{j+\ell} > a_{j+\ell+1}$.
     Then either $a_{j+1} = b_j$ or $a_{j+1} = b_{j+1}$. For the first case, there are two possible cases:
     \begin{equation}\label{eq:inequality1}
        \begin{cases}
        (1) ~a_j > b_j = a_{j+1} = \cdots = a_{j+\ell} > b_{j+\ell},   ~\mathrm{or}\\
        (2) ~a_j > b_j = a_{j+1} = \cdots = a_{j+\ell} = b_{j+\ell} > a_{j+\ell+1}.
        \end{cases}
     \end{equation}
     For the second case, we have two possible cases too:
     \begin{equation}\label{eq:inequality2}
        \begin{cases}
        (3) ~b_j > a_{j+1} = b_{j+1} = \cdots = a_{j+\ell} = b_{j+\ell} > a_{j+\ell+1},  ~\mathrm{or}\\
		(4) ~b_j > a_{j+1} = b_{j+1} =  \cdots = a_{j+\ell} > b_{j+\ell},\\
        \end{cases}
     \end{equation}
     Note that only the case (2) above have more multiplicity of $b$'s than $a$'s, i.e.,
     multiplicity $\ell+1$ and $\ell$ respectively: The cases (1) and (3) have
     the same multiplicities of both $a$ and $b$ while in the case (4) $a$ has multiplicity $\ell$ and $b$ has
     multiplicity $\ell -1$.

	We start with the first inequality of (\ref{eq:inequality2}).
	\begin{lemma}[case (3) of~\eqref{eq:inequality2}]\label{lemma_fiber_zero_a_b}
		Suppose that $b_j > a_{j+1} = b_{j+1} = \cdots = a_{j+\ell} = b_{j+\ell} > a_{j+\ell+1}.$ Then, every solution $(z_1, \cdots, z_k) \in \C^k$ of the equation
		(\ref{equation_det}) satisfies
		\[
			z_{j+1} = \cdots = z_{j+\ell} = 0.
		\]
	\end{lemma}

	\begin{proof}
	 	Observe that each term of the equation (\ref{equation_det})
	 	\[
	 		\det(xI - Z_{(\fa,\fb)}(z)) = (x - z_{k+1}) \cdot \prod_{i=1}^k (x-b_i) - \sum_{i=1}^k \left( \frac{|z_i|^2}{x-b_i} \cdot \prod_{m=1}^k (x-b_m) \right) = 0
		\]
		is divisible by $(x-b_{j+1})^{\ell-1}$ by our assumption. In particular, the first term of the equation
		\[
		(x - z_{k+1}) \cdot \prod_{i=1}^k (x-b_i)
		\]
		is divisible by $(x-b_{j+1})^{\ell}$.
		For each $i \not \in \{j+1, \cdots, j+\ell\}$, so is
		\[
		\frac{|z_i|^2}{x-b_i} \cdot \prod_{m=1}^k (x-b_m)
		\]
		since $b_{j+1} = \cdots = b_{j+\ell}$.
		Taking
		\[
			g(x) := \det(xI - Z_{(\fa,\fb)}(z)) / (x-b_{j+1})^{\ell-1},
		\]
		we have $g(b_{j+1}) = g(a_{j+1}) = 0$ because $x = a_{j+1} = b_{j+1}$ is a solution of~\eqref{equation_det} with multiplicity $\ell$. It yields
		\[
			\left(|z_{j+1}|^2 + \cdots + |z_{j+\ell}|^2\right) \cdot \prod_{\substack{m=1 \\ m \not \in \{j+1, \cdots, j+\ell\}}}^k (b_{j+1}-b_m) = 0.
		\]
		Since
		\[
			\prod_{\substack{m=1 \\ m \not \in \{j+1, \cdots, j+\ell\}}}^k (b_{j+1}-b_m) \neq 0,
		\]
		we deduce that $|z_{j+1}|^2 + \cdots + |z_{j+\ell}|^2 = 0$ and hence $z_{j+1} = \cdots = z_{j+\ell} = 0$.
	\end{proof}
	
	Therefore, under the assumption (case (3) in \eqref{eq:inequality2}) on $\fa$ and $\fb$,
	the equation (\ref{equation_det}) is written by
	\[
			\det(xI - Z_{(\fa,\fb)}(z))
			        = (x-b_{j+1})^{\ell} \cdot
			        \left \{(x - w_{k+1-\ell}) \cdot \displaystyle \prod_{i=1}^{k-\ell} (x-b'_i) - \sum_{i=1}^{k-\ell}
			        \left( \frac{|w_i|^2}{x-b'_i} \cdot \prod_{m=1}^{k-\ell} (x-b'_m) \right) \right \}\\
			        =  0
	\]
		where
		\begin{itemize}
			\item $(b_1'. \cdots, b_{k-\ell}') = (b_1, \cdots, b_j, \hat{b_{j+1}}, \cdots, \hat{b_{j+\ell}}, b_{j+\ell+1}, \cdots, b_k)$,
			\item $(w_1. \cdots, w_{k-\ell}) = (z_1, \cdots, z_j, \hat{z_{j+1}}, \cdots, \hat{z_{j+\ell}}, z_{j+\ell+1}, \cdots, z_k)$, and
			\item $w_{k-\ell+1} = z_{k+1}$.
		\end{itemize}
	Observe that the equation
	\[
		\det(xI - Z_{(\fa,\fb)}(z)) / (x-b_{j+1})^{\ell} = 0
	\]	
	is same as the equation $\det(xI - Z_{(\fa',\fb')}(w)) = 0$ where
	\begin{itemize}
		\item $\fa' = (a_1', \cdots, a_{k-\ell+1}') = (a_1, \cdots, a_j, \hat{a_{j+1}}, \cdots, \hat{a_{j+\ell}}, a_{j+\ell+1}, \cdots, a_{k+1})$ and
		\item $\fb' = (b_1', \cdots, b_{k-\ell}')$.
	\end{itemize}
	Thus, $\widetilde{\mathcal{A}}_{k+1}(\fa,\fb)$ can be identified with $\widetilde{\mathcal{A}}_{k+1-\ell}(\fa', \fb')$ via
	\begin{equation}\label{identificationof}
		\begin{array}{cccc}
				  & \widetilde{\mathcal{A}}_{k+1}(\fa,\fb) & \rightarrow & \widetilde{\mathcal{A}}_{k+1-\ell}(\fa', \fb') \\
				                    &               		(z_1, \cdots, z_j, 0, \cdots, 0, z_{j+\ell+1}, \cdots, z_k)             & \mapsto    &  		(z_1, \cdots, z_j,  \hat{z_{j+1}}, \cdots, \hat{z_{j+\ell}}, z_{j+\ell+1}, \cdots, z_k).\\
		\end{array}
	\end{equation}
	Therefore, we obtain the following corollary.
	
	\begin{corollary}\label{corollary_fiber_zero_a_b}
		For sequences $\fa = (a_1, \cdots, a_{k+1})$ and $\fb = (b_1, \cdots, b_k)$ of real numbers obeying~\eqref{patternaandb}, suppose that there exist $j, \ell \in \Z_{>0}$ such that
		\[
			b_j > a_{j+1} = b_{j+1} = \cdots = a_{j+\ell} = b_{j+\ell} > a_{j+\ell+1}.
		\]
		Setting $\fa'$ (respectively $\fb'$) to be the sequence of real numbers obtained by deleting $a_{j+1}, \cdots, a_{j+\ell}$ (respectively $b_{j+1}, \cdots, b_{j+\ell}$), $\widetilde{\mathcal{A}}_{k+1}(\fa, \fb)$ can be identified with $\widetilde{\mathcal{A}}_{k+1-\ell}(\fa', \fb')$ under~\eqref{identificationof}.
	\end{corollary}
	
	The following two lemmas below are about the cases of (4) in \eqref{eq:inequality2} and (1) of \eqref{eq:inequality1}.
	Since they can be proven by using exactly same method of the proof of Lemma \ref{lemma_fiber_zero_a_b}, we omit the proofs.
	
	\begin{lemma}[case (1) of \eqref{eq:inequality1}]\label{lemma_fiber_zero_b_a}
		Suppose that $a_j > b_j = a_{j+1} = \cdots = a_{j+\ell} > b_{j+\ell}.$ Then
		$(z_1, \cdots, z_k) \in \widetilde{\mathcal{A}}_{k+1}(\fa,\fb)$ if and only if
		\begin{itemize}
			\item $z_{j} = \cdots =z_{j+\ell-1} = 0$, and
			\item $(z_1, \cdots, z_{j-1}, z_{j+\ell}, \cdots, z_{k}) \in \widetilde{\mathcal{A}}_{k-\ell+1}(\fa',\fb')$
		\end{itemize}
		where $\fa'$ is obtained by deleting $\{a_{j+1}, \cdots, a_{j+\ell} \}$ from $\fa$ and
		$\fb' = \{b_1', \cdots, b_{k-\ell}'\}$ is obtained by deleting $\{b_{j}, \cdots, b_{j+\ell-1} \}$ from $\fb$.
	\end{lemma}
	
	\begin{lemma}[case (4) of \eqref{eq:inequality2}]\label{lemma_fiber_zero_a_a}
		Suppose that $b_j > a_{j+1} = \cdots = a_{j+\ell} > b_{j+\ell}.$ Then
		$(z_1, \cdots, z_k) \in \widetilde{\mathcal{A}}_{k+1}(\fa,\fb)$ if and only if
		\begin{itemize}
			\item $z_{j+1} = \cdots =z_{j+\ell-1} = 0$, and
			\item $(z_1, \cdots, z_{j}, z_{j+\ell}, \cdots, z_{k}) \in \widetilde{\mathcal{A}}_{k-\ell+2}(\fa',\fb')$
		\end{itemize}
		where $\fa' = \{a_1', \cdots, a_{k-\ell+2}'   \}$ is obtained by deleting $\{a_{j+2}, \cdots, a_{j+\ell} \}$ from $\fa$ and
		$\fb' = \{b_1', \cdots, b_{k-\ell+1}' \}$ is obtained by deleting $\{b_{j+1}, \cdots, b_{j+\ell-1} \}$ from $\fb$.
	\end{lemma}
	
	It remains to take care of the case (2) of \eqref{eq:inequality1}.
	
	\begin{lemma}[case (2) of \eqref{eq:inequality1}]\label{lemma_fiber_sphere_b_b}
		Suppose that
		\[
			a_j > b_j = a_{j+1} = \cdots = a_{j+\ell} = b_{j+\ell} > a_{j+\ell+1}
		\]
		Then there exists a unique positive real number $C_j > 0$ such that
		\[
			|z_j|^2 + \cdots + |z_{j+\ell}|^2 = C_{j}.
		\]
		for any $(z_1, \cdots, z_k) \in \widetilde{\mathcal{A}}_{k+1}(\fa,\fb)$.
	\end{lemma}

	\begin{proof}
		By Lemma \ref{lemma_fiber_zero_a_b}, \ref{lemma_fiber_zero_b_a}, and \ref{lemma_fiber_zero_a_a},
		we may reduce $\fa = \{a_1, \cdots, a_{k+1} \}$ and $\fb = \{ b_1, \cdots, b_k \}$ to
		\[
			\fa' = \{a_1', \cdots, a_{r+1}' \}, \quad \text{and} \quad \fb' = \{b_1', \cdots, b_r' \}
		\]
		for some $r > 0$ so that
		there are no subsequences of type (3), (4) of (\ref{eq:inequality2}) or (1) of \eqref{eq:inequality1} in $(a_1' \geq b_1' \geq \cdots \geq a_r' \geq b_r' \geq a_{r+1}')$.
		Also, the above series of lemmas implies that $\widetilde{\mathcal{A}}_{k+1}(\fa,\fb)$ is identified with $\widetilde{\mathcal{A}}_{r+1}(\fa',\fb')$ under the
		 identification of $w = (w_1, \cdots, w_r)$ with
		suitable sub-coordinates $(z_{i_1}, \cdots, z_{i_r})$ of $(z_1, \cdots, z_{k+1})$.
		Therefore, it is enough to prove Lemma \ref{lemma_fiber_sphere_b_b}
		in the case where $(a_1, b_1, \cdots, a_k, b_k, a_{k+1})$ does not contain any pattern of type
	        (3), (4) of \eqref{eq:inequality2} or (1) of \eqref{eq:inequality1}.
		
		We temporarily assume that
		\[
			a_j > b_j = a_{j+1} = \cdots = a_{j+\ell} = b_{j+\ell} > a_{j+\ell+1}.
		\]
		is the unique pattern
		of type (2) of \eqref{eq:inequality1} in $(a_1, b_1, \cdots, a_k, b_k, a_{k+1})$.
		Then the equation (\ref{equation_det}) is written as
		\[
			\det(xI - Z) = (x-b_j)^{\ell} \cdot g(x)
		\]
		where
		\[
			\begin{array}{ccl}
				g(x) & = & (x-z_{k+1}) \displaystyle \prod_{\substack{i=1 \\ i \not\in \{j+1, \cdots, j+\ell\}}}^k (x-b_i)
				                   -  \sum_{i=1}^k \left( \frac{|z_i|^2}{x-b_i} \cdot \prod_{\substack{m=1 \\ m \not\in \{j+1, \cdots, j+\ell\}}}^k (x-b_m) \right) \\[0.2em]
			\end{array}\vspace{0.2cm}
		\]
		is a polynomial of degree $(k-\ell+1)$ with respect to $x$.
		For the sake of simplicity, we denote by
		\[
			\displaystyle B(x) := \prod_{\substack{m=1 \\ m \not\in \{j+1, \cdots, j+\ell\}}}^k (x-b_m).
		\]
		Since our assumption says that $\displaystyle \frac{1}{x-b_{j}} = \cdots = \frac{1}{x-b_{j+\ell}}$,
		the second part of $g(x)$ can be written by
		\begin{displaymath}
			\begin{array}{ccl}
			\displaystyle \sum_{i=1}^k \left( \frac{|z_i|^2}{x-b_i} \cdot B(x) \right) & = &
			\displaystyle
			\left( \displaystyle \frac{(|z_j|^2 + \cdots + |z_{j+\ell}|^2)}{x-b_j} + \sum_{\substack{i=1 \\ i \not\in \{j, \cdots, j+\ell\}}}^k  \frac{|z_i|^2}{x-b_i} \right) \cdot B(x)  \\
			\end{array}
		\end{displaymath}
		By substituting $\fa' = \{a_1' \cdots, a_{k+1-\ell}'\}$ and $\fb' = \{b_1', \cdots, b_{k-\ell}' \}$ where
		\begin{itemize}
			\item $a_i' = a_i$ \quad and  \quad $b_i' = b_i$ \quad for $1 \leq i \leq j$,
			\item $a_i' = a_{i+\ell}$ \quad for $j+1 \leq i \leq k-\ell+1$, and
			\item $b_i' = b_{i+\ell}$ \quad for $j+1 \leq i \leq k-\ell$,
		\end{itemize}
		$\fa' $ and $\fb'$ satisfies
		\[
			a_1' > b_1' > \cdots > a_{k-\ell}' > b_{k-\ell}' > a_{k+1-\ell}'.
		\]
		Then we have $g(x) = \det(xI - Z_{(\fa',\fb')}(w))$ where
		\begin{itemize}
			\item $|w_i|^2 =|z_i|^2$ \quad for $1 \leq i \leq j-1$,
			\item $|w_j|^2 = |z_j|^2 + \cdots + |z_{j+\ell}|^2$,
			\item $|w_i|^2 = |z_{i+\ell}|^2$ \quad for $j+1 \leq i \leq k-\ell$, and
			\item $w_{k-\ell+1} = \sum_{i=1}^{k-\ell+1} a_i' - \sum_{i=1}^{k-\ell}b_i' = \sum_{i=1}^{k+1} a_i - \sum_{i=1}^{k}b_i = z_{k+1}.$
		\end{itemize}
		Thus Lemma \ref{lemma_fiber_torus} implies that there exist positive constants $C_1, \cdots, C_{k-\ell}$ such that
		\[
			\widetilde{\mathcal{A}}_{k-\ell+1}(\fa',\fb') = \left\{(w_1, \cdots, w_{k-\ell}) \in \C^{k-\ell} ~|~ |w_j|^2 = C_j,~j=1,\cdots,k-\ell \right\}.
		\]
		In particular, we have $|w_j|^2 = |z_j|^2 + \cdots + |z_{j+\ell}|^2 = C_j$.
		
		It remains to prove the case where  $(a_1, b_1, \cdots, a_k, b_k, a_{k+1})$ contains more than one pattern of
                type (2) of \eqref{eq:inequality1}. However, since all patterns of
        	type (2) of \eqref{eq:inequality1} are disjoint from one another, we can apply the same argument to each
		pattern inductively. This completes the proof.
	\end{proof}
	
	Now, we are ready to prove Proposition \ref{proposition_A_k_S_k_bundle}.
	
	\begin{proof}[Proof of Proposition \ref{proposition_A_k_S_k_bundle}]
		For a given sequence $a_1\geq  b_1 \geq  \cdots \geq a_k \geq b_k \geq a_{k+1}$, let us consider the $W$-shaped block $W_k(a,b)$ with
		walls defined by strict inequalities $a_j > b_j$ or $b_j > a_{j+1}$ for each $j=1,\cdots, k$.  (See Figure \ref{figure_block_a_b}.)
		Note that each pattern of type (2) in \eqref{eq:inequality1} corresponds to an $M$-shaped block inside of $W_k(\fa,\fb)$.
		More specifically, if
		\[
			a_j > b_j = a_{j+1} = \cdots = a_{j+\ell} = b_{j+\ell} > a_{j+\ell+1}.
		\]
		is one of patterns of type (2) in \eqref{eq:inequality1} for some $j$, then it corresponds to a simple closed region which is an $M$-shaped block $M_{\ell+1}$.
		In particular, we have
		\[
			|M_{\ell+1}| = 2\ell + 1= \dim \left\{ (z_j, \cdots, z_{j+\ell}) \in \C^{\ell+1}~|~ |z_j|^2 + \cdots + |z_{j+\ell}|^2 = C_j \right\} = \dim S^{2\ell+1}.
		\]
		for a positive real number $C_j$.
		Combining the series of Lemma \ref{lemma_fiber_zero_a_b}, \ref{lemma_fiber_zero_b_a}, \ref{lemma_fiber_zero_a_a}, and \ref{lemma_fiber_sphere_b_b},
		we see
		$\widetilde{\mathcal{A}}_{k+1}(\fa,\fb) \cong S_{k}(\fa,\fb)$.
		Note that $\mathcal{A}_{k+1}(\fa,\fb)$ is an $\widetilde{\mathcal{A}}_{k+1}(\fa,\fb)$-bundle over $\mathcal{O}_\fb$ via
		\begin{equation}\label{equation_fibration_over_flag_manifold}
			\begin{array}{ccccc}
					\rho_{k+1}  \colon & \mathcal{A}_{k+1}(\fa,\fb) & \rightarrow & \mathcal{O}_\fb \\
					                    &                x             & \mapsto    &  x^{(k)}.\\
			\end{array}
		\end{equation}
		Thus $\mathcal{A}_{k+1}(\fa,\fb)$ is an $S_k(\fa,\fb)$-bundle over $\mathcal{O}_\fb$.
	\end{proof}

	\begin{corollary}\label{corollary_first_part}
		Let $f$ be a face of the Gelfand-Cetlin polytope $\Delta_\lambda$ and $\gamma$ be the face of the ladder diagram
		$\Gamma_\lambda$ corresponding to $f$. For any point $\textbf{\textup{u}}$ in the interior of $f$, the fiber $\Phi_\lambda^{-1}(\textbf{\textup{u}})$ has
		an iterated bundle structure given by
		\[
			\Phi_\lambda^{-1}(\textbf{\textup{u}}) = \bar{S_{n-1}}(\gamma) \xrightarrow {p_{n-2}} \bar{S_{n-2}}(\gamma) \rightarrow \cdots
			 \xrightarrow{p_1} \bar{S_1}(\gamma) = S_1(\gamma)
		\]
		where
		$p_{k-1} : \bar{S_k}(\gamma) \rightarrow \bar{S_{k-1}}(\gamma)$ is an $S_k(\gamma)$-bundle over $\bar{S_{k-1}}(\gamma)$ for $k=1,\cdots, n-1$.
		In particular, $\Phi_\lambda^{-1}(\textbf{\textup{u}})$ is of dimension
		\[
			\dim \Phi_\lambda^{-1}(\textbf{\textup{u}}) = \sum_{k=1}^{n-1} \dim S_k(\gamma).
		\]
	\end{corollary}
	\begin{proof}
		For each $(i,j) \in \Z^2_{\geq 1}$, we denote by $\Phi_\lambda^{i,j} \colon \mathcal{O}_\lambda \rightarrow \R$ be the component of $\Phi_\lambda$ which
		corresponds to the unit box $\square^{(i,j)}$ of $\Gamma_\lambda$
		whose top-right vertex is located at $(i,j)$ in $\Gamma_\lambda$.
		For each $k \in \Z_{>1}$ and $1 \leq i \leq k$,
		let us define
		\[
			a_i(k) := \Phi_\lambda^{i, k+1-i}(\textbf{\textup{u}}) \quad \mathrm{and} \quad b_i(k) := a_i(k-1).
		\]
		provided with $a_1(1) := \Phi_\lambda^{1,1}(\textbf{\textup{u}})$. Let $\fa{(k)} := (a_1(k), \cdots, a_k(k))$ and $\fb{(k)} := (b_1(k), \cdots, b_{k-1}(k))$.
		By applying Proposition \ref{proposition_A_k_S_k_bundle} repeatedly and observing that
		$S_k(\gamma) = S_k(\fa{(k+1)}, \fb{(k+1)})$, we describe the fiber $\Phi_\lambda^{-1}(\textbf{\textup{u}})$ as the total space of an iterated bundle as in~\eqref{figure_iterated_bundle}.
		The dimension formula immediately follows.
	\end{proof}

\begin{equation}\label{figure_iterated_bundle}
	\xymatrix{
		\quad \overline{S_{n-1}(\gamma)} \quad \ar[r] \ar[d]^{{p_{n-1}}}& \quad \cdots \quad \ar[r] \ar[d] & \iota_{(n-1)}^* \left( \mcal{A} (\frak{a}{(n)} , \frak{b}{(n)})\right)   \ar[r] \ar[d]^{\iota^*_{(n-1)} \rho_n}  & \,  \mcal{A} (\frak{a}{(n)}, \frak{b}{(n)}) \, \ar@{^{(}->}[r]^{\,\,\,\,\,\,\,\,\, \iota_{(n)}} \ar[d]^{\rho_n} & \mcal{O}_{\frak{a}{(n)}}\\
		\quad \overline{S_{n-2}(\gamma)} \quad \ar[r] \ar[d]^{{p_{n-2}}} & \quad \cdots \quad \ar[r] \ar[d] & \,\,\, \mcal{A}(\frak{a}{(n-1)}, \frak{b}{(n-1)}) \,\,\, \ar@{^{(}->}[r]^{\,\,\,\,\,\,\,\,\,\,\,\,\,\,\,\, \iota_{(n-1)}} \ar[d]^{\rho_{n-1}}   &  \quad \mcal{O}_{\frak{a}{(n-1)}} \quad  & \\
		\quad \quad \vdots \quad \quad \ar[r] \ar[d] & \quad {\cdots} \quad   \ar@{^{(}->}[r] \ar[d] & \quad \quad \mcal{O}_{\frak{a}{(n-2)}} \quad \quad &  &\\
		\overline{S_1(\gamma)} = \mcal{A}(\frak{a}{(2)}, \frak{b}{(2)}) \ar@{^{(}->}[r]^{\quad \quad \,\,\,\,\,\,\,\,\, \iota_{(2)}} \ar[d]^{{p_1} = \rho_{2}} & \quad \cdots \quad &  & & \\
		\overline{S_0(\gamma)} = \mcal{O}_{\frak{a}{(1)}} &   &  & &}
\end{equation}
\vspace{0.2cm}

		To complete the proof of Theorem \ref{theorem_main}, it remains to verify that $\Phi^{-1}_\lambda(\textbf{\textup{u}})$ is an isotropic submanifold
		of $(\mathcal{O}_\lambda, \omega_\lambda)$ for every $\textbf{\textup{u}} \in \Delta_\lambda$.
		Recall the definition of the KKS from Section~\ref{ssecSymplecticStructureOnMathcalOLambda}.
		For a fixed positive integer $k>1$, let $\fa = (a_1, \cdots, a_{k+1})$ and $\fb = (b_1, \cdots, b_k)$ be sequences of real numbers satisfying~\eqref{patternaandb}
		and let $\rho_{k+1} \colon \mathcal{A}_{k+1}(\fa,\fb) \rightarrow \mathcal{O}_\fb$ be the map defined by $\rho_{k+1}(x) = x^{(k)}$.
		Then $\rho_{k+1}$ makes $\mathcal{A}_{k+1}(\fa,\fb)$ into a $\widetilde{\mathcal{A}}_{k+1}(\fa,\fb)$-bundle over $\mathcal{O}_\fb$. 
		See Proposition \ref{proposition_A_k_S_k_bundle}.

		For any $x \in \mathcal{A}_{k+1}(\fa,\fb) \subset \mathcal{O}_\fa \subset \mathcal{H}_{k+1}$, let $V_x \subset T_x \mathcal{A}_{k+1}(\fa,\fb)$
		be the vertical tangent space
		at $x$ with respect to $\rho_{k+1}$ and let $H_x$ be the subspace of $T_x \mathcal{A}_{k+1}(\fa,\fb)$ generated by $U(k)$-action
		where
		$U(k)$ acts on $\mathcal{A}_{k+1}(\fa,\fb)$ as a subgroup of $U(k+1)$ via the embedding
		\[
		\begin{array}{ccccl}
				i_k \colon& U(k) & \hookrightarrow & U(k+1) \\[0.8em]
				        &  A & \mapsto &
		\begin{pmatrix}A & 0\\[0.3em]
                  	0 & 1\\
		  \end{pmatrix}. \\
		 \end{array}
		\]
                Then we can see that
                \[
                	(\rho_{k+1})_*|_{H_x} \colon H_x \rightarrow T_{\rho_{k+1}(x)} \mathcal{O}_\fb
		\]
		is surjective since $\rho_{k+1}$ is $U(k)$-invariant and the $U(k)$-action on $\mathcal{O}_\fb$ is transitive.
		Let  $(i_k)_* \colon \mathfrak{u}(k) \rightarrow \mathfrak{u}(k+1)$ be the induced Lie algebra monomorphism.
                Then the kernel of $\ker (\rho_{k+1})_*|_{H_x}$ is given by
                 \[
                 	\begin{array}{ccl}
	                 	\ker (\rho_{k+1})_*|_{H_x} & = & \left\{ [(i_k)_*(X), x] ~|~ X \in \mathfrak{u}(k) ,~ [X, x^{(k)}]= 0   \right\} \\
	                 	                                                              & = & \left\{ [(i_k)_*(X), x] ~|~ X \in T_e U(k)_{x^{(k)}} \right\} \\
                 	\end{array}
		\]
		where $x^{(k)}$ is the $(k \times k)$ leading principal minor of $x$ and $U(k)_{x^{(k)}}$ is the stabilizer of $x^{(k)} \in \mathcal{O}_\fb$ for the $U(k)$-action.

		From now on, we assume that
		$U(k)$ acts on ${\mathcal{A}}_{k+1}(\fa,\fb)$ via $i_k$, unless stated otherwise.
		\begin{lemma}\label{lemma_transitive}
			$U(k)$ acts transitively on $\mathcal{A}_{k+1}(\fa,\fb)$.
		\end{lemma}
		\begin{proof}
			Since any element $x \in \mathcal{A}_{k+1}(\fa,\fb)$ is conjugate to an element of the following form
			\[
			Z_{(\fa,\fb)}(z) =   \begin{pmatrix} b_1 & & 0 & \bar{z}_1\\
                	  & \ddots & & \vdots\\
                  	0 & & b_k & \bar{z}_k\\
                  	z_1 & \hdots & z_k & z_{k+1}
		  \end{pmatrix} \in \widetilde{\mathcal{A}}_{k+1}(\fa,\fb) \subset \mathcal{A}_{k+1}(\fa,\fb)
		  \]
		with respect to the $U(k)$-action, it is enough to show that the isotropy subgroup $U(k)_{I_\fb}$ of $I_\fb$
		acts on $\widetilde{\mathcal{A}}_{k+1}(\fa,\fb)$ transitively where $I_\fb$ is the diagonal matrix whose $(i,i)$-th entry is $b_i$ for $i=1,\cdots,k$.
		
		Now, let us assume that
		\[
			b_{i_0} := b_1 = \cdots = b_{i_1} > b_{i_1 + 1} = \cdots = b_{i_2} > \cdots > b_{i_{r-1} + 1} = \cdots = b_{i_r} := b_k.
		\]
		for some $r \geq 1$ provided with $i_0 = 0$ and $i_r = k$.
		Then it is not hard to show that $U(k)_{I_b} = U(k_1) \times \cdots \times U(k_r)$ where $k_j = i_j - i_{j-1}$ for $j=1,\cdots, r$.
		For each $j$, we know that each $(z_{i_j+1}, \cdots, z_{i_{j+1}}) \in \C^{k_{j+1}}$ satisfies either
		\begin{itemize}
			\item $|z_{i_j+1}|^2 + \cdots + |z_{i_{j+1}}|^2 = 0$, or
			\item $|z_{i_j+1}|^2 + \cdots + |z_{i_{j+1}}|^2 = C_{j+1}$ for some positive constant $C_{j+1} \in \R_{>0}$
		\end{itemize}
		by Lemma \ref{lemma_fiber_zero_a_b}, \ref{lemma_fiber_zero_b_a}, \ref{lemma_fiber_zero_a_a}, and \ref{lemma_fiber_sphere_b_b}.
		In the latter case, $U(k)_{I_b}$-action is written as
		\[
		\begin{pmatrix} g & 0\\
                	  0 & 1\\
		\end{pmatrix}
		\cdot
		\begin{pmatrix} I_b & \bar{z}^t\\
                	  z & z_{k+1}\\
		\end{pmatrix}
		\cdot
		\begin{pmatrix} g^{-1} & 0\\
                	  0 & 1\\
		\end{pmatrix}
		=
		\begin{pmatrix} I_b & g\bar{z}^t\\
                	  zg^{-1} & z_{k+1}\\
		\end{pmatrix}
		\]
		for every $g \in U(k)_{I_b}$ and $z = (z_1, \cdots, z_k)$.
		Note that every $g \in U(k)_{I_b}$ is of the form
		\[
		g = \begin{pmatrix} g_1 & 0 & 0 & \cdots & 0\\
                	  0 & g_2 & 0 & \cdots & 0\\
                	  0 & 0 & \ddots & \cdots & \vdots \\
                	  \vdots & \vdots & \vdots & \ddots & \vdots\\
                	  0 & \cdots & \cdots & \cdots & g_r
		\end{pmatrix}
		\]		
		where $g_i \in U(k_i)$ for $i=1,\cdots,r$. Thus each $g \in U(k)_{I_b}$ acts on $(z)_{j+1} := (z_{i_j+1}, \cdots, z_{i_{j+1}}) \in \C^{k_{j+1}}$
		by $(z)_{j+1} \cdot g_{j+1}^{-1}$ which is equivalent to the standard linear $U(k_{j+1})$-action on the sphere
		$S^{2k_{j+1} - 1} \subset \C^{k_j}$ of radius $\sqrt{C_{j+1}}$.
		Therefore, the action is transitive.
		\end{proof}
		
		\begin{lemma}\label{lemma_isotropic}
			For each $x \in \mathcal{A}_{k+1}(\fa,\fb)$ and any $\xi, \eta \in T_x \mathcal{A}_{k+1}(\fa,\fb)$, we have
			\[
				(\omega_\fa)_x(\xi, \eta) = (\omega_\fb)_{\rho_{k+1}(x)}((\rho_{k+1})_* \xi, (\rho_{k+1})_* \eta).
			\]
			In particular, the vertical tangent space $V_x \subset T_x \mathcal{A}_{k+1}(\fa,\fb)$ of $\rho_{k+1}$ is contained in $\ker (\omega_\fa)_x$.
		\end{lemma}
		
		\begin{proof}
			Note that Lemma \ref{lemma_transitive} implies that any tangent vector in $T_x \mathcal{A}_{k+1}(\fa,\fb)$ can be written as $[(i_k)_*(X), x]$ for some
			$X \in \mathfrak{u}(k)$ where
			\[
				(i_k)_*(X) = 		\begin{pmatrix}X & 0\\[0.3em]
                  	            0 & 0\\
		  \end{pmatrix} \in \mathfrak{u}(k+1).
			\]
			Thus for any $\xi, \eta \in T_x \mathcal{A}_{k+1}(\fa,\fb)$, there exist $X, Y \in \mathfrak{u}(k)$ such that
			\[
				\xi = [(i_k)_*(X), x], \quad \eta = [(i_k)_*(Y), x].
			\]
			Therefore, we have
			\[
				(\omega_\fa)_x(\xi, \eta) = \mathrm{tr}(ix[(i_k)_*(X), (i_k)_*(Y)]) = \mathrm{tr}(ix^{(k)}[X,Y]) = (\omega_\fb)_{x^{(k)}} ([X, x^{(k)}] , [Y, x^{(k)}])	
			\]
			since the $(k+1, k+1)$-th entry of the matrix $x[(i_k)_*(X), (i_k)_*(Y)]$ is zero by direct computation.
			Since $\rho_{k+1}$ is $U(k)$-invariant, we obtain that $[X, x^{(k)}] = (\rho_{k+1})_*(\xi)$ and
			$[Y, x^{(k)}] = (\rho_{k+1})_*(\eta)$.			
			This completes the proof.
		\end{proof}

	\begin{proposition}\label{proposition_second_part}
		For any $\textbf{\textup{u}} \in \Delta_\lambda$, the fiber $\Phi_\lambda^{-1}(\textbf{\textup{u}})$ is an isotropic submanifold of
		$(\mathcal{O}_\lambda,\omega_\lambda)$, i.e.,
		$\omega|_{\Phi_\lambda^{-1}(\textbf{\textup{u}})} = 0$.
	\end{proposition}

	\begin{proof}
		Suppose that $\gamma$ is a face of $\Gamma_\lambda$
		such that the corresponding face $f_\gamma$ contains $\textbf{\textup{u}}$ in its interior.
		Let $x \in \Phi^{-1}_\lambda(\textbf{\textup{u}})$ and let $\xi, \eta \in T_x \Phi^{-1}_\lambda(\textbf{\textup{u}})$.
		Then Corollary \ref{corollary_first_part} says that $\Phi^{-1}_\lambda(\textbf{\textup{u}})$ is the total space of an iterated bundle
		\[
				\Phi_\lambda^{-1}(\gamma) = \bar{S_{n-1}}(\gamma) \xrightarrow {p_{n-2}} \bar{S_{n-2}}(\gamma) \rightarrow \cdots
		 		\xrightarrow{p_1} \bar{S_1}(\gamma) = S_1(\gamma).
		\]
		described in~\eqref{figure_iterated_bundle}.

		For each integer $k$ with $ 1 < k \leq n$ and $1 \leq i \leq k$,
		let us define
		\[
			a_i(k) := \Phi_\lambda^{i, k+1-i}(\textbf{\textup{u}}).
		\]
		provided with $a_1(1) := \Phi_\lambda^{1,1}(\textbf{\textup{u}})$ and let $\fa(k) := (a_1(k), \cdots, a_k(k))$.
		In particular, we have $\fa(n) = \lambda = \{ \lambda_1, \cdots, \lambda_n \}$.
		Then
		Lemma \ref{lemma_isotropic} implies that
		\[
			(\omega_{\fa(n)})_x (\xi, \eta) = (\omega_{\fa(n-1)})_{\rho_n(x)}((\rho_n)_*(\xi), (\rho_n)_*(\eta)).
		\]
		Since
		$p_{n-2}$ is the restriction of $\rho_n$ to $\bar{S}_{n-1}(\gamma) \subset \mathcal{A}(\fa(n), \fa(n-1))$,
		both $(\rho_n)_*(\xi)$ and $(\rho_n)_*(\eta)$ are lying on $T_{\pi_{n-2}(x)} \bar{S}_{n-2}(\gamma) \subset
		T_{p_{n-2}(x)} \mathcal{A}(\fa(n-1), \fa(n-2))$.

		Thus we can apply Lemma \ref{lemma_isotropic}
		inductively so that we have
		\[	\begin{array}{ccl}
			(\omega_{\fa(n)})_x (\xi, \eta) & = & (\omega_{\fa(n-1)})_{\rho_n(x)}((\rho_n)_*(\xi), (\rho_n)_*(\eta)) \\
                                             & = & (\omega_{\fa(n-2)})_{\rho_{n-1} \circ \rho_n(x)}((\rho_{n-1} \circ \rho_n)_*(\xi), (\rho_{n-1} \circ \rho_n)_*(\eta)) \\
                                             & = & \cdots \\
			                     & = & (\omega_{\fa(2)})_{\rho_2 \circ \cdots \circ \rho_n(x)}((\rho_2 \circ \cdots \circ \rho_n)_*(\xi), (\rho_2 \circ \cdots \circ \rho_n)_*(\eta))\\
			                     & = & 0.
			\end{array}
		\]
		The last equality follows from $\rho_2 \colon \mathcal{A}(\fa(2), \fa(1)) \rightarrow \{ a_1(1) = \Phi_\lambda^{1,1}(\textbf{\textup{u}}) \} = \mathrm{point}.$
	\end{proof}
	
	\begin{proof}[Proof of Theorem~\ref{theorem_main}]
		It follows from Corollary~\ref{corollary_first_part} and  Proposition~\ref{proposition_second_part}.
	\end{proof}

%------------------------------------------------------------------------
\section{Degenerations of fibers to tori}
\label{secDegenerationsOfFibersToTori}

In this section, we study the topology of GC fibers via toric degenerations and describe how each fiber of a GC system 
degenerates to a torus fiber of a toric moment map. 

Let $\lambda$ be given in \eqref{equation_GC-pattern} and let $f$ be a face of the GC polytope $\Delta_\lambda$ of dimension $r$. 
It is shown in Section ~\ref{secIteratedBundleStructuresOnGelfandCetlinFibers} that a fiber $\Phi_\lambda^{-1}(\textbf{\textup{u}})$ is the total space of an iterated bundle 
	\begin{equation}\label{equation_iterated_S}
		\Phi_\lambda^{-1}(\textbf{\textup{u}}) \cong \bar{S_{n-1}}(\gamma_f) \xrightarrow {p_{n-1}} \bar{S_{n-2}}(\gamma_f)
		 \rightarrow \cdots
		 \xrightarrow{p_2} \bar{S_1}(\gamma_f) \xrightarrow{p_1} \bar{S_0}(\gamma_f) := \mathrm{point}
	\end{equation}
where $\gamma_f$ is the face of  $\Gamma_\lambda$ corresponding to $f$, see Corollary \ref{corollary_first_part}.
The main theorem of this section further claims that 
every $S^1$-factor appeared in any stage of the iterated bundle in \eqref{equation_iterated_S} comes out as a trivial factor. 
Namely, letting $\bar{S_{i}}(\gamma)$ be $(S^1)^{m_i} \times Y_i$ such that $\sum_{i=1}^{n-1} m_i = r$, the fiber is written as the product
$$
\Phi_\lambda^{-1}(\textbf{\textup{u}}) = (S^1)^r \times Y(\textbf{\textup{u}})
$$
where $Y(\textbf{\textup{u}})$ is the total space of an iterated bundle whose fibers at stages are $Y_{\bullet}$'s, see Theorem~\ref{theorem_contraction}. Indeed, $Y(\textbf{\textup{u}})$ shrinks to a point through a toric degeneration, which illustrates how fibers degenerate into toric fibers. As an application, it provides a more concrete description of the GC fiber. Furthermore, we compute the fundamental group and the second homotopy group of $\Phi_\lambda^{-1}(\textbf{\textup{u}})$.

Recall that for a given K\"{a}hler manifold $(X,\omega)$, 
a flat family $\pi \colon \mcal{X} \rightarrow \C$ of algebraic varieties is called a 
{\em toric degeneration} $(X, \omega)$ if there exists an algebraic embedding $i \colon \mcal{X} \hookrightarrow \p^N \times \C$ such that 
\begin{enumerate}
	\item the diagram
		\begin{equation}\label{equation_def_toric_degeneration}
		    \xymatrix{
			      \mcal{X}  ~\ar[drr]_{\pi} \ar@{^{(}->}[rr]^{i}
			      & & \p^N \times \C \ar[d]^{q} \\ 
			      & & \C 
			    }
		\end{equation} commutes where $q \colon \p^N \times \C \rightarrow \C$ is the projection to the second factor, 
	\item $\pi^{-1}(\C^*)$ is isomorphic to $ X \times \C^*$	as a complex variety, and 
	\item For the product K\"{a}hler form $\widetilde{\omega} := \omega_{\mathrm{FS}} \oplus \omega_{\mathrm{std}}$ 
	on $\p^N \times \C$ where 
	$\omega_{\mathrm{FS}}$ is a multiple of the Fubini-Study form on $\p^N$ and $\omega_{\mathrm{std}}$ is the standard K\"{a}hler form 
	on $\C$, 
		\begin{itemize}
			\item $(X_1, \widetilde{\omega}|_{X_1})$ is symplectomorphic to  $(X, \omega)$, and 
			\item $X_0$ is a projective toric variety (in $\p^N$) and $\widetilde{\omega}|_{X_0}$ is a torus invariant K\"{a}hler form 
			denoted by $\omega_0$
		\end{itemize}
	where $X_t := \pi^{-1}(t) \cong i(\pi^{-1}(t)) \subset \p^N \times \{t \}$ is a projective variety for every $t \in \C$. 
\end{enumerate}	

Let $\mcal{X}^{\mathrm{sm}} \subset \mcal{X}$ be the smooth locus of $\mcal{X}$.  
The Hamiltonian vector field, denoted by $\widetilde{V}_\pi$, for the imaginary part $\mathrm{Im}(\pi)$ of $\pi$ is defined on $\mcal{X}^{\mathrm{sm}}$.
By the holomorphicity, $\pi$ satisfies the Cauchy-Riemann equation which induces $\nabla\mathrm{Re}(\pi) = - \widetilde{V}_\pi$ where 
$\nabla$ denotes the gradient vector field with respect to a K\"{a}hler metric associated with $\widetilde{\omega}$. 
We call
\[
	V_\pi := \widetilde{V}_\pi / ||\widetilde{V}_\pi||^2
\]
the {\em gradient-Hamiltonian vector field of $\pi$}, see Ruan \cite{R}. 
Then the one-parameter subgroup generated by $V_\pi$ induces a symplectomorphism 
\begin{equation}\label{equation_symplecto_open_dense}
	\phi \colon (\mcal{U}, \omega) \rightarrow (\phi(\mcal{U}), \omega_0)
\end{equation}
on an open dense subset $\mcal{U}$ of $X$ ($\cong X_1$) such that $\phi(\mcal{U}) = X_0^{\mathrm{sm}}$ and it is extended to 
a surjective continuous map 
\[
	\phi \colon X \rightarrow X_0
\]
defined on the whole $X$, see Harada-Kaveh \cite[Theorem A]{HK} for more details. 

We may also consider a toric degeneration of a K\"{a}hler manifold {\em ``equipped with a completely integrable system''} as follows.
Consider a triple $(X,\omega,\Theta)$ where $\Theta = (\Theta_\alpha)_{\alpha \in \mcal{I}}$ is a (continuous) completely integrable system on $(X, \omega)$
and $\mcal{I}$ is the index set for $\Theta$ such that $|\mcal{I}| = \dim_\C X_0$. 
We call $\pi \colon \mcal{X} \rightarrow \C$ a {\em toric degeneration} of $(X, \omega, \Theta)$ if $\pi$ is a toric degeneration of 
$(X,\omega)$ and $\Theta = \Phi \circ \phi$ where $\Phi \colon X_0 \rightarrow \R^{|\mcal{I}|}$ is a toric moment map on $(X_0, \omega_0)$, see \cite[Definition 1.1]{NNU}.
\begin{equation}\label{equation_toric_degeneration_diagram_Section 7}
	\xymatrix{
		(X_1,\omega_1)\cong (X,\omega)  \ar[dr]_{\Theta} \ar[rr]^{ \phi}
                            & & (X_0, \omega_0) \subset \p^N \times \{0\}
      \ar[dl]^{\Phi} \\
   & \Delta_0 &}
\end{equation}
The Hamiltonian vector field of each component $\Phi_\alpha$ of $\Phi$ ($ = \{ \Phi_\alpha \}_{\alpha \in \mcal{I}}$) is globally defined on $X_0$, even though $X_0$ is singular, by the following reason. 
Note that $X_0 \subset \p^N \times \{0 \}$ is a projective toric variety, which is the Zariski closure of the single $(\C^*)^{|\mcal{I}|}$-orbit on $X_0$.
The $(\C^*)^{|\mcal{I}|}$-action on $X_0$ extends to the linear Hamiltonian action on $\p^N$ with respect to $\omega_{\mathrm{FS}}$. 
We denote by $(S^1)^{|\mcal{I}|}$ the maximal compact subgroup of $(\C^*)^{|\mcal{I}|}$, by $\widetilde{\Phi} = (\widetilde{\Phi}_\alpha)_{\alpha \in \mcal{I}}$
a moment map for the $(S^1)^{|\mcal{I}|}$-action on $\p^N$, and by $\xi_\alpha$ the fundamental vector field of $\widetilde{\Phi}_\alpha$ on $\p^N$ for each $\alpha \in \mcal{I}$.
Then each component $\Phi_\alpha$ coincides with the restriction of $\widetilde{\Phi}_\alpha$ to $X_0$. 
Since $X_0$ is $T^{|\mcal{I}|}$-invariant, the trajectory of the flow of $\xi_\alpha$ passing through any point of $X_0$ 
should be on $X_0$. In other words, the restriction $\xi_\alpha|_{X_0}$ should be tangent\footnote{Every toric variety is a stratified space \cite{LS}
so that each point in $X_0$ is contained in an open smooth stratum and each vector field $\xi_\alpha$ is tangent to the stratum.} 
to $X_0$, and therefore the Hamiltonian vector field of $\Phi_\alpha$ is defined on the whole $X_0$. 

Now, let $\mcal{V}_\alpha$ be the open dense subset of $X$
on which $\Theta_\alpha$ is smooth. 
Then the Hamiltonian vector field, denoted by $\zeta_\alpha$, of $\Theta_\alpha$ is defined on $\mcal{V}_\alpha$. 
For any subset $\mcal{I}' \subset \mcal{I}$, we let
\begin{equation}\label{domainofsmoothofcomponent}
	\displaystyle \mcal{V}_{\mcal{I}'} := \displaystyle \bigcap_{\alpha \in \mcal{I}'} \mcal{V}_\alpha
\end{equation}
so that the Hamiltonian vector field of $\Theta_\alpha$ is defined on $\mcal{V}_{\mcal{I}'}$ for every $\alpha \in \mcal{I}'$. 
If $\Theta_\alpha$ is a periodic Hamiltonian, i.e., if $\Theta_\alpha$ generates a circle action for every $\alpha \in \mcal{I}'$, 
then $\mcal{V}_{\mcal{I}'}$ admits the $T^{|\mcal{I}'|}$-action generated by $\{ \zeta_\alpha \}_{\alpha \in \mcal{I}'}$. 
Note that $\mcal{V}_{\mcal{I}'}$ is open dense in $X$ so that $\mcal{U} \cap \mcal{V}_{\mcal{I}'}$ 
is also open dense in $X$ where $\mcal{U}$ is in \eqref{equation_symplecto_open_dense}.

\begin{lemma}\label{lemma_equivariant}
	For any $\alpha \in \mcal{I}$ and $p \in \mcal{V}_\alpha$, we have 
	\[
		\phi(\exp(t \, \zeta_\alpha) \cdot p) =\exp(t \, \xi_\alpha) \cdot \phi(p)
	\]
	for every $t \in \R$.  
\end{lemma}
\begin{proof}
	Fix $\alpha \in \mcal{I}$.
	From the fact that 
	\begin{itemize}
		\item $\phi^*\omega_0 = \omega$ on $\mcal{U} \cap \mcal{V}_\alpha$, and 
		\item $\Theta = \Phi \circ \phi$, 
	\end{itemize}	
	it follows that
	\[
		\omega_0(\phi_*(\zeta_\alpha), \phi_*(\cdot)) = \omega(\zeta_\alpha, \cdot ) = d\Theta_\alpha(\cdot) = d\Phi_\alpha \circ d\phi(\cdot) = \omega_0(\xi_\alpha, \phi_*(\cdot))
	\]
	so that $\phi_*(\zeta_\alpha) = \xi_\alpha$ on $\mcal{U} \cap \mcal{V}_\alpha$.
	Since $\mcal{U} \cap \mcal{V}_\alpha$ is open dense in $\mcal{V}_\alpha$ and $\xi_\alpha$ is defined on whole $X_0$, the equality 
	\[
		\phi_*(\zeta_\alpha) = \xi_\alpha
	\]
	holds on $\mcal{V}_\alpha$. 
	This completes the proof by the uniqueness of a solution of first-order ODE's.  
\end{proof}

Let $\mcal{I}' \subset \mcal{I}$ and suppose that $\Theta_\alpha$ is a periodic Hamiltonian on $\mcal{V}_{\mcal{I}'}$ for every $\alpha \in \mcal{I}'$. 
Since $\Phi_\alpha$ is also a periodic Hamiltonian on $X_0$, we deduce the following immediately from Lemma \ref{lemma_equivariant}.

\begin{corollary}\label{corollary_torus_equivariant}
	Let $\mcal{I}' \subset \mcal{I}$ such that $\{ \Theta_\alpha \}_{\alpha \in \mcal{I}'}$ are periodic Hamiltonians on $\mcal{V}_{\mcal{I}'}$. 
	Then $\phi$ is $T^{|\mcal{I}'|}$-equivariant on $\mcal{V}_{\mcal{I}'}$. 
\end{corollary}

We will apply Corollary \ref{corollary_torus_equivariant} to GC systems. Recall that for any $\lambda$ given in \eqref{equation_GC-pattern}, Nishinou-Nohara-Ueda \cite{NNU}
built a toric degeneration of the GC system $\Phi_\lambda$ on a partial flag manifold $(\mcal{O}_\lambda, \omega_\lambda, \Phi_\lambda)$. 
We first describe their toric degeneration of $(\mcal{O}_\lambda, \omega_\lambda, \Phi_\lambda)$ and a continuous map $\phi \colon \mcal{O}_\lambda \rightarrow X_0$ given in 
\eqref{equation_symplecto_open_dense} obtained by the gradient-Hamiltonian flow,  
where $X_0$ is the central fiber of the toric degeneration. (We also refer the reader to \cite{KM} or \cite{NNU} for more details.) 

In \cite{KM}, Kogan and Miller constructed an $(n-1)$-parameter family $F \colon \mcal{X} \rightarrow \C^{n-1}$, called a {\em toric degeneration in stage}, 
of projective varieties that can be factored as 
\begin{equation}{\label{equation_KoM}}
	F = q \circ \iota, \quad \quad 
	\mcal{X} \stackrel{\iota}  \hookrightarrow  P \times \C^{n-1} \stackrel {q} \rightarrow \C^{n-1}, \quad 
	P = \prod_{k=1}^r \p_{n_k}\footnote{For $m \in \Z_+$, we denote by $\p_m := \p(\wedge^m \C^n) = \p^{{n \choose m}-1}$.}
\end{equation}
where $\iota$ is an algebraic embedding with a K\"{a}hler form $\widetilde{\omega}$ on $P \times \C^{n-1}$ such that 
\begin{itemize}
	\item $(F^{-1}(1,\cdots,1), \widetilde{\omega}|_{F^{-1}(1,\cdots,1)}) \cong (\mcal{O}_\lambda, \omega_\lambda)$ and 
	\item $F^{-1}(0,\cdots,0)$ is isomorphic to the GC toric variety $X_0$ whose moment map image is $\Delta_\lambda$ 
		with the torus-invaraint K\"{a}hler form $\widetilde{\omega}|_{F^{-1}(0,\cdots,0)}$ on $X_0$. 
\end{itemize}
See \cite[Section 5 and Remark 5.2]{NNU} for more details. Following \cite{NNU}, we denote the coordinates of $\C^{n-1}$ by $(t_2, \cdots, t_n)$ and 
$F^{-1}(1,\cdots,1, t = t_k , 0,\cdots, 0)$ by $X_{k,t}$ for $2 \leq k \leq n$ and $t \in \C$. 
Then the set 
\[
	\{ X_{k,t} \}_{2 \leq k \leq n, t \in \C}
\] can be regarded as a family of algebraic varieties in $P$ via the embedding $\iota$ where 
$X_{n,1} \subset P$ is the image of the Pl\"{u}cker embedding of $\mcal{O}_\lambda$
and $X_{2,0} \subset P$ is the toric variety $X_0$ associated with $\Delta_\lambda$.

Let $T^{\frac{n(n-1)}{2}}$ be the compact subtorus of $(\C^*)^{\frac{n(n-1)}{2}} \subset X_0$ and consider a decomposition 
\[
	T^{\frac{n(n-1)}{2}} \cong T^1 \times T^2 \times \cdots \times T^{n-1}.
\]
For each $k = 1,\cdots,n-1$, we denote the $i$-th coordinate of $T^k$ by $\tau_{i,j}$
\footnote{
For the consistency of~\eqref{equation_coordinate}, we use the index $(i,j)$. 
} where $i + j = k+1$. 
Then, each $S^1$-action on $X_0$ corresponding to $\tau_{i,j}$ can be extended to the linear Hamiltonian $S^1$-action on $P$ and we denote a corresponding moment map by 
\begin{equation}{\label{phiijaa}}
	\Phi^{i,j} : P \rightarrow \R.
\end{equation}
On the other hand, recall that $U(n)$ acts on $P$ in a Hamiltonian fashion with a moment map $\mu^{(n)} : P \rightarrow \frak{u}(n)^*$ and consider the sequence of subgroups of $U(n)$ 
\[
	U(1) \subset U(2) \subset \cdots \subset U(n-1) \subset U(n), \quad U(k) := \begin{pmatrix} U(k) & 0 \\ 0 & I_{n-k} \end{pmatrix}.
\]
Then each $U(k)$ also acts on $P$ in a Hamiltonian fashion and the moment map induced from $\mu^{(n)}$ is given by 
\[
	\mu^{(k)} \colon P \rightarrow \frak{u}(k)^* \cong \{\text{$(k \times k)$-Hermitian matrices} \}
\]
for $k=1,\cdots,n-1$. For each pair $(i,j) \in (\Z_+)^2$ with $i + j = k + 1$, define 
\begin{equation}\label{phiijbb}
	\Phi_\lambda^{i,j} \colon P \rightarrow \R
\end{equation}
which assigns the $i$-th largest eigenvalue of $\mu^{(k)}(p)$ for every $p \in P$. 

\begin{remark}
If one follows the notations used in \cite[Section 5]{NNU}, then we may express as
\[
	\tau_{i,j} = \tau_i^{(j)}, \quad \Phi^{i,j} = v^{(k)}_i, \quad \Phi^{i,j}_\lambda = \lambda^{(k)}_i, \quad i+j = k+1.
\]
\end{remark}

An important fact is that 
a fiber of $F$ in \eqref{equation_KoM} is not invariant neither under the $U(n)$-action nor under the $T^1 \times T^2 \times \cdots \times T^{n-1}$-action on $P$, 
but $X_{k,t}$ is invariant under the $U(k-1)$ action and the $T^k \times \cdots \times T^{n-1}$ action for every $k \geq 2$ and $t \in \C$.
The following theorem states that the maps $\Phi^{i,j}_\lambda$'s in \eqref{phiijbb}
and $\Phi^{i,j}$'s in \eqref{phiijaa} defined on $P$ induces a completely integrable system on $X_{k,t}$ and how the GC system $\Phi_\lambda$ on $X_{n,1} \cong \mcal{O}_\lambda$
degenerates into the toric moment map $\Phi$ on $X_{2,0} \cong X_0$ {\em in stages}. See also Section 5 and Section 7 of \cite{NNU}.

\begin{theorem}[Theorem 6.1 in \cite{NNU}]\label{theorem_NNU_6_1}
For every $k \geq 2$ and $t \in \C$, the map
\[
	\Phi_{k,t} := \left.(\underbrace{\Phi^{1,1}_\lambda}_{1}, \cdots, \underbrace{\Phi^{1, k-1}_\lambda, \cdots, \Phi^{k-1,1}_\lambda}_{k-1}, 
	\underbrace{\Phi^{1, k}_{\vphantom{\lambda}}, \cdots, \Phi^{k,1}_{\vphantom{\lambda}}}_{k}, \cdots, 
	\underbrace{\Phi^{1,n-1}_{\vphantom{\lambda}}, \cdots, \Phi^{n-1,1}_{\vphantom{\lambda}}}_{n-1})\right|_{X_{k,t}} 
	: X_{k,t} \rightarrow \R^{\dim \Delta_\lambda}
\]
is a completely integrable system on $X_{k,t}$ in the sense of Definition \ref{definition_CIS_continuous}. Moreover, 
we have $\Phi_\lambda^{i, k - i} = \Phi_{\vphantom{\lambda}}^{i, k - i}$
on $X_{k, 0} = X_{k-1, 1}$ for every $i=1,\cdots, k-1$.
\end{theorem}

Note that $\Phi_{k,t}$'s are related to one another via the {\em gradient-Hamiltonian flows} introduced by Wei-Dong Ruan \cite{R}. 
For each $m=1,\cdots, n-1$, let $F_m$ be the $m$-th component of $F$ in \eqref{equation_KoM} and 
let $\widetilde{V}_m$ be the Hamiltonian vector field of $\mathrm{Im}(F_m)$
on the smooth locus $\mcal{X}^{\mathrm{sm}}$ of $\mcal{X}$. Then the gradient-Hamiltonian vector field is defined by 
\[
	V_m := \widetilde{V}_m / ||\widetilde{V}_m||^2. 
\]
The flow of $V_m$, which we denote by $\phi_{m,t}$ where $t$ is a time parameter, 
preserves the fiberwise symplectic form and so $\phi_{m,t}$ induces a symplectomorphism  
on an open subset of each fiber on which $\phi_{m,t}$ is smooth. 
As a corollary, we have the following.
\begin{corollary}[Corollary 7.3 in \cite{NNU}]\label{corollary_NNU_deform}
	 The gradient-Hamiltonian vector field $V_k$ gives a deformation of $X_{k,t}$ 
	 preserving the structure of completely integrable systems.
 	 In particular, we have the following commuting diagram for every $t \geq 0$ : 
	  \begin{equation}\label{equation_preserve_integrable_system}
		    \xymatrix{
			      X_{k,1}  \ar[dr]_{\Phi_{k,1}} \ar[rr]^{\phi_{k,1-t}}
			      & & X_{k,t} \ar[dl]^{\Phi_{k,t}} \\ 
			      & \Delta_{\lambda} &
			    }
	  \end{equation}
\end{corollary}

By Corollary \ref{corollary_NNU_deform}, we obtain a continuous map 
\begin{equation}\label{phicontimapx}
	\phi \colon \mcal{O}_\lambda = X_{n,1} \rightarrow X_0 = X_{2,0}
\end{equation}
where $\phi = \phi_{2,1} \circ \cdots \circ \phi_{n,1}$ and it satisfies $\Phi \circ \phi = \Phi_\lambda$. 
Note that 
\[
	\phi \colon \Phi_\lambda^{-1}(\mathring{\Delta}_\lambda) \rightarrow \Phi^{-1}(\mathring{\Delta}_\lambda)
\]
is a symplectomorphism. 

Now, for a given $\textbf{u} \in \Delta_\lambda$, we investigate which component $\Phi_\lambda^{i,j}$ is smooth on $\Phi_\lambda^{-1}(\textbf{u})$. 
Let $\gamma$ be an $r$-dimensional face of $\Gamma_\lambda$ and let $f_\gamma$ be the corresponding face of $\Delta_\lambda$. 
Let 
	\begin{equation}\label{equation_index_cycle}
		\mcal{I}_{\mathcal{C}(\gamma)} := \{ (i,j) ~|~ (i,j) = v_\sigma \quad ~\text{for some minimal cycle}~\sigma~\text{of}~\gamma \}
	\end{equation}
so that $|\mcal{I}_{\mathcal{C}(\gamma)}| = r$. (See Figure \ref{figure_M_1_block_simple_closed_region}.)

\begin{lemma}\label{lemma_smooth_component_free_torus_action}
	Each $\Phi_\lambda^{i,j}$ is smooth on $\Phi_\lambda^{-1}(\mathbf{u})$ for every $(i,j) \in \mcal{I}_{\mathcal{C}(\gamma)}$.
	In particular, $\Phi_\lambda^{i,j}$ is smooth on $\Phi_\lambda^{-1}(\mathring{f}_\gamma)$. 
	Furthermore, $\{ \Phi_\lambda^{i,j} \}_{(i,j) \in \mcal{I}_{\mathcal{C}(\gamma)}}$ generates a smooth fiberwise $T^r$-action on 
	$\Phi_\lambda^{-1}(\textbf{\textup{u}})$ for each $\textbf{\textup{u}} \in \mathring{f}_\gamma$.
\end{lemma}

\begin{proof}
	The smoothness of each $\Phi_\lambda^{i,j}$ on $\Phi_\lambda^{-1}(\mathring{f}_\gamma)$ follows 
	from the condition $(i,j) = v_\sigma$ and Proposition \ref{proposition_GS_smooth}.
	Also, we have seen in Section \ref{ssecSmoothnessOfPhiLambda} that each $\Phi_\lambda^{i,j}$ 
	generates a smooth circle action on an open dense subset of $\mathcal{O}_\lambda$ on which $\Phi_\lambda^{i,j}$  is smooth.
	Since all components of $\Phi_\lambda$ Poisson-commute with each other, it finishes the proof.
\end{proof}

We recall the following well-known fact on toric varieties.
Let $\Delta \subset \R^N \cong \frak{t}^*$ be a convex polytope of dimension $N$ and let $X_\Delta$ be the corresponding projective toric variety
where $\frak{t}$ is the Lie algebra of the maximal compact torus $T^N$ in $(\C^*)^N \subset X_0$ and $\frak{t}^*$ is the dual of $\frak{t}$. 
Let $f$ be an $r$-dimensional face of $\Delta$ and suppose that $F_1, \cdots, F_m$ are facets of 
$\Delta$ such that $f = \cap_{i=1}^m F_i$. Also, we let $v_i \in \frak{t}$ be the inward primitive integral normal vector of $F_i$ for $i=1,\cdots,m$

\begin{lemma}[Exercise 12.4.7.(d) in \cite{CLS}]\label{lemma_freely}
	Let $\xi_1, \cdots, \xi_r$ be primitive integral vectors in $\frak{t}$ which generates an $r$-dimensional subtorus $T^r$ of $T^N$. 
	Then the $T^r$ acts on $\Phi^{-1}(\mathring{f})$ freely if
	$\Z^N \cong \langle v_1, \cdots, v_m, \xi_1, \cdots, \xi_r \rangle$.	
\end{lemma}

We then obtain the following.

\begin{proposition}\label{proposition_toric_degenerate_equivariant_map}
	Fix $\textbf{\textup{u}} \in \mathring{f}_\gamma$ and consider the $T^r$-action on $\Phi_\lambda^{-1}(\textbf{\textup{u}})$ generated by 
	$\{ \Phi_\lambda^{i,j} \}_{(i,j) \in \mcal{I}_{\mathcal{C}(\gamma)}}$ as given in Lemma \ref{lemma_smooth_component_free_torus_action}.
	Then the $T^r$-action on $\Phi_\lambda^{-1}(\textbf{\textup{u}})$ is free. Furthermore, $\Phi_\lambda^{-1}(\textbf{\textup{u}})$ becomes a trivial principal bundle over 
	$\Phi_\lambda^{-1}(\textbf{\textup{u}}) / T^r$, that is,  
	\[
		\Phi_\lambda^{-1}(\textbf{\textup{u}}) \cong T^r \times \Phi_\lambda^{-1}(\textbf{\textup{u}}) / T^r.
	\]
\end{proposition}

\begin{proof}
	We first show that the $T^r$-action is free on $\Phi^{-1}(\textbf{\textup{u}})$. 
	For each $(i,j) \in \mcal{I}$\footnote{See \eqref{equation_index_global} and \eqref{collectionofindicesset} for the definition of $\mcal{I}$ and $\mcal{I}_\lambda$, respectively.}, we denote by 
	\[
		\xi_{i,j} := \left( e_{k,l} \right) \in \R^{\frac{n(n-1)}{2}}, \quad \begin{cases} \text{$e_{k,l} = 1$ if $(k,l) = (i,j)$ and $e_{k,l} = 0$ otherwise} & \text{if $(i,j) \in \mcal{I}_\lambda$} \\
		\text{$e_{k,l} = 0$ for every $(k,l)$} & \text{if $(i,j) \not \in \mcal{I}_\lambda$}.
		\end{cases}
	\]
	By the min-max principle \eqref{equation_GC-pattern} and the dimension formula given in Definition \ref{definition_face}, 
	an inward primitive integral normal vector $v_F$ for any facet $F$ of $\Delta_\lambda$ is either
	\[
		v_{vert}^{i,j} := \xi_{i,j+1} - \xi_{i,j}, \quad \text{or} \quad v_{hor}^{i,j} := -\xi_{i+1,j} + \xi_{i,j}
	\]
	for some $(i,j) \in \mcal{I}_\lambda$. In particular, if $F$ contains $f_\gamma$, then $v_F$ is either
	\[
		\begin{cases}
			v_{vert}^{i, j}, \quad \text{$\square^{i,j}$ and $\square^{i,j+1}$ are in the same simple region of $\gamma$, or} \\
			v_{hor}^{i, j}, \quad \text{$\square^{i,j}$ and $\square^{i+1,j}$ are in the same simple region of $\gamma$}
		\end{cases}
	\] for some $(i,j) \in \mcal{I}_\lambda$. Then, it is not hard to see that 
	\[
		\{ v_F \}_{f_\gamma \subset F} \cup \{\xi_{i,j} \}_{(i,j) \in \mcal{I}_{\mcal{C}(\gamma)}} 
	\]
	generates the full lattice $\Z^N$ where $N = \dim \Delta_\lambda$.
	Therefore the $T^r$-action generated by $\{\Phi^{i,j} \}_{(i,j) \in \mcal{I}_{\mcal{C}(\gamma)}}$ on $\Phi^{-1}(\mathring{f}_\gamma)$ is free by Lemma \ref{lemma_freely}, and hence 
	the $T^r$-action on each fiber $\Phi^{-1}(\textbf{\textup{u}})$ is free for every $\textbf{\textup{u}} \in \mathring{f}_\gamma$.
	Since 
	\[
		\phi_{\textbf{\textup{u}}} := \phi|_{\Phi_\lambda^{-1}({\textbf{\textup{u}}})} \colon \Phi_\lambda^{-1}(\textbf{\textup{u}}) 
		\to \Phi_{\vphantom{\lambda}}^{-1}(\textbf{\textup{u}}) \cong T^r. 
	\]
	is $T^r$-equivariant by Corollary \ref{corollary_torus_equivariant}, we see that the $T^r$-action on $\Phi_\lambda^{-1}(\textbf{\textup{u}})$ is also free. 
	
	The freeness of the $T^r$-action on $\Phi_\lambda^{-1}({\textbf{\textup{u}}})$ implies that the map
	$\phi_{\textbf{\textup{u}}} := \phi|_{\Phi_\lambda^{-1}({\textbf{\textup{u}}})}$ becomes a principal bundle map such that 
	\[
	\xymatrix{
		\Phi_\lambda^{-1}({\textbf{\textup{u}}}) \ar[d]_{/T^r} \ar[r]^{\phi_{\textbf{\textup{u}}}} & \Phi^{-1}({\textbf{\textup{u}}}) \cong T^r \ar[d]^{/ T^r}\\
		\Phi_\lambda^{-1}({\textbf{\textup{u}}}) / T^r \ar[r] & \mathrm{point} 
		}
	\] commutes. In particular, $\Phi_\lambda^{-1}({\textbf{\textup{u}}})$ is a pull-back bundle of the trivial $T^r$-bundle over a point so that $\Phi_\lambda^{-1}({\textbf{\textup{u}}})$ 
	is a trivial $T^r$-bundle as desired.
\end{proof}

To sum up, we can describe how each fiber of a GC system deforms into a torus fiber of a moment map of $X_0$ via a toric degeneration as follows. 

\begin{theorem}\label{theorem_contraction}
	Let $\gamma$ be a face of the ladder diagram $\Gamma_\lambda$ of dimension $r$ and let $f_\gamma$ be the corresponding face of the Gelfand-Cetlin polytope $\Delta_\lambda$.
	For each point ${\textbf{\textup{u}}}$ in the relative interior $\mathring{f}_\gamma$, all $S^1$-factors that appeared in each stage of the iterated bundle structure $\bar{S_{\bullet}}(\gamma)$ 
	of $\Phi_\lambda^{-1}({\textbf{\textup{u}}})$ given in Theorem \ref{theorem_main} are factored out. That is, 
	\[
		\Phi_\lambda^{-1}({\textbf{\textup{u}}}) \cong T^r \times \bar{S_{\bullet}}(\gamma)'
	\] where 
	$\bar{S_{\bullet}}(\gamma)'$ is the total space of the iterated bundle which can be obtained by the construction of $\bar{S_{\bullet}}(\gamma)$
	ignoring all $S^1$-factors appeared in each stage. Furthermore, the continuous map $\phi$ in~\eqref{phicontimapx} on each fiber $\Phi_\lambda^{-1}({\textbf{\textup{u}}})$
	is the projection map 
	\[ 
		\Phi_\lambda^{-1}({\textbf{\textup{u}}}) \cong T^r \times \bar{S_{\bullet}}(\gamma)' \stackrel{\phi} \longrightarrow T^r \cong \Phi^{-1}({\textbf{\textup{u}}}). 
	\]
\end{theorem}

\begin{proof} 
	Consider the iterated bundle structure of $\Phi_\lambda^{-1}({\textbf{\textup{u}}})$ given in Theorem \ref{theorem_main} : 
	\begin{equation}\label{equation_iterated_bundle}
		\begin{array}{ccccccccc}
		\Phi_\lambda^{-1}({\textbf{\textup{u}}}) \cong \bar{S_{n-1}}(\gamma) & \xrightarrow {p_{n-1}} & \bar{S_{n-2}}(\gamma) &
		 \rightarrow & \cdots & 
		 \xrightarrow{p_2} & \bar{S_1}(\gamma) & \xrightarrow{p_1} & \bar{S_0}(\gamma) := \mathrm{point} \\
		 \hookuparrow & & \hookuparrow & & \hookuparrow & & \hookuparrow & &  \\
		 S_{n-1}(\gamma) & & S_{n-2}(\gamma)  & & \cdots & & S_1(\gamma)  \\
		 \end{array}
	\end{equation}	
	where $S_k(\gamma)$ is the fiber of $p_k \colon \overline{S_k}(\gamma) \rightarrow \overline{S_{k-1}}(\gamma)$ at the $k$-th stage
	defined in \eqref{equation_block_sphere}. 
	Each $S_k(\gamma)$ can be factorized into $S_k(\gamma) = (S^1)^{r_k} \times Y_k$ where $Y_k$ is either a point or a product of odd-dimensional spheres 
	without any $S^1$-factors. (See the proof of Proposition \ref{proposition_bundle}.)
	Then we claim that 
	\begin{enumerate}
		\item there is a one-to-one correspondence between the $S^1$-factors that appeared in each stage and the elements in $\mcal{I}_{\mcal{C}(\gamma)}$, 
		\item $(S^1)^{r_k}$ acts on $\overline{S_k}(\gamma)$ fiberwise with respect to $p_k \colon \overline{S_k}(\gamma) \rightarrow \overline{S_{k-1}}(\gamma)$, and 
		\item the torus action on $\Phi_\lambda^{-1}({\textbf{\textup{u}}}) \cong \bar{S_{n-1}}(\gamma)$ generated by $\{ \Phi_\lambda^{i,j} \}_{(i,j) \in \mcal{I}_{\mcal{C}(\gamma)}, i+j-1=k}$
		is an extension of the $(S^1)^{r_k}$-action on $\overline{S_k}(\gamma)$ given in (2). 
	\end{enumerate}
	The first statement (1) is straightforward since each $(i,j) \in \mcal{I}_{\mcal{C}(\gamma)}$ corresponds to an $M_1$-block in 
	$W_{i+j-1}(\gamma)$ containing a bottom vertex of $W_{i+j-1}$
	so that each $(i,j) \in \mcal{I}_{\mcal{C}(\gamma)}$ assigns an $S^1$-factor in $S_{i+j-1}(\gamma)$. See Section \ref{ssecWShapedBlocks}. \
	The third statement (3) is also clear since each $\Phi_\lambda^{i,j}$ with $i+j-1 = k$ descends to a function $\Phi_{\lambda_k}^{i,j}$ on $\bar{S_k}(\gamma)$ where 
	\[
		\lambda_k = (\Phi_\lambda^{1,k}({\textbf{\textup{u}}}), \cdots, \Phi_\lambda^{k,1}({\textbf{\textup{u}}})). 
	\]

	For the second statement (2), fix $k \geq 1$ and consider the $k$-th stage
	\begin{equation}\label{bundleatstagek}
		\begin{array}{ccc}
			S_k(\gamma) = (S^1)^{r_k} \times Y_k & \hookrightarrow & \overline{S_k}(\gamma) \subset (\mathcal{O}_{\frak{a}}, \omega_\frak{a})\\
			                                     &    &  \downarrow p_k\\
			                                      &    &  \overline{S_{k-1}}(\gamma) \subset (\mcal{O}_{\frak{b}}, \omega_\frak{b}) \\
		\end{array}
	\end{equation}
	of $\overline{S_\bullet}(\gamma)$. 
	As we have seen in the diagram~\eqref{figure_iterated_bundle},  $\overline{S_k}(\gamma)$ is a subset of $\mcal{A}(\frak{a}, \frak{b})$ where $\frak{a} = (a_1, \cdots, a_{k+1})$ and $\frak{b} = (b_1, \cdots, b_k)$ with 
	\begin{itemize}
		\item $a_i = \Phi_\lambda^{i, k+1-i}({\textbf{\textup{u}}})$ for $1 \leq i \leq k+1$ and 
		\item $b_j = \Phi_\lambda^{j, k-j}({\textbf{\textup{u}}})$ for $1 \leq j \leq k$. 
	\end{itemize}
	Note that $\Phi_{\frak{a}}^{i,k-i}$ generates a circle action
	on $\mcal{O}_{\frak{a}}$ whenever $a_i < b_i < a_{i+1}$, see Section \ref{ssecSmoothnessOfPhiLambda}. For any smooth functions $f$ on $\mcal{O}_\frak{b}$ and 
	$\hat{f} = f \circ p_k$ on $\mcal{O}_\frak{a}$, we denote by $\xi_f$ and $\xi_{\hat{f}}$ the Hamiltonian vector fields for $f$ and $\hat{f}$, respectively. 
	Then it follows that $\xi_{\hat{f}}$ is projectable under $p_k$ and it satisfies $(p_k)_* \xi_{\hat{f}} = \xi_f$, i.e., $dp_k ( \xi_{\hat f})(x) = \xi_f(p_k(x))$ for every $x \in \mcal{O}_{\frak{b}}$ since
	\begin{equation}\label{equation_pull_back}
		\begin{array}{ccl}
			\omega_\frak{b}((p_k)_* \xi_{\hat{f}}, (p_k)_* (\cdot)) & = & (p_k)^*\omega_\frak{b}(\xi_{\hat{f}}, \cdot) \\
			                                                                         & = & \omega_\frak{a}(\xi_{\hat{f}}, \cdot) \\
			                                                                         & = & d\hat{f}(\cdot) = df((p_k)_* (\cdot)) \\
			                                                                         & = & \omega_\frak{b}(\xi_f, (p_k)_* (\cdot))
		\end{array}
	\end{equation}
	where the second equality comes from Lemma \ref{lemma_isotropic}.
	Also, note that $\Phi_{\frak{b}}^{i,k-i}$ is a constant function on $\mcal{O}_{\frak{b}}$ and 
	$\Phi_{\frak{a}}^{i,k-i} = \Phi_{\frak{b}}^{i,k-i} \circ p_k$. By applying \eqref{equation_pull_back}, we can see that the Hamiltonian flow generated by 
	$\Phi_{\frak{a}}^{i,k-i}$ preserves the $(k \times k)$ leading principal minor, and therefore its Hamiltonian vector field is 
	tangent to the vertical direction of $p_k$. 
	
	Once (1), (2), and (3) are satisfied, its iterated bundle in \eqref{equation_iterated_bundle} descends to 
	\begin{equation}\label{equation_iterated_bundle_quotient}
		\begin{array}{ccccccccc}
		\Phi_\lambda^{-1}(\textbf{\textup{u}}) / (S^1)^r \cong \bar{S_{n-1}}(\gamma)' & \xrightarrow {p_{n-1}'} & \bar{S_{n-2}}(\gamma)' &
		 \rightarrow & \cdots & 
		 \xrightarrow{p_2'} & \bar{S_1}(\gamma)' & \xrightarrow{p_1'} & \bar{S_0}(\gamma)' := \mathrm{point} \\
		 \hookuparrow & & \hookuparrow & & \hookuparrow & & \hookuparrow & &  \\
		 Y_{n-1} & & Y_{n-2}  & & \cdots & & Y_1  \\
		 \end{array}
	\end{equation}	
	where $\overline{S_k}(\gamma)' = \overline{S_k}(\gamma) / (S^1)^{r_k + \cdots + r_1}$ and the $(S^1)^{r_k + \cdots + r_1}$-action is generated by 
	$\{ \Phi_\lambda^{i,j} \}_{(i,j) \in \mcal{I}_{\mcal{C}(\gamma)}, i+j-1 \leq k}$. Since $\Phi_\lambda^{-1}(\textbf{\textup{u}}) \cong T^r \times Y(\textbf{\textup{u}})$ for some $Y(\textbf{\textup{u}})$ by Proposition \ref{proposition_toric_degenerate_equivariant_map},
	we have $Y(\textbf{\textup{u}}) \cong \Phi_\lambda^{-1}(\textbf{\textup{u}}) / T^r \cong \overline{S_\bullet}(\gamma)'$, which completes the proof.
\end{proof}

As an application of Theorem~\ref{theorem_contraction}, one can provide a more explicit description of GC fibers. As mentioned in Remark~\ref{remark_su(3)}, an iterated bundle in Theorem~\ref{theorem_main} is in general \emph{not} trivial. Generally speaking, a torus bundle over a torus might be non-trivial, e.g. Kodaira-Thurston example, a $2$-torus bundle over a $2$-torus whose first betti number is $3$, see \cite{Th}. Yet, Theorem~\ref{theorem_contraction} guarantees that all torus factors in the iterated bundle can be taken out from $\Phi_\lambda^{-1}(\textbf{\textup{u}})$. Using this observation, in some case, the iterated bundle can be characterized explicitly. 

\begin{example}
\begin{enumerate}
\item Let $\mcal{O}_\lambda \simeq \mcal{F}(6)$ be the co-adjoint orbit associated with $\lambda = (5,3,1,-1,-3,-5)$. 
Consider the face $\gamma_1$ defined in Figure~\ref{Figure_diffeotype}. 
The diffeomorphic type of the fiber over a point $\textbf{\textup{u}}$ in the relative interior of $\gamma_1$ is the product of $(S^1)^7$ and $Y(\textbf{\textup{u}})$ by Theorem~\ref{theorem_contraction}. 
Here, $Y(\textbf{\textup{u}})$ is diffeomorphic to $SU(3)$ because $Y(\textbf{\textup{u}})$ is the total space of the $S^3$-bundle over $S^5$ from Remark~\ref{remark_su(3)}. In sum,
$$
\Phi^{-1}_\lambda(\textbf{\textup{u}}) \simeq (S^1)^7 \times SU(3).
$$
\item Let $\lambda = (3,3,3, -3,-3,-3)$. 
Then, the co-adjoint orbit $\mcal{O}_\lambda$ is $\mathrm{Gr}(3,6)$. 
Consider the face $\gamma_2$ defined in Figure~\ref{Figure_diffeotype}. 
We claim that the diffeomorphic type of the fiber over a point in the relative interior of $\gamma_2$ is 
$$
\Phi^{-1}_\lambda(\textbf{\textup{u}}) \simeq (S^1)^3 \times (S^3)^2.
$$
Because of Theorem~\ref{theorem_contraction}, the fiber is of the form $(S^1)^3 \times Y(\textbf{\textup{u}})$
where $Y(\textbf{\textup{u}})$ is an $S^3$-bundle over $S^3$. Since every $S^3$-bundle over $S^3$ is trivial (see Steenrod \cite{St}), we have $Y(\textbf{\textup{u}}) \simeq (S^3)^2$.
\end{enumerate}
\end{example}

\begin{figure}[ht]
	\scalebox{1}{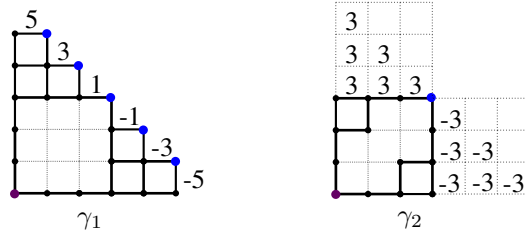}
	\caption{\label{Figure_diffeotype} Gelfand-Cetlin fibers.}	
\end{figure}

Another application of Theorem \ref{theorem_contraction} is to compute the first and the second homotopy groups of each $\Phi_\lambda^{-1}(\textbf{\textup{u}})$ as follows.
Let $\textbf{\textup{u}} \in \Delta_\lambda$ and let $f$ be the face of $\Delta_\lambda$ containing $\textbf{\textup{u}}$ in its relative interior. 
Also, let $\gamma$ be the face of $\Gamma_\lambda$ corresponding to $f$.
For each $k=1, \cdots, n-1$, the fibration~\eqref{bundleatstagek} 
induces the long exact sequence of homotopy groups given by
\begin{equation}\label{equation_homotopy_long_exact_sequence}
	\begin{array}{ccccccccccc}
		\cdots  &
		\rightarrow & \pi_2(S_k(\gamma)) & \rightarrow & \pi_2(\bar{S_k}(\gamma))&
		\rightarrow &  \pi_2(\bar{S_{k-1}}(\gamma))& & & & \\
		& \rightarrow & \pi_1(S_k(\gamma)) & \rightarrow & \pi_1(\bar{S_k}(\gamma)) & \rightarrow &
		\pi_1(\bar{S_{k-1}}(\gamma)) &
		\rightarrow & \pi_0(S_k(\gamma)) & \rightarrow & \cdots
	\end{array}
\end{equation}

\begin{proposition}\label{proposition_pi_2_zero_pi_1_k}
	Let $\textbf{\textup{u}} \in \Delta_\lambda$. Then the followings hold.
	\begin{itemize}
		\item $\pi_2(\Phi_\lambda^{-1}(\textbf{\textup{u}})) = 0$.
		\item If $\textbf{\textup{u}}$ is a point in the relative interior of an $r$-dimensional face of $\Delta_\lambda$, then
		\[
			\pi_1(\Phi_\lambda^{-1}(\textbf{\textup{u}})) \cong \Z^r.
		\]
	\end{itemize}
\end{proposition}
\begin{proof}
	Since $S_k(\gamma)$ in \eqref{bundleatstagek} is a point or a product space of odd dimensional spheres,
	we have $\pi_2(S_k(\gamma)) = 0$ for every $k=1,\cdots, n-1$. 
	Note that $\pi_2(\bar{S_1}(\gamma)) = \pi_2(S_1(\gamma)) = 0$. 
	Therefore, it easily follows that $\pi_2(\bar{S_k}(\gamma)) = 0$ for every $k$ by the induction on $k$. 
	The second statement is deduced from Theorem \ref{theorem_contraction}, since $\bar{S_{\bullet}}(\gamma)'$ is simply connected. 
\end{proof}

\begin{corollary}\label{corollary_Lagrangian_tori}
	For a point $\textbf{\textup{u}} \in \Delta_\lambda$, the fiber $\Phi_\lambda^{-1}(\textbf{\textup{u}})$ is a Lagrangian torus if and only if $\textbf{\textup{u}}$ is an interior point of
	$\Delta_\lambda$
\end{corollary}

\begin{proof}
	The ``if'' statement follows immediately from Theorem \ref{theorem_contraction}, and the ``only if'' part follows from Proposition \ref{proposition_pi_2_zero_pi_1_k}.
\end{proof}

%------------------------------------------------------------------------------------------------------------
\vspace{0.2cm}
\part{Non-displaceability of Lagrangian fibers}\label{Part_nondisplaceabilityofLag}

%------------------------------------------------------------------------------------------------------------
\vspace{0.2cm}
\section{Introduction of Part~\ref{Part_nondisplaceabilityofLag}}\label{Sec_Intropart2}

In the second part of this article, we focus on detecting displaceable and non-displaceable Lagrangian 
\footnote{See Definition~\ref{def_nondisanddis} for (non-)displacealbe Lagrangian submanifolds}
GC fibers. 

There has been research on displaceability and non-displaceability of GC fibers on partial flag manifolds. 
Nishinou, Nohara, and Ueda showed that the Lagrangian GC torus fiber at the center is non-displaceable. To show it, they calculated the potential function after constructing a toric degeneration from a GC system to a toric moment map and found a weak bounding cochain such that the deformed Floer cohomology is non-vanishing.

\begin{theorem}[Theorem 12.1 in \cite{NNU}]\label{theorem_NNU}
For any non-increasing sequence $\lambda$ of real numbers, the Gelfand-Cetlin system $\Phi_\lambda \colon (\mathcal{O}_\lambda, \omega_\lambda) \rightarrow \Delta_\lambda$ admits a non-displaceable Lagrangian torus fiber $\Phi_\lambda^{-1}(\textbf{\textup{u}}_0)$ over the center $\textbf{\textup{u}}_0$ of the Gelfand-Cetlin polytope $\Delta_\lambda$.
\end{theorem}

Over the Novikov ring over the field of complex numbers, Nohara and Ueda \cite{NU2} found bounding cochains making deformed Floer cohomology of $U(2)$-fiber non-zero in $\mathcal{O}_\lambda \cong \mathrm{Gr}(2,4)$ for 
\begin{equation}\label{eq_lambdagr}
\lambda = \{ \lambda_1 = \lambda_2 = 2 > \lambda_3 = \lambda_4 = -2 \}.
\end{equation}
The authors explicitly classified moduli spaces of holomorphic discs bounded by the fiber over the center of $\gamma$, 
using the homogeneity of the fiber in $\mathrm{Gr}(2,4)$. 
Here, the notion of homogeneous Lagrangians was introduced by Evans and Lekili, see \cite[Definition 1.1.1]{EL}.

\begin{theorem}[Theorem 1.2 in \cite{NU2}]\label{theorem_NU}
For the sequence $\lambda$ in~\eqref{eq_lambdagr}, let $\gamma$ be the one-dimensional Lagrangian face of $\Gamma_\lambda$ in Example \ref{example_gr24_W_block}. Then, the Gelfand-Cetlin $U(2)$-fiber over the center of the face $\gamma$ is non-displaceable. Moreover, any other Gelfand-Cetlin $U(2)$-fibers located at its interior are displaceable. 
\end{theorem}

For a sequence
\begin{equation}\label{eq_lambdagrrr}
\lambda = (\lambda_1 = \cdots = \lambda_n = n > \lambda_{n+1} = \cdots = \lambda_{2n} = - n),
\end{equation}
the co-adjoint orbit $\mcal{O}_\lambda$ gives rise to $\mathrm{Gr}(n,2n)$ and the $(n \times n)$-square $\gamma$ in $\Gamma_\lambda$ corresponds to the one-dimensional face whose fiber at the interior is diffeomorphic to $U(n)$. 
According to the result of the third named author in \cite{Oh}, any Lagangian submanifold that is a fixed point set of an anti-symplectic involution of a monotone Hermitian symmetric space has non-zero Floer cohomology over the field of characteristic two. 
By Iriyeh, Sakai, and Tasaki in \cite{IST}, the $U(n)$-fiber is the fixed point set of an anti-symplectic involution on $\mathrm{Gr}(n,2n)$, a Hermitian symmetric space. Using those results, Evans and Lekili \cite{EL2} observed that the $U(n)$-fiber at the center of the face $\gamma$ has non-zero Floer cohomology over the field in characteristic two. 

\begin{theorem}[Corollary 7.4.2 in \cite{EL2}]\label{theorem_EL}
For the sequence $\lambda$ in~\eqref{eq_lambdagrrr}, the Lagrangian $U(n)$-fiber at the center of one-dimensional edge $\gamma$ is non-displaceable. 
Moreover, the Lagrangian $U(N)$-fiber at the center of $\gamma$ split-generates the Fukaya category of $\mathrm{Gr}(N, 2N)$ in characteristic $2$ for $N = 2^s$ with a positive integer $s$ . 
\end{theorem}

In this paper, we discuss non-displaceable GC fibers on complete flag manifolds equipped with a monotone KKS form. In contrast to the case of the $U(n)$-fiber in $\mathrm{Gr}(n,2n)$, a non-toric fiber can not be realized as an orbit of a global holomorphic action in general because it has torus factors from torus actions not extending to the ambient manifold. Because of this feature, it seems that the arguments in Theorem~\ref{theorem_NU} and Theorem~\ref{theorem_EL} cannot be applied to this case. 

Let
$$
	\lambda = \{ \lambda_1 > \lambda_2 > \cdots > \lambda_n \}
$$
be a decreasing sequence of real numbers so that the co-adjoint orbit $\mathcal{O}_\lambda$ is diffeomorphic to a complete flag manifold. 
When the KKS form $\omega_\lambda$ is monotone, we will show non-displaceability of toric fibers and non-toric fibers as well. 
To show it, we first locate a half-open line segment $[\textbf{\textup{u}}_0,\textbf{u}_\gamma) \subset \Delta_\lambda$ where $\textbf{\textup{u}}_0$ is the center
\footnote{For any moment polytope $\Delta$, Fukaya, Oh, Ohta, and Ono \cite[Proposition 9.1]{FOOOToric1} described a unique interior point (which they denoted by $\textbf{u}_0$ or $P_K$) and we call it the \emph{center} of $\Delta$. 
The center $\textbf{u}_0$ of $\Delta$ is a point over which the corresponding toric fiber is non-displaceable, see Theorem 1.5 and Proposition 4.7 in \cite{FOOOToric1} for more details. 
Notice that the center is \emph{not} meant to be the barycenter of a polytope. The center and the barycenter are in general different. 
} 
of $\Delta_\lambda$ and $\textbf{u}_\gamma$ is a point in the relative interior of some Lagrangian face $\gamma$ of $\Delta_\lambda$. Then, in some cases, we are able to show that each torus fiber $\Phi_\lambda^{-1}(\textbf{\textup{u}})$ ($\textbf{\textup{u}} \in [\textbf{\textup{u}}_0,\textbf{u}_\gamma)$) is non-displaceable by showing that a Lagrangian Floer cohomology (with a certain bulk-deformation parameter and a weak bounding cochain) is non-zero. Since the non-toric Lagrangian submanifold $\Phi_\lambda^{-1}(\textbf{u}_\gamma)$ is realized as the ``limit" of non-displaceable Lagrangian tori, we deduce that the fiber $\Phi_\lambda^{-1}(\textbf{u}_\gamma)$ is also non-displaceable, see Proposition~\ref{closednessofnondisp} for the precise statement.

We describe line segments over which the fibers are non-displaceable in the GC polytope $\Delta_\lambda$ explicitly. If the symplectic form $\omega_\lambda$ is monotone, the center of the polytope $\Delta_\lambda$ is expressed as follows.
Recall that the form $\omega_\lambda$ is monotone if and only if $\lambda_n - \lambda_{n-1} = \lambda_{n-1} - \lambda_{n-2} = \cdots = \lambda_2 - \lambda_1$
by Proposition \ref{proposition_monotone_lambda}. By scaling $\omega_\lambda$ if necessary, we may assume that
\begin{equation}\label{givensequencemonotone}
	\lambda = \left\{ \lambda_{i} := n - 2i + 1 \,\colon\, i = 1, \cdots, n \right\},
\end{equation}
which is the case where $[\omega_\lambda] = c_1(T\mathcal{O}_\lambda)$. The polytope $\Delta_\lambda$ is of dimension $\frac{n(n-1)}{2}$. It turns out that $\Delta_\lambda$ is a reflexive
\footnote{A convex lattice polytope $\mathcal{P}$ is called {\em reflexive} if its dual polytope $\mathcal{P}^*$ is also a lattice polytope.}
polytope, see \cite[Corollary 2.2.4]{BCKV} or \cite[Lemma 3.12]{NNU}.
One well-known fact on a reflexive polytope is that there exists a unique lattice point in its interior, that is exactly the center of the polytope.

We start from the simplest case where the co-adjoint orbit $\mcal{O}_\lambda$ of a sequence $\lambda = \{ \lambda_ 1 > \lambda_2 > \lambda_3 \}$. In this case, Pabiniak investigated displaceable GC fibers. 

\begin{theorem}[\cite{Pa}]\label{displaceablefl3}
For $\lambda = \{ \lambda_ 1 > \lambda_2 > \lambda_3 \}$, let $\mcal{O}_\lambda$ be the co-adjoint orbit, which is a complete flag manifold $\mcal{F}(3)$ equipped with $\omega_\lambda$.
\begin{enumerate}
\item If $\omega_\lambda$ is not monotone, i.e. $(\lambda_1 - \lambda_2) \neq (\lambda_2 - \lambda_3)$, then all  Gelfand-Cetlin fibers but one over the center are displaceable.
\item If $\omega_\lambda$ is monotone i.e. $(\lambda_1 - \lambda_2) = (\lambda_2 - \lambda_3)$, then all  Gelfand-Cetlin fibers but the fibers over the line segment
\begin{equation}\label{equation_linesegmentforf3}
	I := \left\{ (u_{1,1}, u_{1,2}, u_{2,1}) = (0, a -t, -a +t) \in \R^3 ~\colon~ 0 \leq t \leq a \right\}
\end{equation}
are displaceable where $2a = \lambda_1 - \lambda_2$. Observe that the line segment $I$ is the red line in Figure~\ref{figure_GCmomnetf3}.
\end{enumerate}
\end{theorem}

Note that the line segment $I$ connects the center $(0, a, -a)$ of $\Delta_\lambda$ and the position $(0, 0, 0)$ of the Lagrangian $3$-sphere.
Combining with Theorem~\ref{theorem_NNU}, Theorem~\ref{displaceablefl3} finishes the classification of displaceable and non-displaceable fibers of the GC system $\Phi_\lambda$ for the non-monotone case.
The following theorem asserts every GC fiber in the family $\{\Phi_\lambda^{-1}(\textbf{\textup{u}}) ~|~ \textbf{\textup{u}} \in I\}$ is non-displaceable. Thus, together with Theorem~\ref{displaceablefl3}, our result provides the complete classification of displaceable and non-displaceable Lagrangian GC fibers when $\omega_\lambda$ is monotone.
We postpone its proof until Section~\ref{pfofmaincorforfullflag3}.

\begin{thmx}[Theorem \ref{theorem_maincompleteflag3}]\label{theoremC}
Let $\lambda = \{ \lambda_1 = 2 > \lambda_2 = 0 > \lambda_3 = -2 \}$.
Consider the co-adjoint orbit $\mathcal{O}_\lambda$, a complete flag manifold $\mcal{F}(3)$ equipped with the monotone Kirillov-Kostant-Souriau symplectic form $\omega_\lambda$, Then the Gelfand-Cetlin fiber over a point $\textbf{\textup{u}} \in \Delta_\lambda$ is non-displaceable if and only if $\textbf{\textup{u}} \in I$ where
\begin{equation}\label{linesegmentforf3}
I := \left\{ (u_{1,1}, u_{1,2}, u_{2,1}) = (0, 1 -t, -1 +t) \in \R^3 ~|~ 0 \leq t \leq 1 \right\}
\end{equation}
In particular, the Lagrangian 3-sphere $\Phi_\lambda^{-1}(0,0,0)$ is non-displaceable.
\end{thmx}

\begin{figure}[H]
	\scalebox{1.7}{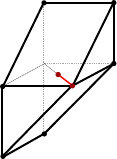}
	\caption{\label{figure_GCmomnetf3} The positions of non-displaceable GC Lagrangian fibers in $\mcal{F}(3)$.}	
\end{figure}

\begin{remark}
The fiber over the center $(0,1,-1)$ of $\Delta_\lambda$ is known to be non-displaceable by Theorem \ref{theorem_NNU}.
Also, Nohara and Ueda \cite{NU2} calculated a Floer cohomology of the Lagrangian 3-sphere $\Phi^{-1}_\lambda(0,0,0)$, which turns out to be zero over the Novikov field $\Lambda$.
Nevertheless, Theorem~\ref{theoremC} says that it is non-displaceable.
\end{remark}

Next, we deal with a general case for an arbitrary positive integer $n \geq 4$ where $\lambda$ is given as in~\eqref{givensequencemonotone}. In this case, the GC polytope $\Delta_\lambda$ is a reflexive polytope whose center is 
\[
\left( u _{i, j} := j - i ~\colon~ {i + j \leq n} \right) \in \Delta_\lambda \subset \R^{n(n-1)/2}.
\]
Consider the face $f_m$ of $\Delta_\lambda$ defined by
\[
\{ u_{i,j} = u_{i, j+1} ~\colon~  1 \leq i \leq m, 1 \leq j \leq m-1\} \cup \{ u_{i+1,j} = u_{i, j} ~\colon~  1 \leq i \leq m-1, 1 \leq j \leq m\}.
\]
for any integer $m$ satisfying $2 \leq m \leq \left\lfloor \frac{n}{2} \right\rfloor$.
Note that there are $\left( \left\lfloor \frac{n}{2} \right\rfloor - 1 \right)$ such faces in $\Delta_\lambda$.
For instance, we have two faces $f_2$ and $f_3$ for the case where $n=7$
(see Figure~\ref{figureF7}), and three faces $f_2$, $f_3$ and $f_4$
for $n=8$ (see Figure~\ref{figureF8}).
Furthermore, those faces can be filled by $L$-shaped blocks so that they are Lagrangian by Corollary~\ref{corollary_L_fillable}.
Regarding $f_m$ as a polytope, the center of $f_m$ admits a unique lattice point in its interior, whose components are
given by
\begin{equation}\label{endpointofIMT}
u_{i,j}  :=
\begin{cases}
0  \quad &\mbox{if \,} \max (i,j) \leq m \\
j - i \quad &\mbox{if \,} \max (i,j) > m.
\end{cases}
\end{equation}

Candidates for non-displaceable Lagrangian fibers are the fibers over the line segment $I_m \subset \Delta_\lambda$
connecting the center of $\Delta_\lambda$ and the center
of $f_m$ for each $m \geq 2$.
Explicitly,
the line segment $I_m$ is parameterized by $\{I_m(t) \in \Delta_\lambda ~\colon~ 0 \leq t \leq 1\}$ where
\begin{equation}\label{IMT}
I_m (t) :=
\begin{cases}
u_{i,j}(t) = (j - i) - (j - i) \, t \quad &\mbox{if \,} \max (i,j) \leq m \\
u_{i,j}(t) = (j - i) \quad &\mbox{if \,} \max (i,j) > m.
\end{cases}
\end{equation}
We denote by $L_m(t)$ the Lagrangian GC fiber over the point $I_m(t)$, that is $L_m(t) := \Phi_\lambda^{-1}(I_m(t))$.
Now, we state our second main theorem in this section. Again its proof will be
postponed to Section~\ref{section:decompositionofpote} and Section~\ref{solvabilityofthesplitleadingtermequ}.

\begin{thmx}[Theorem~\ref{theorem_completeflagmancotinuum}]\label{theoremD}
Let $\lambda = \{ \lambda_{i} := n - 2i + 1 \,|\, i = 1, \cdots, n \}$ be an $n$-tuple of real numbers for an arbitrary integer $n \geq 4$.
Consider the co-adjoint orbit $\mathcal{O}_\lambda$, a complete flag manifold $\mcal{F}(n)$ equipped with the monotone Kirillov-Kostant-Souriau symplectic form $\omega_\lambda$.
Then each Gelfand-Cetlin fiber $L_m(t)$ is non-displaceable Lagrangian for every $2 \leq m \leq \left\lfloor \frac{n}{2} \right\rfloor$. In particular, there exists a family of non-displaceable non-torus Lagrangian fibers
\[
	\left\{L_m(1) ~\colon~2 \leq m \leq \left\lfloor \frac{n}{2} \right\rfloor \right\}
\] 
of the Gelfand-Cetlin system $\Phi_\lambda$ where $L_m(1)$ is diffeomorphic to $U(m) \times T^{\frac{n(n-1)}{2} - m^2}$. 
\end{thmx}

\begin{example} 
\begin{enumerate}
\item
A monotone complete flag manifold $\mcal{F}(7)$ admits (at least) two line segments $I_2$ and $I_3$ in the GC polytope over which the fibers are non-displaceable, see Figure~\ref{figureF7}. Particularly, it has non-displaceable Lagrangian submanifolds diffeomorphic to $U(2) \times T^{17}$ and $U(3) \times T^{12}$. 
\begin{figure}[ht]
	\scalebox{1}{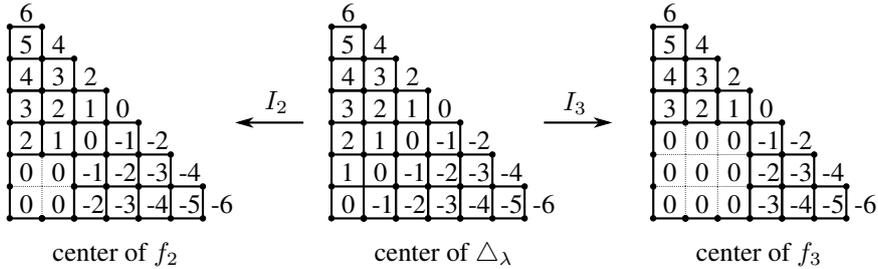}
	\caption{\label{figureF7} Positions of non-displaceable Lagrangian GC fibers in $\mathcal{F}(7)$.}	
\end{figure}
\item
A monotone complete flag manifold $\mcal{F}(8)$ admits (at least) three line segments $I_2$, $I_3$ and $I_4$ in the GC polytope over which the fibers are non-displaceable, see Figure~\ref{figureF8}. Particularly, it has non-displaceable Lagrangian submanifolds diffeomorphic to $U(2) \times T^{24}$, $U(3) \times T^{19}$ and $U(4) \times T^{12}$. 
\begin{figure}[ht]
	\scalebox{1}{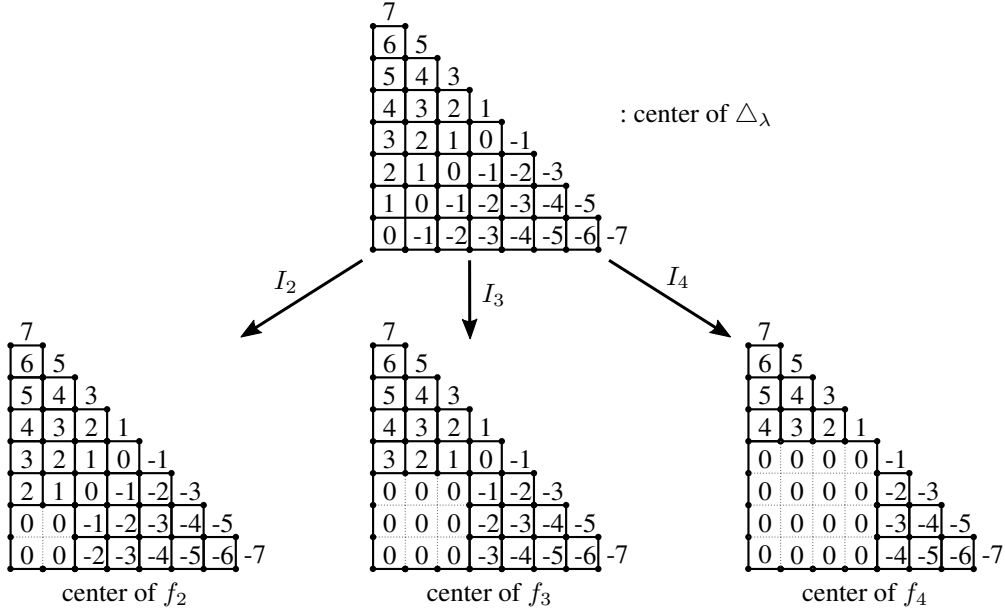}
	\caption{\label{figureF8} Positions of non-displaceable Lagrangian GC fibers in $\mathcal{F}(8)$.}	
\end{figure}
\end{enumerate}
\end{example}

\begin{remark}
	The third named author with Fukaya, Ohta, and Ono \cite{FOOOToric2} found a continuum of non-displaceable toric fibers on some compact toric manifolds including a non-monotone toric blowup of $\C P^2$ at two points, see also Woodward \cite{W}. Using the degeneration models, they also produced a continuum of non-displaceable Lagrangian tori on $\C P^1 \times \C P^1$ and the cubic surface respectively in \cite{FOOOS2S2} and \cite{FOOOsp}. Vianna \cite{V} also showed a continuum of non-displaceable tori in $(\C P^1)^{2n}$. In the case of toric orbifolds, dealing with more restrictive classes of Hamiltonian isotopies, non-displaceable toric fibers usually exist in abundance, see Woodward \cite{W}, Wilson-Woodward \cite{WW}, and Cho-Poddar \cite{CP}.
\end{remark}

Before dealing with non-displaceable GC fibers, we study displaceability of fibers in GC systems in Section \ref{secDisplaceabilityOfLagrangianFibers}. In the toric case, McDuff \cite{McD} and Abreu-Borman-McDuff \cite{ABM} developed the method of {\em probes} to detect displaceable toric fibers. The probe method can be also applied to Lagrangian \emph{torus} fibers in GC systems because they are related to the toric fibers via a (local) symplectomorphism in~\eqref{equation_toric_degeneration_diagram}, however, it is not applicable to non-torus fibers. In the case where a co-adjoint orbit is $\mcal{F}(3)$, Pabiniak \cite{Pa} investigated displaceable fibers. 

We develop several numerical and combinatorial criteria for detecting displaceable GC fibers which can be applied to \emph{both} torus fibers and non-torus fibers in GC systems.
Even though our criteria are not exhaustive to classify all displaceable fibers, it is enough to detect almost all displaceable fibers in that the non-displaceable fibers should be located over a measure zero set of the polytope $\Delta_\lambda$.
In particular, we are able to displace all non-torus fibers in some cases. 

\begin{thmx}[Corollary \ref{corollary_gr_2x}]\label{theoremE}
	Let $p$ be a prime number.
	Then every non-torus Lagrangian Gelfand-Cetlin fiber of the complex Grassmannian $\mathrm{Gr}(2,p)$ is displaceable.
\end{thmx}

The rest of the paper is organized as follows. In Section~\ref{secDisplaceabilityOfLagrangianFibers}, we introduce several numerical criteria testing displaceability of GC fibers and the proof of Theorem~\ref{theoremE} will be discussed. Section~\ref{secReviewOfLagrangianFloerTheory} is devoted to review Lagrangian Floer theory and to explain Theorem~\ref{theoremC}. To prove Theorem~\ref{theoremD}, the split leading term equation is introduced and its solvability is discussed in Section~\ref{secDecompositionsOfTheGradientOfPotentialFunction} and~\ref{secSolvabilityOfSplitLeadingTermEquation}. Section~\ref{sec_bulkdefbyschucy} focuses on calculation of the potential function deformed Schubert cycles.

%------------------------------------------------------------------------------------------------------------	
\section{Criteria for displaceablility of Gelfand-Cetlin fibers}
\label{secDisplaceabilityOfLagrangianFibers}

For a given moment polytope,
McDuff \cite{McD} and Abreu-Borman-McDuff \cite{ABM} developed the techniques (using {\em probes})
to detect positions of displaceable toric moment fibers by using combinatorial data of the polytope.
In contrast to the toric cases, any partial flag manifold always possesses a non-torus Lagrangian GC fiber
unless it is diffeomorphic to a projective space, see Corollary \ref{corollary_unless_projective}.
The probe method also can be applied to some extent to the case of GC systems, see \cite{Pa} for example.
However, it works only for \emph{torus} fibers over interior points of a GC polytope.
In this section, we provide several numerical and combinatorial criteria for displaceability of both torus and non-torus Lagrangian
fibers of GC systems. (See Proposition~\ref{proposition_diagonal_all_equal},~\ref{proposition_gr_symmetry_case_one},~\ref{proposition_gr_symmetry_case_more_two}.)

\begin{definition}\label{def_nondisanddis}
	Let $(M,\omega)$ be a symplectic manifold and let $Y$ be a subset.
	We say that $Y$ is {\em displaceable} if there exists a Hamiltonian diffeomorphism $\phi \in \mathrm{Ham}(M,\omega)$ such that
	\[
		\phi(Y) \cap \bar{Y} = \emptyset.
	\]
	If there is no such a diffeomorphism, then we say that $Y$ is {\em non-displaceable}.
\end{definition}

\begin{definition}
	Let $\lambda$ be a sequence of non-increasing real numbers given in \eqref{lambdaidef}.
	We say that a face $\gamma$ is {\em displaceable} if $\Phi_\lambda^{-1}(\textbf{\textup{u}})$ is displaceable for every $\textbf{\textup{u}} \in \mathring{\gamma}$.
\end{definition}

\subsection{Testing by diagonal entries}
\label{ssecTestingByDiagonalEntries}~

\vspace{0.2cm}

For a sequence $\lambda = \{\lambda_1, \cdots, \lambda_n\}$ in~\eqref{lambdaidef},
let
$
	\Phi_\lambda \colon \mathcal{O}_\lambda \rightarrow \Delta_\lambda \subset \R^{\frac{n(n-1)}{2}}
$
be the GC system.
For each lattice point $(i,j) \in (\Z_{\geq 1})^2$ with $2 \leq i+ j \leq n$,
we denote by $\Phi_\lambda^{i,j} \colon \mathcal{O}_\lambda \rightarrow \R$ the $(i,j)$-th component of $\Phi_\lambda$.
We denote the coordinate system of $\R^{\frac{n(n-1)}{2}}$ by $(u_{i,j})_{2 \leq i+j \leq n}$ as described in Figure \ref{figure_GC_to_ladder}.

Note that the symmetric group $\mathfrak{S}_n$ can be regarded as a subgroup of $U(n) \subset \mathrm{Ham}(\mathcal{O}_\lambda, \omega_\lambda)$.
Namely, each element $w \in \mathfrak{S}_n$ is represented by the $(n \times n)$ elementary matrix, which
we still denote by $w$, whose $(i, w(i))$-th entry is
$1$ for each $i=1,\cdots, n$ and other entries are zero.
Then each $x = (x_{i,j})_{1 \leq i,j \leq n}\in \mathcal{O}_\lambda \subset  \mathcal{H}_n$ and $w \in \mathfrak{S}_n$ satisfies  
\[
	(w \cdot x)_{ij} = (w x w^{-1})_{ij}= x_{w(i), w(j)}
\]
for every $1 \leq i,j \leq n$.
The following lemma says the diagonal entries of any element $x \in \Phi^{-1}_\lambda(\textbf{\textup{u}})$ are completely determined by $\textbf{\textup{u}}$.

\begin{lemma}\label{lemma_diagonal_elt}
	Let $\textbf{\textup{u}} = (u_{i,j})_{i,j \geq 1} \in \Delta_\lambda$.
	Then for any $x \in \Phi_\lambda^{-1}(\textbf{\textup{u}})$, we have $x_{1,1} = u_{1,1}$ and
	\[
		x_{k, k} = \sum_{i+j = k+1} u_{i,j} - \sum_{i+j = k} u_{i,j}.
	\]
	for $ 1 < k \leq n$ where
	\[
		u_{i,j} := \lambda_i \quad \textup{for} \quad i+j = n+1.
	\]
\end{lemma}
\begin{proof}
	It is straightforward from the fact that
	\[
		\mathrm{tr}(x^{(k)}) = x_{1,1} + \cdots + x_{k,k} = \sum_{i+j = k+1} u_{i,j}
	\]
	for every $k=1,\cdots,n$ where $x^{(k)}$ denotes the $(k \times k)$ leading principal minor matrix of $x$.
\end{proof}

For the sake of simplicity,
let $d_1(\textbf{\textup{u}}) := u_{1,1}$ and
\begin{equation}\label{dku}
d_k(\textbf{\textup{u}}) :=  \sum_{i+j = k+1} u_{i,j} - \sum_{i+j = k} u_{i,j}
\end{equation}
for $\textbf{\textup{u}} \in \Delta_\lambda$ and $1 < k \leq n$ so that $x_{k,k} = d_k(\textbf{\textup{u}})$ for every $x \in \Phi^{-1}_\lambda(\textbf{\textup{u}})$ by Lemma \ref{lemma_diagonal_elt}. 
We then state a numerical criterion for displaceable fibers.

\begin{proposition}\label{proposition_diagonal_all_equal}
	Let $\textbf{\textup{u}}$ be a point in the Gelfand-Cetlin polytope $\Delta_\lambda$.
	If the fiber $\Phi_\lambda^{-1}(\textbf{\textup{u}})$ is non-displaceable, then 
$$
d_1(\textbf{\textup{u}}) = \cdots = d_n(\textbf{\textup{u}}).
$$
\end{proposition}

\begin{proof}
Note that $\Phi_\lambda^{1,1}(x) = x_{1,1}$ and
$\Phi_\lambda^{1,1}(w \cdot x) = x_{w(1),w(1)}$
for $w \in \mathfrak{S}_n$.
If $d_1(\textbf{\textup{u}}) \neq d_k(\textbf{\textup{u}})$ for some $k$ with $1 \leq k \leq n$, then
\[
	\Phi_\lambda^{1,1}(x)  = x_{1,1} = d_1(\textbf{\textup{u}}) \neq d_k(\textbf{\textup{u}})  = x_{k,k} = x_{w(1),w(1)} = \Phi_\lambda^{1,1}(w \cdot x)
\]
for every
$x \in \Phi^{-1}_\lambda(\textbf{\textup{u}})$
where $w$ is the transposition $(1,k)$.
Thus, 
$$
w \cdot \Phi_\lambda^{-1}(\textbf{\textup{u}}) \cap \Phi_\lambda^{-1}(\textbf{\textup{u}}) = \emptyset
$$ 
and hence $\Phi_\lambda^{-1}(\textbf{\textup{u}})$ is displaceable.
This completes the proof.
\end{proof}

In a GC system, Proposition~\ref{proposition_diagonal_all_equal} implies that almost all fibers are displaceable. 

\begin{corollary}
For any Gelfand-Cetlin system $\Phi_\lambda$, there exists a dense open subset $\mcal{U}$ of the Gelfand-Cetlin polytope $\Delta_\lambda$ such that for each $\textbf{\textup{u}} \in \mcal{U}$ the fiber $\Phi_\lambda^{-1}(\textbf{\textup{u}})$ is displaceable. 
\end{corollary}

\begin{remark}
This contrasts to the case of toric orbifolds. There is a symplectic toric orbifold which contains an open subset $\mcal{U}$ of a moment polytope such that every fiber over a point in 
$\mcal{U}$ is non-displaceable, see Wilson-Woodward \cite{WW} and Cho-Poddar \cite{CP}.
\end{remark}

\begin{example}\label{example_non_disp}
We demonstrate how to apply Proposition \ref{proposition_diagonal_all_equal} to detect displaceable fibers.
\begin{enumerate}
\item (Complex projective spaces)
	For
	$\lambda = \{\lambda_1= n, ~\lambda_2 = \cdots = \lambda_n = 0\}$, the co-adjoint oribt $(\mathcal{O}_\lambda, \omega_\lambda)$ is a complex projective space
	$\C P^{n-1}$ equipped with a (multiple of) the Fubini-Study form $\omega_\lambda$.
	In this case, every component $\Phi_\lambda^{i,j}$ of the GC system $\Phi_\lambda$ is a \emph{constant} function unless $i = 1$ so that $\Delta_\lambda$ is an $(n-1)$-dimensional simplex.
	For any $\textbf{\textup{u}} \in \Delta_\lambda$ and $x \in \Phi_\lambda^{-1}(\textbf{\textup{u}})$, we have
	\[
	\begin{cases}
		x_{1,1} = d_1(\textbf{\textup{u}}) = u_{1,1}, \\
		x_{k,k} = d_k(\textbf{\textup{u}}) = u_{1,k} - u_{1, k-1} \mbox{for  $1 \leq k \leq n-1$}, \\
		x_{n,n} = d_n(\textbf{\textup{u}}) = u_{1,n} - u_{1,n-1} = n - u_{1, n-1}.
	\end{cases}
 	\]
	Then $d_1(\textbf{\textup{u}}) = \cdots =d_n(\textbf{\textup{u}})$ implies that $d_k(\textbf{\textup{u}}) = 1$ for every $1 \leq k \leq n$.
	Therefore, by Proposition \ref{proposition_diagonal_all_equal}, if $d_k(\textbf{\textup{u}}) \neq 1$ for some $k$, then $\Phi^{-1}_\lambda(\textbf{\textup{u}})$ is displaceable. 
	The only possible candidate for a non-displaceable fiber is $\Phi^{-1}(\textbf{\textup{u}}_0)$ where $\textbf{\textup{u}}_0$ is the center of $\Delta_\lambda$, i.e.,
	$\textbf{\textup{u}}_0 = (u_{i,j})$ with $u_{1,j} = j$ for $j=1,\cdots,n-1$.
	Note that 
	it has been shown by Cho \cite{Cho} that $\Phi^{-1}_\lambda(\textbf{\textup{u}}_0)$ is non-displaceable.
	Therefore, $\Phi_\lambda^{-1}(\textbf{\textup{u}}_0)$ is a \emph{stem}.\footnote{A fiber of a moment map is called a {\em stem} if all other fibers are displaceable. See \cite{EP2}.}

\item (Complete flag manifold $\mcal{F}(3)$)
	Let $\lambda = \{\lambda_1, 0 , -\lambda_2\}$ for $\lambda_1, \lambda_2 > 0$ so that 
	$(\mathcal{O}_\lambda, \omega_\lambda)$ is a complete flag manifold $\mcal{F}(3)$ and it
	admits a unique proper Lagrangian face, a vertex $v_3$, of $\Delta_\lambda$ in Example \ref{example_123},
	see Figure \ref{figure_F123_non_torus_fiber} below.
	For any $x \in \Phi_\lambda^{-1}(v_3)$, we can easily see that $x_{1,1} = x_{1,2} = x_{2,1} = x_{2,2} = 0$,
	and $x_{3,3} = \lambda_1 - \lambda_2$.
	By Proposition \ref{proposition_diagonal_all_equal}, we can conclude that $\Phi^{-1}(v_3)$ is
	displaceable whenever $\lambda_1 \neq \lambda_2$. 
	We have reproduced the result of Pabiniak \cite{Pa} on displaceability of a (unique) non-torus GC Lagrangian fiber in $\mcal{F}(3)$.
	In Section \ref{pfofmaincorforfullflag3}, $\Phi^{-1}(v_3)$ will be shown to be non-displaceable when
	$\lambda_1 = \lambda_2$.

	\begin{figure}[H]
		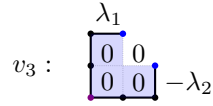
		\caption{\label{figure_F123_non_torus_fiber} The Lagrangian face of $\Gamma_\lambda$ in $\mcal{F}(3)$}
	\end{figure}

\item (Complex Grassmannian $\mathrm{Gr}(2,4)$)
	Let $\lambda = \{t,t,0,0\}$ for $t > 0$. By Corollary \ref{corollary_L_fillable},
	the edge $e$ in Figure \ref{figure_Lagrangian_Gr24} of $\Delta_\lambda$
	is the unique proper Lagrangian face of $\Gamma_\lambda$. 	For a positive real number $a$ with $0 \leq a \leq t$, we consider the point $r_a \in e$ given by $u_{1,1} = u_{1,2} = u_{2,1} = u_{2,2} = a$.
	Nohara-Ueda \cite{NU2} proved that the every fiber over the edge except the fiber $\Phi^{-1}_\lambda(r_{{t}/{2}})$ is displaceable and moreover $\Phi^{-1}_\lambda(r_{{t}/{2}})$ is non-displaceable.
	Our combinatorial test (Proposition \ref{proposition_diagonal_all_equal}) easily tells us that $\Phi^{-1}_\lambda(r_a)$ is displaceable unless $2a = t$.
	\begin{figure}[H]
		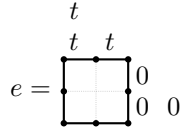
		\caption{\label{figure_Lagrangian_Gr24} Lagrangian face of $\Gamma_\lambda$ in $\mathrm{Gr}(2,4)$}
	\end{figure}

\item (Complex Grassmannian $\mathrm{Gr}(2,6)$)
	Let $\lambda = \{6,6,0,0,0,0\}$. By Example \ref{example_gr26}, there are exactly four proper Lagrangian faces
	$\gamma_2$, $\gamma_3$, $\gamma_4$, and $\gamma_{2,4}$ of $\Gamma_\lambda$ as follows.
	\begin{figure}[H]
		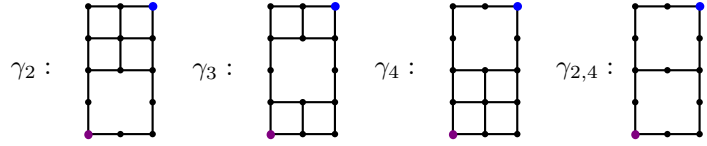
		\caption{\label{figure_4faces_gr26} Four proper Lagrangian faces of $\Gamma_\lambda$}
	\end{figure}
	
	First, consider two faces $\gamma_3$ and $\gamma_{2,4}$. By Proposition \ref{proposition_diagonal_all_equal}, we can easily check that
	every Lagrangian fiber over a point in $\gamma_3$ (resp. $\gamma_{2,4}$) is displaceable except for the fiber at $\textbf{\textup{u}}_3$
	(resp. $\textbf{\textup{u}}_{2,4}$) described in Figure \ref{figure_nondisp_point_gr26_24}.
	\begin{figure}[H]
		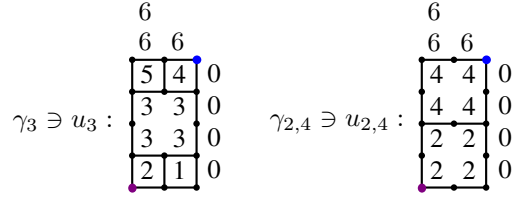
		\caption{\label{figure_nondisp_point_gr26_24} Lagrangian faces $\gamma_3$ and $\gamma_{2,4}$}
	\end{figure}
	
	For the other faces $\gamma_2$ and $\gamma_4$, again by Proposition \ref{proposition_diagonal_all_equal},
	every fiber is displaceable except for the one-parameter families of Lagrangian fibers $\textbf{\textup{u}}_2(t)$ and $\textbf{\textup{u}}_4(t)$ ($-1 < t < 1$)
	described in Figure \ref{figure_nondisp_point_gr26_13}
	in $\gamma_2$ and $\gamma_4$, respectively.
	\begin{figure}[H]
		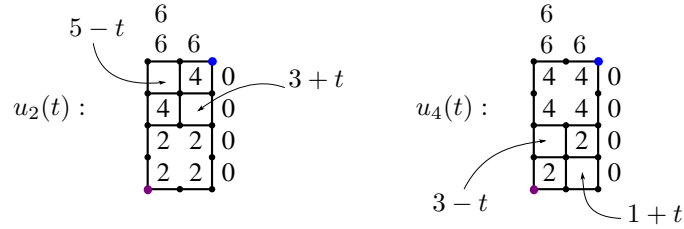
		\caption{\label{figure_nondisp_point_gr26_13} Lagrangian faces $\gamma_2$ and $\gamma_4$}
	\end{figure}
	In Section~\ref{ssecSymmetryOnGammaLambdaComplexGrassmannianCases},
	we will give another method for detecting displaceability of fibers and show that $\gamma_2$ and $\gamma_4$ are
	 indeed displaceable. Consequently, there are exactly two non-torus Lagrangian fiber
	$\Phi^{-1}_\lambda(\textbf{\textup{u}}_3)$ and $\Phi^{-1}_\lambda(\textbf{\textup{u}}_{2,4})$ that might be non-displaceable.
\end{enumerate}
\end{example}

\subsection{Symmetry on $\Gamma_\lambda$ (complex Grassmannian cases)}
\label{ssecSymmetryOnGammaLambdaComplexGrassmannianCases}~
\vspace{0.2cm}

In this section, we study displaceability of non-torus Lagrangian fibers of GC systems on complex Grassmannians.
Let $\mathrm{Gr}(k,n)$ be the Grassmannian of complex $k$-dimensional subspaces of $\C^n$.
Since $\mathrm{Gr}(k,n) \cong \mathrm{Gr}(n-k,n)$, without loss of generality, we assume that $n-k \geq k$.
Let $\lambda = \{ \underbrace{t,\cdots,t}_{k \textup{ times}}, ~\underbrace{0,\cdots,0}_{n-k \textup{ times}}\}$ for $t > 0$  so that
\[
	\mathrm{Gr}(k,n) \cong \mathcal{O}_\lambda.
\]
We start with the following series of algebraic lemmas which seems to be well-known.
\begin{lemma}\label{lemma_same_eigen_value}
	Let $A$ be any complex $(n \times k)$ matrix.
	Then $AA^*$ and $A^*A$ have the same non-zero eigenvalues with the same multiplicities.
\end{lemma}

\begin{proof}
	Recall that a singular value decomposition yields 
	\[
		A = U\Sigma V^*
	\]
	where $U$ is an $(n \times n)$ unitary matrix, $\Sigma$ is an $(n \times k)$ matrix such that $\Sigma_{i,j} = 0$ unless $i = j$, and $V$ is a
	$(k \times k)$ unitary matrix. Then $AA^* = U\Sigma \Sigma^* U^*$ and $A^*A = V\Sigma^* \Sigma V^*$ are unitary diagonalizations of
	$AA^*$ and $A^*A$ respectively. Since
	$\Sigma \Sigma^*$ and $\Sigma^* \Sigma$ have the same nonzero eigenvalues with
	the same multiplicities, this completes the proof.
\end{proof}

\begin{lemma}\label{lemma_k_n_matrix}
	Any element $x \in \mathcal{O}_\lambda$ can be expressed by
	\[
		x = X X^*
	\]	
	for some $(n \times k)$ matrix $X = [v_1, \cdots, v_k]$ such that
	\begin{itemize}
		\item $|v_i|^2 = t$ for every $i=1,\cdots, k$, and
		\item $\langle v_i, v_j \rangle = 0$ for every $i \neq j$
	\end{itemize}
	where $\langle \cdot, \cdot \rangle$ means the standard hermitian inner product on $\C^n$. In particular, we have
	\[
		X^*X = tI_k
	\]
	where $I_k$ is the $(k \times k)$ identity matrix.
\end{lemma}

\begin{proof}
	Let $x \in \mathcal{O}_\lambda$.
	Since $x$ is semi-positive definite by our choice of $\lambda$, there exists an $(n \times n)$ lower triangular matrix $L$ having non-negative diagonal entries such that
	\begin{equation}\label{equation_Cole}
		x = LL^*.
	\end{equation}
	The expression \eqref{equation_Cole} is called a {\em Cholesky factorization of $x$}, see \cite[Corollary 7.2.9]{HJ}.
	Let $L = U\Sigma V^*$ be a singular value decomposition of $L$. In this case, the matrices $U, \Sigma,$ and $V$ are all $(n \times n)$ matrices.
	Then we have
	\[
		x = LL^* = U\Sigma \Sigma^* U^* \quad \text{where} \quad 
		 \Sigma \Sigma^* = \begin{pmatrix} tI_k & 0 \\
                	  0 & O_k\\
                  \end{pmatrix}.
	\]	
        Let $D = D^* = \sqrt{\Sigma \Sigma^*}$ so that $x = UDD^*U^* $ and denote by $U = [u_1, \cdots, u_n]$. Since $UD = \sqrt{t}\cdot [u_1, \cdots, u_k, 0, \cdots, 0]$,
        we have $UDD^*U^* = UD(UD)^* = XX^*$ where $X$ is taken to be the $(n \times k)$ matrix
        \[
        	X = [v_1, \cdots, v_k] = \sqrt{t}\cdot [u_1, \cdots, u_k].
	\]
        This finishes the proof.
\end{proof}

Now, we consider Lagrangian faces of particular types in a ladder diagram $\Gamma_\lambda$ described as follows. Let $^{(0,i)} \square_k$ denote the $(k \times k)$
square whose upper-left corner is $(0,i)$ so that the vertices of $^{(0,i)} \square_k$ are $(0, i), (k, i), (k, i -k)$ and $(0, i-k)$. Let $\gamma_i$ be the graph drawn by all positive paths not passing through the interior of the square $^{(0,i)} \square_k$. Hence, $\gamma_i$ contains exactly one $(k \times k)$-sized simple closed region $^{(0,i)}\square_k$
and the other simple closed regions in $\gamma_i$ are unit squares, see Figure \ref{gr_nk_ladder}. Note that $\gamma_i$ is a Lagrangian face. 

\begin{figure}[H]
		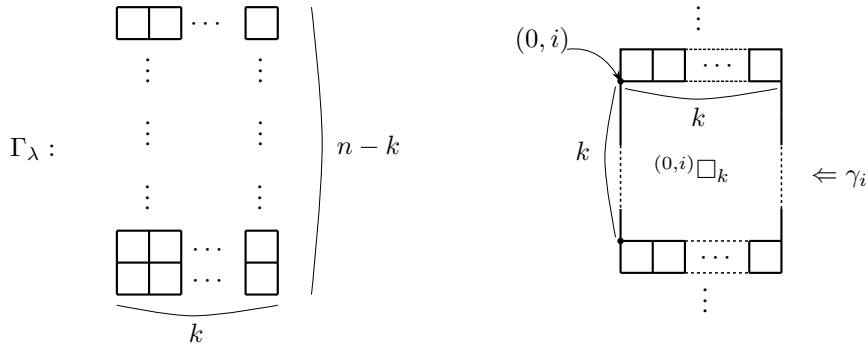
		\caption{\label{gr_nk_ladder} Ladder diagram $\Gamma_\lambda$ and Lagrangian faces $\{\gamma_i\}_{k \leq i \leq n-k}$}
\end{figure}

\begin{proposition}\label{proposition_gr_symmetry_case_one}
	For each $k \leq i \leq n-k$, there exists a permutation matrix $w \in U(n)$ such that
	\[
	w \cdot \Phi_\lambda^{-1}(\mathring{\gamma}_i) \subset \Phi_\lambda^{-1}(\gamma_{n-i}).
	\] In particular,  
	\[
		w \cdot \Phi^{-1}_\lambda(\mathring{\gamma}_i) \cap \Phi^{-1}_\lambda(\mathring{\gamma}_i) = \emptyset.
	\]
	 unless $n = 2i$. 
	Consequently, the face $\gamma_i$ is displaceable provided that $n \neq 2i$.
\end{proposition}
\begin{proof}
	Let $u$ be any point  in $\mathring{\gamma}_i$.
	Then we have
	\[
		u_{1, i} = u_{2, i-1} = \cdots = u_{k,i-k+1},
	\]
	which implies that for every $x \in \Phi^{-1}_\lambda(\textbf{\textup{u}})$, the $i$-th leading principal minor $x^{(i)}$ of $x$ has eigenvalues
	$u_{1,i}$ with multiplicity $k$ and $0$ with multiplicity $i - k$.
	
	Now, choose any $x \in \Phi^{-1}_\lambda(\textbf{\textup{u}})$. Then Lemma \ref{lemma_k_n_matrix} implies that
	there exists an $(n \times k)$ matrix
	\[
		X = \left( v_1, \cdots, v_k \right) = \left(\begin{array}{c} w_1 \\ w_2 \\ \vdots \\ w_n \\ \end{array}\right)
	\] such that $x = XX^*$ and it satisfies
	\begin{itemize}
		\item $|v_j|^2 = t$ for every $j=1, \cdots, k$, and
		\item $\langle v_j, v_{j'} \rangle = 0$ for every $j \neq j'$.
	\end{itemize}
	Now, we divide $X$ into two submatrices $X^{(i)}$ and $\check{X}^{(i)}$ where 
	\[
		X^{(i)} := \left(v_1^{(i)}, \cdots, v_k^{(i)} \right) = \left(\begin{array}{c} w_1 \\ w_2 \\ \vdots \\ w_i \\ \end{array}\right), \quad \quad
		\check{X}^{(i)} := \left(\check{v}_1^{(i)}, \cdots, \check{v}_k^{(i)}\right) = \left(\begin{array}{c} w_{i+1} \\ \vdots \\ w_n \\ \end{array}\right).
	\]
	In other words, $X^{(i)}$ (resp. $\check{X}^{(i)}$) is the $(i \times k)$ (resp. $((n-i) \times k)$)
	submatrix of $X$ obtained by deleting all $\ell$-th rows of $X$ for $\ell > i$ (resp. $\ell \leq i )$.
	Thus we have 
	\[
		x^{(i)} = X^{(i)}(X^{(i)})^*.
	\]
	Then Lemma \ref{lemma_same_eigen_value} implies that $x^{(i)} = X^{(i)}(X^{(i)})^*$ and $(X^{(i)})^*X^{(i)}$ have the same non-zero eigenvalue
	$u_{1, i}$ with the same multiplicity $k$. Since $(X^{(i)})^*X^{(i)}$ is a $(k \times k)$ matrix, we get
	\[
		(X^{(i)})^*X^{(i)} = u_{1, i} \cdot I_k.
	\]
	In particular, we have
	\begin{itemize}
		\item $|v_j^{(i)}|^2 = u_{1, i}$ for every $j=1,\cdots, k$, and
		\item $\langle v_j^{(i)}, v_{j'}^{(i)} \rangle = 0$ for every $j \neq j'$,
	\end{itemize}
	i.e., the columns of $X^{(i)}$ are orthogonal to each other and have the same square norm equal to $u_{1,i}$.
	This implies that $\check{X}^{(i)} = [\check{v}_1^{(i)}, \cdots, \check{v}_k^{(i)}]$ satisfies
	\begin{itemize}
		\item $|\check{v_j}^{(i)}|^2 = t - u_{1,i}$ for every $j=1, \cdots, k$, and
		\item $\langle \check{v}_j^{(i)}, \check{v}_{j'}^{(i)} \rangle  = 0$ for every $j \neq j'$
	\end{itemize}
	so that we have
	\[
		(\check{X}^{(i)})^*\check{X}^{(i)} = (t - u_{1, i}) \cdot I_{k}.
	\]
	Now, let $w \in \mathfrak{S}_n \subset U(n)$ be any permutation satisfying
	$
		w(i+j) = j
	$
	for $j=1,\cdots, n-i$. Then
	$
		(w  X)^{(n-i)} = \check{X}^{(i)}
	$
	and the following matrix
	\[
		w \cdot x = wxw^{-1} = w X X^*w^* = (w X)(w X)^*.
	\]
	has the $(n-i)$-th leading principal minor
	\[
		(w \cdot x)^{(n-i)} = (w  X)^{(n-i)} ((w  X)^{(n-i)})^* = (\check{X}^{(i)})(\check{X}^{(i)})^*.
	\]
	Again by Lemma \ref{lemma_same_eigen_value}, the nonzero eigenvalue of $(w \cdot x)^{(n-i)}$ is $t - u_{1,i}$ with multiplicity $k$ and zero with multiplicity
	$n-i-k$. Thus we have
	\[
		\Phi_\lambda^{1, n-i}(w \cdot x) = \Phi_\lambda^{2, n-i-1}(w \cdot x) = \cdots = \Phi_\lambda^{k, n-i-k+1}(w \cdot x) = t - u_{1,i}.
	\]
	Therefore, we have $\Phi_\lambda(w \cdot x) \in \gamma_{n-i}$.
	In particular if $n-i \neq i$, then
	$\Phi_\lambda(w \cdot x)$ is not in $\mathring{\gamma}_i$ since $\mathring{\gamma}_i \cap \gamma_{n-i} = \emptyset$.
\end{proof}

\begin{example}\label{example_gr26_13}
	Let us reconsider the Lagrangian faces $\gamma_2$ and $\gamma_4$ in Example \ref{example_non_disp} (4), which is the case where $n=6$ and $k=2$. Then
	\[
		n = 6 \neq 2 \cdot 2 = 2i ~\text{ for}~i=2,  \,\,\text{and} \,\, n = 6 \neq 2 \cdot 4 = 2i ~\text{ for} ~i=4.
	\]
	Thus the faces $\gamma_2$ and $\gamma_4$ are displaceable by Proposition \ref{proposition_gr_symmetry_case_one}.
	For the face $\gamma_3$ in Example \ref{example_non_disp} (4), however, we have
	\[
		n = 6 = 2 \cdot 3 = 2i \quad \text{and} \quad i = 3
	\]
	so that we cannot determine displaceability of $\gamma_3$ by using  Proposition \ref{proposition_gr_symmetry_case_one}.
\end{example}

Now, we extend Proposition \ref{proposition_gr_symmetry_case_one} to Lagrangian faces containing multiple squares of type ${}^{(0,\bullet)} \square_k$ in $\Gamma_\lambda$. 
For a sequence $(i_1, \cdots, i_r)$ satisfying
\begin{equation}\label{conditionforboxes}
\begin{cases}
&k \leq i_1 < \cdots < i_r \leq n-k, \\
&i_{s+1} - i_{s} \geq k \,\, \textup{for each } s,
\end{cases}
\end{equation}
let $\gamma_{i_1, \cdots, i_r}$ be the Lagrangian face of $\Gamma_\lambda$ which contains $r$ simple
closed regions $\{C_1 \cdots, C_r \}$ where $C_s$ is the square ${}^{(0,i_s)} \square_k$ and the other simple closed regions of $\gamma_{i_1, \cdots, i_r}$
are of size $1 \times 1$, see $\gamma_{2,4}$ in Figure \ref{figure_4faces_gr26} for example. 

\begin{proposition}\label{proposition_gr_symmetry_case_more_two}
	Suppose $\{ i_1 , \cdots , i_r \}$ is given satisfying~\eqref{conditionforboxes}.
	Then there is a permutation $w \in \mathfrak{S}_n$ such that
	\begin{equation}\label{permutingfaces}
		w \cdot \Phi_\lambda^{-1}(\mathring{\gamma}_{i_1, \cdots, i_r}) \cap \Phi_\lambda^{-1}(\mathring{\gamma}_{i_1, \cdots, i_r}) = \emptyset
	\end{equation}
	unless $i_s = s \cdot i_1$ for every $s=1,\cdots, r+1$ provided that $i_{r+1} := n$.
\end{proposition}

\begin{remark}
	Proposition \ref{proposition_gr_symmetry_case_one} can be obtained by Proposition \ref{proposition_gr_symmetry_case_more_two}
	by taking $r = 1$.
\end{remark}

\begin{proof}
	Let $\textbf{\textup{u}} \in \mathring{\gamma}_{i_1, \cdots, i_r}$. Then $\textbf{\textup{u}}$ satisfies
	\[
		u_{1, i_s } = u_{2, i_s -1} = \cdots = u_{k, i_s -k +1}
	\]
	for every $s=1,\cdots, r$.
	
	For any $x \in \Phi^{-1}_\lambda(\textbf{\textup{u}})$, Lemma \ref{lemma_k_n_matrix} implies that
	there exists an $(n \times k)$ matrix $X = [v_1, \cdots, v_k]$ such that
	 \begin{itemize}
	 	\item $x = XX^*$,
	 	\item $|v_j|^2 = t$ for every $j=1,\cdots,k$, and
	 	\item $\langle v_j , v_{j'} \rangle = 0$ for $j \neq j'$.
	 \end{itemize}
	 Then we can divide $X$ into $r+1$ submatrices $X^{(i_1)}_{(i_0)}, X^{(i_2)}_{(i_1)}, \cdots, X^{(i_r)}_{(i_{r-1})}, X^{(i_{r+1})}_{(i_r)}$ of $X$
	 (provided that $i_0 := 0$ and $i_{r+1} := n$) where
	 \[
	 	X^{(i_{s+1})}_{(i_s)} = \left({v_1}^{(i_{s+1})}_{(i_s)}, \cdots, {v_k}^{(i_{s+1})}_{(i_s)}\right) := \left(\begin{array}{c} w_{i_s + 1} \\ \vdots \\ w_{i_{s+1}} \\ \end{array}\right)
	 \]
	 for each $s=0,\cdots, r$. In other words,
	$X^{(i_{s+1})}_{(i_s)}$ is the $((i_{s+1} - i_s) \times k)$ submatrix of $X$ obtained by deleting all $\ell$-th rows of $X$ for $\ell > i_{s+1}$ and $\ell \leq i_s$.
	By using Lemma \ref{lemma_same_eigen_value} and Lemma \ref{lemma_k_n_matrix},
	it is not hard to show that
	 \begin{itemize}
	 	\item $|{v_j}^{(i_{s + 1})}_{(i_s)}|^2 = u_{1, i_{s+1}} - u_{1, i_s}$ for every $j=1,\cdots,k$ and $s=0,\cdots,r$, and
	 	\item $\langle {v_j}^{(i_{s + 1})}_{(i_s)}, {v_{j'}}^{(i_{s + 1})}_{(i_s)} \rangle = 0$ for $j \neq j'$.
	 \end{itemize}

Now, suppose that
\begin{equation}\label{permufacesss}
w \cdot \Phi_\lambda^{-1}(\mathring{\gamma}_{i_1, \cdots, i_r}) \cap \Phi_\lambda^{-1}(\mathring{\gamma}_{i_1, \cdots, i_r}) \neq \emptyset
\end{equation}
for every $w \in \mathfrak{S}_n$. We claim that
	\[
		i_s = s \cdot i_1 \quad\quad ~\text{for every}~ s =1,\cdots, r+1 \quad \text{where} \quad i_{r+1} := n.
	\]
	To show this, consider $w:= w(s,s') \in \mathfrak{S}_n \subset U(n)$ for any $s, s' \in \{1,\cdots, r+1\}$ with $s < s'$ defined by
	 \[
	 	\begin{array}{c}
		\{1, \cdots, i_s, \underline{i_s + 1, \cdots, i_{s+1}}, i_{s+1} + 1, \cdots, i_{s'}, \underline{i_{s'} + 1, \cdots, i_{s'+1}}, i_{s'+1} + 1, \cdots, n\}   \\[0.5em]
		w \bigg \downarrow \\[0.5em]
		\{1, \cdots, i_s, \underline{i_{s'} + 1, \cdots, i_{s'+1}} , i_{s+1} + 1, \cdots, i_{s'}, \underline{i_s + 1, \cdots, i_{s+1}}, i_{s'+1} + 1, \cdots, n\}. \\
		\end{array}
	 \]
	Then we can easily see that the permutation $w$ swaps the position of $X^{(i_{s+1})}_{(i_s)}$ with $X^{(i_{s'+1})}_{(i_{s'})}$ in $X$ so that 
	\[
		\Phi_\lambda(w \cdot x) \in \gamma_{i_1', \cdots, i_r'}
	\]
	for some $k \leq i_1'< \cdots < i_r' \leq n-k$ where
	\[
		i_{s+1}' = i_s + (i_{s'+1} - i_{s'}).
	\]
	By our assumption \eqref{permufacesss}, we have
	$\gamma_{i_1', \cdots, i_r'} \cap \mathring{\gamma}_{i_1, \cdots, i_r} \neq \emptyset$.
	Since two faces $\gamma_{i_1', \cdots, i_r'}$ and $\gamma_{i_1, \cdots, i_r}$
	have the same dimensions, they must coincide and hence 
	\[
		i_{s+1}' = i_s + (i_{s'+1} - i_{s'}) = i_{s+1}.
	\]
	Therefore, we may deduce that 
	\[
		i_{r+1} - i_r = i_r - i_{r-1} = \cdots = i_2 - i_1 = i_1 - i_0 = i_1,
	\]
	which completes the proof.	
\end{proof}

\begin{corollary}[Theorem~\ref{theoremE}]\label{corollary_gr_2x}
	Let $\lambda = \{t,t,0,\cdots, 0\}$ so that $\mathcal{O}_\lambda \cong \mathrm{Gr}(2,n)$.
	If $n$ is prime, then every proper Lagrangian face of the Gelfand-Cetlin system on $(\mathcal{O}_\lambda, \omega_\lambda)$ is displaceable.
\end{corollary}

\begin{proof}
	Note that every proper Lagrangian face of $\Gamma_\lambda$ is of the form $\gamma_{i_1, \cdots, i_r}$ for some $i_1, \cdots, i_r$.
	By Proposition \ref{proposition_gr_symmetry_case_more_two}, $i_1$ divides $n$ where $2 \leq i_1 \leq n-2$ which is impossible since $n$ is prime.
	Thus there is no proper non-displaceable Lagrangian face of $\Gamma_\lambda$.
\end{proof}

%------------------------------------------------------------------------------------------
\vspace{0.2cm}
\section{Lagrangian Floer theory on Gelfand-Cetlin systems}\label{secReviewOfLagrangianFloerTheory}

The aim of this section is to review Lagrangian Floer theory which will be used to prove the results in Section~\ref{Sec_Intropart2}. After briefly recalling Lagrangian Floer theory and its deformation developed by the third named author with Fukaya, Ohta, and Ono in a general context, we review work of Nishinou, Nohara, and Ueda about the calculation of the potential function of a GC system. Then, using the combinatorial description of Schubert cycles in complete flag manifolds by Kogan, we will express the potential function deformed by a combination of Schubert cycles of codimension two as a Laurent series. Finally, combining those ingredients, we provide the proof of Theorem~\ref{theoremC}.

%------------------------------------------------------------------------------------------
\subsection{Potential functions of Gelfand-Cetlin systems}\label{reviewpotentialpotentialfunc}~
\vspace{0.2cm}

Let $\Lambda$ be the Novikov field over the field of complex numbers, defined by
\begin{equation}\label{Novikovfield}
\Lambda := \left\{ \sum_{j=1}^{\infty} a_{j} T^{\lambda_j}  \, \bigg{|} \, a_j \in \C, \lambda_j \in \R, \lim_{j \rightarrow \infty} \lambda_j = \infty \right\}.
\end{equation}
It is algebraically closed by Lemma A.1 in \cite{FOOOToric1}. It comes with the valuation
$$
\frak{v}_T : \Lambda \backslash \{ 0 \} \to \R, \quad \frak{v}_T \left(\sum_{j=1}^{\infty} a_{j} T^{\lambda_j}  \right) := \inf_{j} \, \{ \lambda_j : a_j \neq 0 \}.
$$
We also play with two subrings of $\Lambda$ given by
\begin{align*}
\Lambda_0 &:= \frak{v}_T^{-1}[0, \infty) \cup \{ 0 \} = \left\{ \sum_{i=1}^{\infty} a_{i} T^{\lambda_i} \in \Lambda  \, \bigg{|} \, \lambda_i \geq 0 \right\} \\
\Lambda_+ &:= \frak{v}_T^{-1}(0, \infty) \cup \{ 0 \} = \left\{ \sum_{i=1}^{\infty} a_{i} T^{\lambda_i} \in \Lambda  \, \bigg{|} \, \lambda_i > 0 \right\}. \,
\end{align*}
Let $\Lambda_U$ be the collection of unitary elements of $\Lambda$. That is,
\begin{align*}
\Lambda_U &:= \Lambda_0 \backslash \Lambda_+ = \left\{ \sum_{i=1}^{\infty} a_{i} T^{\lambda_i} \in \Lambda_0  \, \bigg{|} \, \frak{v}_T \left( \sum_{i=1}^{\infty} a_{i} T^{\lambda_i} \right) = 0 \right\}.
\end{align*}

For a compact relatively spin Lagrangian submanifold $L$ in a compact symplectic manifold $(X, \omega)$, thanks to the work of Fukaya \cite{Fuk}, one can associate a sequence of $A_\infty$-structure maps $\{ \frak{m}_k \}_ {k \geq 0}$ on the $\Lambda_0$-valued de Rham complex of $L$, which comes from moduli spaces of holomorphic discs bounded by $L$. Following the procedure in \cite{FOOOc} for instance, the $A_\infty$-algebra can be converted into the canonical model on $H^\bullet (L; \Lambda_0)$. By an abuse of notation, the structure maps of the canonical model are still denoted by $\frak{m}_k$'s because the canonical model is only dealt with from now on.

A solution $b \in H^1(L; \Lambda_+)$ of the (weak) Maurer-Cartan equation
$$
\sum_{k=0}^\infty \frak{m}_k (b^{\otimes k}) \equiv 0 \mod \textup{PD}[L]
$$
is said to be a \emph{(weak) bounding cochain}. The value of the \emph{potential function} $\frak{PO}$ at a bounding cochain $b$ is assigned to be the multiple of the Poincar\'{e} dual $\textup{PD}[L]$ of $L$. Namely,
$$
\sum_{k=0}^\infty \frak{m}_k (b^{\otimes k}) = \frak{PO}(b) \cdot \textup{PD}[L].
$$
Since $\textup{PD}[L]$ is the strict unit of the $A_\infty$-algebra, the deformed map
$$
\frak{m}^{b}_1 (h) = \sum_{l,k} \frak{m}_{l+k+1}(b^{\otimes l}, h, b^{\otimes k})
$$
becomes a differential and thus the Floer cohomology (deformed by $b$) over $\Lambda_0$ can be defined by
$$
{HF} ((L, b); \Lambda_0) := \textup{Ker} (\frak{m}^b_1) \, / \, \textup{Im} (\frak{m}^b_1).
$$
The Floer cohomology (deformed by $b$) over $\Lambda$ is defined by
$$
{HF} ((L, b); \Lambda) := {HF} ((L, b); \Lambda_0) \otimes_{\Lambda_0} \Lambda.
$$
The reader is referred to \cite{FOOO, FOOOToric1, FOOOToric2, FOOOToric3} for details.

Now, we specialize to the case of a Lagrangian GC torus fiber $\Phi_\lambda^{-1}(\textbf{u})$ in a co-adjoint orbit $\mcal{O}_\lambda$. As a deformation of the action of the Borel subgroup, Kogan and Miller \cite{KM} realized the toric degeneration of a flag manifold given by Gonciulea-Lakshmibai \cite{GL} and Batyrev et al. \cite{BCKV}.
Using the degeneration of a (partial) flag manifold to the GC toric variety (in stages), Nishinou-Nohara-Ueda \cite{NNU} constructed a (not-in-stages) degeneration of the GC system of $\mcal{O}_\lambda$ to the moment map of the toric variety. 

\begin{theorem}[Theorem 1.2 in \cite{NNU}]\label{NNUtoricdeg}
For any non-increasing sequence $\lambda = \{\lambda_1 \geq \cdots \geq \lambda_n \}$, there exists a toric degeneration of the Gelfand-Cetlin system $\Phi_\lambda$ on the co-adjoint orbit $(\mcal{O}_\lambda, \omega_\lambda)$ in the following sense.
\begin{enumerate}
\item There is a flat family $f \colon \mcal{X} \to I = [0,1]$ of algebraic varieties and a symplectic form $\widetilde{\omega}$ on $\mcal{X}$ such that
\begin{enumerate}
\item ${X}_0 := f^{-1}(0)$ is the toric variety associated with the Gelfand-Cetlin polytope $\Delta_\lambda$ and $\omega_0 := \widetilde{\omega}|_{X_0}$ is a torus-invariant K{\"a}hler form.
\item ${X}_1 := f^{-1}(1)$ is the co-adjoint orbit $\mcal{O}_\lambda$ and $\omega_1 = \widetilde{\omega}|_{X_1}$ is the Kirillov-Kostant-Souriau symplectic form {$\omega_\lambda$}.
\end{enumerate}
\item There is a family $\{\Phi_t \colon X_t \rightarrow \Delta_\lambda\}_{0 \leq t \leq 1}$ of completely integrable systems such that $\Phi_0$ is the moment map for the torus action on $X_0$ and $\Phi_1$ is the Gelfand-Cetlin system.
\item Let $\Delta^{\textup{sm}}_\lambda := \Delta_\lambda \backslash \Phi_0 \left( \textup{Sing} (X_0) \right)$ and
$X^{\textup{sm}}_t := \Phi_t^{-1}(\Delta^{\textup{sm}}_\lambda )$ where $\textup{Sing} (X_0)$ is the set of singular points of $X_0$. 
Then, there exists a flow $\phi_t$ on $\mcal{X}$ such that for each $0 \leq t \leq s$,
the restricted flow $\phi_t |_{X^{\textup{sm}}_{s}}  : X^{\textup{sm}}_{s} \to X^{\textup{sm}}_{s-t}$ respects the symplectic structures and the complete integrable systems:
\[
	\xymatrix{
		  (X^{\textup{sm}}_{s}, \omega_s)  \ar[dr]_{\Phi_s} \ar[rr]^{ \phi_t|_{X^{\textup{sm}}_{s}} }  & & (X^{\textup{sm}}_{s-t}, \omega_{s-t})
      \ar[dl]^{\Phi_{s-t}} \\
  & \Delta^{\textup{sm}}_\lambda &}
\]
\end{enumerate}
\end{theorem}
Let $\phi^\prime_s \colon X_s \to X_0$ be a (continuous) extension of the flow $\phi_s \colon X^{\textup{sm}}_s \to X^{\textup{sm}}_0$ in Theorem~\ref{NNUtoricdeg} (\cite[Section 8]{NNU}). The extended map $\phi^\prime_s$ effectively transports data for Floer theory from the toric moment map to a nearby integrable system. 
As the deformation in Theorem~\ref{NNUtoricdeg} is through Fano varieties and the GC toric variety admits a small resolution at the singular loci, any holomorphic discs bounded by $L$ intersecting the loci collapsing to the singular loci of $X_0$ must have the Maslov index strictly greater than two so that such discs do \emph{not} contribute to the potential function. Furthermore, because the Fredholm regularity is an open condition, the holomorphic discs of Maslov index two intersecting the toric divisor in the toric variety $X_0$ give rise to regular holomorphic discs at $X_s$ for sufficiently small $s$. Set $X := X_s$ and let $L$ be any Lagrangian torus fiber of $\Phi_s \colon X_s \to \Delta_\lambda$. Combining those with the results of Cho-Oh \cite{CO}, Nishinou-Nohara-Ueda proved the followings.   

\begin{enumerate}
\item Each Lagrangian torus fiber $L$ does not bound any non-constant holomorphic discs whose classes are of Maslov index less than or equal to zero. \cite[Lemma 9.20]{NNU}
\item Every class $\beta \in \pi_2 (X, L)$ of Maslov index two is Fredholm regular. \cite[Proposition 9.17]{NNU}
\item There is a one-to-one correspondence between the holomorphic discs of Maslov index two bounded by a Lagrangian Gelfand-Cetlin torus fiber and the facets of Gelfand-Cetlin polytope. \cite[Lemma 9.22]{NNU}
\item For each class $\beta$ of Maslov index two, the open Gromov-Witten invariant $n_\beta$, which counts the holomorphic discs passing through a generic point in $L$ and representing $\beta$, is $1$. \cite[Proposition 9.16]{NNU}
\end{enumerate}

The above conditions imply that every $1$-cochain in $H^1(L; \Lambda_0)$ is a weak bounding cochain and hence the potential function can be defined on $H^1(L; \Lambda_0)$. To extend deformation space from $H^1(L; \Lambda_+)$ to $H^1(L; \Lambda_0)$, one should consider Floer theory twisted with flat non-unitary line bundles, see Cho \cite{Cho2}. Moreover, the potential function can be written as
\begin{equation}\label{certainpotential}
\frak{PO} \left( L ; b \right) = \sum_{\beta} n_\beta \cdot \exp(\pa \beta \cap b) \, T^{\omega(\beta) / 2 \pi}
\end{equation}
where the summation is taken over all homotopy classes in $\pi_2 (X, L)$ of Maslov index two. Setting
$$
\Gamma(n) = \{ (i, j) ~\colon~ 2 \leq i + j \leq n\},
$$
we fix the basis $\{ \gamma_{i,j} ~\colon~ (i,j) \in \Gamma(n) \}$ for $H^1(L; \Z)$ dual to the basis for $H_1(L, \Z)$ which consists of the circles generated by $\{ u_{i,j}: (i,j) \in \Gamma(n) \}$. Take the exponential variables
\footnote{In \cite{NNU}, Nishinou, Nohara, and Ueda set
$$
y_{i,j} := e^{x_{i,j}} T^{u_{i,j}}
$$
so that it is different from our $y_{i,j}$ in~\eqref{exponencoord}. To keep track of the valuations of holomorphic discs, we prefer to take $y_{i,j}$ as the expotential variable without weights. 
}
\begin{equation}\label{exponencoord}
y_{i,j} := e^{x_{i,j}}
\end{equation}
where $b$ is expressed as the linear combination  $\sum_{(i,j) \in \Gamma(n)} x_{i,j} \cdot \gamma_{i,j}$. Then, the potential function can be expressed as a Laurent polynomial $\frak{PO}(\textup{\textbf{y}})$ with respect to $\{ y_{i,j} ~\colon~ (i,j) \in \Gamma(n) \}$. Setting $u_{i,n+1-i}:= \lambda_{i}$, keep in mind that $\Delta_\lambda$ is defined by
$$
\left\{ (u_{i,j}) \in \R^{n(n-1)/2} ~\colon~ u_{i, j+1} - u_{i,j} \geq 0, \, u_{i,j} - u_{i+1, j} \geq 0 \, \right\}.
$$
(see Theorem~\ref{theorem_equiv_CG_LD} for instance.) 

\begin{theorem}[Theorem 10.1 in \cite{NNU}]
Setting
$$
y_{i,n+1-i} := 1 \quad u_{i,n+1-i}:= \lambda_{i},
$$
we consider the Lagrangian torus $L$ over a point $\{ u_{i,j} : (i,j) \in \Gamma(n) \}$. Then, the potential function on $L$ is written as
\begin{equation}\label{potentialfunctionoriginal}
\frak{PO}(L; \textup{\textbf{y}}) = \sum_{(i,j)} \left( \frac{y_{i, j+1}}{y_{i,j}} \, T^{ u_{i,j+1} - u_{i,j}} + \frac{y_{i,j}}{y_{i+1, j}} \, T^{ u_{i,j} - u_{i+1, j}}  \right)
\end{equation}
where the summation is taken over $\Gamma(n) = \{ (i, j) : 2 \leq i + j \leq n\}$.
\end{theorem}

\begin{remark}
As we are mainly concerned with non-displaceability, it is enough to show that $L_s$ is non-displaceable because of the homotopy invariance of $A_\infty$-structures.
\end{remark}

For $t$ with $0 \leq t < 1$, the potential function of the Lagrangian torus $L_m(t)$ over $I_m(t)$ is arranged as follows:
\begin{align}\label{potential}
\begin{split}
\frak{PO}(L_m(t); \textbf{y}) &= \left( \sum_{i = 1}^{m} \sum_{j=1}^{m-1} \frac{y_{i, j+1}}{y_{i,j}}+ \sum_{i= 1}^{m-1} \sum_{j=1}^{m} \frac{y_{i, j}}{y_{i+1,j}}\right) T^{1-t} + \left(  \sum_{\substack{\max(i, j) \geq m+1 }} \left( \frac{y_{i, j+1}}{y_{i,j}}+ \frac{y_{i, j}}{y_{i+1,j}} \right) \right) T^{1}  \\
&+ \left( \sum_{i=0}^{m-1} \left( \frac{y_{m-i,m+1}}{y_{m-i,m}} + \frac{y_{m,m-i}}{y_{m+1,m-i}} \right) T^{1 + it}   \right)
\end{split}
\end{align}
For simplicity, we frequently omit $L_m(t)$ in $\frak{PO}(L_m(t); \textbf{y})$ if $L_m(t)$ is clear in the context. The logarithmic derivative with respect to $y_{i,j}$ is denoted by
\begin{equation}\label{shortnotofderi}
\pa (i,j)(\textbf{y}) := y_{i,j} \frac{\pa \frak{PO}}{\pa y_{i,j}}.
\end{equation}

\begin{example}
In the complete flag manifold $\mcal{F}(5) \simeq \mcal{O}_\lambda$ where $\lambda = \{ 4, 2, 0, -2, -4\}$, some logarithmic derivatives of $\frak{PO}(L_2(t); \textbf{y})$ are as follows.
\begin{align*}
\pa (1,1)(\textbf{y}) &= \left( - \frac{y_{1, 2}}{y_{1,1}} + \frac{y_{1,1}}{y_{2, 1}} \right) T^{1-t}, \\
\pa (2,2)(\textbf{y}) &= \left( - \frac{y_{1, 2}}{y_{2,2}} + \frac{y_{2,2}}{y_{2,1}}\right) T^{1-t} + \left( - \frac{y_{2, 3}}{y_{2,2}} + \frac{y_{2,2}}{y_{3, 2}} \right) T^{1}, \\
\pa (2,3)(\textbf{y}) &= \left( - \frac{1}{y_{2,3}} - \frac{y_{1,3}}{y_{2,3}} + {y_{2,3}} + \frac{y_{2,3}}{y_{2, 2}} \right) T^{1}.
\end{align*}
\end{example}

%------------------------------------------------------------------------------------------
\vspace{0.2cm}
\subsection{Bulk-deformations by Schubert cycles}\label{section:bulkdeform}~
\vspace{0.2cm}

We will apply Lagrangian Floer theory deformed by cycles of an ambient symplectic manifold, developed by the third named author with Fukaya, Ohta, and Ono in \cite{FOOO, FOOOToric2, FOOOToric3}, in order to show non-displaceability. For a Lagrangian torus $L$ from Section~\ref{reviewpotentialpotentialfunc}, we deform the underlying $A_\infty$-algebra by empolying moduli spaces of holomorphic discs with interior marked points passing through a combination of designated ambient cycles at the image of interior marked points. We are particularly interested in a combination of cycles $\scr{D}_j$ of degree two not intersecting $L$
$$
\frak{b}:= \sum_{j=1}^B \frak{b}_j \cdot \scr{D}_j,
$$
which is called a \emph{bulk-deformation parameter}. As the $A_\infty$-algebra is deformed, the potential function is also deformed. The deformed potential function is denoted by $\frak{PO}^\frak{b}$.

\begin{remark} For the purpose of the present paper, we use only cycles of degree two to deform Lagrangian
Floer theory because the deformed potential function is computable. It also simplifies construction of
virtual fundamental cycles needed for the definition of open Gromov-Witten invariants
\cite[Definition 6.7]{FOOOToric2}, denoted by $n_\beta({\bf p})$.
\end{remark}

We first recall the formula for the potential function of a toric fiber $L$ deformed by a combination of toric divisors $$
\frak{b}:= \sum_{j=1}^B \frak{b}_j \cdot \scr{D}_j,
$$
in a compact toric manifold $X$ from \cite{FOOOToric2}.

\begin{theorem}[\cite{FOOOToric2}]\label{Foootoric2potenti} The bulk-deformed potential function, also called the potential function with bulk, is written as
\begin{equation}\label{bulkdeformedpotential}
\frak{PO}^\frak{b} \left( L ; b \right) = \sum_{\beta} n_\beta \cdot \exp \left( \sum_{j=1}^B \left( \beta \cap \scr{D}_{j} \right) \frak{b}_{j} \right) \exp(\pa \beta \cap b) \, T^{\omega(\beta) / 2 \pi}.
\end{equation}
where the summantion is taken over all homotopy classes in $\pi_2 (X, L)$ of Maslov index two.
\end{theorem}

In the derivation of this in \cite{FOOOToric2}, the properties that
the relevant ambient cycles  are smooth and $T^n$-invariant are used. Since our Schubert cycles
are neither smooth nor $T^n$-invariant in general, we will provide details of
the proof of this theorem for the current GC case modifying the arguments used in the proof of \cite[Proposition 4.7]{FOOOToric2} similarly as done in \cite[Section 9]{NNU}, see Section~\ref{sec_bulkdefbyschucy}. The upshot is that we still have the same formula for the potential function with bulk in the current GC case, see~\eqref{bulkdeformedpotential}. 
Again by taking the system of exponential coordinates in~\eqref{exponencoord}, $\frak{PO}^\frak{b}$ in~\eqref{bulkdeformedpotential} becomes a Laurent series with respect to $\{ y_{i,j} : (i,j) \in \Gamma(n) \}$.

\begin{theorem}[Section 8 in \cite{FOOOToric2}]\label{foootoric2nondispl}\label{criticalpointimpliesnondisplaceability}
If the bulk-deformed potential function $\frak{PO}^\frak{b} (L; \textbf{\textup{y}})$ admits a critical point $\textbf{\textup{y}}$ whose components are in $\Lambda_U$, then $L$ is non-displaceable.
\end{theorem}

In our case, we will employ Schubert cycles to deform the potential function $\frak{PO}$. In his thesis \cite{K}, Kogan found an expression of a Schubert cycle in terms of a certain union of the inverse images of faces in the GC system of a complete flag manifold, see also Kogan-Miller \cite{KM}. As we have seen in Section~\ref{secLagrangianFibersOfGelfandCetlinSystems}, due to presence of non-torus fibers, the inverse image of a face might have boundary so that it does \emph{not} represent a cycle. What he proved is that a certain combination of faces can represent a cycle because the boundaries are cancelled out.

We review the result in terms of ladder diagrams. A facet in a GC polytope is said to be \emph{horizontal} (resp. \emph{vertical}) if it is given by $u_{i,j} = u_{i+1,j}$ (resp. $u_{i, j+1} = u_{i,j}$). Let $P^\textup{hor}_{i,i+1}$ (resp. $P^\textup{ver}_{j+1,j}$) be the union of horizontal (resp. vertical) facets between the $i$-th column and the $(i+1)$-th column (resp. the $(j+1)$-th row and the $j$-th row) of the ladder diagram. That is,
\begin{align*}
P^\textup{hor}_{i,i+1} := \bigcup_{s=1}^{n-i} \{ u_{i, s} = u_{i+1, s} \}, \quad P^\textup{ver}_{j+1, j} := \bigcup_{r=1}^{n-j} \{ u_{r, j+1} = u_{r, j} \}
\end{align*}
for $1 \leq i, j  \leq n - 1$ where $\{ u_{\bullet, \bullet} = u_{\bullet, \bullet}\}$ denotes the facet given by the equation inside.
Let 
$$
\scr{D}^\textup{hor}_{i,i+1} := \Phi^{-1}_\lambda \left( P^\textup{hor}_{i,i+1} \right), \quad
\scr{D}^\textup{ver}_{j+1,j} := \Phi^{-1}_\lambda \left( P^\textup{ver}_{j+1,j} \right),
$$
which are respectively called a \emph{horizontal} and \emph{vertical Schubert cycle} (of degree two) because of Theorem~\ref{thm_KoganMiller}.

\begin{example}
Consider the co-adjoint orbit $\mcal{O}_\lambda \simeq \mcal{F}(6)$ where $\lambda = (5, 3, 1, -1, -3, -5)$.
\begin{enumerate}
\item $P^\textup{hor}_{4,5}$ is the union of two horizontal facets
$$
P^\textup{hor}_{4,5} = \{ u_{4,2} = -3 \} \cup \{ u_{4,1} = u_{5,1} \}
$$
as in Figure~\ref{HorizontalP45}.
\item $P^\textup{ver}_{4,3}$ is the union of three vertical facets
$$
P^\textup{ver}_{4,3} = \{ 1 = u_{3,3} \} \cup \{ u_{2,4} = u_{2,3} \} \cup \{u_{1,4} = u_{1,3} \}
$$
as in Figure~\ref{VerticalP43}.
\end{enumerate}
\begin{figure}[ht]
	\scalebox{1}{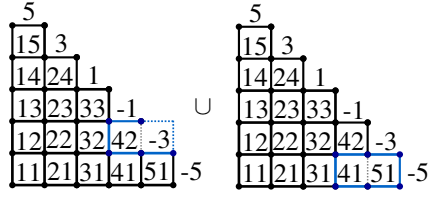}
	\caption{\label{HorizontalP45} $P^\textup{hor}_{4,5}$ in $\mathcal{F}(6)$.}	
\end{figure}
\begin{figure}[ht]
	\scalebox{1}{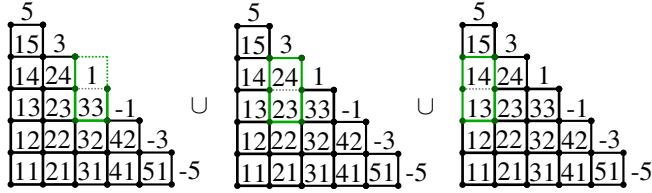}
	\caption{\label{VerticalP43} $P^\textup{ver}_{4,3}$ in $\mathcal{F}(6)$.}	
\end{figure}
\end{example}

By following the combinatorial process playing with \emph{reduced pipe dreams} in \cite{K,KM}, observe that the Schubert varieties associated with the adjacent transpositions
\footnote{
An adjacent transposition is a transposition of the form $(i, i +1)$. 
},
having a complex codimension one, are corresponding to either unions of horizontal facets or unions of vertical facets. The opposite Schubert varieties are corresponding to the other, see Remark 9 in \cite{KM}.

\begin{theorem}[Theorem 2.3.1 in \cite{K}, Theorem 8 in \cite{KM}]\label{thm_KoganMiller}
The inverse image $\scr{D}^\textup{hor}_{\bullet, \bullet+1}$ (or $\scr{D}^\textup{ver}_{\bullet+1, \bullet}$) represents a Schubert cycle of degree two. 
For any (opposite) Schubert divisor $X_w$, there exists either $\scr{D}^\textup{hor}_{\bullet, \bullet+1}$ or $\scr{D}^\textup{ver}_{\bullet+1, \bullet}$ such that it represents the cycle $[X_w]$ of the divisor.
\end{theorem}

Now, we apply~\eqref{bulkdeformedpotential}, which is a counterpart of~\eqref{Foootoric2potenti}, to calculate the bulk-deformed potential function. Because of the condition (3) in Section~\ref{reviewpotentialpotentialfunc}, let
$$
\beta^{i,j}_{i+1, j}  \,\,\, \left(\textup{resp. } \beta^{i,j+1}_{i, j} \right)
$$
be the homotopy class in $\pi_2(\mcal{O}_\lambda, L)$ represented by a holomorphic disc of Maslov index $2$ corresponding to $u_{i,j} = u_{i+1,j}$ (resp. $u_{i,j+1} = u_{i,j}$).

\begin{lemma}\label{intersectionnumbercal}
Let $\scr{D}$ be either a horizontal or a vertical Schubert cylce in $X_s$. Then, we have
$$
\beta^{i,j}_{i+1,j} \cap \scr{D} =
\begin{cases}
1 \quad \mbox{if } \scr{D} = \scr{D}^\textup{hor}_{i, i+1} \\
0 \quad \mbox{otherwise, }
\end{cases}
\quad
\beta^{i,j+1}_{i,j} \cap \scr{D} =
\begin{cases}
1 \quad \mbox{if } \scr{D} = \scr{D}^\textup{ver}_{j+1, j} \\
0 \quad \mbox{otherwise. }
\end{cases}
$$
\end{lemma}

\begin{proof}
Let $\phi^\prime_s \colon X_s \to X_0$ be a (continuous) extension of the flow $\phi_s \colon X^{\textup{sm}}_s \to X^{\textup{sm}}_0$
in Theorem~\ref{NNUtoricdeg}, see \cite[Section 8]{NNU}.
Let $\varphi \colon (\mathbb{D}, \pa \mathbb{D}) \to (X_s, L_s)$ be a holomorphic disc in the class $\beta^{i,j}_{i+1,j}$ of Maslov index two for example.
Then, we have a (topological) disc $\phi^\prime_s \circ \varphi \colon (\mathbb{D}, \pa \mathbb{D}) \to (X_s, L_s) \to (X_0, L_0)$, representing $(\phi^\prime_s)_* \beta^{i,j}_{i+1,j}$. Note that there exists a holomorphic disc $\varphi_0$ by \cite{CO} in the class $(\phi^\prime_s)_* \beta^{i,j}_{i+1,j} = [\phi^\prime_s \circ \varphi]$. Meanwhile, by our choice of $\scr{D}$, $\phi^\prime_s (\scr{D})$ is the union of the components over either $P^\textup{hor}_{i, i+1}$ or $P^\textup{ver}_{j+1, j}$.
Since the flow $\phi^\prime_s$ gives rise to a symplectomorphism from $X^{\textup{sm}}_s$ to $X^{\textup{sm}}_0$ and the image of the disc $\varphi$ is contained in $X^{\textup{sm}}_s$, the (local) intersection number should be preserved through the flow $\phi^\prime_s$. To calculate the intersection number, we consider a small resolution $p \colon \widetilde{X}_0 \to X_0$. Because the intersection happens only outside the singular loci at $X_0$, we can lift the divisor and the disc $\varphi_0$ to $\widetilde{\scr{D}}_0$ and $\widetilde{\varphi}_0$ in $\widetilde{X}_0$ without any change of the intersection number. Then, we have
$$
\beta^{i,j}_{i+1,j} \cap \scr{D} = [ \varphi] \cap \scr{D} = [ \widetilde{\varphi}_0 ]\cap \widetilde{\scr{D}}_0,
$$
which completes the proof.
\end{proof}

We take
\begin{equation}\label{bulkparak}
\frak{b} :=  \sum_{i} \frak{b}^\textup{hor}_{i, i+1} \cdot \varphi^\prime_{1-s} \left( \scr{D}^\textup{hor}_{i, i+1}  \right)
 + \sum_{j} \frak{b}^\textup{ver}_{j+1, j} \cdot \varphi^\prime_{1-s} \left( \scr{D}^\textup{ver}_{j+1, j} \right)
\end{equation}
where $\frak{b}^\textup{hor}_{i, i+1}, \frak{b}^\textup{ver}_{j+1, j} \in \Lambda_0$ and $\varphi^\prime_{1-s} \colon X_1 \to X_{s}$. By abuse of notation for simplicity, we denote $\varphi^\prime_{1-s} \left( \scr{D}^\textup{hor}_{i, i+1}  \right)$ (resp. $\varphi^\prime_{1-s} \left( \scr{D}^\textup{ver}_{j+1, j} \right)$) by $\scr{D}^\textup{hor}_{i, i+1}$ (resp. $\scr{D}^\textup{ver}_{j+1, j}$). By the homotopy invariance of the $A_\infty$-structures, we may calculate the (bulk-deformed) Floer cohomology of $L_{m,s}(t)$ in $X_s$ for $s$ sufficiently close to $0$. In particular, non-displaceability of $L_{m}(t)$ can be achieved as long as the Floer cohomology of $L_{m,s}(t)$ is non-zero. Whenever turning on a bulk-deformation, this process passing to $L_{m,s}(t)$ and $\varphi^\prime_{1-s} (\scr{D})$  in $X_s$ will be taken into consideration.  Also, depending on the position $t$ of a Lagrangian torus $L_m( \, \cdot \,)$, we need to consider different $\frak{b}^\textup{hor}_{i, i+1}$ and $\frak{b}^\textup{ver}_{j+1, j}$.

\begin{corollary}\label{formulaforbulkdeformedpotentialoursitu}
Taking a bulk-deformation parameter as in~\eqref{bulkparak}, the deformed potential function is expressed as
\begin{equation}
\frak{PO}^\frak{b}(L; \textup{\textbf{y}}) = \sum_{(i,j)} \left( \exp \left( \frak{b}^\textup{hor}_{i,i+1} \right) \frac{y_{i,j}}{y_{i+1, j}}  T^{ u_{i,j} - u_{i+1, j}} + \exp \left( \frak{b}^\textup{ver}_{j+1,j} \right) \frac{y_{i, j+1}}{y_{i,j}}  T^{ u_{i,j+1} - u_{i,j}} \right).
\end{equation}
\end{corollary}

Note that the gradient of the bulk-deformed potential is given as follows:
\begin{equation}\label{thegradientofbulkdeformedpotential}
\begin{split}
\pa^\frak{b} (i, j) (\textup{\textbf{y}}) := y_{i,j} \frac{\pa \frak{PO}^{\frak{b}}}{\pa y_{i,j}}(\textup{\textbf{y}}) &= - c^\textup{ver}_{j+1, j}  \frac{y_{i, j+1}}{y_{i, j}} T^{u_{i,j+1} - u_{i,j}} - c^\textup{hor}_{i-1, i} \frac{y_{i-1, j}}{y_{i,j}} T^{u_{i-1,j} - u_{i,j}}\\
&+ c^\textup{hor}_{i, i+1} \frac{y_{i,j}}{y_{i+1, j}} T^{u_{i,j} - u_{i+1,j}} + c^\textup{ver}_{j, j-1} \cdot \frac{y_{i, j}}{y_{i, j-1}} T^{u_{i,j} - u_{i,j-1}}
\end{split}
\end{equation}
where $c^\textup{hor}_{i,i+1} = \exp (\frak{b}^\textup{hor}_{i, i+1})$ and $c^\textup{ver}_{j+1,j} = \exp (\frak{b}^\textup{ver}_{j+1, j})$.

%------------------------------------------------------------------------------------------
\vspace{0.2cm}
\subsection{The non-displaceable Gelfand-Cetlin fibers in $\mcal{F}(3)$}\label{pfofmaincorforfullflag3}~
\vspace{0.2cm}

In this section, the case of $\mcal{F}(3)$ will be discussed in details. The following theorem will be proven. 

\begin{theorem}[Theorem~\ref{theoremC}]\label{theorem_maincompleteflag3}
Let $\lambda = \{ \lambda_1 = 2 > \lambda_2 = 0 > \lambda_3 = -2 \}$.
Consider the co-adjoint orbit $\mathcal{O}_\lambda$, a complete flag manifold $\mcal{F}(3)$ equipped with the monotone Kirillov-Kostant-Souriau symplectic form $\omega_\lambda$, Then the Gelfand-Cetlin fiber over a point $\textbf{\textup{u}} \in \Delta_\lambda$ is non-displaceable if and only if $\textbf{\textup{u}} \in I$ where
\begin{equation}\label{linesegmentforf3}
I := \left\{ (u_{1,1}, u_{1,2}, u_{2,1}) = (0, 1 -t, -1 +t) \in \R^3 ~|~ 0 \leq t \leq 1 \right\}
\end{equation}
In particular, the Lagrangian 3-sphere $\Phi_\lambda^{-1}(0,0,0)$ is non-displaceable.
\end{theorem}

Before starting our proof, we explain an alternative description for the fibers over the red line in Figure~\ref{figure_GCmomnetf3} following Chan-Pomerleano-Ueda \cite{CPU}. They proved the homological mirror symmetry for the conifold by realizing Strominger-Yau-Zaslow(SYZ) mirror symmetry on the smoothing $Y_\varepsilon$ of the conifold where
\begin{align*}
Y_\varepsilon = \{ (u_1, v_1, u_2, v_2) \in \C^4 ~\colon~ u_1 v_1 - u_2 v_2 + \varepsilon = 0\} 
\end{align*}
It can be embedded into
$$
\mathring{Y}_\varepsilon =\{ (u_1, v_1, u_2, v_2, z) \in \C^5 ~\colon~ u_1 v_1 = z - a, u_2 v_2 =  z - b\}
$$
where $a$ and $b$ are positive real numbers such that $\varepsilon = b - a$.
For the purpose of doing Strominger-Yau-Zaslow(SYZ) mirror symmetry, they came up with a double conic fibration on the complement of the anticanonical divisor $\{z=0\}$ in $\mathring{Y}_\varepsilon$. Consider the projection to the $z$-variable $\mathring{Y}_\varepsilon \to \C$. Note that we have singular fibers over two points $z=a$ and $z=b$. Also, it admits the following fiberwise $T^2$-action
$$
\begin{cases}
\theta_1 * (u_1, v_1) = (e^{\sqrt{-1}\theta} u_1, e^{- \sqrt{-1}\theta} v_1) \\
\theta_2 * (u_2, v_2) = (e^{\sqrt{-1}\theta} u_2, e^{- \sqrt{-1}\theta} v_2).
\end{cases}
$$
Any $T^2$-orbit satisfying $|u_i| = |v_i|$ for $i= 1, 2$ is called an \emph{equator}. By collecting the equators over a circle centered at the origin in the $z$-plane, we obtain a (special) Lagrangian torus fibration (with respect to the holomorphic volume form $d \log z \wedge d \log u_1 \wedge d \log u_2$), see Figure~\ref{Fig_doubleconicfib}. The fibration carries the exactly two walls at $z = a$ and $z = b$. Also, collecting the equators over the line segment connecting $a$ and $b$, we obatin two solid tori that forms Lagrangian $S^3$. This picture will serve as a local model for a Lagrangian torus fibration on $\mcal{F}(3)$ equipped with a (non-standard) symplectic form and the tori consisting of equators over simply closed curve containing $a$, $b$ and $0$ in the $z$-plane will be in our interest.

\vspace{0.2cm}
\begin{figure}[ht]
	\scalebox{1}{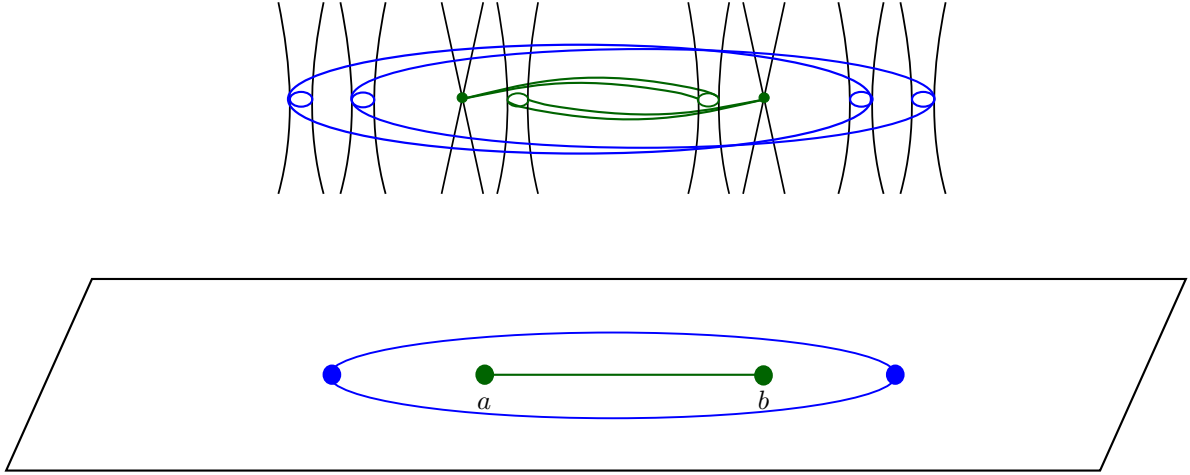}
	\caption{\label{Fig_doubleconicfib} Double conic fibration.}	
\end{figure}

Returning to the case of $\mcal{F}(3)$, we have a toric degeneration of algebraic varieties consisting of 
$$
X_\varepsilon = \{ ([Z_1: Z_2: Z_3], [Z_{12}: Z_{23}: Z_{13}]) \in \CP^2 \times \CP^2 ~\colon~ \varepsilon Z_1 Z_{23} - Z_2 Z_{13} + Z_3 Z_{12} = 0\}.
\footnote{
The conventions constructing toric degenerations on \cite{NNU} and \cite{KM} are different. In  \cite{NNU}, the diagonal term in the Pl{\"u}cker embedding survives, whereas in \cite{KM} the anti-diagonal term does. Here, we are following the convention of \cite{KM}. 
}
$$
For any $\varepsilon \neq 0$, $X_\varepsilon$ is diffeomorphic to a complete flag manifold $\mcal{F}(3)$. When $\varepsilon = 0$, we have a toric variety $X_0$ at the center. A generic one $X_\varepsilon$ ($\varepsilon \neq 0$) can be viewed as a compactification of $Y_\varepsilon$. By getting rid of divisors $Z_1 = 0$ and $Z_{23} = 0$ in $X_\varepsilon$ and setting 
$$
u_1 := \frac{Z_3}{Z_1}, v_1 := \frac{Z_{12}}{Z_{23}}, u_2 := \frac{Z_2}{Z_1}, v_2 := \frac{Z_{13}}{Z_{23}},
$$
we obtain a torus fibration on an open subvariety of $X_\varepsilon$ by following the above recipe with respect to the symplectic form induced from $X_\varepsilon$. 

As we are concerned with Lagrangians, we may set $a = 0$ and $b = \varepsilon$ and then fix a family $\{\gamma(t) ~\colon~ 0 < t \leq 1 \}$ of simply closed curves enclosing $a$ and $b$ at the interior and converging to the line segment $[a,b]$. The torus fibers consisting of the equators over the simple curves in the family degenerates into the Lagrangian $S^3$. We would like to show that those tori are non-displaceable. 

As $\varepsilon \to 0$, $X_\varepsilon$ degenerates into a toric variety $X_0$. A torus over the simple closed curve enclosing $a$ and $b$ is Lagrangian isotopic to a torus over a simple closed curve centered at the origin, which comes from the toric fibers. 
In this case, according to Knutson-Miller \cite[Section 1.3]{KnM}, the (opposite) Schubert divisors on $X_\varepsilon$ can be defined as matrix Schubert variety. In terms of the Pl{\"u}cker coordinates, their results can be written as 
$$
X_{132} = \{ Z_{12} = 0\} = \{v_1 = 0\},\, X_{213} = \{ Z_{1} = 0\} = \{u_1 = 0\}
$$
where $X_w$ is the Schubert variety associated with $w$ in the symmetric group $S_3$. By \cite{KM}, these correspond to
\begin{equation}\label{eq_ver}
\scr{D}^\textup{ver}_{2,1} = X_{132}, \, \scr{D}^\textup{ver}_{3,2} =  X_{213}
\end{equation}
By taking the involution where, 
$$
Z_1 \leftrightarrow \overline{Z}_{23}, Z_2 \leftrightarrow \overline{Z}_{13}, Z_3 \leftrightarrow \overline{Z}_{12} 
$$
we also have 
\begin{equation}\label{eq_hor}
\scr{D}^\textup{hor}_{1,2} = \{ Z_{3} = 0\},\, \scr{D}^\textup{hor}_{2,3} = \{ Z_{23} = 0 \}.
\end{equation}
To deform the Floer theory, we employ a combination of the vertical and horizontal divisors in~\eqref{eq_ver} and~\eqref{eq_hor}.
Note that $\scr{D}^\textup{ver}_{2,1}$ and $\scr{D}^\textup{hor}_{1,2}$ are over $z=0$ and $\scr{D}^\textup{ver}_{3,2}$ and $\scr{D}^\textup{hor}_{2,3}$ are at $\infty$. Therefore, any of the divisors do \emph{not} intersect the torus fibers. 
\begin{equation}\label{eq_combverhor}
\frak{b} = \frak{b}^\textup{ver}_{2,1} \cdot \scr{D}^\textup{ver}_{2,1} + \frak{b}^\textup{hor}_{1,2} \cdot \scr{D}^\textup{hor}_{1,2} + \frak{b}^\textup{ver}_{3,2} \cdot \scr{D}^\textup{ver}_{3,2} + \frak{b}^\textup{hor}_{2,3} \cdot \scr{D}^\textup{hor}_{2,3}
\end{equation}

For the proof of Theorem~\ref{theorem_maincompleteflag3}, we need the following proposition.

\begin{proposition}\label{closednessofnondisp}
Let $\Phi \colon X \to \Delta \subset \R^d$ be a completely integrable system (Definition~\ref{definition_CIS_continuous}) such that $\Phi$ is proper. If there exists a sequence $\{ \textbf{\textup{u}}_i  : i \in \N \}$ such that
\begin{enumerate}
\item Each $\Phi^{-1}(\textbf{\textup{u}}_i)$ is non-displaceable.
\item The sequence $\textbf{\textup{u}}_i$ converges to some point $\textbf{\textup{u}}_\infty$ in $\Delta$.
\end{enumerate}
then $\Phi^{-1}(\textbf{\textup{u}}_\infty)$ is also non-displaceable.
\end{proposition}

\begin{proof}
For a contradiction, suppose that $\Phi^{-1}(\textbf{\textup{u}}_\infty)$ is displaceable. There is a (time-dependent) Hamiltonian diffeomorphism $\phi$ and an open set $U$ containing $\Phi^{-1}(\textbf{\textup{u}}_\infty)$ in $X$ such that
$$
\phi(U) \cap U = \emptyset.
$$
For each $i$, there exists a point $x_i \in \Phi^{-1}(\textbf{\textup{u}}_i)$ such that $x_i \notin U$ since $\Phi^{-1}(\textbf{\textup{u}}_i)$ is non-displaceable. It implies that any subsequence of $\{x_i\}$ cannot converge to a point in $U$. On the other hand, passing to a subsequence, we may assume that $x_i$ converges to $x_\infty$ for some $x_\infty \in X$ since $\Phi$ is proper. By the continuity of $\Phi$, we then have
$$
\textbf{\textup{u}}_\infty = \lim_{i \to \infty} \textbf{\textup{u}}_i = \lim_{i \to \infty} \Phi (x_i) = \Phi (x_\infty).
$$
It leads to a contradiction that $x_\infty \in \Phi^{-1} (\textbf{\textup{u}}_\infty) \subset U$.
\end{proof}

We now start the proof of Theorem~\ref{theorem_maincompleteflag3}. Since 

\begin{proof}[Proof of Theorem~\ref{theorem_maincompleteflag3}]

For any fixed $t$ with $0 \leq t < 1$, let $L(t)$ be the Lagrangian torus fiber over $(0, 1-t, -1+t) \in I$ in~\eqref{linesegmentforf3}. Let $L_s(t)$ be the fiber corresponding to $L(t)$ in $X_s$ via a toric degeneration of completely integrable systems in Theorem~\ref{NNUtoricdeg}. By taking $s > 0$ sufficiently close to $0$, the potential function of $L_s(t)$ can be arranged as
$$
\frak{PO}(\textbf{y}) = \left( \frac{y_{1,2}}{y_{1,1}} + \frac{y_{1,1}}{y_{2,1}} + y_{1,2} + \frac{1}{y_{2,1}} \right) T^{1-t} + \left( \frac{1}{y_{1,2}} + y_{2,1} \right) T^{1+t}.
$$

We use a combination of Schubert cycles in~\eqref{bulkparak} (or equivalently~\eqref{eq_combverhor}) to deform the potential function. A strategy we take is to postpone determining bulk-deformation parameters. Namely, we start with a tentative a parameter, determine solutions for $\textbf{y}$ first, and then adjust the parameter to make the chosen $\textbf{y}$ a critical point.

Take a \emph{tentative} bulk-parameter $\frak{b}^\prime := \frak{b}^\textup{ver}_{2,1} \cdot \scr{D}^\textup{ver}_{2,1}$ such that $\exp (\frak{b}^\textup{ver}_{2,1}) = 1 + T^{2t}$, i.e.,  
$$
\frak{b}^\textup{ver}_{2,1} = T^{2t} - \frac{1}{2} T^{4t} + \cdots \in \Lambda_+. 
$$
By Corollary~\ref{formulaforbulkdeformedpotentialoursitu}, the potential function is deformed into
$$
\frak{PO}^{\frak{b}^\prime}(\textbf{y}) = \left( \frac{y_{1,2}}{y_{1,1}} + \frac{y_{1,1}}{y_{2,1}} + y_{1,2} + \frac{1}{y_{2,1}} \right) T^{1-t} + \left( \frac{y_{1,2}}{y_{1,1}} + \frac{1}{y_{2,1}} + \frac{1}{y_{1,2}} + y_{2,1} \right) T^{1+t},
$$
whose logarithmic derivatives are
$$
\begin{cases}
\displaystyle y_{1,1} \frac{\pa \frak{PO}^{\frak{b}^\prime}}{\pa y_{1,1}} (\textbf{y}) =  \left( - \frac{y_{1,2}}{y_{1,1}} + \frac{y_{1,1}}{y_{2,1}} \right) T^{1-t} + \left( - \frac{y_{1,2}}{y_{1,1}} \right) T^{1+t} \\ \\
\displaystyle y_{1,2} \frac{\pa \frak{PO}^{\frak{b}^\prime}}{\pa y_{1,2}} (\textbf{y}) =  \left(  \frac{y_{1,2}}{y_{1,1}} + {y_{1,2}} \right) T^{1-t} + \left( \frac{y_{1,2}}{y_{1,1}} - \frac{1}{y_{1,2}} \right) T^{1+t} \\ \\
\displaystyle y_{2,1} \frac{\pa \frak{PO}^{\frak{b}^\prime}}{\pa y_{2,1}} (\textbf{y}) =  \left( - \frac{y_{1,1}}{y_{2,1}} - \frac{1}{y_{2,1}} \right) T^{1-t} + \left( - \frac{1}{y_{2,1}} + y_{2,1} \right) T^{1+t}. \\
\end{cases}
$$

We set $y_{1,2} = 1, y_{2,1} = 1$ and take $y_{1,1}$ as the solution of $(y_{1,1})^2 = 1 + T^{2t}$ satisfying $y_{1,1} \equiv -1 \mod T^{>0}$. It is easy to see that $y_{1,1} \frac{\pa \frak{PO}^{\frak{b}^\prime}}{\pa y_{1,1}} (\textbf{y}) = 0$. Note that $y_{1,1}$ is of the form
$$
y_{1,1} \equiv - 1 - \frac{1}{2} T^{2t} \mod T^{>2t}.
$$

We now adjust a bulk-deformation parameter from $\frak{b}^\prime$ to $\frak{b}$ in order for the chosen $(y_{1,1}, y_{1,2}, y_{2,1})$ to be a critical point of $\frak{PO}^\frak{b}$. Let
$$
\frak{b} := \frak{b}^\prime + \frak{b}^\textup{ver}_{3,2} \cdot \scr{D}^\textup{ver}_{3,2} + \frak{b}^\textup{hor}_{2,3} \cdot \scr{D}^\textup{hor}_{2,3}.
$$
Since $\scr{D}^\textup{ver}_{3,2}$ and $\scr{D}^\textup{hor}_{2,3}$ do \emph{not} intersect with the homotopy classes corresponding to $\scr{D}^\textup{ver}_{2,1}$ in $\pi_2(\mcal{O}_\lambda, \Phi^{-1}_\lambda(t))$, we still have 
$$
y_{1,1} \frac{\pa \frak{PO}^{\frak{b}}}{\pa y_{1,1}} (\textbf{y}) = y_{1,1} \frac{\pa \frak{PO}^{\frak{b}^\prime}}{\pa y_{1,1}} (\textbf{y})  = 0.
$$ 
Plugging the chosen $y_{i,j}$'s, we have
$$
\begin{cases}
\displaystyle y_{1,2} \frac{\pa \frak{PO}^{\frak{b}}}{\pa y_{1,2}} (\textbf{y}) = \left( -\frac{1}{2} - \exp (\frak{b}^\textup{ver}_{3,2}) \right) T^{1+t} + \frak{P}^{(1,2)} \cdot T^{1+t}  \\ \\
\displaystyle y_{2,1} \frac{\pa \frak{PO}^{\frak{b}}}{\pa y_{2,1}} (\textbf{y}) = \left( -\frac{1}{2} + \exp (\frak{b}^\textup{hor}_{2,3}) \right) T^{1+t} + \frak{P}^{(2,1)} \cdot T^{1+t}. \\
\end{cases}
$$
for some constant $ \frak{P}^{(1,2)}, \frak{P}^{(2,1)}\in \Lambda_+$. By choosing $\frak{b}^\textup{ver}_{3,2}, \frak{b}^\textup{hor}_{2,3} \in \Lambda_0$ so that 
$$
\begin{cases}
\displaystyle \exp (\frak{b}^\textup{ver}_{3,2}) = -\frac{1}{2} + \frak{B}^{(1,2)} \\ \\
\displaystyle \exp (\frak{b}^\textup{hor}_{2,3}) = \frac{1}{2} - \frak{B}^{(2,1)}. \\
\end{cases}
$$
we can make $\frak{PO}^\frak{b}(\textbf{y})$ admit a critical point. By Theorem~\ref{criticalpointimpliesnondisplaceability}, $L_\varepsilon(t)$ has a non-vanishing (bulk-)deformed Floer cohomology. By the hamiltonian invariance of $A_\infty$-structures, so does $L(t)$ and therefore it is non-displaceable. 

In sum, each torus fiber over the line segment in~\eqref{linesegmentforf3} is non-displaceable. Furthermore, Proposition~\ref{closednessofnondisp} yields non-displaceability of the Lagrangain $3$-sphere. 
\end{proof}

%------------------------------------------------------------------------------------------
\section{Decompositions of the gradient of potential function}
\label{secDecompositionsOfTheGradientOfPotentialFunction}
\label{section:decompositionofpote}

In this section, in order to prove Theorem~\ref{theoremD}, we introduce the split leading term equation of the potential function in~\eqref{potential}, which is the analogue of the leading term equation in~\cite{FOOOToric1, FOOOToric2}. We discuss the relation between its solvability and non-triviality of Floer cohomology under a certain bulk-deformation.

%------------------------------------------------------------------------------------------
\vspace{0.2cm}
\subsection{Outline of Section~\ref{section:decompositionofpote} and Section~\ref{solvabilityofthesplitleadingtermequ}}~
\vspace{0.2cm}

Due to Theorem~\ref{criticalpointimpliesnondisplaceability}, in order to show that the GC torus fiber $L_m(t)$ for each $t$ with $0 \leq t < 1$ is non-displaceable, it suffices to find a bulk-deformation parameter $\frak{b}$ such that $\frak{PO}^\frak{b}$ admits a critical point $\textbf{y}$. Section~\ref{section:decompositionofpote} and Section~\ref{solvabilityofthesplitleadingtermequ} will be occupied to discuss how to determine them.

Before giving the outline, we explain why this process takes so long by pointing out the differences from the case of toric fibers in a symplectic toric manifold. In the toric case, the (generalized) leading term equation was introduced to detect non-displaceable toric fibers effectively in \cite[Section 11]{FOOOToric2}. Roughly speaking, it consists of the initial terms of the gradient of a (bulk-deformed) potential function with respect to a suitable choice of exponential variables. It is proven therein that there always exists a bulk-deformation parameter $\frak{b}$ so that the complex solution becomes a critical point of the bulk-deformed potential function $\frak{PO}^\frak{b}$ as soon as the leading term equation admits a solution whose components are in $\C \backslash \{0\}$.  Indeed, the positions where the leading term equation is solvable are characterized by the intersection of certain tropicalizations in \cite{KL}. The key features for proving the above statements are in order. First, there is a one-to-one correspondence between the \emph{honest} holomorphic discs bounded by a toric fiber of Maslov index $2$ and the facets of the moment polytope. Second, the preimage of each facet represents a cycle of degree $2$. Therefore, all terms corresponding to the facets can be independently controlled, yielding a parameter $\frak{b}$ and providing a solution of the equation.

In a GC system, however, the inverse image of a facet may not represent a cycle of degree $2$ so that the terms of $\frak{PO}$ cannot be independently controlled.
\footnote{
Because of this feature, Bernstein-Kushnirenko theorem \cite{Be, Ku} cannot be applied in our situation. 
}
Thus, the above statements are not expected to hold anymore. But for a family of Lagrangian tori $L_m(t)$ in $\mcal{O}_\lambda \simeq \mcal{F}(n)$ with the monotone symplectic form $\omega_\lambda$, we will show the existence of a bulk-deformation parameter $\frak{b}$ and a critical point $\textbf{y}$. In Section~\ref{section:decompositionofpote}, we define the \emph{split leading term equation} (See Definition~\ref{splittingleadingtermequa}), which replaces the role of the leading term equation in the toric case. We then demonstrate how to determine a bulk-deformation parameter and extend a solution of the split leading term equation to a critical point of the bulk-deformed potential function. In Section~\ref{solvabilityofthesplitleadingtermequ}, we show that the split leading term equation always admits a solution. In general, finding a solution of a general system of multi-variable equations is not simple at all even with the aid of a computer. Yet, in this case, we are able to find a solution, guided by ladder diagrams regarding as the containers of exponential variables.  

Let $B(m)$ be the sub-diagram consisting of $(m \times m)$ lower-left unit boxes in the ladder diagram $\Gamma(n) := \Gamma(1, \cdots, n)$ of $\mcal{F}(n)$. The diagrams $\Gamma(n)$ and $B(m)$ are often regarded as collections of double indices as follows:
\begin{equation}\label{gammanbm}
\begin{split}
&\Gamma(n) = \{ (i, j) ~\colon~ 2 \leq i + j \leq n\}\\
&B(m) = \{ (i, j) ~\colon~ 1 \leq i, j \leq m\}.
\end{split}
\end{equation}
Recalling~\eqref{potential}, the potential function of $L_m(t)$ is arranged as several groups with repsect to the energy levels. Observe that the valuation of $\pa (i,j)(\textbf{y})$ for $(i, j) \in B(m)$ is $(1-t)$ and that of $\pa (i,j)(\textbf{y})$ for $(i, j) \in \Gamma(n) \backslash B(m)$ is $1$.

We will decompose the gradient of the potential function deformed by $\frak{b}$ in~\eqref{bulkparak} into two pieces along the boundary of $B(m)$. We are planning to determine a critical point in the following steps.
\begin{enumerate}
\item Find a solution $y^\C_{i,j} \in \C \backslash \{0\}$ of the system consisting of the equations $\pa^\frak{b}(i,j)(\textbf{y}) \equiv 0 \mod T^{>1}$ in $\Gamma(n) \backslash  B(m) \cup \{(m,m)\}$ and equations relating the variables adjacent to $B(m)$ in Section~\ref{solvabilityofthesplitleadingtermequ}.
\item Find a solution $y^\C_{i,j} \in \C \backslash \{0\}$ of $\pa^\frak{b}(i,j)(\textbf{y}) \equiv 0 \mod T^{>1-t}$ in $B(m)$ in Section~\ref{symmetriccomplexsolinbm}.
\item Determine a solution $y_{i,j} \in \Lambda_U$ of $\pa^\frak{b}(i,j)(\textbf{y}) = 0$ in $B(m)$  such that $y_{i,j} \equiv y_{i,j}^\C \mod T^{>0}$ in Section~\ref{insidebm}.
\item Determine a solution $y_{i,j} \in \Lambda_U$ of $\pa^\frak{b}(i,j)(\textbf{y}) = 0$ in $\Gamma(n) \backslash  B(m)$ such that $y_{i,j} \equiv y_{i,j}^\C \mod T^{>0}$ in Section~\ref{outsideofbm}.
\end{enumerate}
The \emph{split leading term equation} (See Definition~\ref{splittingleadingtermequa}) arises in the first step $(1)$. In this section, assuming that the split leading term equation is solvable, we explain how to complete the remaing steps $(2), (3)$ and $(4)$. The next section focuses on solving the split leading term equation.

\begin{example}
In the co-adjoint orbit $\mcal{O}_\lambda$ of a sequence $\lambda = \{ 5, 3, 1, -1, -3, -5 \}$ for instance, the potential function of $L_2(t)$ is arranged as follows: 
$$
\frak{PO}(L_2(t); \textbf{y}) = \left( \frac{y_{1, 2}}{y_{1,1}}+ \frac{y_{1, 1}}{y_{2,1}} + \frac{y_{1, 2}}{y_{2,2}} + \frac{y_{2, 2}}{y_{2,1}} \right) T^{1-t} + \left(  \frac{y_{1, 4}}{y_{1,3}} + \frac{y_{1, 3}}{y_{2,3}} + \cdots \right) T^{1} + \left( \frac{y_{1, 3}}{y_{1,2}} + \frac{y_{2, 1}}{y_{3,1}} \right) T^{1 + t}.
$$
In this example, the valuation of partial derivatives of $\frak{PO}$ jumps along the red line in Figure~\ref{Decomposition of the gradient of the potential function}.

Turning on bulk-deformation, according to~\eqref{thegradientofbulkdeformedpotential}, A complex number $y_{i,j} \in \C \backslash \{0\}$ has to satisfy
\begin{equation}\label{firstsysfrominb2}
\begin{cases}  - c^\textup{ver}_{2,1} \cdot \frac{y_{1, 2}}{y_{1,1}}+ c^\textup{hor}_{1,2} \cdot \frac{y_{1, 1}}{y_{2,1}} = 0, \,\, c^\textup{ver}_{2,1} \cdot \frac{y_{1, 2}}{y_{1,1}} + c^\textup{hor}_{1,2} \cdot \frac{y_{1, 2}}{y_{2,2}} = 0, \\ \\
 - c^\textup{hor}_{1,2} \cdot \frac{y_{1, 1}}{y_{2,1}} - c^\textup{ver}_{2,1} \cdot \frac{y_{2, 2}}{y_{2,1}} = 0, \,\, - c^\textup{hor}_{1,2} \cdot \frac{y_{1, 2}}{y_{2,2}} + c^\textup{ver}_{2,1} \cdot \frac{y_{2, 2}}{y_{2,1}} = 0, \\
\end{cases}
\end{equation}
which comes from the initial parts of the partial derivatives inside $B(2)$, and
\begin{equation}\label{secondsysfromoutb2}
\begin{cases}
& - c^\textup{ver}_{6,5} \cdot \frac{1}{y_{1,5}} + c^\textup{hor}_{1,2} \cdot {y_{1,5}} + c^\textup{ver}_{5,4} \cdot \frac{y_{1,5}}{y_{1,4}} = 0, \, - c^\textup{ver}_{5,4} \cdot \frac{1}{y_{2,4}} - c^\textup{hor}_{1,2} \cdot \frac{y_{1,4}}{y_{2,4}} + c^\textup{hor}_{2,3} \cdot {y_{2,4}} + c^\textup{ver}_{4,3} \cdot \frac{y_{2,4}}{y_{2,3}} = 0, \cdots \\ \\
& - c^\textup{ver}_{5,4} \cdot \frac{y_{1,5}}{y_{1,4}} + c^\textup{hor}_{1,2} \cdot \frac{y_{1,4}}{y_{2,4}} + c^\textup{ver}_{4,3} \cdot \frac{y_{1,4}}{y_{1,3}} = 0, \, - c^\textup{ver}_{4,3} \cdot \frac{y_{2,4}}{y_{2,3}} - c^\textup{hor}_{1,2} \cdot \frac{y_{1,3}}{y_{2,3}} + c^\textup{hor}_{2,3} \cdot \frac{y_{2,3}}{y_{3,3}} + c^\textup{ver}_{3,2} \cdot \frac{y_{2,3}}{y_{2,2}} = 0, \cdots \\ \\
& -  c^\textup{ver}_{4,3} \cdot \frac{y_{1,4}}{y_{1,3}} + c^\textup{hor}_{1,2} \cdot \frac{y_{1,3}}{y_{2,3}} = 0, \,\,  -  c^\textup{ver}_{3,2} \cdot \frac{y_{2,3}}{y_{2,2}} + c^\textup{hor}_{2,3} \cdot \frac{y_{2,2}}{y_{3,2}} = 0, \,\, -  c^\textup{ver}_{2,1} \cdot \frac{y_{3,2}}{y_{3,1}} + c^\textup{hor}_{3,4} \cdot \frac{y_{3,1}}{y_{4,1}} = 0,\\
\end{cases}
\end{equation}
which comes from the initial parts of the partial derivatives inside $\Gamma(6) \backslash B(2) \cup \{(2,2)\}$, see Figure~\ref{Decomposition of the gradient of the potential function}. Solving the first system~\eqref{firstsysfrominb2} is related to the step $(2)$ and solving the second system~\eqref{secondsysfromoutb2} is related to the step $(1)$.
\begin{figure}[ht]
	\scalebox{1}{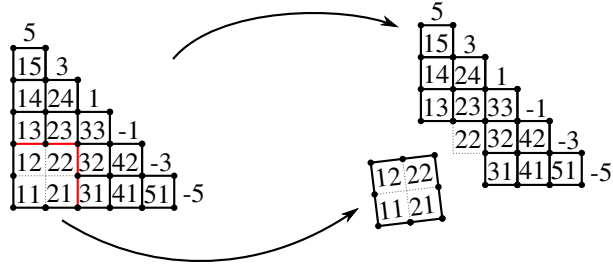}
	\caption{\label{Decomposition of the gradient of the potential function} Decomposition of the gradient of the potential function in $\mcal{F}(6)$.}	
\end{figure}
\end{example}

\begin{remark}
More generally, one might consider the fibers over the line segment connecting the center of $\Delta_\lambda$ and the center of another Lagrangian face $\gamma$. Depending on $\gamma$ in the ladder diagram $\Gamma(n)$ in some cases, one might decompose the potential function into several pieces along the simply connected regions that are not unit-sized blocks. The reader is referred to \cite{CKO} in order to consult a geometric implication of the potential function \emph{inside} the decomposed blocks. Also, it is discussed therein when the potential function has a critical point. 

The split leading term equation would be formed by the system from the outside of the decomposed blocks, providing the complex part of a critical point with a suitable choice of a complex solution within the decomposed blocks. This intuition motivates us to name it the ``{split}" leading term equation. In this article, the general form will not be discussed as the split leading term equation from $I_m(t)$ in~\eqref{IMT} is only dealt with.
\end{remark}

%------------------------------------------------------------------------------------------
\vspace{0.2cm}
\subsection{Split leading term equation}~
\vspace{0.2cm}

We now define the split leading term equation arising from the potential function of $L_m(t)$. In this case, it suffices to take a bulk-deformation parameter
\begin{equation}\label{bulkparak2}
\frak{b} :=  \sum_{i \geq k} \frak{b}^\textup{hor}_{i, i+1} \cdot \scr{D}^\textup{hor}_{i, i+1}  + \sum_{j \geq k} \frak{b}^\textup{ver}_{j+1, j} \cdot \scr{D}^\textup{ver}_{j+1, j}
\end{equation}
instead of the form~\eqref{bulkparak}. Therefore, we should set
$$
\begin{cases}
c^\textup{hor}_{i, i+1} := \exp \left(\frak{b}^\text{hor}_{i,i+1} = 0 \right) = 1 \quad \mbox{for } i < k, \\
c^\textup{ver}_{j+1, j} := \exp \left(\frak{b}^\text{ver}_{j+1,j}= 0 \right) = 1  \quad \mbox{for } j < k.
\end{cases}
$$
In particular, we see that $\pa^\frak{b}(i,j)$ in~\eqref{thegradientofbulkdeformedpotential} coincides with $\pa(i,j)$ in~\eqref{shortnotofderi} for all $1 \leq i, j < k$. 

\begin{definition}\label{splittingleadingtermequa} Let $k = \left\lceil  n/2 \right\rceil$, that is $n = 2k -1$ or $2k$. We set
\begin{equation}\label{settingup}
\begin{cases}
&c^\textup{hor}_{i, i+1} := 1 \quad \mbox{for } i < k, \quad c^\textup{ver}_{j+1, j} := 1 \quad \mbox{for } j < k \\
&y_{i,m} := \infty \quad \mbox{for } i < m, \quad y_{m, j} := 0 \quad \mbox{for } j < m\\
&y_{\bullet,0} := \infty, \,\, y_{0, \bullet} := 0, \,\, y_{i, n+ 1 - i} := 1 \quad \mbox{for } 1 \leq i \leq n \\
\end{cases}
\end{equation}
The \emph{split leading term equation of} $\Gamma(n)$ \emph{associated with} $B(m)$ is the system of the following equations:
\begin{equation}\label{splitleadingtermequ}
\begin{cases}
\pa_m^\frak{b} (i, j) (\textup{\textbf{y}}) = 0 \\
\pa_m (l) (\textup{\textbf{y}}) = 0
\end{cases}
\end{equation}
for all $(i,j) \in \Gamma(n) \backslash B(m) \cup \{ (m,m) \}$ and all $l$ with $1 \leq l < m$. Here,
\begin{align}\label{pabmijy}
\pa_m^\frak{b} (i, j) (\textup{\textbf{y}}) &:= - c^\textup{ver}_{j+1, j} \cdot \frac{y_{i, j+1}}{y_{i, j}} - c^\textup{hor}_{i-1, i} \cdot \frac{y_{i-1, j}}{y_{i,j}} + c^\textup{hor}_{i, i+1} \cdot \frac{y_{i,j}}{y_{i+1, j}} + c^\textup{ver}_{j, j-1} \cdot \frac{y_{i, j}}{y_{i, j-1}} \\
\label{pamly} \pa_m (l) (\textup{\textbf{y}}) &:= (-1)^{m+1-l} \cdot \frac{y_{l,m+1}}{y_{m,m}} + \frac{y_{m,m}}{y_{m+1,l}},
\end{align}
and $c^\textup{hor}_{i, i+1}$'s and $c^\textup{ver}_{j+1, j}$'s for $i, j \geq k$ are non-zero complex numbers.
\end{definition}

We explain how the split leading term equation of $\Gamma(n)$ associated with $B(m)$ can be written. Cutting the boxes $B(m) \backslash \{ (m,m) \}$ off from the diagram $\Gamma(n)$, \eqref{pabmijy} comes from $\pa^\frak{b} (i, j)(\textbf{y})$ for $(i,j)$'s on the cut diagram as explained in~\eqref{secondsysfromoutb2} for the case $\mcal{F}(6)$.
In addition to them, we impose~\eqref{pamly} to relate the variables adjacent to $B(m)$. It will be explained in \eqref{y1m+1} and~\eqref{yjm+1} why~\eqref{pamly} appears.
 
\begin{remark}\label{whenm2kchoiceof}
Furthermore, we may take $\frak{b}^\textup{hor}_{k,k+1} = 0$ so that $c^\textup{hor}_{k,k+1} = 1$ if $n=2k$. See Remark~\ref{separatecasemk} to see why.
\end{remark}

\begin{example}\label{splitleadingtermgamm5b2}
The split leading term equation of $\Gamma(5)$ associated with $B(2)$ consists of
$$
\begin{cases}
&\displaystyle - c^\textup{ver}_{5,4} \cdot \frac{1}{y_{1,4}} + y_{1,4} + c^\textup{ver}_{4,3} \cdot \frac{y_{1,4}}{y_{1,3}} = 0, \,\, - c^\textup{ver}_{4,3} \cdot \frac{1}{y_{2,3}} - \frac{y_{1,3}}{y_{2,3}} + y_{2,3} + \frac{y_{2,3}}{y_{2,2}} = 0, \\ \\
&\displaystyle -\frac{1}{y_{3,2}} - \frac{y_{2,2}}{y_{3,2}} + c^\text{hor}_{3,4} \cdot y_{3,2} + \frac{y_{3,2}}{y_{3,1}} = 0, - \frac{1}{y_{4,1}} - c^\textup{hor}_{3,4} \cdot \frac{y_{3,1}}{y_{4,1}} + c^\textup{hor}_{4,5} \cdot y_{4,1} = 0, \\  \\
&\displaystyle - c^\textup{ver}_{4,3} \cdot \frac{y_{1,4}}{y_{1,3}} + \frac{y_{1,3}}{y_{2,3}} = 0, \,\, -\frac{y_{3,2}}{y_{3,1}} + c^\textup{hor}_{3,4} \cdot \frac{y_{3,1}}{y_{4,1}} = 0,\,\, -\frac{y_{2,3}}{y_{2,2}} + \frac{y_{2,2}}{y_{3,2}} = 0, \,\, \frac{y_{1,3}}{y_{2,2}} + \frac{y_{2,2}}{y_{3,1}} = 0.
\end{cases}
$$
\end{example}

\begin{remark}
We would like to address that the split leading term equation is \emph{not} same as the initial part of the gradient of the potential function. Note that $\pa^\frak{b} (2,2) = 0$ is of the form 
$$
\left( - \frac{y_{1,2}}{y_{2,2}} + \frac{y_{2,2}}{y_{2,1}} \right) T^{1-t} + \left( -\frac{y_{2,3}}{y_{2,2}} + \frac{y_{2,2}}{y_{3,2}} \right) T^1 = 0.
$$
Its initial part will be used to obtain a critical point of the potential function within $B(2)$ in Section~\ref{symmetriccomplexsolinbm}. As we have seen in Example~\ref{splitleadingtermgamm5b2}, the split leading term equation captures the terms with the second energy level
$$
-\frac{y_{2,3}}{y_{2,2}} + \frac{y_{2,2}}{y_{3,2}} = 0
$$
as well. 
\end{remark}

The main theorem of this section is as follows.
\begin{theorem}\label{splitleadingtermequationimpliesnonzero}
Let $\lambda = \{ \lambda_{i} := n - 2i + 1 ~\colon~ i = 1, \cdots, n \}$ be an $n$-tuple of real numbers for an arbitrary integer $n \geq 4$. Consider the co-adjoint orbit $\mathcal{O}_\lambda$, a complete flag manifold $\mcal{F}(n)$ equipped with the monotone form $\omega_\lambda$.
Fix one Lagrangian Gelfand-Cetlin torus $L_m(t)$ over $I_m(t)$ for $0 \leq t < 1$ in $\mathcal{O}_\lambda$.
If the split leading term equation~\eqref{splitleadingtermequ} admits a solution $\{ y^\C_{i,j} \in \C \backslash \{0\} ~\colon~ (i,j) \in \Gamma(n) \backslash B(m) \cup \{ (m,m) \} \}$ each component of which is a non-zero complex number for some nonzero complex numbers $c^{\textup{hor}, \C}_{i, i+1}$'s and $c^{\textup{ver}, \C}_{j+1, j}$'s ($i, j \geq k$), then there exists a bulk-deformation parameter $\frak{b}$ (depending on $m$ and $t$) of the form~\eqref{bulkparak} such that
\begin{enumerate}
\item The bulk-deformed potential function $\frak{PO}^\frak{b}(\textbf{\textup{y}})$ has a critical point $ \{ y_{i,j} \in \Lambda_U ~\colon~ (i,j) \in \Gamma(n) \}$ satisfying 
$$
y^\C_{i,j} \equiv y^{\phantom{\C}}_{i,j} \mod T^{>0} \quad \text{for } (i,j) \in \Gamma(n) \backslash B(m) \cup \{ (m,m) \}.
$$
\item Also, $c^{\textup{hor}, \C}_{i,i+1} \equiv \exp (\frak{b}^{\textup{hor} \phantom{, \C}}_{i,i+1}), \, c^{\textup{ver}, \C}_{j+1, j} \equiv \exp (\frak{b}^{\textup{ver} \phantom{, \C}}_{j+1,j}) \mod T^{>0}$.
\end{enumerate}
\end{theorem}

The existence of a solution for the split leading term equation~\eqref{splitleadingtermequ} implies that the assumption of the following lemma is satisfied. We will repeatedly employ it in order to extend a solution in $\C \backslash \{0\}$ to that in $\Lambda_U$ of the gradient of the (bulk-deformed) potential function.

\begin{lemma}\label{extensionlemma}
Let $c_{j+1, j}, c_{i-1, i}, c_{i, i+1}$ and $c_{j, j-1}$ be elements in $\Lambda_U$. Suppose that we are given $y_{i-1, j} \in \Lambda_U \cup \{ 0 \}$, $y_{i,j-1} \in \Lambda_U \cup \{ \infty\}$ and $y_{i+1, j}, y_{i,j}$ (resp. $y_{i,j}, y_{i, j+1})$ $\in \Lambda_U$. If there is a \emph{non-zero} complex solution $y^\C_{i, j+1}$ (resp. $y^\C_{i+1, j}$) for
\begin{equation}\label{complexlevel}
- c^\C_{j+1, j} \cdot \frac{y^\C_{i,j+1}}{y^\C_{i,j}} - c^\C_{i-1, i} \cdot \frac{y^\C_{i-1,j}}{y^\C_{i,j}} + c^\C_{i, i+1} \cdot \frac{y^\C_{i,j}}{y^\C_{i+1,j}} + c^\C_{j, j-1} \cdot \frac{y^\C_{i,j}}{y^\C_{i,j-1}} = 0,
\end{equation}
then there exists a unique element $y_{i,j+1}$ (resp. $y_{i+1, j}$) $\in \Lambda_U$ that solves
\begin{equation}\label{equationwithoutcomplexlevel}
- c_{j+1, j} \cdot \frac{y_{i,j+1}}{y_{i,j}} - c_{i-1, i} \cdot \frac{y_{i-1,j}}{y_{i,j}} + c_{i, i+1} \cdot \frac{y_{i,j}}{y_{i+1,j}} + c_{j, j-1} \cdot \frac{y_{i,j}}{y_{i,j-1}} = 0.
\end{equation}
Here, for $y \in \Lambda_U$, $y^\C$ denotes a unique complex number such that $y^\C \equiv y \mod T^{>0}$.

Furthermore, assume in addition that $c_{\bullet, \bullet}$'s are non-zero complex numbers and
$$
\frak{v}_T \left(y^{\phantom{\C}}_{i,j} - y^\C_{i,j} \right) > \lambda, \frak{v}_T \left(y^{\phantom{\C}}_{i-1,j} - y^\C_{i-1,j} \right) > \lambda \text{ and } \frak{v}_T \left(y^{\phantom{\C}}_{i,j-1} - y^\C_{i,j-1} \right) > \lambda.
$$
Then, $\frak{v}_T (y_{i+1,j}) = \lambda$ if and only if $\frak{v}_T (y_{i,j+1}) = \lambda$.
\end{lemma}

\begin{proof}
The proof immediately follows from the observation that $y_{i+1,j}$ (or $y_{i,j+1}$) in~\eqref{equationwithoutcomplexlevel} can be expressed as a rational function in terms of the other variables. 
\end{proof}

\begin{figure}[ht]
	\scalebox{0.7}{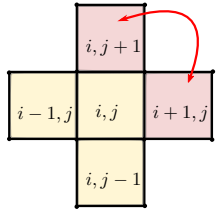}
	\caption{\label{graphicdeslemma} Graphical description of Lemma~\ref{extensionlemma}.}	
\end{figure}

%------------------------------------------------------------------------------------------
\subsection{Symmetric complex solutions within $B(m)$}\label{symmetriccomplexsolinbm}~
\vspace{0.2cm}

In this section, we focus on the system of equations
\begin{equation}\label{paijmBm}
\pa{(i,j)} (\textbf{y}) := y_{i,j} \frac{\pa \frak{PO}}{\pa y_{i, j}}(\textbf{y}) = 0 \quad \mbox{ for all } (i, j) \in B(m)
\end{equation}
within $B(m)$. For an index $(i, j) \in B(m)$, ignoring the variables outside of $B(m)$, the initial part of $\pa(i, j) (\textbf{y}) $ is denoted by $\pa_m{(i,j)} (\textbf{y})$. By~\eqref{potential}, we obtain
\begin{equation}\label{paijinm}
\pa_m{(i,j)} (\textbf{y}) = - \frac{y_{i, j+1}}{y_{i,j}} - \frac{y_{i-1, j}}{y_{i,j}} + \frac{y_{i,j}}{y_{i+1, j}} + \frac{y_{i, j}}{y_{i, j-1}}
\end{equation}
where $y_{0, \bullet}, y_{\bullet, m+1,}, y_{\bullet, 0}$ and $y_{m+1, \bullet}$ are respectively set to be $0, 0, \infty$ and $\infty$. The goal of the section is to find a ``{symmetric}" complex solution for~\eqref{paijinm}, see Proposition~\ref{propsymmetricsol}.

\begin{lemma}\label{minussol}
Let $c$ be a non-zero complex number. If there is a solution $\{ y_{i,j}^\C \in \C \backslash \{0\} ~\colon~ (i,j) \in B(m) \}$ of the system of equations
\begin{equation}\label{paijmbfy}
\pa_m (i, j) (\textbf{\textup{y}}) = 0 \quad \mbox{for }\, (i,j) \in B(m),
\end{equation}
then
$$
\widetilde{y^\C_{i,j}} := c \cdot y^\C_{i,j} \quad \mbox{for }\, (i,j) \in B(m)
$$
also forms a solution of~\eqref{paijmbfy}.
\end{lemma}

\begin{proof}
It follows from $\pa_m(i,j) (\textbf{\textup{y}}) = \pa_m (i,j) (c \cdot \textbf{\textup{y}})$.
\end{proof}

\begin{lemma}\label{symmetricsol}
There exists a solution $\{ {y}^\C_{i,j} \in \C \backslash \{ 0\}  ~\colon~ \, i + j \leq m + 1 \}$ of the system of equations
\begin{equation}\label{paijmbfyijm}
\pa_m{(i,j)} (\textup{\textbf{y}}) = 0 \quad \mbox{for }\, i + j \leq m
\end{equation}
such that
\begin{equation}\label{symmetricproperty}
y^\C_{i,j} = (y^\C_{j,i})^{-1}.
\end{equation}
\end{lemma}

\begin{proof}
We claim that
\begin{equation}\label{solutionofpamij}
\displaystyle y^\C_{i,j} :=
\begin{cases}
1 \quad &\mbox{for } i = j \\
\displaystyle \prod_{r=0}^{j - i - 1} (2i + 2r) \quad &\mbox{for } i < j \\
\displaystyle \prod_{r=0}^{i - j - 1} (2j + 2r)^{-1} \quad &\mbox{for } i > j
\end{cases}
\end{equation}
forms a solution for~\eqref{paijmbfyijm} satisfying~\eqref{symmetricproperty}.

For the case where $ i < j$, we see
$$
\displaystyle y_{i,j}^\C =  \prod_{r=0}^{j - i - 1} (2i + 2r) = \left( \prod_{r=0}^{j - i - 1} (2i + 2r)^{-1} \right)^{-1} = \left( {y_{j,i}^\C} \right)^{-1}
$$
and
\begin{align*}
\pa_m(i,j) ({\textbf{\textup{y}}}) &= - \frac{y^\C_{i, j+1}}{y^\C_{i,j}} - \frac{y^\C_{i-1,j}}{y^\C_{i,j}} + \frac{y^\C_{i,j}}{y^\C_{i+1,j}} +\frac{y^\C_{i,j}}{y^\C_{i,j-1}} \\
\displaystyle &= - \frac{\prod_{r=0}^{j - i} (2i + 2r)}{\prod_{r=0}^{j - i - 1} (2i + 2r)} - \frac{\prod_{r=0}^{j - i} (2(i - 1) + 2r)}{\prod_{r=0}^{j - i - 1} (2i + 2r)} + \frac{\prod_{r=0}^{j - i - 1} (2i + 2r)}{\prod_{r=0}^{j -i - 2} (2(i+1) + 2r)} + \frac{\prod_{r=0}^{j - i - 1} (2i + 2r)}{\prod_{r=0}^{j - i - 2} (2i + 2r)} \\
\displaystyle &= - 2j - (2i - 2) + 2i + (2j - 2) = 0.
\end{align*}
The case for $i > j$ follows from $\pa_m (i,j) ({\textbf{\textup{y}}}) = - \pa_m (j,i) ({\textbf{\textup{y}}})$. When $i = j$,
$$
\pa_m(i,i) ({\textbf{\textup{y}}}) = - \frac{y^\C_{i, i+1}}{y^\C_{i,i}} - \frac{y^\C_{i-1,i}}{y^\C_{i,i}} + \frac{y^\C_{i,i}}{y^\C_{i+1,i}} +\frac{y^\C_{i,i}}{y^\C_{i,i-1}} = - {y^\C_{i, i+1}} - {y^\C_{i-1,i}} + \frac{1}{y^\C_{i+1,i}} +\frac{1}{y^\C_{i,i-1}} = 0.
$$
\end{proof}

We are ready to prove the existence of a symmetric solution in the sense of~\eqref{symmetricproperty}.
\begin{proposition}\label{propsymmetricsol}
There exists a solution $\{ {y}^\C_{i,j} \in \C \backslash \{ 0\}  ~\colon~ \, (i,j) \in B(m) \}$ of the system~\eqref{paijmbfy} of equations such that~\eqref{symmetricproperty} holds for $(i,j) \in B(m)$.
\end{proposition}

\begin{proof}
We start with a solution $\{ y^\C_{i,j} \in \C \backslash \{0\} ~\colon~ i+j \leq m+1 \}$ from Lemma~\ref{symmetricsol}. For an index $(i,j)$ with $i + j > m + 1$, we take
\begin{equation}\label{fijfullm}
y_{i,j}^\C := (-1)^{i+j-m-1} \, y^\C_{m+1-j, m+1-i}.
\end{equation}
We show that~\eqref{fijfullm} forms a solution for~\eqref{paijmbfy}. For $(i,j) \in B(m)$ with $i + j \geq m + 2$, it is straightforward to see
\begin{align*}
\pa_m{(i,j)} (\textbf{\textup{y}}) &= - \pa_m{(m+1-j, m+1-i)} (\textbf{\textup{y}}) = 0
\end{align*}
by Lemma~\ref{symmetricsol}. For an index $(i,j)$ with $i + j = m + 1$,
\begin{align*}
\pa_m{(i,j)} (\textbf{\textup{y}}) &= - \frac{y^\C_{i, j+1}}{y^\C_{i,j}} - \frac{y^\C_{i-1,j}}{y^\C_{i,j}} + \frac{y^\C_{i,j}}{y^\C_{i+1,j}} +\frac{y^\C_{i,j}}{y^\C_{i,j-1}} = \frac{y^\C_{i-1, j}}{y^\C_{i,j}} - \frac{y^\C_{i-1,j}}{y^\C_{i,j}} + \frac{y^\C_{i,j}}{y^\C_{i+1,j}} -\frac{y^\C_{i,j}}{y^\C_{i+1,j}} = 0.
\end{align*}
From~\eqref{symmetricproperty} for $(i,j)$ with $i + j \leq m + 1$, it follows~\eqref{symmetricproperty} for $(i,j)$ with $i + j > m + 1$ because
\begin{align*}
y_{i,j}^\C = (-1)^{i+j-m-1} \, y^\C_{m+1-j, m+1-i} = (-1)^{i+j-m-1} \, (y^\C_{m+1-i, m+1-j})^{-1} = (y_{j,i}^\C)^{-1}.
\end{align*}
Thus, we have just found a symmetric solution $\{ y^\C_{i,j} ~\colon~ (i,j) \in B(m) \}$ such that $y_{i,i} = \pm 1$.
\end{proof}

\begin{corollary}\label{corsymmetricsol}
For any non-zero complex number $c$, there exists a solution $\{ {y}^\C_{i,j} \in \C \backslash \{ 0\}  ~\colon~ \, (i,j) \in B(m) \}$ of the system~\eqref{paijmbfy} of equations such that
\begin{enumerate}
\item $y^\C_{i,j} \cdot y^\C_{j,i} = c^2$ .
\item $y^\C_{m,m} = c^{\phantom{\C}}$
\item $y^\C_{i,i} = \pm \, c^{\phantom{\C}}$ for any $1 \leq i < m$.
\end{enumerate}
\end{corollary}

\begin{proof}
The component $y^\C_{m,m}$ of a solution from Proposition~\ref{propsymmetricsol} is either $1$ or $-1$. By multiplying by $\pm c$, because of Lemma~\ref{minussol}, we have another solution satisfying $(1), (2)$ and $(3)$.
\end{proof}

%------------------------------------------------------------------------------------------
\vspace{0.2cm}
\subsection{Inside of $B(m)$}\label{insidebm}~
\vspace{0.2cm}

Assume that the split leading term equation of $\Gamma(n)$ associated with $B(m)$ has a solution for some non-zero complex numbers $c^{\textup{hor},\C}_{i, i+1}$'s and $c^{\textup{ver},\C}_{j+1, j}$'s for $i, j \geq k$. Let
\begin{equation}\label{seedabovebox}
y^\C_{m,m}, \, y^\C_{1, m+1}, \cdots, y^\C_{m,m+1}
\end{equation}
where $y^\C_{i,j}$ is the $(i,j)$-th component of a solution of the split leading term equation, which is a non-zero complex number.  
By Corollary~\ref{corsymmetricsol}, we obtain a symmetric complex solution such that $c$ becomes the $(m,m)$-component $y_{m,m}^\C$ of the solution. 
In order to emphasize that~\eqref{seedabovebox} is pre-determined, let 
$$
d^{\phantom{\C}}_{i,j} := y_{i,j}^\C.
$$
Setting it as the initial part for a solution and using Lemma~\ref{extensionlemma}, we extend it to a solution of~\eqref{paijmBm} in $\Lambda_U$. For a pictorial outline of Section~\ref{insidebm}, see Figure~\ref{Stepsinside}. 

\vspace{0.2cm}
\begin{figure}[ht]
	\scalebox{1}{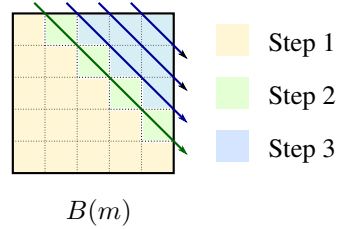}
	\caption{\label{Stepsinside} Pictorial outline of Section~\ref{insidebm}.}	
\end{figure}

\bigskip

\noindent{\bf Step 1. $(i, j) \in B(m)$ with $i + j \leq m + 1$}
\smallskip

We begin by taking $y^{\phantom{\C}}_{i,j} := {y^\C_{i,j}} \in \C \backslash \{0\} \subset \Lambda_U $ where $\{y^\C_{i,j} ~\colon~ (i,j) \in B(m) \}$ is a solution satisfying $y^\C_{m,m} = d^{\phantom{\C}}_{m,m}$ from Corollary~\ref{corsymmetricsol} for all indices $(i,j)$'s with $i + j \leq m+1$. Then, the equations $\pa (i,j) (\textbf{y}) = 0$ for $i+j \leq m$ hold because $\pa (i,j) (\textbf{y})= \pa_m (i,j) (\textbf{y}) \, T^{1-t}$ by~\eqref{shortnotofderi}.

\bigskip

\noindent{\bf Step 2. $(i, j) \in B(m)$ with $i + j = m + 2$}
\smallskip

Next, we determine all entries $y_{i,j}$'s of the anti-diagonal given by $i + j = m + 2$ within $B(m)$. For this purpose, we decompose the equation $\pa (1,m) (\textbf{y}) = 0$ in the system~\eqref{shortnotofderi} into two pieces as follows:
\begin{align*}
\pa{(1,m)} (\textbf{y})
&= \left(\frac{y_{1,m}}{y_{2,m}} + \frac{y_{1,m}}{y_{1,m-1}} \right) T^{1-t} +  \left( - \frac{y_{1,m+1}}{y_{1,m}} \right) T^{1+(m-1)t} \\
&= \left(\frac{y_{1,m}}{y_{2,m}} + \frac{y_{1,m}}{y_{1,m-1}} \right) T^{1-t} +  \left( - \frac{y_{1,m+1}}{y_{1,m}} {\color{blue} + a_{1} \frac{y_{1,m}}{y_{1,m-1}}} {\color{blue} - a_1 \frac{y_{1,m}}{y_{1,m-1}}} \right) T^{1+(m-1)t} \\
&= \left( \frac{y_{1,m}}{y_{2,m}} + \frac{y_{1,m}}{y_{1,m-1}}  {\color{blue} - a_1 \frac{y_{1,m}}{y_{1,m-1}}} T^{mt} \right)  T^{1-t}  + \left( - \frac{y_{1,m+1}}{y_{1,m}} {\color{blue} + a_1 \frac{y_{1,m}}{y_{1,m-1}}} \right) T^{1+(m-1)t}
\end{align*}
where $a_1$ will be determined shortly. 

In order to solve $\pa (1,m) (\textbf{y}) = 0$, it suffices to find a solution of the following system
\begin{align}\label{(m,1)equ}
\begin{cases}
\displaystyle \pa{(1,m)}^{(1)} (\textbf{y}) := \frac{y_{1,m}}{y_{2,m}} + \frac{y_{1,m}}{y_{1,m-1}}  { - a_1 \frac{y_{1,m}}{y_{1,m-1}}} T^{mt} = 0 \\ \\
\displaystyle \pa{(1,m)}^{(2)} (\textbf{y}) := - \frac{y_{1,m+1}}{y_{1,m}} { + a_1 \frac{y_{1,m}}{y_{1,m-1}}}  = 0 \\
\end{cases}
\end{align}
By our choice of $y_{i,j}$'s so far, we have
$$
\frac{y^{\phantom{\C}}_{1,m}}{y^{\phantom{\C}}_{1,m-1}} = \frac{y^\C_{1,m}}{y^\C_{1,m-1}} \neq 0.
$$
Then, $y_{1,m+1} = d_{1,m+1}$ uniquely determines the value $a_1\in \C \backslash \{0\}$ from $\pa{(1,m)}^{(2)} (\textbf{y}) =0$. Then, there exists a unique $y_{2,m} \in \Lambda_U$ such that $\pa (1,m)^{(1)} (\textbf{y}) = 0$.

By applying Lemma~\ref{extensionlemma} succesively, we can determine the remaining entries of the anti-diagonal containing $y_{2,m}$ within $B(m)$. Namely, from a pre-determined $y_{i+1,m-i+1} \in \Lambda_U$ with $\frak{v}_T (y^{\phantom{\C}}_{i+1,m-i+1} - {y}^\C_{i+1,m-i+1}) = mt$, we determine a solution $y_{i+2,m-i} \in \Lambda_U$ of $\pa(i+1,m-i) (\textbf{y}) = 0$ so that
\begin{equation}\label{valofy}
\frak{v}_T \left(y^{\phantom{\C}}_{i+2,m-i} - {y}^\C_{i+2,m-i}\right) = mt.
\end{equation}
Then, all anti-diagonal entries $y_{i+2,m-i}$'s inside $B(m)$ are chosen to obey
$$
\begin{cases}
\pa(1, m)^{(1)} (\textbf{y}) = 0 \\
\pa(i+1,m-i)(\textbf{y}) = 0 \quad \mbox{for  } 1 \leq  i \leq m-2.
\end{cases}
$$
Plugging the previously determined $y_{i,j}$'s into $\pa (m,1) (\textbf{y}) = 0$, we convert $\pa (m,1) (\textbf{y}) = 0$ into $\pa (m,1)^{(2)} (\textbf{y}) = 0$ of the form~\eqref{pa1m2}. By solving $\pa (m,1)^{(2)} (\textbf{y}) = 0$, we obtain $y_{m+1,1}$.

We would like to find a sufficient condition that $y_{m+1,1}$ exists in $\Lambda_U$ such that $y^\C_{m+1,1} \equiv y^{\phantom{\C}}_{m+1,1} \mod T^{>0}$. Because of~\eqref{valofy}, we put
\begin{equation}\label{yismi}
y^{\phantom{\C}}_{i+2,m-i} \equiv {y}^\C_{i+2,m-i} + A_i \cdot T^{mt} \mod T^{>mt}
\end{equation}
where $A_i \in \C \backslash \{0\}$. A straightforward calculation via consideration of $\pa(i+1,m-i)(\textbf{y}) = 0$ gives us the following lemma.

\begin{lemma}\label{aibi}
A recurrence relation for $A_i$'s is
$$
\begin{cases}
A_0 &= - a_1 \cdot y^\C_{1,m-1} \\
A_{i} &= - \frac{\,\, \left({y}^\C_{i+2,m-i} \right)^2 }{\,\, \left({y}^\C_{i+1,m-i} \right)^2} \, A_{i-1}.
\end{cases}
$$
\end{lemma}

From~\eqref{yismi} and Lemma~\ref{aibi}, it follows
\begin{align*}
\pa{(m,1)} (\textbf{y}) T^{t-1} &= - \frac{y_{m-1,1}}{y_{m,1}} - \frac{y_{m,2}}{y_{m,1}} + \frac{y_{m,1}}{y_{m+1,1}} T^{mt} \\
&\equiv \left( - \frac{y^\C_{m-1,1}}{y^\C_{m,1}} - \frac{y^\C_{m,2}}{y^\C_{m,1}} \right) +  \left( - \frac{A_{m-2}}{ y^\C_{m,1}} + \frac{y^\C_{m,1}}{y^{\phantom{\C}}_{m+1,1}} \right) T^{mt} &\mod{ T^{>mt}} \\
&\equiv \left( (-1)^{m} a_1 \, \frac{y^\C_{1,m-1}}{ y^\C_{m,1}} \left( \prod_{i=1}^{m-2} \frac{\,\, ({y}^\C_{i+2,m-i})^2 }{\,\, ({y}^\C_{i+1,m-i})^2} \right) \, + \frac{y^\C_{m,1}}{y^{\phantom{\C}}_{m+1,1}} \right) T^{mt} &\mod{ T^{>mt}}.
\end{align*}
By Corollary~\ref{corsymmetricsol}, we have
\begin{itemize}
\item $y^\C_{i,j} \cdot y^\C_{j,i} = (d^{\phantom{\C}}_{m,m})^2$, $(y^\C_{i,i})^2 = (d^{\phantom{\C}}_{m,m})^2$
\item ${y_{1,m-1}^\C} + {y_{2,m}^\C} = 0$
\item $y^\C_{1,m-1} \cdot y^\C_{m-1,1} = (d^{\phantom{\C}}_{m,m})^2$,
\end{itemize}
Using them, we simplify the above expression as follows.
\begin{align*}
\pa{(m,1)} (\textbf{y}) T^{t-1} 
&\equiv \left( (-1)^{m} \,  a_1\frac{ y^\C_{1,m-1} }{y^\C_{m,1}} \, \frac{(d^{\phantom{\C}}_{m,m})^2}{(y^\C_{2,m})^2} + \frac{y^\C_{m,1}}{y^{\phantom{\C}}_{m+1,1}} \right) T^{mt} &\mod{ T^{>mt}}\\
&\equiv \left( (-1)^{m} \,  a_1\frac{ 1 }{y^\C_{m,1}} \, \frac{(d^{\phantom{\C}}_{m,m})^2}{ y^\C_{1,m-1} } + \frac{y^\C_{m,1}}{y^{\phantom{\C}}_{m+1,1}} \right) T^{mt} &\mod{ T^{>mt}}\\
&\equiv \left( (-1)^{m} a_1 \, \frac{ y^\C_{m-1,1} }{y^\C_{m,1}} + \frac{y^\C_{m,1}}{y^{\phantom{\C}}_{m+1,1}} \right) T^{mt} &\mod{ T^{>mt}}.
\end{align*}
Thus, $\pa(m,1)(\textbf{y}) = 0$ yields
\begin{equation}\label{pa1m2}
\pa(m,1)^{(2)}(\textbf{y}) := (-1)^{m} a_1 \, \frac{ y_{m-1,1} }{y_{m,1}} + \frac{y_{m,1}}{y_{m+1,1}} + \frak{P}(m,1)^{(2)}(a_1) = 0
\end{equation}
for some constant $\frak{P}(m,1)^{(2)}(a_1) \in \Lambda_+$ (depending on $a_1$). Then, $y_{m+1,1} \in \Lambda_U$ can be determined so that $\pa(m,1)^{(2)}(\textbf{y}) = 0$ holds. Because of Corollary~\ref{corsymmetricsol} and~\eqref{(m,1)equ}, we observe
\begin{equation}\label{y1m+1}
y_{m+1,1} \equiv (-1)^{m+1} \, \frac{1}{a_1} \, \frac{(y_{m,1})^2}{y_{m-1,1}} \equiv (-1)^{m+1} \, \frac{(d_{m,m})^2}{a_1} \, \frac{ y_{1,m-1}}{ (y_{1,m})^2} \equiv (-1)^{m+1} \frac{(d_{m,m})^2}{y_{1,m+1}} \mod{ T^{>0}},
\end{equation}
which explains why the equation $\pa_m(1) (\textbf{y}) = 0$ in the system \eqref{splitleadingtermequ} appears. In other words,~\eqref{pamly} provides a sufficient condition to solve $y_{m+1,1}$ over $\Lambda_U$ in~\eqref{pa1m2} such that $y^\C_{m+1,1} \equiv y^{\phantom{\C}}_{m+1,1} \mod T^{>0}$.

\bigskip

\noindent{\bf Step 3. $(i, j) \in B(m)$ with $m + 2 < i + j \leq 2m $}
\smallskip

Now, we determine all elements $y_{i,j}$'s for $(i, j) \in B(m)$ satisfying $m+2 < i + j < 2m$. For an index $j$ with $2 < j < m$, we decompose $\pa{(j,m)} (\textbf{y})$ as follows:
\begin{align*}
\displaystyle \pa{(j,m)} (\textbf{y})&= \left(- \frac{y_{j-1,m}}{y_{j,m}} + \frac{y_{j,m}}{y_{j+1,m}} + \frac{y_{j,m}}{y_{j,m-1}} \right) T^{1-t} +  \left( - \frac{y_{j,m+1}}{y_{j,m}} \right) T^{1+(m-j)t}  \\
&= \left( \left(- \frac{y_{j-1,m}}{y_{j,m}} + \frac{y_{j,m}}{y_{j+1,m}} + \frac{y_{j,m}}{y_{j,m-1}} \right) + \left( {\color{blue} - a_j \frac{y_{j,m}}{y_{j+1,m}} - a_j \frac{y_{j,m}}{y_{j,m-1}}} \right) T^{(m-j+1)t} \right) T^{1-t} \\
&+ \left( - \frac{y_{j,m+1}}{y_{j,m}} {\color{blue} + a_j \frac{y_{j,m}}{y_{j+1,m}} + a_j \frac{y_{j,m}}{y_{j,m-1}}} \right) T^{1+(m-j)t}
\end{align*}
In order to have a solution of $\pa (j,m) (\textbf{y}) = 0$, we decompose it into the following equations.
\begin{align*}
\begin{cases}
\displaystyle \pa{(j,m)}^{(1)} (\textbf{y})  := \left(- \frac{y_{j-1,m}}{y_{j,m}} + \frac{y_{j,m}}{y_{j+1,m}} + \frac{y_{j,m}}{y_{j,m-1}} \right) - a_j \left( {  \frac{y_{j,m}}{y_{j+1,m}} + \frac{y_{j,m}}{y_{j,m-1}}} \right) T^{(m-j+1)t} = 0   \\ \\
\displaystyle \pa{(j,m)}^{(2)} (\textbf{y}) := - \frac{y_{j,m+1}}{y_{j,m}} + a_j \left( { \frac{y_{j,m}}{y_{j+1,m}} + \frac{y_{j,m}}{y_{j,m-1}}} \right)= 0. \\
\end{cases}
\end{align*}
Due to the following lemma, it is enough to find a solution of the following system to solve $\pa(j,m)(\textbf{y}) = 0$.

\begin{lemma} A solution of the system
\begin{align}\label{solutdecompos}
\begin{cases}
\displaystyle \pa{(j,m)}^{(1)} (\textbf{\textup{y}}) = 0 \\
\displaystyle \pa{(j,m)}^{(2)} (\textbf{\textup{y}}) - a_j \cdot \pa{(j,m)}(\textbf{\textup{y}}) \, T^{t-1} = - \frac{y_{j,m+1}}{y_{j,m}}  + a_j \left( \frac{y_{j-1,m}}{y_{j,m}}  + \frac{y_{j,m+1}}{y_{j,m}} \, T^{(m-j+1)t}\right) = 0
\end{cases}
\end{align}
is also a solution of $\pa(j,m)(\textbf{\textup{y}}) = 0$.
\end{lemma}

\begin{proof}
Note that
$$
\pa(j,m)(\textbf{y}) = \pa(j,m)^{(1)}(\textbf{y}) \, T^{1-t} + \pa(j,m)^{(2)}(\textbf{y}) \, T^{1 + (m -j)t}.
$$
A solution of~\eqref{solutdecompos} satisfies
$$
\pa(j,m)(\textbf{y}) =  \pa(j,m)^{(2)}(\textbf{y}) \, T^{1 + (m -j)t} = a_j \cdot \pa(j,m)(\textbf{y}) \, T^{(m-j+1)t},
$$
which gives rise to $\pa(j,m)(\textbf{y}) = 0$.
\end{proof}

Suppose that we are given a solution
$$
\{ y_{r,s} \in \Lambda_U ~\colon~ (r,s) \in B(m) , \, r+s \leq m+j\}
$$
of $\pa (r,s)(\textbf{y}) = 0$
for all $(r,s) \in B(m)$ and $r+s < m +j$ such that
$$
\frak{v}_T (y^{\phantom{\C}}_{r,s} - y_{r,s}^\C) \geq (m - j + 2)t
$$
as the induction hypothesis. 
Since each $y_{r,s}^\C$ is non-zero by our choice, we then obtain ${y^\C_{j-1,m}}/{y^\C_{j,m}} \neq 0$
so that $\pa{(j,m)}^{(2)} (\textbf{\textup{y}}) - a_j \cdot \pa{(j,m)}(\textbf{\textup{y}}) T^{t-1} = 0$ determines a unique value $a_j \in \Lambda_U$ from $y_{j,m+1} = d_{j,m+1}$.
Then, the equation $\pa (j,m)^{(1)} (\textbf{y}) = 0$ yields
$$
\frac{y^\C_{j,m}}{y_{j+1,m}} \left(1 - a_j \, T^{(m-j+1)t} \right) \equiv \left( \frac{y^\C_{j-1,m}}{y^\C_{j,m}} - \frac{y^\C_{j,m}}{y^\C_{j,m-1}} \right) + a_j \frac{y^\C_{j-1,m}}{y^\C_{j,m}} \, T^{(m-j+1)t}  \mod T^{>(m-j+1)t}.
$$

Keeping in mind that $y_{r,s}^\C$'s from Corollary~\ref{corsymmetricsol} satisfy
$$
\pa_m (j,m)(\textbf{y}) = - \frac{y^\C_{j-1,m}}{y^\C_{j,m}}  + { \frac{y^\C_{j,m}}{y^\C_{j+1,m}} + \frac{y^\C_{j,m}}{y^\C_{j,m-1}}} = 0,
$$
we obtain
$$
 \frac{y^\C_{j,m}}{y^\C_{j+1,m}} \neq 0 \, \textup{ and } \, \frac{y^\C_{j-1,m}}{y^\C_{j,m}} - \frac{y^\C_{j,m}}{y^\C_{j,m-1}} \neq 0
$$
since $y_{j,m}^\C$ is non-zero. Hence $y_{j+1,m} \in \Lambda_U$ with $\frak{v}_T (y^{\phantom{\C}}_{j+1,m} - y_{j+1,m}^\C) = (m-j+1)t$.

Suppose that
$$
\{ y_{r,s} \in \Lambda_U : (r,s) \in B(m),\, r+s \leq m+j \} \cup \{ y_{r,s} \in \Lambda_U : (r,s) \in B(m), \, r+s = m+j + 1, \, r > m-i\}
$$
are given and $\frak{v}_T (y^{\phantom{\C}}_{i+j,m-i+1} - y^\C_{i+j,m-i+1})  = (m - j + 1) t$.
By Lemma~\ref{extensionlemma}, the equation $\pa (i+j,m-i) = 0$ determines $y_{i+j+1,m-i} \in \Lambda_U$ such that
\begin{equation}\label{valofy2}
\frak{v}_T (y^{\phantom{\C}}_{i+j+1,m-i} - y^\C_{i+j+1,m-i}) = (m - j + 1)t.
\end{equation}

In order to find $y_{m+1,j}$, we convert $\pa (m,j) (\textbf{y}) = 0$ into $\pa (m,j)^{(2)} (\textbf{y}) = 0$ by inserting the previously determined $y_{i,j}$'s. For $0 \leq i \leq m - j -1$, due to~\eqref{valofy2}, we may set
$$
y^{\phantom{\C}}_{i+j+1,m-i} \equiv {y}^\C_{i+j+1,m-i} + A_i \cdot T^{(m - j + 1) t} \mod T^{> (m - j + 1) t}
$$
where $A_i \in \C \backslash \{0\}$. As in Lemma~\ref{aibi}, we derive the following lemma.
\begin{lemma}\label{aibi2}
A recurrence relation for $A_i$'s is
$$
\begin{cases}
\displaystyle A_0 &= - a_j \cdot \frac{(y^\C_{j+1,m})^2}{(y^\C_{j,m})^2} \cdot y^\C_{j-1,m} \\
\displaystyle A_{i} &= - \frac{\,\, \left({y}^\C_{i+j+1,m-i} \right)^2 }{\,\, \left({y}^\C_{i+j,m-i} \right)^2} \, A_{i-1}.
\end{cases}
$$
\end{lemma}

By Lemma~\ref{aibi2} and Corollary~\ref{corsymmetricsol},
\begin{align*}
\pa{(m,j)} (\textbf{y}) T^{t-1} &=  \left(- \frac{y_{m,j+1}}{y_{m,j}} - \frac{y_{m-1,j}}{y_{m,j}} + \frac{y_{m,j}}{y_{m,j-1}} \right) +  \frac{y_{m,j}}{y_{m+1,j}} T^{(m-j+1)t} \\
&\equiv  \left(- \frac{y^\C_{m,j+1}}{y^\C_{m,j}} - \frac{y^\C_{m-1,j}}{y^\C_{m,j}} + \frac{y^\C_{m,j}}{y^\C_{m,j-1}} \right) + \left( - \frac{A_{m-j-1}}{y^\C_{m,j}} + \frac{y^\C_{m,j}}{y^{\phantom{\C}}_{m+1,j}} \right) T^{(m-j+1)t} &\mod{ T^{>(m-j+1)t}} \\
&\equiv \left( (-1)^{m-j+1} a_j \, \frac{y^\C_{j-1,m}}{ y^\C_{m,j}} \left( \prod_{i=0}^{m-j-1} \frac{\,\, ({y}^\C_{i+j+1,m-i})^2 }{\,\, ({y}^\C_{i+j,m-i})^2}  \right)  + \frac{y^\C_{m,j}}{y^{\phantom{\C}}_{m+1,j}} \right) T^{(m-j+1)t} &\mod{ T^{>(m-j+1)t}} \\
&\equiv \left( (-1)^{m-j+1} a_j \, \frac{1}{ y^\C_{m,j}} \frac{(d^{\phantom{\C}}_{m,m})^2}{y^\C_{m,j-1}} \frac{(y^\C_{m,j})^2}{(d^{\phantom{\C}}_{m,m})^2}+  \frac{y^\C_{m,j}}{y^{\phantom{\C}}_{m+1,j}} \right) T^{(m-j+1)t} &\mod{ T^{>(m-j+1)t}} \\
&\equiv \left( (-1)^{m-j+1} a_j \, \frac{y^\C_{m,j}}{y^\C_{m,j-1}} +  \frac{y^\C_{m,j}}{y^{\phantom{\C}}_{m+1,j}} \right) T^{(m-j+1)t} &\mod{ T^{>(m-j+1)t}}
\end{align*}
which yields
\begin{equation}\label{pajm2}
\pa(m,j)^{(2)}(\textbf{y}) = \left( (-1)^{m-j+1} a_j \, \frac{y^\C_{m,j}}{y^\C_{m,j-1}} + \frac{y^\C_{m,j}}{y^{\phantom{\C}}_{m+1,j}} \right) + \frak{P}(m,j)^{(2)} (a_1, \cdots, a_j) = 0.
\end{equation}
for some $\frak{P}(m,j)^{(2)} (a_1, \cdots, a_j) \in \Lambda_+$. We then have
\begin{equation}\label{yjm+1}
y_{m+1,j} \equiv (-1)^{m-j} \, \frac{y_{m,j-1}}{a_j} \equiv (-1)^{m-j} \, \frac{(d_{m,m})^2}{a_j \cdot y_{j-1,m}} \equiv (-1)^{m-j} \frac{(d_{m,m})^2}{y_{j,m+1}} \mod{ T^{>0}}.
\end{equation}
which explains why the equation $\pa_m(j) (\textbf{y}) = 0$ in the system \eqref{splitleadingtermequ} appears. In other words,~\eqref{pamly} provides a sufficient condition to solve $y_{m+1,j}$ over $\Lambda_U$ in~\eqref{pajm2} such that $y^{\phantom{\C}}_{m+1,j} = y^\C_{m+1,j}$.

Finally, we convert $\pa (m,m)(\textbf{y}) = 0$ into $\pa (m,m)^{(2)}(\textbf{y}) = 0$ as follows. Inserting $y_{m-1,m}, y_{m,m}, y_{m,m-1}$ and $y_{m,m+1} =  d_{m,m+1}$ into $\pa (m,m)(\textbf{y}) = 0$, we derive
\begin{equation}\label{pamm2}
\pa (m,m)^{(2)}(\textbf{y}) = \left( - \frac{y^{\phantom{\C}}_{m,m+1}}{y^\C_{m,m}} + \frac{y^\C_{m,m}}{y^{\phantom{\C}}_{m+1, m}} \right) + \frak{P}(m,m)^{(2)}(\textbf{a}) = 0
\end{equation}
for some $\frak{P}(m,m)^{(2)}(\textbf{a}) \in \Lambda_+$. We obtain
$$
y_{m+1, m} \equiv \frac{(y_{m,m})^2}{y_{m,m+1}} \mod T^{>0}
$$
and determine $y_{m+1,m}$ in $\Lambda_U$.

In summary, the above discussion is summarized as follows.
\begin{proposition}
For any tuple $(d_{m,m}, d_{1, m+1}, \cdots, d_{m,m+1})$ of non-zero complex numbers, there exist
\begin{itemize}
\item $y_{i,j} \in \Lambda_U$ \, for $(i,j) \in B(m) $,
\item $y_{i, m+1} \in \Lambda_U$ \, for $1 \leq i \leq m $,
\item $y_{m+1,j} \in \Lambda_U$ \, for $1 \leq j \leq m $
\end{itemize}
satisfying
\begin{enumerate}
\item $y_{m,m} \equiv d_{m,m} \mod T^{>0}$,
\item $y_{i, m+1} = d_{i, m+1}$ for each $i = 1, \cdots, m$,
\item $\pa (i,j) (\textbf{\textup{y}}) = 0$ \, for $(i,j) \in B(m)$,
\item $\displaystyle (-1)^{m+1-l} \, \frac{y_{l,m+1}}{y_{m,m}} +  \frac{y_{m,m}}{y_{m+1,l}} \equiv 0 \mod T^{>0}$.
\end{enumerate}
\end{proposition}

%------------------------------------------------------------------------------------------
\vspace{0.2cm}
\subsection{Outside of $B(m)$} \label{outsideofbm}~
\vspace{0.2cm}

Suppose that we are given a complex solution 
$$
\{ y^\C_{i,j} \in \C \backslash \{0\} :  (i,j) \in \Gamma(n) \backslash B(m) \cup \{(m,m)\} \}
$$ 
for~\eqref{splitleadingtermequ} together with non-zero complex numbers $c^{\textup{ver}, \C}_{i+1,i}$'s and $c^{\textup{hor}, \C}_{j,j+1}$'s, which is the hypothesis of Theorem~\ref{splitleadingtermequationimpliesnonzero}. In this section, we discuss how to determine a bulk-deformation parameter $\frak{b}$ in~\eqref{bulkparak2} from $c^{\textup{ver}, \C}_{i+1,i}$'s and $c^{\textup{hor}, \C}_{j,j+1}$'s and how to extend it to a solution in $\Lambda_U$ from $y^\C_{i,j}$'s for $\pa^\frak{b}(i,j) (\textbf{y}) = 0$. Assume that $m < k = \left\lceil n/2 \right\rceil$. For the case $m = k$, see Remark~\ref{separatecasemk}. Here is a pictorial outline of the section. 

\vspace{0.2cm}
\begin{figure}[ht]
	\scalebox{1}{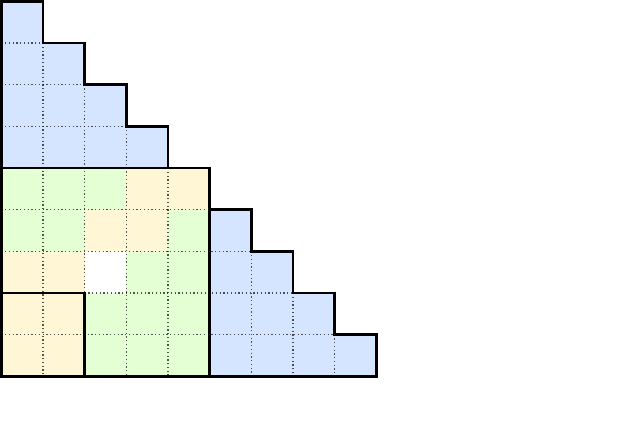}
	\caption{\label{Outside} Pictorial outline of Section~\ref{outsideofbm}.}	
\end{figure}

\bigskip

\noindent{\bf Step 1. $(i,j) \in B(m) \cup \mcal{I}_{\textup{seed}}$}
\smallskip

Let 
\begin{equation}\label{IndexsetforSeed}
\mcal{I}_{\textup{seed}} := \left\{ (m,m), ({1,m+1}), \cdots, ({m,m+1}), ({m+1,m+1}), ({m+1,m+2}), \cdots, \left({\left\lfloor \frac{n}{2} \right\rfloor, \left\lceil \frac{n}{2} \right\rceil}\right) \right\} 
\end{equation}
and
\begin{equation}\label{IndexsetforSeedbackslashmm}
\mcal{I}_{\textup{initial}} :=  \mcal{I}_{\textup{seed}} \backslash \{(m,m)\}.
\end{equation}

\begin{remark}
We will define a \emph{seed} in Definition~\ref{definitionofseeds} to generate a candidate for a solution for the split leading term equation. The set~\eqref{IndexsetforSeed} is the collection of indices where the corresponding variables will be generically chosen as the initial step. 
\end{remark}

We start to take $y^{\phantom{\C}}_{i,j} := y^\C_{i,j}$ for $(i,j) \in \mcal{I}_\textup{initial}$. Then, we fix a complex solution in $B(m)$ from Corollary~\ref{corsymmetricsol} such that $c = y_{m,m}^\C$. 

\bigskip

\noindent{\bf Step 2. $(i,j) \in  B(k) \backslash B(m)$}
\smallskip

By following Section~\ref{insidebm}, the chosen element $y_{1,m+1} \in \C \backslash \{0\}$ determines $y_{i+1,m-i+1}$'s in $\Lambda_U$ for $1 \leq i \leq m-1$. Moreover, we have $y_{m+1,1} \in \Lambda_U$. Again by Section~\ref{insidebm}, we also find $y_{i,j}$'s in $\Lambda_U$ for $(i,j)$ with $i \leq m+1$, $j \leq m+1$ and $i+j = m+3$ satisfying~\eqref{yjm+1}. If $ m + 1 = k$, then proceed to the next anti-diagonal. If $m + 1 < k$, then it remains to determine $y_{1, m+2}$ and $y_{m+2, 1}$ in this anti-diagonal within $B(k) \backslash B(m)$. Since the hypothesis of Lemma~\ref{extensionlemma} is fulfilled at the equations $\pa (1,m+1) (\textbf{y}) = 0$ and $\pa (m+1,1) (\textbf{y}) = 0$ by our standing assumption, they are determined in $\Lambda_U$. Proceeding inductively, we fill up all $y_{i,j}$'s for $(i, j) \in B(k) \backslash B(m)$ obeying
\begin{enumerate}
\item $y_{i,j} \in \Lambda_U$ such that $y^{\phantom{\C}}_{i,j} \equiv y^\C_{i,j} \mod T^{>0}$,
\item $\pa (i, j) (\textbf{y}) = 0$ for $(i,j) \in B(k-1)$.
\end{enumerate}

\bigskip

\noindent{\bf Step 3. $(i,j) \in  \Gamma(n) \backslash B(k)$ and $\frak{b}$}
\smallskip

We determine $y_{i,j}$'s for $(i,j) \in \Gamma(n) \backslash B(k)$ and $\frak{b}^\textup{hor}_{i, i+1}$'s and $\frak{b}^\textup{ver}_{j+1, j}$'s in~\eqref{bulkparak2} over $\Lambda_U$. Notice that
$$
\pa (i, j) (\textbf{y}) = \pa^\frak{b} (i, j)(\textbf{y}) \,\, \textup{for } (i, j) \in B(k-1)
$$
because of our choice of $\frak{b}$, and thus we may keep $\{ y_{i,j} \in \Lambda_U: (i,j) \in B(k) \}$ as a solution of $\pa^\frak{b} (i, j)(\textbf{y}) = 0$.

From now on, we focus only on the case where $n = 2k - 1$ because the case $n = 2k$ can be similarly dealt with. In this case, there are $(k^2 - 1)$ variables in $B(k)$. As all variables $y_{k-1,k}, y_{k-2,k}$ and $y_{k-1,k-1}$ in 
$$
\pa (k-1,k)(\textbf{y}) = - \frac{1}{y_{k-1,k}} - \frac{y_{k-2,k}}{y_{k-1,k}} + y_{k-1, k} + \frac{y_{k-1,k}}{y_{k-1, k-1}}
$$ 
have been already determined by previous inductive steps, we do \emph{not} have any extra variables to make $\pa(k-1,k)(\textbf{y}) = 0$ hold. It is time to adjust the equation $\pa (k-1,k)(\textbf{y}) = 0$ by selecting $\frak{b}^\textup{ver}_{k+1, k}$ suitably. By~\eqref{thegradientofbulkdeformedpotential}, we have
$$
\pa^\frak{b}(k-1,k) (\textbf{y}) := - \exp( \frak{b}^\textup{ver}_{k+1, k} ) \cdot \frac{1}{y_{k-1,k}} - \frac{y_{k-2,k}}{y_{k-1,k}} + {y_{k-1,k}} + \frac{y_{k-1,k}}{y_{k-1, k-1}}.
$$
From the following equation 
$$
 - c^{\textup{ver}, \C}_{k+1, k}  \cdot \frac{1}{y^\C_{k-1,k}} - \frac{y^\C_{k-2,k}}{y^\C_{k-1,k}} + {y^\C_{k-1,k}} + \frac{y^\C_{k-1,k}}{y^\C_{k-1, k-1}} = 0,
$$
one equation in the split leading term equation, it follows that 
$$
- \frac{y^\C_{k-2,k}}{y^\C_{k-1,k}} + {y^\C_{k-1,k}} + \frac{y^\C_{k-1,k}}{y^\C_{k-1, k-1}} \neq 0,
$$
otherwise $ - c^{\textup{ver}, \C}_{k+1, k}  = 0$. 
Since
$$
- \frac{y^{\phantom{\C}}_{k-2,k}}{y^{\phantom{\C}}_{k-1,k}} + {y^{\phantom{\C}}_{k-1,k}} + \frac{y^{\phantom{\C}}_{k-1,k}}{y^{\phantom{\C}}_{k-1, k-1}} \equiv - \frac{y^\C_{k-2,k}}{y^\C_{k-1,k}} + {y^\C_{k-1,k}} + \frac{y^\C_{k-1,k}}{y^\C_{k-1, k-1}} \neq 0 \mod T^{>0},
$$
there exists a unique bulk-deformation parameter $\frak{b}^\textup{ver}_{k+1, k} \in \Lambda_0$ such that
\begin{enumerate}
\item $\pa^\frak{b} (k-1,k)(\textbf{y}) = 0$
\item $\exp (\frak{b}^\textup{ver}_{k+1, k}) \equiv c^{\textup{ver}, \C}_{k+1, k} \mod T^{>0}$
\end{enumerate}
Notice that $\frak{b}^\textup{ver}_{k+1, k}$ does not only deforms $\frac{1}{y_{k-1,k}}$, but also deforms $\frac{y_{j,k+1}}{y_{j,k}}$ into $\exp( \frak{b}^\textup{ver}_{k+1, k} )  \cdot \frac{y_{j,k+1}}{y_{j,k}}$ for all $j$ with $1 \leq j < k-1$ as in Corollary~\ref{formulaforbulkdeformedpotentialoursitu}. Therefore, we need to solve the deformed equation
$$
\pa^\frak{b}(j,k) (\textbf{y}) := -\exp(\frak{b}^\textup{ver}_{k+1, k}) \cdot \frac{y_{j,k+1}}{y_{j,k}} - \frac{y_{j-1,k}}{y_{j,k}} + \frac{y_{j,k}}{y_{j+1,k}} + \frac{y_{j,k}}{y_{j,k-1}} = 0
$$
in order to decide $y_{\bullet,k+1} \in \Lambda_U$.

For the induction hypothesis, assume that $y_{r,s}$'s for $s \leq j$ and $\frak{b}^\textup{ver}_{s, s-1}$'s for $s \leq j$ are determined. We pick a bulk-deformation parameter $\frak{b}^\textup{ver}_{j+1, j} \in \Lambda_0$ so that
\begin{enumerate}
\item $\pa^\frak{b}(n - j, j)(\textbf{y}) = 0$
\item $\exp (\frak{b}^{\textup{ver} \phantom{,\C}}_{j+1, j}) \equiv c^{\textup{ver}, \C}_{j+1, j} \mod T^{>0}$.
\end{enumerate}
After fixing $\frak{b}^\textup{ver}_{j+1, j} $, we determine $y_{\bullet,j+1} \in \Lambda_U$. Hence, all entries above $B(k)$ together with $\frak{b}^\textup{ver}_{j+1,j}$'s are determined in this way. Symmetrically, we can choose $\frak{b}^\textup{hor}_{i,i+1}$'s and fill up the other part of $\Gamma(n) \backslash B(k)$. Hence, Theorem~\ref{splitleadingtermequationimpliesnonzero} is now verified.

\begin{remark}\label{separatecasemk}
We outline the proof of Theorem~\ref{splitleadingtermequationimpliesnonzero} when $n = 2k$ and $m = k$. In this case, taking $c = 1$ for $y^\C_{m,m}$, Corollary~\ref{corsymmetricsol} will give us the initial parts $y^\C_{i,j}$'s of $y^{\phantom{\C}}_{i,j}$'s for $(i,j) \in B(m)$. We then follow Section~\ref{insidebm} to extend to $y_{i,j}$'s in $\Lambda_U$. If one uses both $\frak{b}^\textup{ver}_{m+1,m}$ and $\frak{b}^\textup{hor}_{m,m+1} $ to deform $\pa(m,m) = 0$, then we have two extra variables $c^\textup{ver}_{m+1,m}$ and $c^\textup{hor}_{m,m+1}$ in $\pa^\frak{b}(m,m) = 0$. For our convenience, recall that we have chosen $\frak{b}^\textup{hor}_{m,m+1} = 0$ in Remark~\ref{whenm2kchoiceof}. Now, we need to take $\frak{b}^\textup{ver}_{m+1,m}$ so that $\pa^\frak{b}(m,m) = 0$. Since $y^\C_{m,m} = 1$, we get $\frak{b}^\textup{ver}_{m+1,m} \in \Lambda_+$, which yields that $1 = c^{\textup{ver},\C}_{m+1,m} = \exp (\frak{b}^{\textup{ver}\phantom{,\C}}_{m+1,m}) \mod T^{>0}$. After fixing $\frak{b}^{\textup{ver}\phantom{,\C}}_{m+1,m}$, we solve $y_{\bullet, m+1}$ by solving $\pa^\frak{b}(\bullet, m)^{(2)} = 0$ where
\begin{align*}
\pa^\frak{b}{(1,m)}^{(2)} (\textbf{y}) &:= - \exp (\frak{b}^{\textup{ver}\phantom{,\C}}_{m+1,m}) \, \frac{y_{1,m+1}}{y_{1,m}} { + a_1 \frac{y_{1,m}}{y_{1,m-1}}}  = 0 \\
\pa^\frak{b}{(j,m)}^{(2)} (\textbf{y}) &:= - \exp (\frak{b}^{\textup{ver}\phantom{,\C}}_{m+1,m}) \, \frac{y_{j,m+1}}{y_{j,m}} + a_j \left( { \frac{y_{j,m}}{y_{j+1,m}} + \frac{y_{j,m}}{y_{j,m-1}}} \right)= 0 \quad \mbox{for } j > 2.
\end{align*}
The remaining steps are similar to the case for $m < k$ in Section~\ref{outsideofbm}.
\end{remark}

%------------------------------------------------------------------------------------------
%------------------------------------------------------------------------------------------
\section{Solvability of split leading term equation}\label{secSolvabilityOfSplitLeadingTermEquation}\label{solvabilityofthesplitleadingtermequ}

This section aims to verify the assumption for Theorem~\ref{splitleadingtermequationimpliesnonzero} when the split leading term equation~\eqref{splitleadingtermequ} comes from the line segment $I_m \subset \Delta_\lambda$ in~\eqref{IMT}. To find its solution, we introduce a \emph{seed} generating a candidate for a solution and prove that there exists a ``good" choice of seeds such that the candidate is indeed a solution.

%------------------------------------------------------------------------------------------
\vspace{0.2cm}
\subsection{Seeds}~
\vspace{0.2cm}

We begin by the definition of a seed. Recall the notations $\Gamma(n)$ and $B(m)$ in~\eqref{gammanbm}.

\begin{definition}\label{definitionofseeds}
A \emph{seed} of $\Gamma(n)$ \emph{associated with} $B(m)$ consists of the two data $(\textbf{d}, \mcal{I})$.
\begin{itemize}
\item An $(n - m)$-tuple $\textbf{\textup{d}}$ of elements in $\Lambda_U$ 
$$ 
\textbf{\textup{d}} = (d_1, \cdots, d_{n-m})
$$
\item An $(n- m)$-tuple $\mcal{I}$ of double indices
$$
\mcal{I} = \{(m,m), (i_1,j_1), \cdots, (i_{n-m-1},j_{n-m-1})\} \subset
\{(m,m)\}\cup (\Gamma(n) \backslash B(m))
$$
satisfying
\begin{enumerate}
\item the first index is $(m,m)$ 
\item the remaining indices are contained in $\Gamma(n) \backslash B(m)$ such that any two indices must not come from the same anti-diagonal of $\Gamma(n) \backslash B(m)$.
\end{enumerate}
\end{itemize}
\end{definition}

We are particularly interested in seeds $(\textbf{d}, \mcal{I})$ of the form
\begin{itemize}
\item $\textbf{d}$ is a tuple of \emph{non-zero real} numbers. 
\item $\mcal{I} := \mcal{I}_{\textup{seed}}$ in~\eqref{IndexsetforSeed}.
\end{itemize}
Let $\textbf{y}_\mcal{I}$ denote the components of $\textbf{y}$ associated with the set $\mcal{I}$ of indices. Namely,
$$
\textbf{y}_\mcal{I} := \left( y_{m,m}, y_{i_1,j_1}, \cdots, y_{i_{n-m-1},j_{n-m-1}} \right).
$$
Then, as the initial step, we take 
$$
\textbf{y}_\mcal{I} := \textbf{d}.
$$ 
So, the double indices designate the places in which the components of $\textbf{d}$ are plugged. Since $\mcal{I}$ is always taken to be $\mcal{I}_{\textup{seed}}$, $\mcal{I}$ will be often omitted from now on.
We instead set $d_{i,j}$ to denote the component of $\textbf{d}$ corresponding to $(i,j)$. For instance, we have $d_1 = d_{m,m}$. 

Following the procedure in Section~\ref{outsideofbm}, see Figure~\ref{Outside}, we generate the other $y_{i,j}$'s such that $\textbf{y}$ satisfies the split leading term equation with a suitable choice of complex numbers
$$
\textbf{c} := \left( c^\textup{hor}_{k, k+1}, \cdots, c^\textup{hor}_{n-1, n} , c^\textup{ver}_{k+1, k}, \cdots, c^\textup{ver}_{n, n-1} \right).
$$ 
Namely, by isolating one undetermined variable and plugging the determined variables in one equation of the split leading term equation, we can solve the remaining $y_{i,j}$'s and $\textbf{c}$ inductively. However, the undetermined variable might be zero or undefined when generating a candidate from a seed. A \emph{good} choice of seed, we call a \emph{generic} seed, must avoid the issue. 

We would like to find a condition for generic seeds. In the setup of~\eqref{settingup} and Remark~\ref{whenm2kchoiceof}, we put
\begin{equation}\label{tildepabijy}
\widetilde{\pa^\frak{b}_m} (i,j)  (\textup{\textbf{y}}) :=
\begin{cases}
\displaystyle - c^\textup{ver}_{j, j-1} \cdot \frac{1}{y_{i, j-1}} + \frac{1}{(y_{i, j})^2} \left( c^\textup{ver}_{j+1, j} \cdot {y_{i, j+1}}+ c^\textup{hor}_{i-1, i} \cdot {y_{i-1, j}} \right) \quad &\mbox{if \,} i \geq j \\ \\
\displaystyle - c^\textup{hor}_{i-1, i} \cdot {y_{i-1, j}} + (y_{i,j})^2 \left( c^\textup{hor}_{i, i+1} \cdot \frac{1}{y_{i+1, j}} + c^\textup{ver}_{j, j-1} \cdot \frac{1}{y_{i, j-1}} \right) \quad &\mbox{if \,} i < j.\\
\end{cases}
\end{equation}
Note that $\widetilde{\pa_m^\frak{b}} (i,j)  (\textup{\textbf{y}})$ is achieved by isolating $c^\textup{hor}_{i,i+1} \cdot (y_{i+1, j})^{-1}$ and $c^\textup{ver}_{j+1,j} \cdot y_{i, j+1}$ in ${\pa_m^\frak{b}} (i,j)  (\textup{\textbf{y}}) =0$, see ~\eqref{pabmijy}. Moreover,~\eqref{tildepabijy} appears when isolating the undetermined variable so that the expression is required to be non-zero.  

\begin{definition}
A seed $\textbf{\textup{d}}$ is called \emph{generic} if the candidate generated by $\textbf{y}_{\mcal{I}} =  {\bf d}$ satisfies 
\begin{equation}\label{widetildepabijneq}
\widetilde{\pa_m^\frak{b}} (i,j) (\textup{\textbf{y}}) \neq  0 \mod T^{>0}
\end{equation}
for all $(i,j)$'s. 
\end{definition}

\begin{example}
A straightforward calculation asserts that the tuples 
\begin{enumerate}
\item $\textbf{d} = (-1, 1, 1, -1, 1)$
\item $\mcal{I} = \mcal{I}_\textup{seed} = ( (2,2), (1,3), (2,3), (3,3), (3,4) )$
\end{enumerate}
form a generic seed of $\Gamma(7)$ to $B(2)$. The tuples 
\begin{enumerate}
\item $\textbf{d} = (-1, 1, 1, 1, 1)$
\item $\mcal{I} = \mcal{I}_\textup{seed} = ( (2,2), (1,3), (2,3), (3,3), (3,4) )$
\end{enumerate}
form a seed of $\Gamma(7)$ to $B(2)$, but not a generic seed because $\widetilde{\pa^\frak{b}_2} (1,5) (\textbf{y}) = 0$. 
\end{example}

The main proposition of this section is the existence of a generic seed, which will be proven throughout this section.
\begin{proposition}\label{Existenceofgenericseeds}
For each integer $m$ where $2 \leq m \leq k =  \left\lceil n/2 \right\rceil$, a generic seed of $\Gamma(n)$ to $B(m)$ exists.
\end{proposition}

As a corollary, we assert solvability of the split leading term equation.
\begin{corollary}\label{splitleadingtermequationissolv}
The split leading term equation of $\Gamma(n)$ associated with $B(m)$ has a solution each component of which is a non-zero complex number.
\end{corollary}

\begin{proof}
Once a seed has the property~\eqref{widetildepabijneq}, the remaining $y_{i,j}$'s and a sequence $\textbf{c}$ are (uniquely) determined to be in $\C \backslash \{0\}$ by the exactly same process in Section~\ref{outsideofbm}.
\end{proof}

%------------------------------------------------------------------------------------------
\vspace{0.2cm}
\subsection{Pre-generic elements}~
\vspace{0.2cm}

We now introduce a coordinate system $\{z_{i,j} ~\colon~ (i,j) \in \Gamma(n) \backslash B(m) \cup \{(m,m)\} \}$ with respect to which
the system of equations
$$
\pa_m^\frak{b} (i,j) (\textup{\textbf{y}}) = 0, \quad \text{ for }\,  i+j < n
$$
does \emph{not} depend on the choice of a bulk-deformation parameter $\frak{b}$. We define
\begin{equation}\label{coordinatechange}
\begin{cases}
\displaystyle z_{i+1, \bullet} := \left( \prod^{i}_{r=k} {c^\textup{hor}_{r, r+1} } \right)^{-1} y_{i+1,\bullet} \quad &\mbox{if \,} i \geq k \\
\displaystyle z_{\bullet, j+1} := \left( \prod^{j}_{r=k} c^\textup{ver}_{r+1, r} \right) y_{\bullet, j+1} \quad &\mbox{if \,} j \geq k \\
\displaystyle z_{i,j} := y_{i,j} \quad &\mbox{otherwise.}
\end{cases}
\end{equation}
Under this coordinate system, we convert $\pa_m^\frak{b} (i, j) (\textup{\textbf{y}})$ in~\eqref{pabmijy} into
\begin{equation}\label{pambijconvertedtoz}
\pa_m^\frak{b} (i, j) (\textup{\textbf{z}}) :=
\begin{cases}
\displaystyle - \frac{z_{i, j+1}}{z_{i, j}} - \frac{z_{i-1, j}}{z_{i,j}} + \frac{z_{i,j}}{z_{i+1, j}} + \frac{z_{i, j}}{z_{i, j-1}}\, &\mbox{if \,} i + j < n \\ \\
\displaystyle - \left( \prod^{i-1}_{r=k} {c^\textup{hor}_{r, r+1} } \right)^{-1} \frac{1}{z_{i, j}} - \frac{z_{i-1, j}}{z_{i,j}} + \left( \prod^{i}_{r=k} {c^\textup{hor}_{r, r+1} } \right){z_{i,j}} + \frac{z_{i, j}}{z_{i, j-1}}\, &\mbox{if \,} i \geq j,\, i + j = n\\ \\
\displaystyle -  \left( \prod^{j}_{r=k} c^\textup{ver}_{r+1, r} \right) \frac{1}{z_{i, j}} - \frac{z_{i-1, j}}{z_{i,j}} +  \left( \prod^{j-1}_{r=k} c^\textup{ver}_{r+1, r} \right)^{-1} {z_{i,j}}+ \frac{z_{i, j}}{z_{i, j-1}}\, &\mbox{if \,} i < j,\, i + j = n.\\
\end{cases}
\end{equation}
Here, one should interpret that the product over the empty set is $1$. For example,
$$
\prod_{r = k}^{k-1} c^\textup{ver}_{r+1, r} = 1.
$$

We set
\begin{equation}\label{tildepabijz}
\widetilde{\pa_m^\frak{b}} (i,j)(\textup{\textbf{z}}) :=
\begin{cases}
\displaystyle - \frac{1}{z_{i, j-1}} + \frac{1}{(z_{i, j})^2} \left( {z_{i, j+1}} + {z_{i-1, j}} \right) \left( = \frac{1}{z_{i+1,j}}\right) \quad &\mbox{if \,} i \geq j, \, i + j < n \\ \\
\displaystyle - {z_{i-1, j}} + ({z_{i,j}})^2 \left( \frac{1}{z_{i+1, j}} + \frac{1}{z_{i, j-1}} \right) \left(=  z_{i, j+1}\right) \quad &\mbox{if \,} i < j, \, i + j < n \\ \\
\displaystyle - \frac{1}{z_{i, j-1}} + \frac{1}{(z_{i, j})^2} \left( \left( \prod^{i-1}_{r=k} {c^\textup{hor}_{r, r+1} } \right)^{-1} + {z_{i-1, j}} \right) \left( = \left( \prod^{i}_{r=k} {c^\textup{hor}_{r, r+1} } \right)\right) \quad &\mbox{if \,} i \geq j, \, i + j = n \\ \\
\displaystyle - {z_{i-1, j}} + ({z_{i,j}})^2 \left( \left( \prod^{j-1}_{r=k} c^\textup{ver}_{r+1, r} \right)^{-1} + \frac{1}{z_{i, j-1}} \right) 
\left(= \left( \prod^{j}_{r=k} c^\textup{ver}_{r+1, r} \right)\right) \quad &\mbox{if \,} i < j, \,  i + j = n,
\end{cases}
\end{equation}
where $\widetilde{\pa_m^\frak{b}} (i,j)(\textup{\textbf{z}})$ is obtained from isolating the expression in the parentheses in $\pa_m^\frak{b} (i, j) (\textup{\textbf{z}})$.

We then have the following lemma, which says it suffices to check $\widetilde{\pa_m^\frak{b}} (i, j) (\textup{\textbf{z}}) \neq 0$ to show $\widetilde{\pa_m^\frak{b}} (i, j) (\textup{\textbf{y}}) \neq 0$.
\begin{lemma}
$\widetilde{\pa_m^\frak{b}} (i, j) (\textup{\textbf{z}}) \neq 0$ for all indices $(i,j) \in \Gamma(n) \backslash B(m)$ if and only if $\widetilde{\pa_m^\frak{b}} (i, j) (\textup{\textbf{y}}) \neq 0$ for all indices $(i,j) \in \Gamma(n) \backslash B(m)$.
\end{lemma}

\begin{proof}
Under the coordinate change~\eqref{coordinatechange},~\eqref{tildepabijy} is converted into~\eqref{tildepabijz}.
\end{proof}

To show that $\widetilde{\pa_m^\frak{b}} (i, j) (\textup{\textbf{z}}) \neq 0$, we now start to solve~\eqref{pambijconvertedtoz} from $\textbf{y}_\mcal{I} := \textbf{d}$ by isolating the undetermined variable in~\eqref{pambijconvertedtoz}. When $m < k = \lceil n / 2 \rceil$, since $\mcal{I} \subset B(k)$ and $\textbf{y}_\mcal{I} = \textbf{z}_\mcal{I}$ because of~\eqref{coordinatechange}, we may insert $\textbf{d}$ into $\textbf{z}_\mcal{I}$ as the starting point. For the case $m = k$, we take $c^\textup{hor}_{m,m+1} = 1$ and $c^\textup{ver}_{m+1,m} = 1$ (see Remark~\ref{separatecasemk}) and hence $\textbf{z}_\mcal{I} = \textbf{y}_\mcal{I} = \textbf{d}$ as well. For simplicity, we set
\begin{equation}
\textbf{z}_{(l \backslash m)}: =  \{ z_{i,j} \in \C \backslash \{0\} : (i,j) \in \Gamma(l) \backslash B(m) \cup \{ (m,m)\} \}.
\end{equation}
Choosing the component in an anti-diagonal generically, we can easily make the first two equations of~\eqref{tildepabijz} non-zero because of the following lemma. 

\begin{lemma}\label{rationalfunctionzij}
Suppose that the set $\textbf{\textup{z}}_{(r+s-1 \backslash m)}$ is determined. Each variable $z_{r-i,s+i}$ can be expressed as a \emph{non-constant} rational function with respect to $z_{r,s}$.
\end{lemma}

\begin{proof}
We only show the case for $i > 0$ since the case where $i < 0$ can be similarly proven. Let
$$
X(i) := z_{r-i,s+i}.
$$
By~\eqref{pambijconvertedtoz}, a recurrence relation for $X(i)$'s is
\begin{equation}\label{recrelxi}
X(i) = [i] + \frac{[i,i-1]}{X(i-1)}
\end{equation}
where
$$
[i] := - z_{r-i-1, s+i} + \frac{(z_{r-i, s+i-1})^2}{z_{r-i,s+i-2}},\,\, [i, i-1] :=(z_{r-i, s+i-1})^2.
$$
Composing~\eqref{recrelxi} several times, $X(i)$ is expressed as a continued fraction in terms of $X(0)$. Letting $A(0) = 1$ and $B(0) = 0$, it becomes
\begin{equation}\label{X(i)}
X(i) = \frac{A(i) \cdot X(0) + B(i)}{A(i-1) \cdot X(0) + B(i-1)}
\end{equation}
for some constants $A(i)$'s and $B(i)$'s determined by the given set $\textbf{\textup{z}}_{(r+s-1 \backslash m)}$. Thus, $X(i)$ is a rational function with respect to $X(0)$.

To show that every $X(i)$ is \emph{non-constant} with respect to $X(0)$, we investigate properties of $A(i)$'s and $B(i)$'s. By induction, we can show that the terms of $A(i)$ correspond to the partitions of $\{ i, i-1, \cdots, 1 \}$ into one single number or two consecutive numbers. Also, the terms of $B(i)$ correspond to the partitions of $\{ i, i-1, \cdots, 1, 0 \}$ into one single or two consecutive numbers containing the subset $[1,0]$. For instance, $A(3)$ and $B(3)$ are expressed as
\begin{align*}
A(3) &= [3][2][1] + [3][2,1] + [3,2][1], \\
B(3) & = [3][2][1,0] + [3,2][1,0].
\end{align*}
It then follows that
\begin{align*}
&A(i) = [i] \cdot A(i-1) + [i, i-1] \cdot A(i-2) \\
&B(i) = [i] \cdot B(i-1) + [i, i-1] \cdot B(i-2).
\end{align*}

Note that $X(0)$ and $X(1)$ are non-constant functions with respect to $X(0)$. Suppose to the contrary that $X(i)$ is a constant function with the value $C$ and all $X(j)$'s for all $j < i$ are non-constant rational functions with respect to $X(0)$. Let
$$
X(i) := \frac{A(i) \cdot X(0) + B(i)}{A(i-1) \cdot X(0) + B(i-1)} = C.
$$
We then obtain
\begin{align*}
&C \cdot A(i-1) = A(i) = [i] \cdot A(i-1) + [i, i-1] \cdot A(i-2)\\
&C \cdot B(i-1) = B(i) = [i] \cdot B(i-1) + [i, i-1] \cdot B(i-2).
\end{align*}

We claim that $C - [i] \neq 0$. Otherwise, $A(i-2) = B(i-2) = 0$ because $[i, i-1] = (z_{r-i, s+i-1})^2 \neq 0$. It yields that $X(i-2) \equiv 0$, contradicting to the assumption that $X(i-2)$ is not constant.

We then have
\begin{align*}
A(i-1) = C^\prime \cdot A(i-2) \\
B(i-1) = C^\prime \cdot B(i-2)
\end{align*}
where $C^\prime = {[i,i-1]}/{(C - [i])}$.  So, we deduce a contradiction that
$$
X(i-1) = \frac{A(i-1) \cdot X(0) + B(i-1)} {A(i-2) \cdot X(0) + B(i-2)} = C^\prime
$$
is constant. Hence, every $X(i)$ has to be a non-constant rational function.
\end{proof}

\begin{corollary}\label{choiceofzij}
Suppose that the set $\textbf{\textup{z}}_{(r+s-1 \backslash m)}$ is determined. There exists a non-zero real number $d_{r,s}$ such that if we set $z_{r,s} = d_{r,s}$
\begin{equation}\label{paijznonzero}
\widetilde{\pa_m^\frak{b}} (i,j) (\textup{\textbf{z}}) \neq 0
\end{equation}
for all $(i,j)$'s obeying $i + j = r + s - 1$.
\end{corollary}

\begin{proof}
Since each $z_{r-i, s+i}$ is a non-constant rational function with respect to $z_{r,s}$, there are only finitely many $z_{r,s}$'s so that $z_{r-i, s+i}$ is zero or is not defined. Avoid these values when choosing $d_{r,s}$.
\end{proof}

\begin{definition} 
Suppose the set $\textbf{\textup{z}}_{(r+s-1 \backslash m)}$ is given.
For an index $(r,s) \in \Gamma(n) \backslash B(m)$, an element $d_{r,s}$ is said to be \emph{pre-generic with respect to}
$\textbf{\textup{z}}_{(r+s-1 \backslash m)}$ if~\eqref{paijznonzero} holds for any $(i,j) \in \Gamma(r+s) \backslash (B(m) \cup \Gamma(r+s-1))$.
\end{definition}

For the later purpose, we prove the following property of the the pre-generic elements.

\begin{lemma}\label{symmericseedz}
Assume that $ m < k$. Suppose that we have $d_{s, m+1}$'s for $s$ with $1 \leq s \leq m$ such that for each $s$, $d_{s,m+1}$ is pre-generic with respect to the previously determined $\textbf{\textup{z}}_{(s+m \backslash m)}$ by one choice of $d_{m,m}$ and $d_{1,m+1}, \cdots, d_{s-1,m+1}$. Then, regardless of a choice of $d_{m,m} \in \C \backslash \{ 0 \}$, $d_{s,m+1}$ is pre-generic as long as we do not change $d_{1,m+1}, \cdots, d_{s-1,m+1}$.
\end{lemma}

\begin{proof} 
If $m < k$, we see $y_{i,j} = z_{i,j}$ for $(i,j) \in B(m)$ by~\eqref{coordinatechange}.
We claim that
\begin{equation*}
z_{j,i+1} = (-1)^{i+j} \cdot \frac{(d_{m,m})^2}{z_{i+1,j}}.
\end{equation*}
Recall from~\eqref{pamly} that
\begin{equation*}
z_{m+1,i} = (-1)^{i + (m+1) - 1} \cdot \frac{(d_{m,m})^2}{z_{i,m+1}},
\end{equation*}
which provides the initial step for the induction. Next, by the induction hypothesis, we observe
\begin{align*}
\begin{split}
0 &= - \frac{z_{i,j+1}}{z_{i,j}} - \frac{z_{i-1,j}}{z_{i,j}} + \frac{z_{i,j}}{z_{i+1,j}} + \frac{z_{i,j}}{z_{i,j-1}}\\
&= \frac{z_{j,i}}{z_{j+1,i}} + \frac{z_{j,i}}{z_{j,i-1}} + (-1)^{i + j -1}\frac{(d_{m,m})^2}{z_{i+1,j} \cdot z_{j,i}} - \frac{z_{j-1,i}}{z_{j,i}} \\
&= \frac{z_{j,i+1}}{z_{j,i}} + (-1)^{i + j- 1} \frac{(d_{m,m})^2}{z_{i+1,j}\, z_{j,i}}.
\end{split}
\end{align*}
Thus, we obtain
\begin{equation*}
z_{j,i+1} = (-1)^{i+j} \cdot \frac{(d_{m,m})^2}{z_{i+1,j}}.
\end{equation*}
Therefore, $\widetilde{\pa_m^\frak{b}} (i, j) (\textup{\textbf{z}}) \neq 0$ as long as $\widetilde{\pa_m^\frak{b}} (j,i) (\textup{\textbf{z}}) \neq 0$.
\end{proof}

%------------------------------------------------------------------------------------------
\vspace{0.2cm}
\subsection{Generic seeds}~
\vspace{0.2cm}

Applying Corollary~\ref{choiceofzij}, we make the first two expressions in~\eqref{tildepabijz} non-zero by taking one entry of an anti-diagonal generically. To make the last two equations non-zero, we need to select the previous ones more carefully. We deal with the three cases separately. 

\bigskip

\noindent{\bf Case 1. $ n = 2k - 1$.}
\smallskip

We need several lemmas. 

\begin{lemma}\label{reductionkk-1k-1k-1}
Assume that $\textbf{\textup{z}}_{(n-2 \backslash m)}$ is given. Suppose that either $d_{k-1,k-1} = -1$ is pre-generic or $k - 1 =m$. Then, there is a real number $d_{k-1,k-1}$ (sufficiently close to $-1$) and a non-zero real number $d_{k-1,k}$ such that if $z_{k-1,k-1} = d_{k-1,k-1}$ and $z_{k-1,k} = d_{k-1,k}$, 
\begin{equation}\label{pabijzneq0}
\widetilde{\pa_m^\frak{b}}(i,j)(\textup{\textbf{z}}) \neq 0 \mod T^{>0}
\end{equation}
for all $(i,j)$ with $i + j = n-1$ and $i + j = n$. 
\end{lemma}

Note that $\widetilde{\pa_m^\frak{b}}(i,j)(\textup{\textbf{z}})$'s for $(i,j)$ with $i + j = n$ provide the last two expressions of~\eqref{tildepabijz}. 

\begin{proof}
Assuming that $d_{k-1, k-1} = -1$ is pre-generic, by definition, every $z_{i, n - 1- i}$ is defined and becomes non-zero if we set $z_{k-1,k-1} = d_{k-1,k-1} = -1$. By Corollary~\ref{choiceofzij}, we can choose and fix a pre-generic element $d_{k-1,k}$ for $z_{k-1,k}$ so that the entries $z_{i,n-i}$'s are also determined.

We would like to emphasize that $d_{k-1,k-1} = -1$ is never being a component of a generic seed because of the following reason. Recall that the equations $\pa_m^\frak{b} (i, j) (\textup{\textbf{z}}) = 0$'s in~\eqref{pambijconvertedtoz} for $(i,j)$'s with $i + j \geq n$ and $i < j$ read
\begin{equation}\label{lemma1012recall}
\begin{cases}
\prod^{k}_{r=k} c^\textup{ver}_{r+1,r}&= - {z_{k-2,k}} + {(z_{k-1,k})^2} \left( 1 + \frac{1}{z_{k-1,k-1}} \right), \\
\prod^{k+1}_{r=k} c^\textup{ver}_{r+1,r} &= - {z_{k-3,k+1}} + {(z_{k-2,k+1})^2} \left( \left( \prod^{k}_{r=k} c^\textup{ver}_{r+1,r}\right)^{-1} + \frac{1}{z_{k-2,k}} \right), \\
&\cdots \\
\prod^{n-2}_{r=k} c^\textup{ver}_{r+1,r} &= - {z_{1,n-2}} + {(z_{2,n-2})^2} \left( \left( \prod^{n-3}_{r=k} c^\textup{ver}_{r+1,r}\right)^{-1} + \frac{1}{z_{2,n-3}} \right), \\
\prod^{n-1}_{r=k} c^\textup{ver}_{r+1,r} &= - {(z_{1,n-1})^2} \left( \left( \prod^{n-2}_{r=k} c^\textup{ver}_{r+1,r}\right)^{-1} + \frac{1}{z_{1,n-2}} \right).
\end{cases}
\end{equation}
If one chooses $z_{k-1,k-1} = d_{k-1, k-1} = -1$, then from~\eqref{lemma1012recall} we obtain
\begin{align*}
\prod^{j}_{r=k} c^\textup{ver}_{r+1,r} = - {z_{n-j-1,j}}
\end{align*}
for $j = k, k+1, \cdots, n-2$ and
\begin{equation}\label{problematicpamb}
\widetilde{\pa_m^\frak{b}}(1,n-1)(\textup{\textbf{z}})  = \prod^{n-1}_{r=k} c^\textup{ver}_{r+1,r} = 0.
\end{equation}
Thus, a seed $\textbf{d}$ is \emph{not} generic
\footnote{
We also have $\widetilde{\pa_m^\frak{b}}(n-1, 1)(\textup{\textbf{z}}) = 0$ if taking $z_{k-1,k-1} = d_{k-1, k-1} = -1$. 
}.
Nevertheless, we claim that there exists a choice of $d_{k-1, k-1}$ not equal to $-1$ but close to $-1$ so that~\eqref{pabijzneq0} is satisfied.

Note that the fixed $d_{k,k-1}$ remains to be pre-generic even if we perturb the value $z_{k-1, k-1}$ from $-1$ with sufficiently small amount. This is because the expression $\widetilde{\pa_m^\frak{b}}(i,j)(\textup{\textbf{z}})$ for each index $(i, j)$ with $i + j = n - 1$ is a continuous function with respect to $z_{k-1, k-1}$ at $-1$ after inserting $d_{k, k-1}$ into $z_{k, k-1}$. Also, by Lemma~\ref{rationalfunctionzij}, there exists a dense set of pre-generic elements for $d_{k-1, k-1}$. Therefore,~\eqref{pabijzneq0} is satisfied for $i + j = n -1$.

Also, we observe that as $z_{k-1, k-1} \to -1$, because of~\eqref{tildepabijz} and~\eqref{lemma1012recall}, $\widetilde{\pa_m^\frak{b}}(n-j,j)(\textup{\textbf{z}}) \to - z_{n-j-1,j}$ when $j \geq k$. Because $-z_{n -j-1, j} \neq 0$ for $j$ with $k \leq j < n - 1$, we still have $\widetilde{\pa_m^\frak{b}}(n-j,j)(\textup{\textbf{z}}) \neq 0$ for $j$ with $k \leq j < n - 1$ if $d_{k-1,k-1}$ is sufficiently close to $-1$. Finally, we claim that $\widetilde{\pa_m^\frak{b}}(1,n-1)(\textup{\textbf{z}}) \neq 0$ as soon as $z_{k-1, k-1} \neq -1$ so that the problem in~\eqref{problematicpamb} is solved. From ${\pa_m^\frak{b}}(k-1,k)(\textup{\textbf{z}}) = 0$ and $z_{k-1, k-1} \neq -1$, it follows that
$$
c^\textup{ver}_{k+1, k} \neq -z_{k-2, k}.
$$
Combining it with ${\pa_m^\frak{b}}(k-2,k+1)(\textup{\textbf{z}}) = 0$, we obtain
$$
\prod^{k+1}_{r=k} c^\textup{ver}_{r+1,r} \neq -z_{k-3, k+1}.
$$
Proceeding inductively, we deduce $\widetilde{\pa_m^\frak{b}}(1,n-1)(\textup{\textbf{z}}) \neq 0$. The discussion on the part where $j < k$ is omitted because the argument is symmetrical. 

Consequently, we may choose a generic $d_{k-1, k-1}$ sufficiently close to $-1$ so that~\eqref{pabijzneq0} holds for all $(i,j)$ with $i + j = n-1$ and $i + j = n$.

It remains to take care of the case where $k-1 = m$. The index $(k-1,k-1) = (m,m)$ is contained in the box $B(m)$ so that $d_{k-1,k-1}$ can be freely chosen by Corollary~\ref{corsymmetricsol}. Thus, we can apply the exactly same argument as above. 
\end{proof}

By applying a similar argument, we can prove the following lemma. 

\begin{lemma}\label{reductioniii+1i+1}
Suppose that either $d_{i, i} = \pm 1$ is pre-generic for $i > m$ or $i = m$. There is a real number $d_{i, i}$ (sufficiently close to $\pm 1$) and a non-zero real number $d_{i,i+1}$ so that $d_{i+1, i+1} = \mp 1$ becomes pre-generic.
\end{lemma}

We now ready to start the proof of Proposition~\ref{Existenceofgenericseeds} for the case where $n = 2k - 1$ and $m < k := \left\lceil n/2 \right\rceil$.

\begin{proof}[Proof of Proposition~\ref{Existenceofgenericseeds}] 
We start with a tentative choice of $d_{m,m} = \pm 1$. Choosing pre-generic elements from $d_{1,m+1} := z_{1,m+1}$ to $d_{m-1,m+1} := z_{m-1, m+1}$, we find $\textbf{\textup{z}}_{(2m \backslash m)}$ so that that \eqref{pabijzneq0} is satisfied for each index $(i,j)$ with $i + j \leq 2m - 1$. Due to Lemma~\ref{reductioniii+1i+1}, we may select $d_{m,m}$ sufficiently close to $\pm 1$ and $d_{m,m+1}$ so that $d_{m+1, m+1} = \mp 1$ becomes pre-generic. Because of Lemma~\ref{symmericseedz}, note that $d_{1,m+1}, \cdots, d_{m, m+1}$ remain to be pre-generic even if we choose another $d_{m,m}$. Moreover, applying Lemma~\ref{reductioniii+1i+1} repeatedly, we assert that $d_{k-1,k-1} = -1$ is also pre-generic by suitably choosing $d_{\bullet, \bullet}$. Hence, we have \eqref{pabijzneq0} for all indices $(i,j)$'s with $i + j \leq n -2$. Finally, Lemma~\ref{reductionkk-1k-1k-1} says that there is $d_{k-1, k-1}$ and $d_{k,k-1}$ such that~\eqref{pabijzneq0} holds for $i + j = n-1, n$. Thus, we have just found a generic seed.
\end{proof}

\bigskip

\noindent{\bf Case 2. $ n = 2k$ and $m < k$.}
\smallskip

Modifying the proofs of Lemma~\ref{reductionkk-1k-1k-1} and Lemma~\ref{reductioniii+1i+1}, we can prove the following lemma.

\begin{lemma}\label{evenreduction1}
Assume that $\textbf{\textup{z}}_{(n-2 \backslash m)}$  is given. Suppose that $d_{k-1,k} = -1$ is pre-generic. Then, there is a real number $d_{k-1,k}$ (sufficiently close to $-1$) and a non-zero real number $d_{k,k}$ such that if $z_{k-1,k} = d_{k-1,k}$ and $z_{k,k} = d_{k,k}$,
$$
\widetilde{\pa^\frak{b}_m}(i,j)(\textup{\textbf{z}}) \neq 0 \mod T^{>0}
$$
for all $(i,j)$ with $i + j = n-1$ and $i + j = n$.

Suppose that $d_{i-1, i} = \pm 1$ is pre-generic for $i \geq m+1$. There is a real number $d_{i-1, i}$ (sufficiently close to $\pm 1$) and a non-zero real number $d_{i,i}$ so that $d_{i, i+1} = \mp 1$ becomes pre-generic.
\end{lemma}

Also, we need the lemma, which serves as the starting point to obtain the desired $d_{\bullet,\bullet}$'s.

\begin{lemma}\label{dm+11canbegeneric}
$d_{m, m+1} = \pm 1$ can be pre-generic.
\end{lemma}

\begin{proof}
By Lemma~\ref{symmetricsol},
\begin{equation}
\displaystyle \widetilde{z}_{i,m+j} :=
\begin{cases}
1 \, &\mbox{for } i = j \\
\displaystyle \prod_{r=0}^{j - i - 1} (2i + 2r) \, &\mbox{for } i < j \\
\displaystyle \prod_{r=0}^{i - j - 1} (2j + 2r)^{-1} \, &\mbox{for } i > j
\end{cases}
\end{equation}
is a solution of $\pa_m^\frak{b} (i, m+j) (\textup{\textbf{z}}) = 0$ in ~\eqref{pambijconvertedtoz} for $m + i + j < n$. Furthermore, by Lemma~\ref{minussol}, so does
\begin{equation}\label{solutionofpamij2}
z_{i,m+j} := a \cdot \widetilde{z}_{i,m+j}
\end{equation}
for any non-zero complex number $a$. Selecting
$$
a := \prod^{m-2}_{r=0} (2 + 2r),
$$
$d_{m, m+1} = {z}_{m,m+1}$ becomes $1$. Because of Lemma~\ref{symmericseedz}, no matter what we choose any non-zero complex number $d_{m,m}$, $d_{m,m+1}$ is pre-generic (with respect to the previous determined $\textbf{z}_{(2m \backslash m)}$).
\end{proof}

\begin{proof}[Proof of Proposition~\ref{Existenceofgenericseeds} (continued)]
Combining Lemma~\ref{evenreduction1} and Lemma~\ref{dm+11canbegeneric}, we conclude Proposition~\ref{Existenceofgenericseeds} for the case where $n = 2k$ and $m < k$.
\end{proof}

\bigskip

\noindent{\bf Case 3. $ n = 2k$ and $m = k$.}
\smallskip

In this case, we take $d_{m,m} = 1$, see Remark~\ref{separatecasemk}. Because of Lemma~\ref{symmetricsol}, note that
\begin{equation}
\displaystyle \widetilde{z}_{i,m+j} :=
\begin{cases}
1 \, &\mbox{for } i = j \\
\displaystyle \prod_{r=0}^{j -i - 1} (2i + 2r) \, &\mbox{for } i < j \\
\displaystyle \prod_{r=0}^{i - j - 1} (2j + 2r)^{-1} \, &\mbox{for } i > j
\end{cases}, \,\,
\displaystyle \widetilde{z}_{m+i,j} :=
\begin{cases}
\displaystyle {(-1)^{m + i + j -1}} \, &\mbox{for } i = j \\
\displaystyle (-1)^{m + i + j -1} \prod_{r=0}^{i - j - 1} (2j + 2r)^{-1} \, &\mbox{for } i > j \\
\displaystyle (-1)^{m + i + j -1} \prod_{r=0}^{j - i - 1} (2i + 2r) \, &\mbox{for } i < j
\end{cases}
\end{equation}
respectively form a solution of $\pa_m^\frak{b} (i, m+j) (\textup{\textbf{z}}) = 0$ and $\pa_m^\frak{b} (m+i, j) (\textup{\textbf{z}}) = 0$ in ~\eqref{pambijconvertedtoz} for $m + i + j < n$. Also, our choice makes $\pa_m (l) (\textup{\textbf{y}}) = \pa_m (l) (\textup{\textbf{z}}) = 0$ in~\eqref{pamly} because $c^\textup{ver}_{m+1,m} = 1$ and $c^\textup{hor}_{m,m+1} = 1$, see Remark~\ref{whenm2kchoiceof} and Remark~\ref{separatecasemk}.

Furthermore,
\begin{equation}\label{solutionofpamij2}
z_{i,m+j} := a \cdot \widetilde{z}_{i,m+j}, \quad z_{m+i,j} := a^{-1} \cdot \widetilde{z}_{m+i,j}
\end{equation}
are also solutions of~\eqref{pambijconvertedtoz} and~\eqref{pamly} for any non-zero complex number $a$. Thus, we have a one-parameter family of solutions. Then, the expressions $\pa_m^\frak{b} (i, m+j) (\textup{\textbf{z}})$ and $\pa_m^\frak{b} (m+i, j) (\textup{\textbf{z}})$ for $(i,j)$ with $i + j = n$ can be considered as a function with respect to $a$.

\begin{lemma}\label{tempssedmk}
There exists a choice of the variable $a$ such that
\begin{equation}\label{genericamim-i}
\widetilde{\pa_m^\frak{b}} (m-i, m+i) (\textup{\textbf{z}}) \neq 0 \, \textup{ and } \, \widetilde{\pa_m^\frak{b}} (m+i, m-i) (\textup{\textbf{z}}) \neq 0 \mod T^{>0}
\end{equation}
in~\eqref{tildepabijz}.
\end{lemma}

\begin{proof} 
We claim that $\widetilde{\pa^\frak{b}_m}(m-i, m+i) (\textup{\textbf{z}}) / z_{m-i, m+i}$ is a non-constant rational function with respect to $a$. For $i \geq 1$, we observe that
$$
\frac{\widetilde{\pa^\frak{b}_m}(m-1, m+1)(\textup{\textbf{z}})}{z_{m-1, m+1}} = - \frac{z_{m-2,m+1}}{z_{m-1,m+1}} + z_{m-1,m+1} =  - (2m - 4) + a \cdot \left( \prod^{m-3}_{r = 0} \left( 2 + 2r \right)^{-1} \right)
$$
is a non-constant linear function with respect to $a$. By the induction hypothesis, assume that
$$
\frac{\widetilde{\pa^\frak{b}_m}(m-i, m+i)(\textup{\textbf{z}})}{z_{m-i, m+i}} := \frac{P_i (a)}{Q_i (a)}
$$
is a non-constant rational function with respect to $a$. Then, we see
\begin{align*}
\frac{\widetilde{\pa^\frak{b}_m}(m-i-1, m+i+1)(\textup{\textbf{z}})}{z_{m-i-1, m+i+1}} &= \left( - \frac{z_{m-i-2,m+i+1}}{z_{m-i-1,m+i+1}} + \frac{z_{m-i-1,m+i+1}}{z_{m-i-1,m+i}} \right) + \frac{z_{m-i-1,m+i+1}}{{\widetilde{\pa^\frak{b}_m}(m-i,m+i)(\textup{\textbf{z}})}} \\
&= \left( - \frac{\widetilde{z}_{m-i-2,m+i+1}}{\widetilde{z}_{m-i-1,m+i+1}} + \frac{\widetilde{z}_{m-i-1,m+i+1}}{\widetilde{z}_{m-i-1,m+i}} \right) + \frac{\widetilde{z}_{m-i-1,m+i+1}}{\widetilde{z}_{m-i,m+i}} \cdot \frac{Q_i(a)}{P_i(a)},
\end{align*}
which is also a non-constant rational function. Similarly, one can see that $\widetilde{\pa_m^\frak{b}} (m+i, m-i) (\textup{\textbf{z}})$ is also a non-constant rational function for $i \geq 1$. Thus,~\eqref{genericamim-i} is established if choosing $a$ generically.
\end{proof}

We are ready to prove Proposition~\ref{Existenceofgenericseeds} for the case where $n = 2k$ and $m = k :=  \left\lceil n/2 \right\rceil$.

\begin{proof}[Proof of Proposition~\ref{Existenceofgenericseeds} (continued)]
By Lemma~\ref{tempssedmk}, we choose $d_{i, m+1} := a \cdot \widetilde{z}_{i,m+1}$ from~\eqref{solutionofpamij2} as a generic seed. It complete the proof.
\end{proof}

%------------------------------------------------------------------------------------------
\vspace{0.2cm}
\subsection{Proof of Theorem~\ref{theoremD}}~
\vspace{0.2cm}

Finally, we are ready to prove the following theorem.

\begin{theorem}[Theorem \ref{theoremD}]\label{theorem_completeflagmancotinuum}
Let $\lambda = \{ \lambda_{i} := n - 2i + 1 \,|\, i = 1, \cdots, n \}$ be an $n$-tuple of real numbers for an arbitrary integer $n \geq 4$.
Consider the co-adjoint orbit $\mathcal{O}_\lambda$, a complete flag manifold $\mcal{F}(n)$ equipped with the monotone Kirillov-Kostant-Souriau symplectic form $\omega_\lambda$.
Then each Gelfand-Cetlin fiber $L_m(t)$ is non-displaceable Lagrangian for every $2 \leq m \leq \left\lfloor \frac{n}{2} \right\rfloor$. In particular, there exists a family of non-displaceable non-torus Lagrangian fibers
\[
	\left\{L_m(1) ~\colon~2 \leq m \leq \left\lfloor \frac{n}{2} \right\rfloor \right\}
\] 
of the Gelfand-Cetlin system $\Phi_\lambda$ where $L_m(1)$ is diffeomorphic to $U(m) \times T^{\frac{n(n-1)}{2} - m^2}$. 
\end{theorem}

\begin{proof}
By Corollary~\ref{splitleadingtermequationissolv}, the split leading term equation~\eqref{splitleadingtermequ} has a desired solution for some nonzero complex numbers $c^{\textup{ver},\C}_{i+1, i}$'s and $c^{\textup{hor}, \C}_{j, j+1}$'s ($i, j \geq k$). Theorem~\ref{splitleadingtermequationimpliesnonzero} convinces us that for each Lagrangian torus $L_m(t)$ ($0 \leq t < 1$), there exists a suitable bulk-deformation parameter $\frak{b}$ of the form~\eqref{bulkparak} so that the bulk-deformed potential function admits a critical point.  By Theorem~\ref{criticalpointimpliesnondisplaceability}, each GC torus fiber $L_m(t)$ for $0 \leq t < 1$ is non-displaceable. Furthermore, Corollary~\ref{corollary_L_fillable} and Lemma~\ref{closednessofnondisp} imply that $L_m(1)$ is Lagrangian and non-displaceable. Finally, $L_m(1)$ is diffeomorphic to $U(m) \times T^{\frac{n(n-1)}{2} - m^2}$ because of Theorem~\ref{theorem_contraction}. This finishes the proof of Theorem~\ref{theorem_completeflagmancotinuum}.
\end{proof}

%------------------------------------------------------------------------------------------
\section{Calculation of potential function deformed by Schubert cycles}\label{sec_bulkdefbyschucy}

The potential function \emph{with bulk} we use in the present paper was first constructed in
\cite[Section 3.8.5, 3.8.6]{FOOO} and was explicitly computed in \cite[Section 3]{FOOOToric2}
for the toric case. 
Since the main steps of the derivation of \eqref{bulkdeformedpotential} are the
same as that of the proof of \cite[Proposition 4.7]{FOOOToric2} given in Section 7 therein, we will only
explain modifications we need to make to apply them to the current GC case.
Also for the purpose of proving the counter part of Theorem \ref{Foootoric2potenti} in the present paper, the facts that
a Fano manifold $X_\epsilon$ has a toric degeneration and that we have only to consider codimension two cycles also help us to simplify the study of holomorphic discs contributing to the potential functions.
We closely follow \cite[Section 9]{NNU}. 

Let $L$ be a Lagrangian submanifold in a symplectic manifold $X$. Let $\mcal{M}_{k+1; \ell}(X, L; \beta)$ denote the moduli space of stable maps in the class $\beta \in \pi_2(X, L)$ from a bordered Riemann surface $\Sigma$ of genus zero with $(k+1)$ marked points $\{z_s\}_{s=0}^{k}$ on the boundary $\pa \Sigma$ respecting the counter-clockwise orientation and $\ell$ marked points $\{z^+_r \}_{r = 1}^{\ell}$ at the interior of $\Sigma$.
It naturally comes with two types of evaluation maps, at the $i$-th boundary marked point
\begin{equation}\label{evaluationatboundarymarked}
\ev_i \colon \mathcal M_{k+1;\ell}(X, L; \beta) \to L; \quad \ev_i \left( [ \varphi \colon \Sigma \to X, \{z_s\}_{s=0}^{k+1}, \{z^+_r \}_{r = 1}^{\ell}] \right) = \varphi(z_i)
\end{equation}
and at the $j$-th interior marked point
\begin{equation}\label{evaluationatinteriormarked}
\ev_j^{\text{\rm int}} \colon \mathcal M_{k+1;\ell}(X, L; \beta) \to X; \quad \ev_j^{\text{\rm int}} \left( [ \varphi \colon \Sigma \to X, \{z_s\}_{s=0}^{k+1}, \{z^+_r \}_{r = 1}^{\ell}] \right) = \varphi(z^+_j).
\end{equation}
Set $\mcal{M}_{k+1}(X, L; \beta) := \mcal{M}_{k+1; \ell = 0}(X, L; \beta)$, the moduli space without interior marked points and let 
$$
\bev_+ := (\ev_1, \cdots, \ev_k).
$$

Recall that an $A_\infty$-structure with the operators
$$
 \frak m_k = \sum_\beta \frak m_{k,\beta} \cdot T^{\omega(\beta)/  2\pi}, \quad  \frak m_{k,\beta}(b_1, \cdots, b_k): = (\ev_0)!(\bev_+)^*(\pi_1^*b_1 \otimes \cdots \otimes \pi_k^*b_k)
$$
on the de Rham complex $\Omega(L)$ is defined via a smooth correspondence.
\begin{equation}\label{smoothcorrespondence}
	\xymatrix{
                              & {\mcal{M}_{k+1}(X, L; \beta)} \ar[dl]_{\bev_+} \ar[dr]^{\ev_0} &
      \\
 L^k & & L}
\end{equation}
where $\pi_i \colon L^k \to L$ denotes the projection to the $i$-th copy of $L$. 
For a general symplectic manifold, one should choose a system of compatible Kuranishi structures and CF-perturbations on $\mcal{M}_{k+1,l}(X, L; \beta)$'s in order to apply the above smooth correspondence, see \cite{Fuk, FOOO:Kura1} for the details of construction. For a Lagrangian toric fiber $L$ in a $2n$-dimensional toric manifold, by constructing a system of compatible $T^n$-equivariant Kuranishi structures and multi-sections, the smooth correspondence can be applied without adding an auxiliary space for perturbing multi-sections to make it submersive, see \cite[Section 12]{FOOOToric2}. It is because the map $\ev_0$ automatically becomes submersive by the $T^n$-equivariance. 

We now recall Nishinou-Nohara-Ueda's computation of the potential function
of a torus fiber $L_\varepsilon \subset X_\varepsilon$ in \cite{NNU}.
They were able to exploit
the presence of toric degeneration of $X_\epsilon$ to $X_0$ in their computation
the explanation of which is now in order.
For the study of holomorphic discs in $X_0$ which is not smooth,
they used the following notion in Nishinou-Siebert \cite{NS}.

\begin{definition}[Definition 4.1 in \cite{NS}]\label{toricallytransverse}
A holomorphic curve in a toric variety $X$ is called \emph{torically transverse} if it is disjoint
from all toric strata of codimension greater than one. A stable map
$\varphi\colon \Sigma \to X$ is \emph{torically transverse} if $\varphi(\Sigma) \subset X$ is
torically transverse and $\varphi^{-1}(\operatorname{Int}X) \subset \Sigma$ is dense.
Here, $\operatorname{Int}X$ is the complement of the toric divisors in $X$.
\end{definition}

We denote by
$S_0: = \operatorname{Sing}(X_0)$ the singular locus of $X_0$.
Using the classification result \cite{CO} of holomorphic discs attached to
a Lagranigian toric fiber in a smooth toric manifold and the property of the small resolution,
Nishinou-Nohara-Ueda \cite{NNU} proved the following.

\begin{lemma}[Proposition 9.5 and Lemma 9.15 in \cite{NNU}] \label{discsinX0}
Any holomorphic disc $\varphi \colon (\mathbb{D}^2, \partial \mathbb{D}^2) \to (X_0,L_0)$ can be deformed into a holomorphic
disc with the same boundary condition that is torically transverse. Furthermore
the moduli space ${{\mathcal M}}_1(X_0,L_0;\beta)$ is empty if the Maslov
index of $\beta$ is less than two.
\end{lemma}

\begin{lemma}[Lemma 9.9 in \cite{NNU}]\label{nodiscinW0} There is a small neighborhood $W_0$ of
the singular locus $S_0 \subset X_0$ such that no holomorphic discs of Maslov index two intersect $W_0$. 
\end{lemma}

Now let $\phi^\prime_\epsilon \colon X_\epsilon \to X_0$ be a (continuous) extension of the flow
$\phi_\epsilon \colon {X}^\textup{sm}_\epsilon \to {X}^\textup{sm}_0$ given in Theorem~\ref{NNUtoricdeg} (\cite[Section 8]{NNU}).
The following is the key proposition which relates the above mentioned holomorphic discs
in $(X_0,L_0)$ to those of $(X_\epsilon,L_\epsilon)$.

\begin{proposition}[Proposition 9.16 in \cite{NNU}]\label{barM1}
For any $\beta \in \pi_2(X_0,L_0)$ of Maslov index two, there is a positive real numbers $0 < \varepsilon \leq 1$
and a diffeomorphism
$$
\psi \colon {{\mathcal M}}_1(X_0,L_0;\beta) \to {{\mathcal M}}_1 (X_\varepsilon,L_\varepsilon;\beta)
$$
such that the diagram
$$
\xymatrix{
H_*({{\mathcal M}}_1(X_0,L_0;\beta)) \ar[d]_{\psi_*} \ar[r]^{\quad \quad (\ev_0)_*} & H_*(L_0) \ar[d]^{(\phi_\varepsilon)^{-1}_*}\\
H_*({{\mathcal M}}_1(X_\varepsilon,L_\varepsilon;\beta)) \ar[r]^{\quad \quad (\ev_0)_*} & H_*(L_\varepsilon)
}
$$
is commutative.
\end{proposition}

\begin{lemma}[Lemma 9.22 in \cite{NNU}]\label{discsinXepsilon} Let $W_\varepsilon: = (\phi_\varepsilon')^{-1}(W_0)$.
There exists $\varepsilon_0 > 0$ such that
for all $0< \varepsilon \leq \varepsilon_0$, any holomorphic curve bounded by $L_\varepsilon$ in a class
of Maslov index two does not intersect $W_\varepsilon$. 
\end{lemma}

We now combine the diffeomorphism $\psi \colon {{\mathcal M}}_1(X_0,L_0;\beta) \to {{\mathcal M}}_1 (X_\varepsilon,L_\varepsilon;\beta)$
and $\phi_\varepsilon' \colon X_0 \to X_\varepsilon$ to define
an isomorphism between the correspondence 
\begin{equation}\label{smoothcorrespondence0}
	\xymatrix{
                              & {\mcal{M}_{k+1}(X_0, L_0; \beta)} \ar[dl]_{\bev_+} \ar[dr]^{\ev_0} &
      \\
 L_0^k & & L_0}
\end{equation}
and the following correspondence 
\begin{equation}\label{correspondence-epsilon}
	\xymatrix{
                              & {\mcal{M}_{k+1}(X_\varepsilon, L_\varepsilon; \beta)} \ar[dl]_{\bev_+} \ar[dr]^{\ev_0} &
      \\
 L_\varepsilon^k & & L_\varepsilon}.
\end{equation}
Although they did not explicitly mention a choice of compatible systems of Kuranishi structures or
perturbations, Nishinou-Nohara-Ueda \cite{NNU} essentially constructed an $A_\infty$-structure 
on $L_\varepsilon \subset X_\epsilon$ and computed its potential function 
in the same way as on a Fano toric manifolds \cite{CO,FOOOToric1} using Proposition \ref{barM1} and Lemma~\ref{discsinXepsilon}.
Thus, they were able to take advantage of properties of $T^n$-equivariant perturbation in a toric manifold, a open submanifold of a toric variety $X_0$. 
We denote the corresponding compatible system of multi-sections by $\frak s = \frak s_{k+1, \beta}$, 
see \cite{FOOOToric1, FOOOToric2} for the meaning of this notation. 

Next we need to involve bulk deformations for our purpose of a construction of continuum of 
non-displaceable Lagrangian tori in $X$, whose construction is now in order.

Denote $\scr{A}_{GS}^2(\Z)$ be the free abelian group generated
by the horizontal and vertical Schubert cycles of real codimension two 
\begin{equation}\label{eq:basis}
\{\scr{D}_{i,i+1}^\textup{hor} ~\colon~ 1 \leq i \leq n-1 \} \cup \{\scr{D}_{j+1,j}^\textup{ver} ~\colon~ 1 \leq j \leq n-1 \}.
\end{equation}
We recall
\begin{equation}\label{NOintersections}
L \cap \scr{D}_{i,i+1}^\textup{hor} = \emptyset = L \cap \scr{D}_{j+1,j}^\textup{ver}
\end{equation}
for any $i, j$ and so the cap product of $\beta \in \pi_2(X,L)$ with any element thereof is well-defined.
Putting $\scr{A}_{GS}^2(\Lambda_0) := \scr{A}_{GS}^2(\Z) \otimes \Lambda_0,$
any element $\frak b \in \scr{A}_{GS}^2(\Lambda_0)$ can be expressed as
\begin{equation}\label{eq:frakb}
\frak b = \sum_{i=1}^{n-1} \frak b_{i,i+1}^\textup{hor}\scr{D}_{i,i+1}^\textup{hor}
+ \sum_{j=1}^{n-1} \frak b_{j+1,j}^\textup{ver}\scr{D}_{j+1,j}^\textup{ver}
\end{equation}
where $\frak b_{i,i+1}^\textup{hor}, \frak{b}_{j+1,j}^\textup{ver} \in \Lambda_0$.
We formally denote
\begin{equation}\label{formallydenote}
\beta \cap \frak b = \sum_{i=1}^{n-1} \frak b_{i,i+1}^\textup{hor} \left( \beta \cap \scr{D}_{i,i+1}^\textup{hor} \right)
+ \sum_{j=1}^{n-1} \frak b_{j+1,j}^\textup{ver} \left( \beta \cap \scr{D}_{j+1,j}^\textup{ver} \right).
\end{equation}
For simplicity, let us fix an enumeration $\{\scr{D}_j \mid j = 1,\cdots, B\}$ of the elements in \eqref{eq:basis} where $B = 2(n-2)$ and set $\frak{b}_j$ to be the coefficient corresponding to $\scr{D}_j$ in~\eqref{eq:frakb}.

The following is the statement of the counterpart of Theorem \ref{Foootoric2potenti}.

\begin{theorem}\label{mainthmappdendix} Let $\frak b \in \scr{A}_{GS}^2(\Lambda_0)$ and let $L_\varepsilon$ be a torus Lagrangian fiber in $X_\varepsilon$. Then, the bulk-deformed potential function is written as
\begin{equation}\label{bulkdeformedpotential}
\frak{PO}^\frak{b} \left( L_\varepsilon ; b \right) = \sum_{\beta} n_\beta \cdot \exp \left( \beta \cap \frak{b} \right) \cdot \exp(\pa \beta \cap b) \, T^{\omega(\beta) / 2 \pi}.
\end{equation}
where the summation is taken over all homotopy classes in $\pi_2 (X_\varepsilon, L_\varepsilon)$ of Maslov index two.
\end{theorem}

The remaing part of this section is reserved for the proof of Theorem~\ref{mainthmappdendix}.

For a Lagrangian submanifold $L$ of $X$, denote by
$$
ev_i^{\text{\rm int}} \colon \mathcal M^{\text{\rm main}}_{k+1;\ell}(L,\beta) \to X
$$
the evaluation map at the $i$-th interior marked point for $i=1,\ldots,\ell$.
We put $\underline B = \{1,\ldots,B\}$ and denote the set of all maps $\text{\bf p}:
\{1,\ldots,\ell\} \to \underline B$ by
$Map(\ell,\underline B)$. We write $\vert\text{\bf
p}\vert = \ell$ if $\text{\bf p} \in Map(\ell,\underline B)$.
We define a fiber product
\begin{equation}
\mathcal M^{\text{\rm main}}_{k+1;\ell}(L,\beta;\text{\bf p})
= \mathcal M^{\text{\rm main}}_{k+1;\ell}(L,\beta)
{}_{(ev_1^{\text{int}},\ldots,ev_{\ell}^{\text{int}})}
\times_{X^{\ell}} \prod_{i=1}^{\ell} \scr{D}_{\text{\bf p}(i)}
\end{equation}
and consider the evaluation maps
$$
ev_i \colon \mathcal M^{\text{\rm main}}_{k+1;\ell}(L,\beta) \to L
$$
by
$$
ev_i((\Sigma,\varphi,\{z^+_i\}, \{z_i\})) = \varphi(z_i).
$$
It induces
$$
ev_i \colon \mathcal M^{\text{\rm main}}_{k+1;\ell}(L,\beta;\text{\bf p}) \to L.
$$

Note the image of a Schubert horizontal or vertical cycle is (a union of) components of the toric divisor via the map $\phi^\prime_\varepsilon$ so that $\scr{D}_j$ can be regarded as the (union of) components of toric divisors corresponding to Schubert horizontal or vertical cycles . Also, by Proposition \ref{barM1} and Lemma~\ref{discsinXepsilon}, the holomorphic discs of Maslov index two intersect at the smooth locus of divisors. For a toric fiber $L_0$ in $X_0$, there exists a system $\frak s = \{\frak s_{k+1,\beta;\textbf{\textup{p}}}\}$ of $T^n$-equivariant multi-sections on the moduli spaces $\mcal{M}_{k+1,l}(X_0, L_0; \beta; \textbf{\textup{p}})$ for all classes $\beta$ with $\mu(\beta) = 2$, see Lemma 6.5 in \cite{FOOOToric2}. 

As in the correspondence between~\eqref{smoothcorrespondence0} and~\eqref{correspondence-epsilon}, applying a smooth correspondence in  to
\begin{equation}\label{smoothcorrespondence}
	\xymatrix{
                              & {\mcal{M}_{k+1,l}(X_\varepsilon, L_\varepsilon; \beta; \textbf{\textup{p}})} \ar[dl]_{\bev_+} \ar[dr]^{\ev_0} &
      \\
 L^k & & L}
\end{equation}
we define
$$
\frak{q}_{k, \ell; \beta} \left( \textbf{p}; b^{\otimes k} \right) := \left(\ev_0\right)! \left(\bev_+^* (\pi_1^*b \otimes \cdots \otimes \pi_k^*b) \right).
$$
Since $X_\varepsilon$ is Fano, Lemma~\ref{discsinX0} and~\eqref{NOintersections} allows us to remain $\mathcal M_{k+1;\ell}(X_\varepsilon, L_\varepsilon,\beta;\text{\bf p})$ empty if one of the followings is satisfied.
\begin{equation}\label{threeconditionsforempty}
\begin{cases}
(1)\,\,  \mu(\beta) < 0, \\
(2)\,\, \mu(\beta) = 0 \,\, \mbox{and} \,\, \beta \neq 0, \\
(3)\,\, \beta = 0 \,\, \mbox{and} \,\, l > 0.
\end{cases}
\end{equation}
Because of the compatibility of the forgetful map forgetting the boundary marked points, see \cite[Section 5]{Fuk},  
\begin{equation}\label{divisoraxiboundarymarked}
\frak{q}_{k, \ell; \beta} \left( \textbf{p}; b^{\otimes k} \right) = \frac{1}{k!} (\pa \beta \cap b)^k \cdot \frak{q}_{0, \ell; \beta} \left( \textbf{p}; 1 \right).
\end{equation}
Since any moduli spaces satisfying one of the conditions in~\eqref{threeconditionsforempty} are empty, $\frak{q}_{0, \ell; \beta}(\textbf{p}; 1)$ represents a cycle, which yields that the Lagrangian $L_\varepsilon$ is weakly unobstructed with respect to $\frak{b}$ in~\eqref{eq:frakb}. Passing to the canonical model \cite{FOOO, FOOOc}, we obtain that
\begin{equation}\label{opengromovwittendef}
\frak{q}_{0, \ell; \beta}(\textbf{p}; 1) = n_{\beta} (\textbf{p}) \cdot \textup{PD}[L]
\end{equation}
for some $n_{\beta} (\textbf{p}) \in \Q$.  As a consequence, we obtain that every $1$-cochain is a weak bounding cochain with respect to $\frak{b}$. In particular, the potential function with bulk is defined on $H^1(L_\varepsilon; \Lambda_+)$.

Under our situation, $n_\beta (\textbf{p})$ is well-defined. Especially when $\dim \scr{D}_{\bullet} = 2n -2$, $\mu(\beta) = 2$, we recall that this is
precisely the situation where the divisor axiom of the Gromov-Witten theory applies, see \cite[p.193]{CK} and \cite[Lemma 9.2]{FOOOToric2}. In particualr, we can calculate $n_\beta (\textbf{p})$ in the homology level and therefore
\begin{equation}\label{divisoraxiomopengromov}
n_\beta (\textbf{p}) = n_\beta \cdot \prod^{| \textbf{p} |}_{i=1} \left(\beta \cap \scr{D}_{\textbf{p}(i)} \right).
\end{equation}
Following \cite[Section 7]{FOOOToric2} and using~\eqref{divisoraxiboundarymarked},~\eqref{opengromovwittendef},~\eqref{divisoraxiomopengromov} and~\eqref{formallydenote}, we obtain that
\begin{align*}
\sum_{k=0}^\infty \frak{m}^\frak{b}_k \left( b^{\otimes k} \right) &:= \sum_{k=0}^\infty \sum_{\ell=0}^\infty \sum_{\beta; \mu(\beta) = 2} \frac{1}{\ell!} \,\, \frak{q}_{k, \ell; \beta} \left(\frak{b}^{\otimes \ell}; b^{\otimes k} \right) T^{\omega(\beta) / 2 \pi} \\
&= \sum_{\ell=0}^\infty \sum_{\textbf{p}; |\textup{p} | = \ell} \sum_{\beta; \mu(\beta) = 2}  \exp \left( \pa \beta \cap b \right) \cdot \frac{1}{\ell!} \,\, \frak{b}^\textbf{p} \, \frak{q}_{0, \ell; \beta} \left(\textbf{p}; 1 \right) T^{\omega(\beta) / 2 \pi} \\
&= \sum_{\beta; \mu(\beta) = 2}  \left( \sum_{\ell=0}^\infty \sum_{\textbf{p}; |\textup{p} | = \ell}  \exp \left( \pa \beta \cap b \right) \cdot \frac{1}{\ell!} \,\, \frak{b}^\textbf{p} \, n_\beta(\textbf{p}) \right) T^{\omega(\beta) / 2 \pi} \cdot PD[L] \\
&= \sum_{\beta; \mu(\beta) = 2} n_\beta \cdot \left( \sum_{\ell=0}^\infty \frac{1}{\ell!}  \sum_{\textbf{p}; |\textup{p} | = \ell}  \,\, \prod^{| \textbf{p} |}_{i=1} \frak{b}_{\textbf{p}(i)} \left(  \beta \cap \scr{D}_{\textbf{p}(i)} \right)  \right) \cdot \exp(\pa \beta \cap b) \,\cdot T^{\omega(\beta) / 2 \pi}\cdot PD[L]\\
&= \sum_{\beta; \mu(\beta) = 2} n_\beta \cdot \exp \left( \beta \cap \frak{b} \right) \cdot \exp(\pa \beta \cap b) \,\cdot T^{\omega(\beta) / 2 \pi}\cdot PD[L]
\end{align*}
where $\frak{b}^\textbf{p} = \prod_{i=1}^\ell \frak{b}_{\textbf{p}(i)}$. Finally, incorporating with deformation of non-unitary flat line bundle by Cho \cite{Cho2}, we can extend the domain of the bulk-deformed potential function to $H^1(L_\varepsilon; \Lambda_0)$. This completes the proof of Theorem~\ref{mainthmappdendix}.

%------------------------------------------------------------------------
\bibliographystyle{annotation}

\end{document}